\documentclass[a4paper,12pt]{amsart}

\usepackage{amsthm, amsmath, amssymb, bbm}
\usepackage[margin=4cm]{geometry} 
\usepackage[T1]{fontenc}
\usepackage{mathrsfs}
\usepackage{tikz}

\usepackage{enumerate}
\usepackage{graphics}
\usepackage[all]{xy}
\usepackage{color}

\usepackage[colorlinks=true, pdfstartview=FitV, linkcolor=blue, %
citecolor=blue, urlcolor=blue]{hyperref}

\usepackage{marginnote}
\usepackage{ulem}
\newenvironment{rouge}
{\relax\color{red}}
{\hspace*{.3ex}\relax}
\newcommand{\ber}[1][{\mbox{$\bullet$}}]
{\begin{rouge}\marginnote{#1}}
\newcommand{\er}{\end{rouge}}

\newenvironment{bleu}
{\relax\color{blue}}
{\hspace*{.3ex}\relax}
\newcommand{\beb}{\begin{bleu}}
\newcommand{\eb}{\end{bleu}}

\addtocontents{toc}{\protect\setcounter{tocdepth}{1}}

\newcommand{\nc}{\newcommand}
\theoremstyle{plain}
\newtheorem{theorem}{Theorem}[subsection]
\newtheorem*{theorem*}{Theorem}
\newtheorem{corollary}[theorem]{Corollary}
\newtheorem{proposition}[theorem]{Proposition}
\newtheorem{lemma}[theorem]{Lemma}

\newtheorem{sublemma}[theorem]{Sublemma}

\theoremstyle{definition}
\newtheorem{definition}[theorem]{Definition}

\newtheorem*{example*}{Example}
\newtheorem{notation}[theorem]{Notation}
\newtheorem{remark}[theorem]{Remark}
\newtheorem{convention}[theorem]{Convention}

\renewcommand{\emptyset}{\varnothing}

\newcommand{\ooint}[1]{\left]{#1}\right[}

\newcommand{\point}{{\{\mathrm{pt}\}}}

\newcommand{\DUnion}{\bigsqcup\limits}
\newcommand{\dunion}{\sqcup}
\newcommand{\union}{\cup}
\newcommand{\Union}{\bigcup\limits}

\newcommand{\C}{\mathbb{C}}

\newcommand{\R}{\mathbb{R}}

\newcommand{\Z}{\mathbb{Z}}

\newcommand{\op}{\mathrm{op}}

\DeclareMathOperator{\id}{id}

\newcommand{\derd}{\mathsf{D}}
\newcommand{\dere}{\mathsf{E}}
\newcommand{\derr}{\mathsf{R}}
\newcommand{\derl}{\mathsf{L}}

\newcommand{\BDC}{\derd^{\mathrm{b}}}

\newcommand{\DSum}{\bigoplus}
\newcommand{\dsum}[1][]{\mathbin{\oplus_{#1}}}

\newcommand{\ilim}[1][]{\mathop{\varinjlim}\limits_{#1}}

\newcommand{\comp}{\mathbin{\circ}}

\renewcommand{\to}[1][]{\xrightarrow{#1}}
\newcommand{\from}[1][]{\xleftarrow{#1}}
\newcommand{\isofrom}[1][]{\xleftarrow[#1]%
{\raisebox{-.4ex}[0ex][-.4ex]{$\mspace{2mu}\sim\mspace{2mu}$}}}
\newcommand{\isoto}[1][]{\xrightarrow[#1]{%
{\raisebox{-.4ex}[0ex][-.4ex]{$\mspace{1mu}\sim\mspace{2mu}$}}}}

\newcommand{\Endo}[1][]{\mathrm{End}_{\raise1.5ex\hbox to.1em{}#1}}

\newcommand{\Hom}[1][]{\mathrm{Hom}_{\raise1.5ex\hbox to.1em{}#1}}

\newcommand{\RHom}[1][]{\derr\mathrm{Hom}_{\raise1.5ex\hbox to.1em{}#1}}

\newcommand{\Ext}[2][]{\mathrm{Ext}_{\raise1.5ex\hbox to.1em{}#1}^{#2}}

\newcommand{\Mod}{\mathrm{Mod}}

\newcommand{\Tens}[1][]{\mathbin{\otimes_{\raise1.5ex\hbox to-.1em{}#1}}}

\newcommand{\LTens}[1][]{\mathbin{\otimes_{\raise1.5ex\hbox to-.1em{}#1}^{\derl}}}

\newcommand{\Tor}[2][]{\mathrm{Tor}^{\raise1.5ex\hbox to.1em{}#1}_{#2}}

\newcommand{\sheaffont}[1]{\mathcal{#1}}
\def\sha{\sheaffont{A}}
\def\shb{\sheaffont{B}}
\def\shc{\sheaffont{C}}
\def\shd{\sheaffont{D}}
\def\she{\sheaffont{E}}

\def\shi{\sheaffont{I}}

\def\shk{\sheaffont{K}}
\def\shl{\sheaffont{L}}
\def\shm{\sheaffont{M}}
\def\shn{\sheaffont{N}}

\def\shs{\sheaffont{S}}
\def\sht{\sheaffont{T}}

\newcommand{\sect}{\varGamma}
\newcommand{\rsect}{\derr\varGamma}

\newcommand{\shendo}[1][]{{\sheaffont{E}nd}_{\raise1.5ex\hbox to.1em{}#1}}

\renewcommand{\hom}[1][]{{\sheaffont{H}om}_{\raise1.5ex\hbox to.1em{}#1}}
\newcommand{\aut}[1][]{{\sheaffont{A}ut}_{\raise1.5ex\hbox to.1em{}#1}}
\newcommand{\inn}[1][]{{\sheaffont{I}nn}_{\raise1.5ex\hbox to.1em{}#1}}

\newcommand{\rhom}[1][]{{\derr\sheaffont{H}om}_{\raise1.5ex\hbox to.1em{}#1}}

\newcommand{\ext}[2][]{{\sheaffont{E}xt}_{\raise1.5ex\hbox to.1em{}#1}^{#2}}

\newcommand{\thom}[1][]{{\sheaffont{T}hom}_{\raise1.5ex\hbox to.1em{}#1}}

\newcommand{\tens}[1][]{\mathbin{\otimes_{\raise1.5ex\hbox to-.1em{}#1}}}

\newcommand{\ltens}[1][]{\mathbin{\otimes_{\raise1.5ex\hbox to-.1em{}#1}^{\derl}}}

\newcommand{\tor}[2][]{{\sheaffont{T}or}^{\raise1.5ex\hbox to.1em{}#1}_{#2}}

\newcommand{\wtens}{\mathbin{\mathop{\otimes}\limits^{{}_{\mathrm{w}}}}}

\newcommand{\etens}[1][]{\mathbin{\boxtimes_{\raise1.5ex\hbox to-.1em{}#1}}}
\DeclareMathOperator{\supp}{supp}

\newcommand{\oim}[1]{#1_*}

\newcommand{\roim}[1]{\derr#1_*}
\newcommand{\roimv}[1]{\derr#1}
\newcommand{\reim}[1]{\derr#1_{\mspace{.5mu}!}\mspace{2mu}}
\newcommand{\reimv}[1]{\derr#1\mspace{2mu}}
\newcommand{\reeim}[1]{\derr#1_{\mspace{1mu}!!}\mspace{1mu}}
\newcommand{\reeimv}[1]{\derr#1\mspace{1mu}}

\newcommand{\opb}[1]{#1^{-1}}
\newcommand{\opbv}[1]{#1}

\newcommand{\epb}[1]{#1^{\mspace{1.5mu}!}\mspace{2mu}}
\newcommand{\epbv}[1]{#1\mspace{2mu}}

\DeclareMathOperator{\ori}{or}

\newcommand{\tenstop}[1][]{\mathbin{\hat{\otimes}_{\raise1.5ex\hbox to-.1em{}#1}}}

\newcommand{\homtop}[1][]{\sheaffont{L}_{\raise1.5ex\hbox to.1em{}#1}}

\newcommand{\Homtop}[1][]{\mathrm{L}_{\raise1.5ex\hbox to.1em{}#1}}

\newcommand{\D}{\sheaffont{D}}

\renewcommand{\O}{\sheaffont{O}}

\newcommand{\Db}{\sheaffont{D}b}

\DeclareMathOperator{\chv}{char}

\newcommand{\detens}[1][]%
{\mathbin{\boxtimes_{\raise1.5ex\hbox to-.1em{}#1}^{\mspace{2mu}\mathsf{D}}}}

\newcommand{\doim}[1]{{\mathsf{D}#1}_*\mspace{1mu}}
\newcommand{\doimv}[1]{{\mathsf{D}#1}}

\newcommand{\dopb}[1]{{\mathsf{D}#1}^{\mspace{1mu}*}}
\newcommand{\dopbv}[1]{{\mathsf{D}#1}}

\newcommand{\dtens}[1][]{\mathbin{\otimes_{\raise1.5ex\hbox to-.1em{}#1}^{\mathsf{D}}}}

\newcommand{\ddual}{\mathbb{D}}

\newcommand{\good}{\mathrm{good}}
\newcommand{\qgood}{\mathrm{q\text-good}}

\newcommand{\coh}{\mathrm{coh}}

\newcommand{\hol}{\mathrm{hol}}

\newcommand{\reghol}{{\mathrm{rh}}}

\newcommand{\Cfield}{\C}
\newcommand{\iCfield}{\ind\C}
\newcommand{\field}{\mathbf{k}}
\newcommand{\ind}{\mathrm{I}\mspace{2mu}}
\newcommand{\ifield}{\ind\field}

\newcommand{\Rc}{{\R\text-\mathrm{c}}}
\newcommand{\Cc}{{\C\text-\mathrm{c}}}

\newcommand{\reg}{{\operatorname{reg}}}

\newcommand{\shEnd}[1][]{\sheaffont{E}\mspace{-.5mu}nd_{#1}}

\newcommand{\ctens}{\mathbin{\mathop\otimes\limits^+}}
\newcommand{\cetens}{\mathbin{\mathop\boxtimes\limits^+}}
\newcommand{\cihom}{{\shi hom}^+}

\newcommand{\PR}{\mathsf{P}}

\newcommand{\PP}{\mathbb{P}}

\newcommand{\Ci}{\shc^{\infty}}

\newcommand{\dist}{\mathrm{dist}}
\newcommand{\sa}{\mathrm{sa}}

\newcommand{\tmp}{\mathsf{t}}
\newcommand{\dr}{\mathcal{DR}}
\newcommand{\drt}{\mathcal{DR}^\tmp}

\newcommand{\sol}{\mathcal Sol}

\newcommand{\solt}{\mathcal Sol^{\mspace{2.5mu}\tmp}}

\newcommand{\Dbt}{\Db^{\mspace{1mu}\tmp}}
\newcommand{\Dbvt}{\Db^{\mspace{1mu}\tmp,\vee}}

\newcommand{\Cit}{\shc^{\infty,\tmp}}
\newcommand{\Ot}{\O^{\mspace{2mu}\tmp}}
\newcommand{\Ovt}{\Omega^{\mspace{1mu}\tmp}}

\renewcommand{\Re}{\operatorname{Re}}
\renewcommand{\Im}{\operatorname{Im}}

\newcommand{\ihom}[1][]{{\shi hom}_{\raise1.5ex\hbox to.1em{}#1}}
\newcommand{\rihom}[1][]{{\derr\mspace{2mu}\shi hom}_{\raise1.5ex\hbox to.1em{}#1}}
\newcommand{\ii}[1][]{{\sheaffont{I}h}_{\raise1.5ex\hbox to.1em{}#1}}

\renewcommand{\ss}{\operatorname{sing.supp}}

\newcommand{\indlim}[1][]{\mathop{\text{\rm``$\varinjlim$''}}\limits_{#1}}
\newcommand{\prolim}[1][]{\mathop{\text{\rm``$\varprojlim$''}}\limits_{#1}}

\newcommand{\fb}{\mathfrak{b}}
\newcommand{\ft}{\mathfrak{t}}

\renewcommand{\comp}[1][]{\mathbin{\mathop{\circ}\limits_{#1}}}

\newcommand{\dcomp}[1][]{\mathbin{\circ_{\raise1.5ex\hbox to-.1em{}#1}^{\mathsf{D}}}}

\newcommand{\mop}{\mathrm{r}}

\newcommand{\enh}{\mathsf{E}}
\newcommand{\Tmp}{\mathsf{T}}
\newcommand{\OEn}{\O^\enh}
\newcommand{\DbT}{\Db^\Tmp}
\newcommand{\DbE}{\Db^\enh}
\newcommand{\OvE}{\Omega^\enh}
\newcommand{\drE}{\mathcal{DR}^\enh}
\newcommand{\solE}{\mathcal{S}ol^{\mspace{1mu}\enh}}

\newcommand{\fhom}{\mathcal{H}om^\enh}

\nc{\Bec}[2][{\mathrm{b}}]{\dere^{#1}(\ifield_{#2})}
\newcommand{\BEC}[2][\ifield]{\dere^{\mathrm{b}}(#1_{#2})}
\newcommand{\BECRc}[2][\ifield]{\dere^{\mathrm{b}}_\Rc(#1_{#2})}
\newcommand{\BECp}[2][\ifield]{\dere^{\mathrm{b}}_+(#1_{#2})}
\newcommand{\BECm}[2][\ifield]{\dere^{\mathrm{b}}_-(#1_{#2})}
\newcommand{\BECpm}[2][\ifield]{\dere^{\mathrm{b}}_\pm(#1_{#2})}

\newcommand{\Edual}{\dual^\enh}
\newcommand{\Eoim}[1]{{\enh#1}_*}
\newcommand{\Eoimv}[1]{{\enh#1}}
\newcommand{\Eeeim}[1]{{\enh#1}_{!!}}
\newcommand{\Eeeimv}[1]{{\enh#1}}
\newcommand{\Eopb}[1]{{\enh#1}^{-1}}
\newcommand{\Eopbv}[1]{{\enh#1}}
\newcommand{\Eepb}[1]{{\enh\mspace{1mu}#1}^{\mspace{1.5mu}!}}
\newcommand{\Eepbv}[1]{{\enh\mspace{1mu}#1}}

\newcommand{\LE}{\operatorname{L^\enh}}
\newcommand{\RE}{\operatorname{R^\enh}}

\newcommand{\bL}{\check L}
\newcommand{\bM}{\check M}
\newcommand{\bN}{\check N}

\newcommand{\tot}{{\operatorname{tot}}}
\newcommand{\sing}{{\operatorname{sing}}}
\newcommand{\suban}{{\operatorname{suban}}}

\newcommand{\semicolon}{\nobreak \mskip2mu\mathpunct{}\nonscript\mkern-\thinmuskip{;}\mskip6mu plus1mu\relax}

\newcommand{\dual}{\mathrm{D}}

\newcommand{\sep}{\mspace{2mu}}

\newenvironment{myequation}
{\relax\setlength{\arraycolsep}{1pt}\begin{eqnarray}}
{\end{eqnarray}}
\newenvironment{myequationn}
{\relax\setlength{\arraycolsep}{1pt}\begin{eqnarray*}}
{\end{eqnarray*}}

\nc{\eq}{\begin{myequation}}
\nc{\eneq}{\end{myequation}}
\nc{\eqn}{\begin{myequationn}}
\nc{\eneqn}{\end{myequationn}}

\newcommand{\defeq}{\mathbin{:=}}

\newcommand{\bl}{\bigl(}
\newcommand{\br}{\bigr)}
\newcommand{\To}[1][]{\xrightarrow[]{\mspace{10mu}{#1}\mspace{10mu}}}

\newenvironment{myarray}[1]{\relax\setlength{\arraycolsep}{1pt}

\begin{array}{#1}}{\end{array}\relax}

\newcommand{\ba}{\begin{myarray}}
\newcommand{\ea}{\end{myarray}}

\newcommand{\db}[1]{\raisebox{-.3ex}[1.5ex][1ex]{$#1$}}
\newcommand{\hs}{\hspace*}

\newcommand{\be}{\begin{enumerate}}
\newcommand{\ee}{\end{enumerate}}
\newcommand{\bnum}{\be[{\rm(i)}]}
\newcommand{\bna}{\be[{\rm(a)}]}
\nc{\bwr}{\mbox{\large{$\wr$}}}
\nc{\vphi}{\varphi}
\nc{\Proof}{\begin{proof}}
\nc{\QED}{\end{proof}}
\nc{\Cor}{\begin{corollary}}
\nc{\encor}{\end{corollary}}
\nc{\ol}{\overline}
\nc{\vs}{\vspace*}
\nc{\monoto}{\rightarrowtail}
\nc{\tone}{\To[+1]}
\nc{\set}[2]{\left\{#1 \hs{.3ex};\hs{.3ex}#2\right\}}
\nc{\tp}{\tilde{p}}
\nc{\cl}{\colon}
\nc{\cor}{\field}
\newcommand{\soplus}{\mathop{\scalebox{0.8}{$\displaystyle\bigoplus$}}\limits}
\begin{document}

\title[Riemann-Hilbert correspondence]%
{Riemann-Hilbert correspondence for holonomic D-modules}
\author[A.~D'Agnolo]{Andrea D'Agnolo}
\address{Dipartimento di Matematica\\
Universit{\`a} di Padova\\
via Trieste 63, 35121 Padova, Italy}
\email{dagnolo@math.unipd.it}
\thanks{The first author was partially supported by
grant CPDA122824/12, Padova University, and by the project ``Differential
methods in Arithmetic, Geometry and Algebra'', Fondazione Cariparo.}

\author[M.~Kashiwara]{Masaki Kashiwara} 

\thanks{The second author was partially supported by Grant-in-Aid for
Scientific Research (B) 22340005, Japan Society for the Promotion of
Science.}
\address{Research Institute for Mathematical Sciences\\
Kyoto University\\ 
Kyoto 606-8502, Japan
}
\email{masaki@kurims.kyoto-u.ac.jp}


\date{August 12, 2015}

\begin{abstract}
The classical Riemann-Hilbert correspondence establishes an equivalence between the triangulated category of regular holonomic $\D$-modules and that of constructible sheaves.

In this paper, we prove a Riemann-Hilbert correspondence for holonomic $\D$-modules which are not necessarily regular. The construction of our target category is based on the theory of ind-sheaves by Kashiwara-Schapira and influenced by Tamarkin's work.
Among the main ingredients of our proof is  the description of the structure of flat meromorphic connections due to Mochizuki and Kedlaya.
\end{abstract}

\maketitle

\tableofcontents

\section{Introduction}
\numberwithin{equation}{section}

\subsection{}
On a complex manifold, the classical Riemann-Hilbert correspondence establishes an equivalence between the triangulated category of regular holonomic $\D$-modules and that of constructible sheaves (see~\cite{Kas84}).
Here $\D$ denotes the sheaf of differential operators.

In particular, flat meromorphic connections with regular singularities correspond to local systems on the complementary of the singular locus (see~\cite{Del70}).

\subsection{}
The problem of extending the Riemann-Hilbert correspondence to cover the case of holonomic $\D$-modules with irregular singularities has been open for 30 years. Some results in this direction have appeared in the literature.

In the one-dimensional case, classical results of Levelt-Turittin and Hukuhara-Turittin describe the formal structure and the asymptotic expansion on sectors of flat meromorphic connections which are not necessarily regular.
Using these descriptions, Deligne and Malgrange established a Riemann-Hilbert correspondence on a complex curve for holonomic $\D$-modules with a fixed set of singular points (see~\cite{DMR07}). See also the work of Babbitt-Varadarajan~\cite{BB89}.

Recently, Mochizuki~\cite{Moc09,Moc11}
and Kedlaya~\cite{Ked10,Ked11} extended the results of Levelt-Turittin and Hukuhara-Turittin to higher dimensions. Namely, they proved that
any flat meromorphic connection becomes ``good'' after
blowing-ups. Sabbah~\cite{Sab13}
obtained an analogue of the construction by Deligne and Malgrange 
on a complex manifold for ``good'' flat meromorphic connections 
with a fixed singular locus.

\subsection{}
In this paper, we prove a Riemann-Hilbert correspondence for holonomic $\D$-modules on a complex manifold. 
The construction of our target category is based 
on the theory of ind-sheaves by Kashiwara-Schapira~\cite{KS01} 
and influenced by the work of Tamarkin~\cite{Tam08}.
The description of the structure of flat meromorphic connections by
Mochizuki and Kedlaya is one of the key ingredients of our proof.

Let us explain our results in greater detail.

\subsection{}
Let $X$ be a complex manifold.
As we have already mentioned, the Riemann-Hilbert correspondence of~\cite{Kas84} establishes an equivalence
between the triangulated category $\BDC_\reghol(\D_X)$
of regular holonomic $\D_X$-modules and
the triangulated category $\BDC_{\Cc}(\Cfield_X)$
of $\C$-constructible sheaves on $X$.
More precisely, there are functors
\begin{equation}
\label{eq:introRHclass}
\xymatrix@C=10ex{
\BDC_\reghol(\D_X) \ar@<.5ex>[r]^-{\dr_X} 
& \BDC_{\Cc}(\Cfield_X) \ar@<.5ex>[l]^-{\Psi_X}
}
\end{equation}
quasi-inverse to each other.
Here, $\dr_X(\shl) \defeq \Omega_X \ltens[\D_X] \shl$ is the holomorphic de Rham complex with $\Omega_X$ the sheaf of holomorphic differential forms of highest degree,
and $\Psi_X(L) \defeq \thom(\dual_X L,\O_X)[d_X]$ is the complex of holomorphic functions tempered along the dual $\dual_X L$ of $L$.

In particular, a regular holonomic $\D_X$-module $\shl$ can be reconstructed from $\dr_X(\shl)$.

Let $\shm$ be an irregular holonomic $\D_X$-module, and consider the regular holonomic $\D_X$-module $\shm_\reg \defeq \Psi_X(\dr_X(\shm))$. Since $\dr_X(\shm) \simeq \dr_X(\shm_\reg)$, it follows that $\shm$ cannot be reconstructed from $\dr_X(\shm)$.

\subsection{}
The theory of ind-sheaves, that is, of ind-objects in the category of sheaves with compact support, was initiated and developed by Kashiwara-Schapira~\cite{KS01}. In such a framework, one can consider the complex $\Ot_X$ of tempered holomorphic functions, which is an object of the derived category of ind-sheaves $\BDC(\iCfield_X)$.
This is related to the functor $\Psi_X$ in \eqref{eq:introRHclass}, since one has $\rhom(F,\Ot_X) \simeq \thom(F,\O_X)$ for any $\R$-constructible sheaf $F$.

Set $\Ovt_X = \Omega_X\ltens[\O_X]\Ot_X$.
For a holonomic $\D_X$-module $\shm$, the tempered de Rham complex $\drt_X(\shm) \defeq \Ovt_X\ltens[\D_X]\shm$ has been introduced and studied in \cite{KS03} and studied further in~\cite{Mor07,Mor10}.
This complex retains some information on the irregularity of $\shm$.
For example, let $\varphi\in\O_X(*Y)$ be a meromorphic function with poles at a hypersurface $Y$, and denote by $\she^\varphi_{X\setminus Y|X}$ the exponential $\D_X$-module generated by $e^\varphi$ (see Definition~\ref{def:expY}).
Then one has
\begin{equation}
\label{eq:drtintro}
\drt_X(\she^\varphi_{X\setminus Y|X}) \simeq \rihom(\Cfield_{X\setminus Y}, \indlim[a\rightarrow+\infty]\Cfield_{\{x\in X\setminus Y\semicolon -\Re \varphi(x) < a\}})[\dim X],
\end{equation}
where $\ihom$ denotes the inner-hom functor of ind-sheaves and $\Cfield_{X\setminus Y}$ denotes the extension by zero to $X$ of the constant sheaf on $X\setminus Y$.

Since $\drt_X(\she^\varphi_{X\setminus Y|X}) \simeq \drt_X(\she^{2\varphi}_{X\setminus Y|X})$, one cannot reconstruct $\shm$ from $\drt_X(\shm)$.

\subsection{}
Denote by $\tau\in\C\subset \PP$ the affine variable in the complex projective line $\PP$.
In this paper, we will show that $\shm$ can be reconstructed from the tempered de Rham complex $\drt_{X\times \PP}(\shm \detens \she_{\C|\PP}^{-\tau})$, an object of $\BDC(\iCfield_{X\times\PP})$.
In the case where $X$ is a complex curve, we outlined a proof of this fact in \cite{DK12}.
The proof in the general case follows from the arguments in the present paper.
However, in this paper we take as target category a modification of $\BDC(\iCfield_{X\times\PP})$. As we now explain, this is related to a construction by Tamarkin~\cite{Tam08} (see also Guillermou-Schapira~\cite{GS12} for an exposition and some complementary results).

\subsection{}\label{sse:introTamSheaf}
On a real analytic manifold $M$, the microlocal theory of sheaves by Kashiwara-Schapira~\cite{KS90} associates to an object of $\BDC(\Cfield_M)$ its microsupport, a closed conic involutive subset of the cotangent bundle $T^*M$.
In his study of symplectic topology, Tamarkin~\cite{Tam08} uses the techniques of~\cite{KS90} in order to treat involutive subsets of $T^*M$ which are not necessarily conic. To this end, he adds a real variable $t\in\R$ and, denoting by $(t;t^*)\in T^*\R$ the associated symplectic coordinates, considers the quotient category $\BDC(\Cfield_{M\times\R})/\shc_{\{t^*\leq 0\}}$ by the 
category $\shc_{\{t^*\leq 0\}}$ consisting of
objects microsupported on $\{t^*\leq 0\}$. 

An important observation in  \cite{Tam08} is that there are equivalences
\eq
&&{}^\bot\shc_{\{t^*\leq 0\}} \simeq \BDC(\Cfield_{M\times\R})/\shc_{\{t^*\leq 0\}} \simeq\shc_{\{t^*\leq 0\}}^\bot
\label{eq:CC}
\eneq
between the quotient category and the left and right orthogonal categories.
Moreover, such categories can be described without using the notion of microsupport. For example, $\shc_{\{t^*\leq 0\}}$ is the full subcategory of $\BDC(\Cfield_{M\times\R})$ of objects whose convolution with $\Cfield_{\{t \geq 0\}}$ vanishes.

\subsection{}\label{sse:introTam}
Back to our complex manifold $X$,
recall that we aim to reconstruct a holonomic $\D_X$-module $\shm$ 
from the tempered de Rham complex $\drt_{X\times \PP}(\shm \detens \she_{\C|\PP}^{-\tau})$. As we explain in \S\ref{sse:introMoch} below, a special important case is when $\shm = \she^\varphi_{X\setminus Y|X}$ for $\varphi\in\O_X(*Y)$. 
Then, \eqref{eq:drtintro} implies that the tempered de Rham complex is 
described in terms of the ind-sheaf
\begin{equation}
\label{eq:introt-phi}
\indlim[a\rightarrow+\infty]\Cfield_{\{(x,\tau)\in (X\setminus Y)\times\C\semicolon t-\Re \varphi(x) < a\}}.
\end{equation}
Here $t=\Re\tau$ is the real part of the affine coordinate $\tau$ of 
the complex projective line $\PP$.
We are thus led to replace the target category $\BDC(\iCfield_{X\times\PP})$ with 
what we call the category of enhanced ind-sheaves 
and denote by $\BEC[\iCfield]X$.
This is a quotient category of $\BDC(\iCfield_{X\times\PR})$, where $\PR$ is the real projective line.

Let us describe the category $\BEC[\iCfield]X$ in greater detail.

\subsection{}
As a preliminary step, we introduce the notion of bordered space.
A bordered space is a pair $(M,\bM)$ of a topological space $\bM$ and an open subset $M\subset\bM$, and we associate the triangulated category $\BDC(\iCfield_{(M,\bM)}) \defeq \BDC(\iCfield_{\bM})/\BDC(\iCfield_{\bM\setminus M})$ to it. 
There is a  natural fully faithful embedding
\[
\BDC(\Cfield_M)  \subset \BDC(\iCfield_{(M,\bM)})  .
\]

The main example for us is the bordered space $\R_\infty \defeq (\R,\PR)$.
This notion  appears naturally when we deal with ind-sheaves such  
as \eqref{eq:introt-phi}.
For example, for $\varphi=0$ such an ind-sheaf becomes trivial when restricted to $\BDC(\iCfield_{X\times\R})$, but is a non trivial object of $\BDC(\iCfield_{X\times\R_\infty})$.

\subsection{}\label{sse:TDC}
We define the category $\BEC[\iCfield]X$ of enhanced ind-sheaves by
\[
\BEC[\iCfield]X = \BDC(\iCfield_{X\times\R_\infty})/\{K \semicolon K\simeq\opb\pi L \text{ for some } L\in \BDC(\iCfield_X) \}.
\]
Here $\pi\colon X\times\R_\infty\to X$ is the projection.
This is related with Tamarkin's construction as follows. We set
\[
\BECp[\iCfield]X \defeq \BDC(\iCfield_{X\times\R_\infty})/\ind\shc_{\{t^*\leq 0\}},
\]
where $\ind\shc_{\{t^*\leq 0\}}$ is the full subcategory of objects whose convolution with $\Cfield_{\{t\geq 0\}}$ vanishes. 
As in \eqref{eq:CC}, we have
\[
{}^\bot\ind\shc_{\{t^*\leq 0\}} \simeq \BECp[\iCfield]X \simeq \ind\shc_{\{t^*\leq 0\}}^\bot.
\]
Replacing $\Cfield_{\{t\geq 0\}}$ with $\Cfield_{\{t\leq 0\}}$ one obtains the category $\BECm[\iCfield]X$. It turns out that
\[
\BEC[\iCfield]X \simeq \BECp[\iCfield]X \dsum \BECm[\iCfield]X.
\]
This is the target category of our Riemann-Hilbert correspondence.
It is a triangulated tensor category whose tensor product is given 
by the convolution $\ctens$  in the $t$ variable.

\subsection{}
Set $\Cfield_X^\enh \defeq \indlim[a\rightarrow+\infty]\Cfield_{\{t\geq a\}}$.
We say that an object $K$ of $\BEC[\iCfield]X$ is stable if $K\simeq \Cfield_X^\enh\ctens K$. 

There is a natural fully faithful embedding of the category of
ind-sheaves into the category of stable enhanced ind-sheaves
\[
e\colon\BDC(\iCfield_X) \To \BEC[\iCfield]X, \quad F\mapsto \Cfield_X^\enh \tens\opb\pi F.
\]

Denote by $\BDC_\Rc(\Cfield_{X\times\PR})$ the full subcategory of $\BDC(\Cfield_{X\times\PR})$ whose objects have $\R$-constructible cohomology groups.
We say that an object $K$ of $\BEC[\iCfield]X$ is $\R$-constructible if,
for any relatively compact subanalytic open subset $U\subset X$,
there exists $F\in \BDC_\Rc(\Cfield_{X\times\PR})$ such that
\[
\opb\pi\Cfield_U \tens K \simeq \Cfield_X^\enh \ctens F.
\]
Note that such a $K$ is a stable object, and that $\R$-constructibility is a local property on $X$.
We denote by $\BECRc[\iCfield]X$ the full subcategory of
$\BEC[\iCfield]X$ consisting of $\R$-constructible objects.
 
\subsection{}\label{sse:introOT}
We can now state our Riemann-Hilbert correspondence.

The objects of $\BEC[\iCfield]X$ which play a role analogous to the objects $\Ot_X$ and $\Ovt_X$ of $\BDC(\iCfield_X)$ are 
\[
\OEn_X \defeq \epb i\rhom[\D_\PP](\she_{\C|\PP}^\tau,\Ot_{X\times\PP})[2],
\quad \OvE_X \defeq \Omega_X \ltens[\O_X] \OEn_X.
\]
where $i\colon X\times\R_\infty\to X\times\PP$ is the embedding.
It turns out that $\OEn_X$ and $\OvE_X$ are stable objects endowed with a natural $\D_X$-module structure.

\smallskip

Denote by $\BDC_\hol(\D_X)$ the full
subcategory of $\BDC(\D_X)$ consisting of
objects with holonomic cohomologies.
We define the enhanced de Rham functor
\[
\drE_X\colon\BDC_\hol(\D_X)\To \BEC[\iCfield]X,\quad \shm\mapsto \OvE_X \ltens[\D_X] \shm
\]
and the reconstruction functor
\[
\Psi_X^\enh\colon\BECRc[\iCfield]X\To\BDC(\D_X),\quad K\mapsto \fhom(\Edual_X K,\OEn_X)[d_X],
\]
where $\fhom$ is the hom-functor between enhanced ind-sheaves, 
with values in sheaves on $X$, and $\Edual_X$ is a natural duality functor of $\BECRc[\iCfield]X$.

Our main result can be stated as follows. 
\begin{theorem*}
\bnum
\item The functor $\drE_X$ is fully faithful and takes values in $\BECRc[\iCfield]X$.
\item there is an isomorphism
$$\shm\isoto\Psi_X^\enh\bl\drE_X(\shm)\br$$ 
functorial in $\shm\in \BDC_\hol(\D_X)$.
In particular, one can reconstruct $\shm$ from $\drE_X(\shm)$.
\ee
\end{theorem*}

We prove the compatibility of $\drE_X$ with duality. We also prove compatibility with the classical Riemann-Hilbert correspondence \eqref{eq:introRHclass}. 
More precisely, there is a quasi-commutative diagram: 
\[
\xymatrix@C=8ex{
\BDC_\reghol(\D_X) \ar[r]^{\dr_X} \ar[d]
& \BDC_\Cc(\Cfield_X) \ar[r]^{\Psi_X} \ar[d]^e
& \BDC_\reghol(\D_X) \ar[d] \\
\BDC_\hol(\D_X) \ar[r]^{\drE_X} 
& \BECRc[\iCfield]X \ar[r]^{\Psi_X^\enh}
& \BDC(\D_X) .
}
\]

\subsection{}\label{sse:introMoch}
A key ingredient in our proofs is the following (see Lemma~\ref{lem:redux}).
Let $P_X(\shm)$ be a statement concerning a complex manifold $X$ and a holonomic $\D_X$-module $\shm$. 
For example,
\[
P_X(\shm)= \text{``$\shm \isoto \Psi_X^\enh(\drE_X(\shm))$''}.
\]
In order to prove $P_X(\shm)$, 
the results of Mochizuki~\cite{Moc09,Moc11} and
Kedlaya~\cite{Ked10,Ked11} allow one, heuristically speaking, to reduce to the case 
when  $\shm = \she^\varphi_{X\setminus Y|X}$ for $\varphi\in\O_X(*Y)$.

\subsection{}
Recall that irregular holonomic modules are subjected to the Stokes phenomenon. In \S\ref{sse:Stokes} we describe
with an example how the Stokes data are encoded topologically in our construction.

\subsection{}
The contents of this paper are as follows.

Section~\ref{se:notations} fixes notations regarding sheaves, ind-sheaves and $\D$-modules. References are made to \cite{KS90,KS01,Kas03}. We also state some complementary results which are of use in later sections.

In Section~\ref{se:bordered}, we introduce the notion of bordered space and of ind-sheaves on it, and develop the formalism of operations in this context.
We also discuss a natural $t$-structure in the triangulated category of ind-sheaves on a bordered space.

In Section~\ref{se:enhcdind}, we introduce the category $\BEC[\iCfield]M$  of enhanced ind-sheaves, mentioned in \S\ref{sse:TDC}, and develop the formalism of operations in this framework.  We also introduce the notion of $\R$-constructible objects in $\BEC[\iCfield]M$.

Section~\ref{se:tempered} recalls from \cite{KS96,KS01} the construction and main properties of the ind-sheaves of tempered distributions $\Dbt_M$ on a real analytic manifold $M$, and of tempered holomorphic functions $\Ot_X$ on a complex manifold $X$. As explained above, this is a fundamental ingredient of our construction.

In Section~\ref{se:expo}, we prove the isomorphism \eqref{eq:drtintro}. The fundamental example where $X=\C\owns z$ and $\varphi(z)=1/z$ 
has been already treated in \cite{KS03}.

Mochizuki and Kedlaya's results  
on the structure of flat meromorphic connections are recalled in Section~\ref{se:normal}. There, we give a precise formulation of the heuristic argument mentioned in \S\ref{sse:introMoch}.

Section~\ref{se:enhanced} introduces and studies the enhancement $\OEn_X$ of $\Ot_X$ mentioned in \S\ref{sse:introOT}, along with the enhancement $\DbE_M$ of $\Dbt_M$.

Our main results, mentioned in \S\ref{sse:introOT}, are stated and proved in Section~\ref{se:RH}.

\subsection*{Acknowledgments} 
We thank Pierre Schapira, who taught us that the ind-sheaf $\Ot_X$ of tempered holomorphic functions is an appropriate language for the study of irregular holonomic $\D$-modules.

We also thank Takuro Mochizuki for his explanations on the structure of irregular holonomic $\D$-modules.

The first author acknowledges the kind hospitality at RIMS, 
Kyoto University, during the preparation of this paper.

Finally, we wish to thank the anonymous referee for his/her careful  
reading of our manuscript and his/her suggestions to simplify the proof of  
Proposition~\ref{pro:summand}.

\numberwithin{equation}{subsection}

\section{Notations and complements}\label{se:notations}

We fix here some notations regarding sheaves, ind-sheaves and $\D$-modules, and state some complementary results that we will need in later sections.
Our notations follow those in \cite{KS90,KS01,Kas03}, to which we refer for further detail.

\medskip
Let us say that a topological space is \emph{good} if it is Hausdorff, locally compact, countable at infinity and has finite flabby dimension.

In this paper, we take a field $\field$ as base ring. However, after minor modifications, one can take any regular ring as base ring.

For a category $\shc$, we denote by $\shc^\op$ 
the opposite category of $\shc$.
For a ring $A$, we denote by $A^\op$ the opposite ring of $A$.

\subsection{Sheaves}

Let $M$ be a good topological space.
Denote by $\Mod(\field_M)$ the abelian category of sheaves of $\field$-vector spaces on $M$, and by $\BDC(\field_M)$ its bounded derived category.

For a locally closed subset $S\subset M$, denote by $\field_S$ the 
extension by zero to $M$ of the constant sheaf on $S$.

For $f\colon M\to N$ a morphism of good topological spaces,
denote by $\tens$, $\rhom$, $\opb f$, $\roim f$, $\reim f$, $\epb f$ the six
Grothendieck operations for sheaves.
Denote by $\etens$ the exterior product.

We define the duality functor $\dual_M$ of $\BDC(\field_M)$ by
\[
\dual_M F = \rhom(F,\omega_M) \quad\text{for $F\in\BDC(\field_M)$,}
\]
where $\omega_M$ denotes the dualizing complex. If $M$ is a $C^0$-manifold of dimension $d_M$, one has $\omega_M \simeq \ori_M[d_M]$, where $\ori_M$ denotes the orientation sheaf.

\subsection{Ind-sheaves}

The theory of ind-sheaves has been introduced and developed in~\cite{KS01}.

\smallskip 

Let $\shc$ be a category and denote by $\shc^\wedge$ the category of contravariant functors from $\shc$ to the category of sets. Consider the Yoneda embedding $h\colon \shc \to \shc^\wedge$, $X\mapsto\Hom[\shc](\ast,X)$.
The category $\shc^\wedge$ admits small inductive limits. Since $h$ does not commute with inductive limits, one denotes by $\indlim$ instead of $\varinjlim$ the inductive limits taken in $\shc^\wedge$.

An ind-object in $\shc$ is an object of $\shc^\wedge$ isomorphic to 
$\indlim[i\in I] X(i)$  for some functor $X\colon I\to\shc$ with a small filtrant category $I$.
Denote by $\operatorname{Ind}(\shc)$ the full subcategory of $\shc^\wedge$ consisting of ind-objects in $\shc^\wedge$.

\medskip

Let $M$ be a good topological space. The category of ind-sheaves on $M$ is the category $\ind(\field_M) \defeq \operatorname{Ind}(\Mod_c(\field_M))$ of ind-objects in the category $\Mod_c(\field_M)$ of sheaves with
compact support. 
Denote by $\BDC(\ifield_M)$ the bounded derived category of $\ind(\field_M)$.

There is a natural exact embedding $\iota_M\colon\Mod(\field_M) \to \ind(\field_M)$  given by $F\mapsto\indlim (\field_U \tens F)$, for $U$ running over the relatively compact open subsets of $M$.
The functor $\iota_M$ has an exact left adjoint
$\alpha_M\colon\ind(\field_M) \to \Mod(\field_M)$ given by $\alpha_M(\indlim
F_i) = \ilim F_i$. The functor $\alpha_M$ has an exact fully faithful left adjoint  $\beta_M\colon\Mod(\field_M) \to \ind(\field_M)$.
For example, if $Z\subset M$ is a closed subset, one has
\[
\beta_M \field_Z \simeq \indlim[U]\field_{\overline U}\,,
\]
where $U$ ranges over the family of open subsets of $M$ containing $Z$.

For $f\colon M\to N$ a morphism of good topological spaces, denote by $\tens$,
$\rihom$, $\opb f$, $\roim f$, $\reeim f$, $\epb f$ the six Grothendieck
operations for ind-sheaves.
Denote by $\etens$ the exterior product.

Since ind-sheaves form a stack, they have a sheaf-valued hom-functor $\hom$. One has $\rhom \simeq \alpha_M\rihom$.

\smallskip
We will need the following proposition to calculate $\rihom$.

For $a\leq b$ in $\Z$, denote by $\mathsf{C}^{[a,b]}(\Mod(\field_M))$ the category of complexes of sheaves $F^\bullet$ such that $F^k=0$ unless $a\leq k \leq b$.

\begin{proposition}
\label{pro:Prolim}
Let $f\colon M\to N$ be a morphism of good topological spaces.
Let $G\in\BDC(\ifield_M)$ and let $\{F^\bullet_n\}_{n\in\Z_{\geq 0}}$ be an inductive system in $\mathsf{C}^{[a,b]}(\Mod(\field_M))$ for some $a\leq b$ in $\Z$.
Assume that the pro-object
\[
\prolim[n] \roim f \rihom(F^\bullet_n,G) \in \operatorname{Pro}(\BDC(\ifield_N))
\]
is represented by an object of $\BDC(\ifield_N)$.
Then
\[
\roim f \rihom(\indlim[n] F^\bullet_n,G) \simeq
\prolim[n] \roim f \rihom(F^\bullet_n,G).
\]
\end{proposition}

\begin{proof}
Set $S^\bullet_n = \DSum_{k\leq n}F^\bullet_k$ and denote by $\tilde F^\bullet_n$ the mapping cone of the morphism $S^\bullet_{n-1} \to S^\bullet_n$. 
Note that the morphism $\tilde F^\bullet_n \to F^\bullet_n$ induced by the projection $S^\bullet_n\to F^\bullet_n$ is a quasi-isomorphism.
Consider the morphism $S^\bullet_n \to S^\bullet_n \dsum F^\bullet_{n+1} = S^\bullet_{n+1}$ obtained by $\id_{S^\bullet_n}$ and $S^\bullet_n\to F^\bullet_n\to F^\bullet_{n+1}$. This induces a morphism $\tilde F^k_n \to \tilde F^k_{n+1}$ which has a cosection for any $k$ and $n$. Hence, replacing $F^\bullet_n$ with $\tilde F^\bullet_n$, we may assume from the beginning that the morphism $F^k_n \to F^k_{n+1}$ has a cosection for any $k$ and $n$.

We may also assume that $G^\bullet$ is a complex of quasi-injective sheaves, i.e.\ that the functor $\Hom(\ast, G^n)$ is exact in $\Mod(\field_M)$ for any $n\in\Z$.

In order to prove that the morphism
\[
\roim f \rihom(\indlim[n] F^\bullet_n,G^\bullet) \To[u]
\prolim[n] \roim f \rihom(F^\bullet_n,G^\bullet)
\]
is an isomorphism, it is enough to show that $\RHom(H,u)$ is a quasi-isomorphism for any $H\in\Mod(\field_N)$. 

Set $E^\bullet_n = \Hom(F^\bullet_n\tens\opb f H, G^\bullet)$.
Then
\begin{align*}
\varprojlim\limits_n E^\bullet_n &\simeq
\RHom(H,\roim f \rihom(\indlim[n] F^\bullet_n,G^\bullet)), \\
\prolim[n] H^k(E^\bullet_n) &\simeq 
H^k \RHom(H, \prolim[n]\roim f\rihom(F^\bullet_n, G^\bullet)).
\end{align*}
Hence we have to show that
\begin{equation}
\label{eq:Hkprolim}
H^k(\varprojlim\limits_n E^\bullet_n) \to \prolim[n] H^k(E^\bullet_n) \simeq \varprojlim\limits_n H^k(E^\bullet_n)
\end{equation}
is an isomorphism for any $k$. 
Since $E^\bullet_{n+1} \to E^\bullet_n$ is an epimorphism and
$\{H^k(E^\bullet_n)\}_n$ satisfies the Mittag-Leffler condition, we conclude that \eqref{eq:Hkprolim} is an isomorphism.
\end{proof}

Let us recall the results of \cite[\S15.4]{KS06}. These provide useful tools to reduce proofs of many results in the framework of ind-sheaves to analogous results in sheaf theory.

Recall that $\Mod_c(\field_M)$ denotes the category of sheaves with compact support. Then $\BDC(\Mod_c(\field_M))$ is equivalent to the full triangulated subcategory of $\BDC(\field_M)$ consisting of objects with compact support.

\begin{proposition}[{cf.~\cite[\S15.4]{KS06}}]
\label{pro:J}
There exists a canonical functor
\[
J_M\colon \BDC(\ifield_M) \To \operatorname{Ind}\bl\BDC(\Mod_c(\field_M))\br\]
which satisfies the following properties:
\bnum
\item 
For $F\in \BDC(\Mod_c(\field_M))$, and $K\in \BDC(\ifield_M)$,
we have
$$\Hom[\BDC(\ifield_M)](F,K)
\isoto \Hom[{\operatorname{Ind}(\BDC(\Mod_c(\field_M)))}]
\bl J_M(F),J_M(K)\br.$$
\label{J:def}
\item
The functor $J_M$ is conservative, i.e.~a morphism $u$ in $\BDC(\ifield_M)$ is an isomorphism as soon as $J_M(u)$ is an isomorphism.
\label{J:cons}
\item
$J_M(F) \simeq F$ for any $F\in \BDC(\Mod_c(\field_M))$.
\item
$J_M(\indlim F_i) \simeq \ilim J_M(F_i)$ for any filtrant inductive system $\{F_i\}$ in $\Mod(\field_M)$. Here, $\ilim$ denotes the inductive limit in the category  $\operatorname{Ind}(\BDC(\Mod_c(\field_M)))$.
\label{J:ind}
\item
$J_M$ commutes with $\tens$ and $J_M\rihom(F,G) \simeq \rhom(F,J_M(G))$ for $F\in\BDC(\field_M)$ and $G\in\BDC(\ifield_M)$.
Here, $\rhom(F,\ast)$ denotes the endofunctor of $\operatorname{Ind}(\BDC(\Mod_c(\field_M)))$ induced by the endofunctor $\rhom(F,\ast)$ of $\BDC(\Mod_c(\field_M))$.
\label{J:tens}
\item
$H^n J_M(F) \simeq H^n F$ for any $n\in\Z$ and $F\in\BDC(\ifield_M)$.
Here, $H^n$ on the right hand side  is the cohomology functor $\BDC(\ifield_M) \to \ind(\field_M)$, and $H^n$ on the left hand side  
is the functor $\operatorname{Ind}(\BDC(\Mod_c(\field_M))) \to \operatorname{Ind}(\Mod_c(\field_M)) = \ind(\field_M)$ induced by the cohomology functor 
$\BDC(\Mod_c(\field_M)) \to \Mod_c(\field_M)$.
\item
Let $f\colon M\to N$ be a continuous map. Then
\bna
\item
$J_N \circ \reeim f \simeq \reim f \circ J_M$.
\item\label{J:opbepb}
$J_M \circ \opb f \simeq \opb f \circ J_N$ and 
$J_M\circ \epb f \simeq\epb f \circ J_N$.
Here, for $u=\opb f,\epb f$, we denote by the same letter the composition
\[
\operatorname{Ind}(\BDC(\Mod_c(\field_N))) \To[u] \operatorname{Ind}(\BDC(\field_M)) \To
\operatorname{Ind}(\BDC(\Mod_c(\field_M))).
\]
\ee
The last arrow is given by $\indlim[i] F_i\mapsto \indlim[i,U] (F_i)_U$,
where $U$ ranges over the relatively compact open subsets of $M$.
\ee
\end{proposition}

Note that $J_N \circ \roim f \simeq \roim f \circ J_M$ \emph{does not} hold in general.

\medskip

As an example of application of Proposition~\ref{pro:J}, one has the following result.

\begin{corollary}
\label{cor:tensindlim}
Let $G\in\BDC(\field_M)$, $K\in\BDC(\ifield_M)$ and $\{F_i\}$ a filtrant inductive system in $\Mod(\field_M)$. If $\supp G$ is compact, then
\[
\Hom[\BDC(\ifield_M)](G,K\tens\indlim[i] F_i) \simeq
\ilim[i]\Hom[\BDC(\ifield_M)](G,K\tens F_i).
\]
\end{corollary}
\begin{proof}
One has
$J_M(K\tens\indlim[i] F_i)\simeq \ilim[i] J_M(K\tens F_i)$ 
by Proposition~\ref{pro:J} \eqref{J:ind} and  \eqref{J:tens}.
Then the assertion follows from Proposition~\ref{pro:J} \eqref{J:def}.
\end{proof}
Here is another application of Proposition~\ref{pro:J}.

\begin{proposition}\label{pro:opbepb}
Let $f\colon M\to N$ be a continuous map of good topological spaces and $K\in\BDC(\ifield_N)$.
Let $U$ be an open subset of $M$ and $\{V_n\}_{n\in\Z_{\geq 0}}$ an increasing sequence of open subsets of $N$. Assume that
\[
U\cap \overline{\opb f(V_n)} \subset \opb f(V_{n+1})
\quad\text{for any $n\in\Z_{\geq 0}$.}
\]
Then, setting $L=\indlim[n]\field_{V_n}$, there is an isomorphism
\[
\field_U \tens \epb f K \tens \opb f L \isoto \field_U \tens \epb f (K\tens L).
\]
\end{proposition}

\begin{proof}
Since the question is local on $M$, we may assume that $U$ is relatively compact.
By the assumption,
$U\cap\supp(\epb f (K\tens\field_{V_{n}})) \subset \opb f(V_{n+1})$. Thus we have
\begin{align*}
\field_U \tens \epb f (K\tens\field_{V_n})
&\isofrom \field_U \tens \epb f (K\tens\field_{V_n}) \tens \opb f \field_{V_{n+1}}, \\
&\To \field_U \tens \epb f K \tens \opb f \field_{V_{n+1}}.
\end{align*}
By applying $J_M$ and taking the inductive limit with respect to $n$ in $\operatorname{Ind}(\BDC(\Mod_c(\field_M)))$, we obtain a morphism
\[
\varinjlim\limits_n J_M(\field_U \tens \epb f (K\tens\field_{V_n}))
\To
\varinjlim\limits_n J_M(\field_U \tens \epb f K \tens \opb f \field_{V_n}).
\]
By Proposition~\ref{pro:J}~\eqref{J:ind}, \eqref{J:tens} and \eqref{J:opbepb}, 
this gives a morphism
\[
J_M(\field_U \tens \epb f (K\tens L)) \To J_M(\field_U \tens \epb f K \tens \opb f L).
\]
We can easily see  that this is an inverse to the natural morphism
\[
J_M(\field_U \tens \epb f K \tens \opb f L) \To J_M(\field_U \tens \epb f (K\tens L)).
\]
Hence, the statement follows from Proposition~\ref{pro:J}~\eqref{J:cons}.
\end{proof}

We will use the following lemma only in Remark~\ref{rem:vanish}.

\begin{lemma}
\label{lem:Indlim}
Let $M$ be a good topological space and
$\{F^\bullet_n\}_{n\in\Z_{\geq 0}}$ an inductive system in $\mathsf{C}^{[a,b]}(\Mod(\field_M))$ for some $a\leq b$ in $\Z$.
Then
\[
\rihom(\indlim[n] F^\bullet_n, \omega_M) \isofrom
\rihom(\varinjlim\limits_{n} F^\bullet_n, \omega_M).
\]
Here, $\varinjlim\limits_{n} F^\bullet_n$ is the inductive limit of $\{F^\bullet_n\}_{n\in\Z_{\geq 0}}$ in $\mathsf{C}^{[a,b]}(\Mod(\field_M))$.
\end{lemma}

\begin{proof}
By d\'evissage, we may assume that the morphism $F^k_n \to F^k_{n+1}$ has a cosection for each $k$ and $n$, as in the proof of Proposition~\ref{pro:Prolim}, and that all the sheaves $F^k_n$ are soft sheaves.

Then, for any $G\in\Mod_c(\field_M)$,
\begin{align*}
\RHom[\BDC(\ifield_M)]&(G,\rihom(\indlim[n] F^\bullet_n, \omega_M)) \\
&\simeq
\RHom[\BDC(\ifield_M)](G\tens \indlim[n] F^\bullet_n, \omega_M) \\
&\simeq
\RHom[\BDC(\field)](\rsect_c(M; G\tens \indlim[n] F^\bullet_n), \field).
\end{align*}
Since $G\tens F^k_n$ are soft sheaves (see \cite[Lemma 3.1.2]{KS90}),
\[
\rsect_c(M; G\tens \indlim[n] F^\bullet_n)
\simeq
\indlim[n] \sect_c(M; G\tens F^\bullet_n).
\]
Hence
\[
\RHom[\BDC(\field)](\rsect_c(M; G\tens \indlim[n] F^\bullet_n), \field)
\simeq \derr\pi\prolim[n]\sect_c(M; G\tens F^\bullet_n)^*,
\]
where $\pi\colon\operatorname{Pro(\Mod(\field))} \to \Mod(\field)$ is the functor of taking the projective limit (see \cite[Corollary 13.3.16]{KS06}).
Since $\sect_c(M; G\tens F^\bullet_n)^*$ satisfies the Mittag-Leffler condition,
one has
\[
\derr^i\pi\prolim[n]\sect_c(M; G\tens F^\bullet_n)^* \simeq0
\quad\text{for any $i\neq 0$.}
\]
Hence
\begin{align*}
\derr\pi\prolim[n]\sect_c(M; G\tens F^\bullet_n)^*
&\simeq \varprojlim\limits_{n}\sect_c(M; G\tens F^\bullet_n)^* \\
&\simeq (\varinjlim\limits_{n}\sect_c(M; G\tens F^\bullet_n))^* \\
&\simeq \sect_c(M; G\tens \varinjlim\limits_{n}F^\bullet_n)^* \\
&\simeq \RHom(G,\rihom(\varinjlim\limits_{n}F^\bullet_n,\omega_M)).
\end{align*}
This implies that
\[
\RHom(G,\rihom(\varinjlim\limits_{n}F^\bullet_n,\omega_M))
\isoto
\RHom(G,\rihom(\indlim[n]F^\bullet_n,\omega_M))
\]
for any $G\in\Mod_c(\field_M)$, and hence we obtain the desired result.
\end{proof}

\subsection{$\R$-constructible sheaves}

The notion of subanalytic subset and of $\R$-constructible sheaf, usually defined on real analytic manifolds, naturally extend to subanalytic spaces (cf.~\cite[Exercise IX.2]{KS90}).

\begin{definition}
A \emph{subanalytic space} $(M,\shs_M)$ is an $\R$-ringed space which is locally isomorphic to $(Z,\shs_Z)$, where $Z$ is a closed subanalytic subset 
of a real analytic  manifold, and $\shs_Z$ is the sheaf of $\R$-algebras  of real valued subanalytic continuous functions. In this paper, we assume that subanalytic spaces are good topological spaces.

One naturally defines the category of subanalytic spaces. The morphisms are morphisms of $\R$-ringed spaces.
\end{definition}

Let $M$ be a subanalytic space. One says that an object of $\BDC(\field_M)$ is $\R$-constructible if all of its cohomologies are $\R$-constructible.
Denote by $\BDC_\Rc(\field_M)$ the full subcategory of $\R$-constructible objects of $\BDC(\field_M)$. The category $\BDC_\Rc(\field_M)$ is triangulated and is closed under $\tens$, $\rhom$ and the duality functor $\dual_M$.

The following two propositions are classical results (see e.g.~\cite[Propositions 3.4.3, 3.4.4, 8.4.9]{KS90}).

\begin{proposition}
Let $M$ be a subanalytic space and $F\in\BDC_\Rc(\field_M)$. Then the natural morphism
\[
F \to \dual_M\dual_M F
\]
is an isomorphism.
\end{proposition}

\begin{proposition}
\label{pro:Rcdtens}
Let $M$ be a subanalytic space and $N$ a good topological space. Let $p_1\colon M\times N\to M$ and $p_2\colon M\times N\to N$ be the projections. Then for any $F\in\BDC_\Rc(\field_M)$ and $G\in\BDC(\field_N)$ the natural morphism
\[
\opb{p_1}\dual_M F \tens \opb{p_2} G \To \rhom(\opb{p_1}F,\epb{p_2} G)
\]
is an isomorphism.
\end{proposition}

Hence, by applying Corollary~\ref{cor:tensindlim}, we obtain the following proposition.

\begin{proposition}
\label{pro:const}
Let $M$, $N$, $p_1$ and $p_2$ be as in {\rm Proposition~\ref{pro:Rcdtens}}.
Let $F\in\BDC_\Rc(\field_M)$ and $G\in\BDC(\ifield_N)$.
Then there are isomorphisms
\begin{align*}
\epb{p_2}\field_N\tens\opb{p_2}G &\isoto \epb{p_2}G, \\
\opb{p_1}\dual_M F \tens \opb{p_2}G &\isoto \rihom(\opb{p_1}F,\epb{p_2} G).
\end{align*}
\end{proposition}

\begin{corollary}
\label{cor:exthom}
Let $M$ be a subanalytic space and $N$ a good topological space. 
Let $F_1,F_2\in\BDC_\Rc(\field_M)$ and $G_1,G_2\in\BDC(\ifield_N)$.
Then the canonical morphism
\[
\rhom(F_1,F_2) \etens \rihom(G_1,G_2)
\To \rihom(F_1 \etens G_1, F_2\etens G_2)
\]
is an isomorphism.
\end{corollary}

\begin{proof}
Let $p_1\colon M\times N\to M$ and $p_2\colon M\times N\to N$ be the projections.
We have
\[
\opb{p_1} F_2\tens \opb{p_2}G_2 \simeq \rihom(\opb{p_1} \dual_M F_2, \epb{p_2}G_2).
\]
Hence
\begin{align*}
\rihom&(\opb{p_1} F_1 \tens \opb{p_2}G_1,  \opb{p_1} F_2\tens \opb{p_2}G_2) \\
&\simeq \rihom(\opb{p_1} F_1 \tens \opb{p_2}G_1 \tens \opb{p_1} \dual_M F_2, \epb{p_2}G_2) \\
&\simeq\rihom(\opb{p_1} (F_1 \tens \dual_M F_2)  , \rihom(\opb{p_2}G_1, \epb{p_2}G_2)) \\
&\underset{(*)}\simeq \rihom(\opb{p_1} \dual_M\rhom(F_1, F_2)  , \epb{p_2}\rihom(G_1,G_2)) \\
&\simeq \opb{p_1}\rhom(F_1, F_2) \tens \opb{p_2}\rihom(G_1,G_2).
\end{align*}
Here, in $(*)$ we have used
\[
F_1 \tens \dual_M F_2 \simeq \dual_M\rhom(F_1, F_2),
\]
which follows from
\begin{align*}
\dual_M(F_1 \tens \dual_M F_2)
&= \rhom(F_1 \tens \dual_M F_2,\omega_M) \\
&\simeq \rhom(F_1 , \rhom(\dual_M F_2,\omega_M)) \\
&\simeq \rhom(F_1 , F_2) .
\end{align*}
\end{proof}

\subsection{Subanalytic ind-sheaves}

Let $M$ be a subanalytic space.
An ind-sheaf on $M$ is called subanalytic if it is isomorphic to a small filtrant ind-limit of $\R$-constructible sheaves. Let us denote by $\ind_\suban(\field_M)$
the category of subanalytic ind-sheaves. Note that it is stable by kernels, cokernels and extensions in $\ind(\field_M)$.
An object of $\BDC(\ifield_M)$ is called subanalytic if all of its cohomologies are subanalytic.
Denote by $\BDC_\suban(\ifield_M)$ the full subcategory of subanalytic objects in
$\BDC(\ifield_M)$.  It is a triangulated category.%
\footnote{\,In \cite{KS01}, subanalytic ind-sheaves are called ind-$\R$-constructible sheaves, and $\ind_\suban(\field_M)$ and $\BDC_\suban(\ifield_M)$ are denoted by $\ind_\Rc(\field_M)$ and $\BDC_{\ind\Rc}(\ifield_M)$, respectively.}

Let $\operatorname{Op}_{M_\sa}$ be the category of relatively compact subanalytic open subsets of $M$, whose morphisms are inclusions.

\begin{definition}[cf.~\cite{KS96,KS01}]
A \emph{subanalytic sheaf} $F$ is a functor $\operatorname{Op}_{M_\sa}^\op\to\Mod(\field)$ which satisfies
\bnum
\item
$F(\emptyset) = 0$,
\item
For $U,V\in\operatorname{Op}_{M_\sa}$, the sequence
\[
0 \To F(U\cup V) \To[r_1] F(U) \dsum F(V) \To[r_2] F(U\cap V)
\]
is exact. Here $r_1$ is given by the restriction maps and $r_2$ is given by the restriction $F(U) \to F(U\cap V)$ and the opposite of the restriction $F(V) \to F(U\cap V)$.
\ee
Denote by $\Mod(\field_{M_\sa})$ the category of subanalytic sheaves.
\end{definition}

The following result is proved in \cite{KS01}.

\begin{proposition}
The category $\ind_\suban(\field_M)$ of subanalytic ind-sheaves and the category $\Mod(\field_{M_\sa})$ of subanalytic sheaves are equivalent by the functor associating with $F\in\ind_\suban(\field_M)$ the subanalytic sheaf
\[
\operatorname{Op}_{M_\sa} \owns U \longmapsto \Hom[\ind(\field_M)](\field_U,F).
\]
\end{proposition}

In particular, we have

\begin{proposition}
\label{pro:suban0}
Let $K\in\BDC_\suban(\ifield_M)$.
Then $K \simeq 0$ if and only if
\[
\Hom[\BDC(\ifield_M)](\field_U[n],K) \simeq 0
\] 
for any $n\in\Z$ and any relatively compact subanalytic open subset $U\subset M$.
\end{proposition}

We will need the following result.

\begin{lemma}\label{lem:vanrelsuban}
Let $M$ be a subanalytic space and $K\in\BDC_\suban(\ifield_{M\times [0,1]})$.
Then $K \simeq 0$ if and only if $\Hom[\BDC(\ifield_{M\times[0,1]})](\field_U[n],K) \simeq 0$ for any $n\in\Z$ and any relatively compact subanalytic open subset $U\subset M\times [0,1]$ such that each fiber of $U\to M$ is either empty or connected.
\end{lemma}

This follows from Proposition~\ref{pro:suban0} and the following lemma.

\begin{lemma}
\label{lem:relsuban}
Any relatively compact subanalytic open subset of $M\times[0,1]$ is a finite union of subanalytic open sets $U$ such that each fiber of $U\to M$ is either empty or connected.
\end{lemma}

For a similar statement, see \cite[Lemma~3.6]{KS96}.

\subsection{$\D$-modules}

Let $X$ be a complex manifold.
We denote by $d_X$ its (complex) dimension.
Denote by $\O_X$ and $\D_X$ the sheaves of algebras  of holomorphic functions and of
differential operators, respectively. Denote by $\Omega_X$ the invertible sheaf of
differential forms of top degree.

Denote by $\Mod(\D_X)$ the category of left $\D_X$-modules, and by $\BDC(\D_X)$ its bounded derived category.
For $f\colon X\to Y$ a morphism of complex manifolds, denote by $\dtens$,
$\dopb f$, $\doim f$ the operations for $\D$-modules.
Denote by $\detens$ the exterior product.

Let us denote by
\eq
\mop\colon \BDC(\D_X) \isoto \BDC(\D_X^\op)\label{def:r}
\eneq
the equivalence of categories given by the functor $\shm^\mop = \Omega_X\ltens[\O_X]\shm$.
Consider the dual of $\shm\in\BDC(\D_X)$ given by
\[
\ddual_X\shm = \rhom[\D_X](\shm,\D_X\tens[\O_X]\Omega_X^{\otimes-1})[d_X],
\]
where the shift is chosen so that $\ddual_X\O_X\simeq\O_X$.

Denote by $\BDC_\coh(\D_X)$, $\BDC_\qgood(\D_X)$ and $\BDC_\good(\D_X)$ the full
subcategories of $\BDC(\D_X)$ whose objects have coherent, quasi-good and good cohomologies, respectively.
Here, a $\D_X$-module $\shm$ is called \emph{quasi-good} if, for any relatively compact open subset $U\subset X$, $\shm\vert_U$
is the sum of a filtrant family of
coherent $(\O_X\vert_U)$-submodules. 
A $\D_X$-module $\shm$ is called \emph{good} if it is quasi-good and coherent.

Recall that to a coherent $\D_X$-module $\shm$ one associates its characteristic variety $\chv(\shm)$, a closed conic involutive subset of the cotangent bundle $T^*X$. If $\chv(\shm)$ is Lagrangian, $\shm$ is called holonomic. For the notion of regular holonomic $\D_X$-module, refer e.g.\ to \cite[\S5.2]{Kas03}.

Denote by $\BDC_{\hol}(\D_X)$ and $\BDC_{\reghol}(\D_X)$ the full 
subcategories of $\BDC(\D_X)$ whose objects have holonomic and regular
holonomic cohomologies, respectively. 

Note that $\BDC_\coh(\D_X)$, $\BDC_\qgood(\D_X)$, $\BDC_\good(\D_X)$,  $\BDC_{\hol}(\D_X)$ and $\BDC_{\reghol}(\D_X)$ are triangulated categories. 

If $Y\subset X$ is a closed hypersurface, denote by $\O_X(*Y)$ the sheaf of meromorphic functions with poles at $Y$. It is a regular holonomic $\D_X$-module. For $\shm\in\BDC(\D_X)$, set
\[
\shm(*Y) = \shm \dtens \O_X(*Y).
\]
If $Y$ is a closed submanifold of $X$, denoting by $i\colon Y\to X$ the inclusion morphism, one sets
\eq
&&\shb_Y = \doim i\O_Y.
\label{eq:B}
\eneq
Then $\shb_Y$ is concentrated in degree zero, and is a regular holonomic $\D_X$-module.

For $\shm\in\BDC_\coh(\D_X)$, denote by $\ss(\shm)\subset X$ its singular support, that is the set
of points where $\chv(\shm) \defeq \Union\nolimits_{i\in\Z} \chv(H^i\shm)$ is not contained in the zero-section of $T^*X$.

\begin{proposition}[{\cite[Theorem~4.33]{Kas03}}]
\label{pro:Dadj} 
Let $f\colon X\to Y$ be a morphism of complex manifolds. Let $\shm\in\BDC_\good(\D_X)$ and $\shn\in\BDC(\D_Y)$. If $\supp(\shm)$ is proper over $Y$, then $\doim f\shm \in \BDC_\good(\D_Y)$ and there is an isomorphism
\[
\roim f\rhom[\D_X](\shm,\dopb f\shn)[d_X] \simeq
\rhom[\D_Y](\doim f\shm,\shn)[d_Y].
\]
In particular,
\[
\Hom[\BDC(\D_X)](\shm,\dopb f\shn[d_X]) \simeq
\Hom[\BDC(\D_Y)](\doim f\shm,\shn[d_Y]).
\]
\end{proposition}

\begin{proposition}[{\cite[Theorem~4.40]{Kas03}}]
\label{pro:Dadj2}
If $f\colon X\to Y$ is a smooth morphism of complex manifolds, then for $\shm\in\BDC(\D_X)$ and $\shn\in\BDC_\coh(\D_Y)$ we have
\[
\roim f\rhom[\D_X](\dopb f\shn,\shm)[d_X] \simeq
\rhom[\D_Y](\shn,\doim f\shm)[d_Y].
\]
In particular,
\[
\Hom[\BDC(\D_X)](\dopb f\shn,\shm[d_X]) \simeq
\Hom[\BDC(\D_Y)](\shn,\doim f\shm[d_Y]).
\]
\end{proposition}

A \emph{transversal Cartesian diagram} is a commutative diagram
\begin{equation}
\label{eq:transCart}
\ba{c}\xymatrix@C=8ex{
X' \ar[r]^{f'} \ar[d]^{g'} & Y' \ar[d]^{g} \\
X \ar[r]^{f}\ar@{}[ur]|-\square & Y
}\ea
\end{equation}
with $X'\simeq X\times_Y Y'$ and such that the map of tangent spaces
\[
T_{g'(x)}X \dsum T_{f'(x)}Y' \to T_{f(g'(x))}Y
\]
is surjective for any $x\in X'$.

\begin{proposition}\label{pro:transCart}
Consider the transversal Cartesian diagram \eqref{eq:transCart}.
Then, for any $\shm\in\BDC_\good(\D_X)$ such that $\supp(\shm)$ is proper over $Y$,
\[
\dopb g \doim f \shm \simeq \doim {f'}\dopbv{g^{\prime\sep*}}\shm.
\]
\end{proposition}

\section{Bordered spaces}\label{se:bordered}

Let $\bM$ be a good topological space, and $M\subset \bM$ an open subset.
For usual sheaves, the restriction functor $F\mapsto F|_M$ induces an equivalence
\[
\BDC(\field_{\bM}) / \BDC(\field_{\bM\setminus M})
\isoto
\BDC(\field_M).
\]
This is no longer true for ind-sheaves, as seen by the following example.

\begin{example*}
Let $\bM=\R$ and $M = \ooint{0,1}$. Consider the ind-sheaf on $\bM$
\[
\beta_{\bM}\field_{\{0\}} = \indlim[U\owns 0]\field_{\overline U},
\]
where $U$ ranges over the family of open neighborhoods of $0\in\bM$.
Then $\beta_{\bM}\field_{\{0\}}|_M \simeq 0$, but $\beta_{\bM}\field_{\{0\}}\notin\BDC(\ifield_{\bM\setminus M})$.
\end{example*}

Therefore, in the framework of ind-sheaves one should consider the quotient category $\BDC(\ifield_{\bM}) / \BDC(\ifield_{\bM\setminus M})$ attached to the pair $(M,\bM)$. We will call such a pair a \emph{bordered space}.

In this section, we define the category of bordered spaces, develop the formalism of external operations, and define the natural $t$-structure on the derived category of ind-sheaves on a bordered space.

\subsection{Quotient categories}

Let $\shd$ be a triangulated category and $\shn\subset\shd$ a full triangulated subcategory. The quotient category $\shd/\shn$ is defined as the localization $\shd_\Sigma$ of $\shd$ with respect to the multiplicative system $\Sigma$ of morphisms $u$ fitting into a distinguished triangle
\[
X\To[u]Y\To Z\To[+1]
\]
with $Z\in\shn$.

The right orthogonal $\shn^\bot$  and the left orthogonal ${}^\bot\shn$ are the full subcategories of $\shd$
\begin{align*}
\shn^\bot & = \{X\in\shd\semicolon \Hom[\shd](Y,X) \simeq 0\text{ for any }Y\in\shn \}, \\
{}^\bot\shn & = \{X\in\shd\semicolon \Hom[\shd](X,Y) \simeq 0\text{ for any }Y\in\shn \}.
\end{align*}

The following result is elementary (cf.~\cite[Exercise 10.15]{KS06}).

\begin{proposition}
\label{pro:D/N}
Assume that
\eqn
&\text{if $X,Y\in\shd$,  $Z\in\shn$ and
$Z\simeq X\oplus Y$, then one has $X\in \shn$.}
\eneqn
Then the following conditions are equivalent:
\bnum
\item  the composition $\shn^\bot \to \shd \to \shd/\shn$ is an equivalence of categories,
\item the embedding $\shn\to\shd$ has a right adjoint,
\item the quotient functor $\shd\to\shd/\shn$ has a right adjoint,
\item for any $X\in\shd$ there is a distinguished triangle $X'\to X\to X''\to[+1]$ with $X'\in\shn$ and $X''\in\shn^\bot$.
\ee
Similar results hold for the left orthogonal.
\end{proposition}

\subsection{Bordered spaces}

Let $M\subset \bM$ and $N\subset \bN$ be open embeddings of good topological spaces. 
For a continuous map $f\colon M\to N$,
denote by $\Gamma_f$ its graph in $M\times N$, and by $\overline\Gamma_f$ the closure of $\Gamma_f$ in $\bM\times\bN$.

\begin{definition}
The category of \emph{bordered spaces} is the category whose objects are pairs $(M,\bM)$ with $M\subset \bM$ an open embedding of good topological spaces.
Morphisms $f\colon (M,\bM) \to (N,\bN)$ are continuous maps $f\colon M\to N$ such that
\begin{equation}
\label{eq:Hbord}
\overline\Gamma_f \to \bM \text{ is proper}. 
\end{equation}
The composition of $(L,\bL) \to[g] (M,\bM) \to[f] (N,\bN)$ is given by $f\circ g\colon L\to N$ (see Lemma~\ref{lem:bordcom} below), and the identity $\id_{(M,\bM)}$ is given by $\id_{M}$.
\end{definition}

\begin{remark}
The properness assumption \eqref{eq:Hbord} is used in
Lemma~\ref{lem:bcomp} below  
to prove the functoriality of external operations. It is satisfied in particular if either $M=\bM$ or $\bN$ is compact.
\end{remark}

\begin{lemma}\label{lem:bordcom}
Let $f\colon (M,\bM) \to (N,\bN)$ and $g\colon (L,\bL) \to (M,\bM)$ be morphisms of bordered spaces. Then the composition $f\circ g$ is a morphism of bordered spaces.
\end{lemma}

\begin{proof}
Note that $\overline\Gamma_g \times_{\bM} \overline \Gamma_f \to \overline\Gamma_g \times_{\bM} \bM \to \bL$ is proper.
Hence $\overline\Gamma_g \times_{\bM} \overline \Gamma_f \to \bL \times \bN$ is proper. In particular, $\Im(\overline\Gamma_g \times_{\bM} \overline \Gamma_f \to \bL \times \bN)$ is a closed subset of $\bL \times \bN$. Since it contains $\Gamma_{f\circ g}$, it also contains $\overline\Gamma_{f\circ g}$.
Since $\overline\Gamma_{f\circ g} \times_{\bL\times\bN} (\overline\Gamma_g \times_{\bM} \overline \Gamma_f) \to \bL$ is proper, $\overline\Gamma_{f\circ g} \to \bL$ is proper.
\end{proof}

Note that the category of bordered spaces has
\bnum
\item a final object $(\point,\point)$,
\item fiber products. 
\ee
In fact, the fiber product of $f\colon (M,\bM) \to (L,\bL)$ and $g\colon (N,\bN)\to(L,\bL)$ is represented by $(M\times_L N, \overline\Gamma_f \times_{\bL} \overline\Gamma_g)$.

\smallskip
Regarding a space $M$ as the bordered space $(M,M)$, one gets a fully faithful embedding of the category of good topological spaces into that of bordered spaces.

\begin{remark}
For any bordered space $(M,\bM)$, using the identifications $M=(M,M)$ and $\bM = (\bM,\bM)$, there are natural morphisms
\[
M \To (M,\bM) \To \bM.
\]
Note however that $\id_M$ does not necessarily induce a morphism $(M,\bM) \to M$ of bordered spaces.
\end{remark}

If a continuous map $f\cl M\to N$ extends to a continuous map $\check f\cl \bM\to\bN$,
then $f$ induces a morphism of bordered spaces
$(M,\bM) \to (N,\bN)$. However the converse is not true. 
If $f\colon (M,\bM) \to (N,\bN)$ is a morphism of bordered spaces, the map $f\colon M\to N$ does not extend to a continuous map $\check f\colon\bM\to\bN$, in general. 
However, the next lemma shows how one can always reduce to this case.

\begin{lemma}
\label{lem:fborddecomp}
Any morphism of bordered spaces $f\colon (M,\bM) \to (N,\bN)$ decomposes as
\[
(M,\bM) \isofrom[q_1] (\Gamma_f,\overline\Gamma_f) \xrightarrow[\;q_2\;]{} (N,\bN),
\]
where the first arrow is an isomorphism and the maps $q_1\colon \Gamma_f \to M$ and $q_2\colon \Gamma_f \to N$ extend to maps $\check q_1\colon \overline\Gamma_f \to \bM$ and $\check q_2\colon \overline\Gamma_f \to \bN$.
\end{lemma}

\begin{definition}
The derived category of ind-sheaves on a bordered space $(M,\bM)$ is the quotient category
\[
\BDC(\ifield_{(M,\bM)}) \defeq \BDC(\ifield_{\bM}) / \BDC(\ifield_{\bM\setminus M}),
\]
where $\BDC(\ifield_{\bM\setminus M})$ is identified with its essential image in $\BDC(\ifield_{\bM})$ by the fully faithful functor $\reeim i \simeq \roim i$, for $i\colon \bM\setminus M\to \bM$ the closed embedding.
\end{definition}

\begin{remark}
In the framework of subanalytic sheaves, an analogue of $\BDC(\ifield_{(M,\bM)})$ is the derived category of sheaves on some site considered in Definitions 6.1.1~(iv) and 7.1.1 of \cite{KS01}.
\end{remark}

Since the functor $\reeim i \simeq \roim i$ has both a right and a left adjoint, it follows from Proposition~\ref{pro:D/N} that there are equivalences
\[
{}^\bot \BDC(\ifield_{\bM\setminus M}) \simeq \BDC(\ifield_{(M,\bM)}) \simeq \BDC(\ifield_{\bM\setminus M})^\bot.
\]
Let us describe these equivalences more explicitly.

\begin{lemma}
For $F\in\BDC(\ifield_{\bM})$, one has
\begin{align*}
\field_M \tens \rihom(\field_M, F) &\isofrom \field_M \tens F, \\
\rihom(\field_M, \field_M \tens F) &\isoto \rihom(\field_M, F).
\end{align*}
\end{lemma}

\begin{proposition}\label{pro:bord}
Let $(M,\bM)$ be a bordered space.
\bnum
\item
One has
\begin{align*}
\BDC(\ifield_{\bM\setminus M}) 
&= \{F\in\BDC(\ifield_{\bM})\semicolon F \isoto \field_{\bM\setminus M} \tens F \}\\
&= \{F\in\BDC(\ifield_{\bM})\semicolon \field_M \tens F \simeq 0 \}\\
&= \{F\in\BDC(\ifield_{\bM})\semicolon \rihom(\field_{\bM\setminus M}, F) \isoto F \} \\
&= \{F\in\BDC(\ifield_{\bM})\semicolon \rihom(\field_M, F) \simeq 0 \} .
\end{align*}
\item
One has
\begin{align*}
{}^\bot \BDC(\ifield_{\bM\setminus M}) 
&= \{F\in\BDC(\ifield_{\bM})\semicolon \field_M \tens F \isoto F \}\\
&= \{F\in\BDC(\ifield_{\bM})\semicolon \field_{\bM\setminus M} \tens F \simeq 0 \},
\end{align*}
and there is an equivalence
\[
\BDC(\ifield_{(M,\bM)}) \isoto {}^\bot \BDC(\ifield_{\bM\setminus M}), \quad F\mapsto \field_M \tens F,
\]
with quasi-inverse induced by the quotient functor.
\item
One has
\begin{align*}
\BDC(\ifield_{\bM\setminus M})^\bot 
&= \{F\in\BDC(\ifield_{\bM})\semicolon F \isoto \rihom(\field_M, F) \} \\
&= \{F\in\BDC(\ifield_{\bM})\semicolon \rihom(\field_{\bM\setminus M}, F) \simeq 0 \},
\end{align*}
and there is an equivalence
\[
\BDC(\ifield_{(M,\bM)}) \isoto \BDC(\ifield_{\bM\setminus M})^\bot, \quad F \mapsto \rihom(\field_M, F),
\]
with quasi-inverse induced by the quotient functor.
\ee
\end{proposition}

\begin{corollary}\label{cor:bord}
For $F,G\in\BDC(\ifield_{\bM})$ one has
\begin{align*}
\Hom[\BDC(\ifield_{(M,\bM)})](F,G)
&\simeq \Hom[\BDC(\ifield_{\bM})](\field_M\tens F,G) \\
&\simeq \Hom[\BDC(\ifield_{\bM})](F,\rihom(\field_M,G)) \\
&\simeq \Hom[\BDC(\ifield_{\bM})](\field_M\tens F,\field_M\tens G) \\
&\simeq \Hom[\BDC(\ifield_{\bM})](\rihom(\field_M,F),\rihom(\field_M,G)).
\end{align*}
\end{corollary}

There is a quasi-commutative diagram of natural functors
\[
\xymatrix@C=10ex{
\BDC(\field_{\bM}) \ar@{^(->}[r]^-{\iota_{\bM}} \ar[d] & \BDC(\ifield_{\bM}) \ar[d] \\
\BDC(\field_M) \ar@{^(->}[r]^-{\iota_{(M,\bM)}} & \BDC(\ifield_{(M,\bM)}),
}
\]
where the left vertical arrow is the functor of restriction to $M$, the right vertical arrow is the quotient functor, and the bottom arrow is the composition
\[
\BDC(\field_M) \simeq  \BDC(\field_{\bM})/\BDC(\field_{\bM\setminus M}) \To \BDC(\ifield_{(M,\bM)}) .
\]

\begin{notation}
\label{not:fieldNbN}
We sometimes write $\BDC(\field_{(M,\bM)})$ for $\BDC(\field_M)$, when considered as a full subcategory of $\BDC(\ifield_{(M,\bM)})$ by $\iota_{(M,\bM)}$.
\end{notation}

\subsection{Operations}
Let us discuss internal and external operations in the category of bordered spaces.

\begin{definition}
The functors $\tens$ and $\rihom$ in $\BDC(\ifield_{\bM})$ induce well defined functors
\begin{align*}
\tens\; &\colon \BDC(\ifield_{(M,\bM)}) \times \BDC(\ifield_{(M,\bM)}) \to \BDC(\ifield_{(M,\bM)}), \\
\rihom\, &\colon \BDC(\ifield_{(M,\bM)})^\op \times \BDC(\ifield_{(M,\bM)}) \to \BDC(\ifield_{(M,\bM)}).
\end{align*}
\end{definition}

\begin{lemma}
For $F_1,F_2\in\BDC(\ifield_{(M,\bM)})$ one has
\[
\Hom[\BDC(\ifield_{(M,\bM)})](\field_M, \rihom(F_1,F_2)) \simeq \Hom[\BDC(\ifield_{(M,\bM)})](F_1,F_2).
\]
\end{lemma}

\begin{lemma}
\label{lem:adjbord}
For $F_1,F_2,F_3\in\BDC(\ifield_{(M,\bM)})$ one has
\begin{align*}
\rihom(F_1\tens F_2,F_3) &\simeq \rihom(F_1,\rihom(F_2,F_3)), \\
\Hom[\BDC(\ifield_{(M,\bM)})](F_1\tens F_2,F_3) &\simeq \Hom[\BDC(\ifield_{(M,\bM)})](F_1,\rihom(F_2,F_3)).
\end{align*}
\end{lemma}

Let $f\colon (M,\bM) \to (N,\bN)$ be a morphism of bordered spaces, and recall that $\Gamma_f$ denotes the graph of the associated map $f\colon M \to N$.
Since $\Gamma_f$ is closed in $M\times N$, it is locally closed in $\bM\times \bN$. One can then consider
the sheaf $\field_{\Gamma_f}$ on $\bM\times \bN$.

\begin{definition}
\label{def:fbordered}
Let $f\colon (M,\bM) \to (N,\bN)$ be a morphism of bordered spaces.
For $F\in\BDC(\ifield_{\bM})$ and $G\in\BDC(\ifield_{\bN})$, we set
\begin{align*}
\reeim f F &= \reeimv{q_{2\sep!!}}(\field_{\Gamma_f}\tens\opb{q_1}F), &
\roim f F &= \roimv{q_{2\sep*}}\rihom(\field_{\Gamma_f},\epb{q_1}F), \\
\opb f G &= \reeimv{q_{1\sep!!}}(\field_{\Gamma_f}\tens\opb{q_2}G), &
\epb f G &= \roimv{q_{1\sep*}}\rihom(\field_{\Gamma_f},\epb{q_2}G),
\end{align*}
where $q_1\colon\bM\times\bN\to\bM$ and $q_2\colon\bM\times\bN\to\bN$ are the projections.
\end{definition}

\begin{remark}
Considering a continuous map $f\colon M\to N$ as a morphism of bordered spaces with $\bM=M$ and $\bN=N$,  the above functors are isomorphic to the usual external operations for ind-sheaves.
\end{remark}

\begin{lemma}
\label{lem:coper}
The above definition induces well-defined functors
\begin{align*}
\reeim f &\colon \BDC(\ifield_{(M,\bM)}) \To {}^\bot\BDC(\ifield_{\bN\setminus N}) \simeq \BDC(\ifield_{(N,\bN)}), \\
\roim f &\colon \BDC(\ifield_{(M,\bM)}) \To \BDC(\ifield_{\bN\setminus N})^\bot \simeq \BDC(\ifield_{(N,\bN)}), \\
\opb f &\colon \BDC(\ifield_{(N,\bN)}) \To {}^\bot\BDC(\ifield_{\bM\setminus M}) \simeq \BDC(\ifield_{(M,\bM)}), \\
\epb f &\colon \BDC(\ifield_{(N,\bN)}) \To \BDC(\ifield_{\bM\setminus M})^\bot \simeq \BDC(\ifield_{(M,\bM)}).
\end{align*}
\end{lemma}

\begin{proof}
Since the arguments are similar for all functors, let us only discuss $\roim f$. Let $F\in\BDC(\ifield_{\bM})$.

\smallskip\noindent(i)
Assume that $F\simeq\rihom(\field_{\bM\setminus M},F)$. Since $\Gamma_f\cap\opb{q_1}(\bM\setminus M) = \emptyset$, we have
\begin{align*}
\roim f F
&\simeq\roimv{q_{2\sep*}}\rihom(\field_{\Gamma_f},\epb{q_1}F) \\
&\simeq \roimv{q_{2\sep*}}\rihom\bl\field_{\Gamma_f},\epb{q_1}\rihom(\field_{\bM\setminus M},F)\br \\
&\simeq \roimv{q_{2\sep*}}\rihom(\field_{\Gamma_f}\tens\opb{q_1}\field_{\bM\setminus M},\epb{q_1}F) \\
&\simeq \roimv{q_{2\sep*}}\rihom(\field_{\Gamma_f\cap\opb{q_1}(\bM\setminus M)},\epb{q_1}F) \simeq 0.
\end{align*}
This shows that the functor $\roim f\colon\BDC(\ifield_{\bM})\to\BDC(\ifield_{\bN})$ factors through $\BDC(\ifield_{(M,\bM)})$.

\smallskip\noindent(ii)
Since $\opb{q_2}(\bN\setminus N) \cap \Gamma_f = \emptyset$, we have
\begin{align*}
\rihom(\field_{\bN\setminus N},\roim f F)
&\simeq\rihom\bl\field_{\bN\setminus N},\roimv{q_{2\sep*}}\rihom(\field_{\Gamma_f},\epb{q_1}F)\br \\
&\simeq \roimv{q_{2\sep*}}\rihom(\opb{q_2}\field_{\bN\setminus N} \tens \field_{\Gamma_f},\epb{q_1}F) \\
&\simeq \roimv{q_{2\sep*}}\rihom(\field_{\opb{q_2}(\bN\setminus N) \cap \Gamma_f},\epb{q_1}F) \simeq 0.
\end{align*}
This shows that $\roim f F \in \BDC(\ifield_{\bN\setminus N})^\bot$.
\end{proof}

The following lemma is easy to prove.

\begin{lemma}
\label{lem:jM}
Let $j_M\colon(M,\bM)\to \bM$ be the morphism given by the open embedding $M\subset \bM$. Then
\bnum
\item
The functors
\[
\opb{j_M} \simeq \epb{j_M} \colon \BDC(\ifield_{\bM}) \to \BDC(\ifield_{(M,\bM)})
\]
are isomorphic to the quotient functor.
\item
For $F\in\BDC(\ifield_{\bM})$ one has the isomorphisms in $\BDC(\ifield_{\bM})$
\[
\reeimv{j_{M\sep!!}}\opb{j_M} F \simeq \field_M \tens F, \quad
\roimv{j_{M\sep*}}\epb{j_M} F \simeq \rihom(\field_M, F).
\]
\item
The functors $\tens$ and $\rihom$ commute with $\opb{j_M} \simeq \epb{j_M}$.
\item
The functor $\tens$ commutes with $\reeimv{j_{M\sep!!}}$ and
the functor $\rihom$ commutes with $\roimv{j_{M\sep*}}$.
More precisely, for $F_1,F_2\in\BDC(\ifield_{(M,\bM)})$ one has 
\begin{align*}
\reeimv{j_{M\sep!!}}(F_1 \tens F_2) 
&\simeq \reeimv{j_{M\sep!!}}F_1 \tens \reeimv{j_{M\sep!!}}F_2 \\
&\simeq \reeimv{j_{M\sep!!}}F_1 \tens \roimv{j_{M\sep*}}F_2, \\
\roimv{j_{M\sep*}}\rihom(F_1,F_2) 
&\simeq \rihom(\reeimv{j_{M\sep!!}}F_1,\reeimv{j_{M\sep!!}}F_2) \\
&\simeq \rihom(\reeimv{j_{M\sep!!}}F_1,\roimv{j_{M\sep*}}F_2) \\
&\simeq \rihom(\roimv{j_{M\sep*}}F_1,\roimv{j_{M\sep*}}F_2).
\end{align*}
\ee
\end{lemma}

\begin{convention}
In the sequel, to avoid confusion, we distinguish between the objects of $\BDC(\ifield_{\bM})$ and the objects of $\BDC(\ifield_{(M,\bM)})$. In other words,
if $F\in\BDC(\ifield_{\bM})$, we avoid to denote by $F$ its image in the quotient category $\BDC(\ifield_{(M,\bM)})$, and write instead $\opb{j_M}F$ or $\epb{j_M}F$.
\end{convention}

Let us now show that the external operations for bordered spaces satisfy similar properties to the external operations for usual spaces.

\begin{lemma}
\label{lem:badj}
Let $f\colon (M,\bM) \to (N,\bN)$ be a morphism of bordered spaces.
\bnum
\item
The functor $\reeim f$ is left adjoint to $\epb f$.
\item
The functor $\opb f$ is left adjoint to $\roim f$.
\ee
\end{lemma}

\begin{lemma}\label{lem:bcomp}
Let $g\colon (L,\bL) \to (M,\bM)$ and $f\colon (M,\bM) \to (N,\bN)$ be morphisms of bordered spaces. One has
\[
\reeim{(f\circ g)} \simeq \reeim f \circ \reeim g, \qquad
\roim{(f\circ g)} \simeq \roim f \circ \roim g
\]
and
\[
\opb{(f\circ g)} \simeq \opb g \circ \opb f, \qquad
\epb{(f\circ g)} \simeq \epb g \circ \epb f.
\]
\end{lemma}

\begin{proof}
Since the proofs are similar, we treat only the first isomorphism.

For $F\in\BDC(\ifield_{\bL})$, one has
\[
\reeim{(f\circ g)}\opb{j_L} F \simeq \opb{j_N} \reeimv{q_{2\sep!!}}(\field_{\Gamma_{f\circ g}}\tens\opb{q_1}F),
\]
where $q_1$ and $q_2$ are the projections from $\bL\times\bN$ to the corresponding factors. Using the projection formula, one easily checks the isomorphism
\[
\reeim f \reeim g \opb{j_L} F \simeq
\opb{j_N} \reeimv{q_{2\sep!!}}((\field_{\Gamma_g}\circ\field_{\Gamma_f})\tens\opb{q_1}F),
\]
where
\[
\field_{\Gamma_g}\circ\field_{\Gamma_f} \defeq \reeimv{q_{13\sep!!}}(\opb{q_{12}} \field_{\Gamma_g} \tens \opb{q_{23}}\field_{\Gamma_f}),
\]
and $q_{12}$, $q_{23}$ and $q_{13}$ denote the projections from $\bL\times\bM\times\bN$ to the corresponding factors. 
For example, $q_{13}(x,y,z) = (x,z)$. 

Hence, writing explicitly the embedding functor  $\iota$ of sheaves into ind-sheaves, it is enough to show
\[
\iota_{\bL\times\bM}\,\field_{\Gamma_g} \comp \iota_{\bM\times\bN}\,\field_{\Gamma_f} \simeq \iota_{\bL\times\bN}\,\field_{\Gamma_{f\circ g}}.
\]
Recalling that $\iota$ commutes with tensor product, ordinary inverse image, and ordinary direct image, we have
\begin{align*}
\notag
\iota_{\bL\times\bM}\,\field_{\Gamma_g} \comp \iota_{\bM\times\bN}\,\field_{\Gamma_f} 
&\defeq \reeimv{q_{13\sep!!}}(\opb{q_{12}}\iota_{\bL\times\bM}\,\field_{\Gamma_g} \tens \opb{q_{23}}\iota_{\bM\times\bN}\,\field_{\Gamma_f}) \\
&\simeq \reeimv{q_{13\sep!!}}\iota_{\bL\times\bM\times\bN}(\opb{q_{12}}\field_{\Gamma_g} \tens \opb{q_{23}}\field_{\Gamma_f}) \\
&\underset{(*)}\simeq \roimv{q_{13\sep*}}\iota_{\bL\times\bM\times\bN}(\opb{q_{12}}\field_{\Gamma_g} \tens \opb{q_{23}}\field_{\Gamma_f}) \\
&\simeq \iota_{\bL\times\bN}\,\roimv{q_{13\sep*}}(\opb{q_{12}}\field_{\Gamma_g} \tens \opb{q_{23}}\field_{\Gamma_f}) \\
&\underset{(*)}\simeq \iota_{\bL\times\bN}\,\reimv{q_{13\sep!}}(\opb{q_{12}}\field_{\Gamma_g} \tens \opb{q_{23}}\field_{\Gamma_f})\\
&\simeq \iota_{\bL\times\bN}\,\field_{\Gamma_{f\circ g}}.
\end{align*}
Here, in $(*)$, we used the fact that
$\supp(\opb{q_{12}}\field_{\Gamma_g} \tens \opb{q_{23}}\field_{\Gamma_f}) \subset \overline\Gamma_g \times_{\bM} \overline\Gamma_f$ is proper over $\bL\times\bN$, 
which  follows from the same arguments as in the proof of Lemma~\ref{lem:bordcom}.
\end{proof}

\begin{corollary}
If $f\colon (M,\bM) \to (N,\bN)$ is an isomorphism of bordered spaces, then $\roim f \simeq \reeim f$ and $\opb f \simeq \epb f$. Moreover, $\roim f$ and $\opb f$ are quasi-inverse to each other.
\end{corollary}

\begin{lemma}
\label{lem:f=jfj}
Let $f\colon (M,\bM) \to (N,\bN)$ be the morphism of bordered spaces associated with a continuous map $\check f\colon \bM\to \bN$ such that $\check f(M) \subset N$. Then
\bnum
\item
For $F\in\BDC(\ifield_{(M,\bM)})$ there are isomorphisms in $\BDC(\ifield_{(N,\bN)})$
\[
\reeim f F \simeq \opb{j_N}\reeim{\check f}\reeimv{j_{M\sep!!}}F,\quad
\roim f F \simeq \opb{j_N}\roim{\check f}\roimv{j_{M\sep*}}F.
\]
\item
For $G\in\BDC(\ifield_{(N,\bN)})$ there are isomorphisms in $\BDC(\ifield_{(M,\bM)})$
\begin{align*}
\opb f G &\simeq \opb{j_M}\opb{\check f}\reeimv{j_{N\sep!!}}G \simeq \opb{j_M}\opb{\check f}\roimv{j_{N\sep*}}G,\\
\epb f G &\simeq \opb{j_M}\epb{\check f}\roimv{j_{N\sep*}}G \simeq \opb{j_M}\epb{\check f}\reeimv{j_{N\sep!!}}G.
\end{align*}
\ee
\end{lemma}

\begin{proof}
We have a commutative diagram
\[
\xymatrix{
(M,\bM) \ar[r]^-{j_M} \ar[d]^f & \bM \ar[d]^{\check f} \\
(N,\bN) \ar[r]^-{j_N} & \bN.
}
\]
Hence Lemma~\ref{lem:bcomp} implies
\[
\reeimv{j_{N\sep!!}}\reeim f F \simeq \reeim{\check f}\reeimv{j_{M\sep!!}} F.
\]
Then, by Lemma~\ref{lem:jM} we have
\[
\reeim f F \simeq \opb{j_N}\reeimv{j_{N\sep!!}}\reeim f F \simeq \opb{j_N}\reeim{\check f}\reeimv{j_{M\sep!!}} F .
\]
We can similarly obtain the other statements, except
\begin{align*}
\opb{j_M}\opb{\check f}\reeimv{j_{N\sep!!}}G &\simeq \opb{j_M}\opb{\check f}\roimv{j_{N\sep*}}G,\\
\opb{j_M}\epb{\check f}\roimv{j_{N\sep*}}G &\simeq \opb{j_M}\epb{\check f}\reeimv{j_{N\sep!!}}G.
\end{align*}
Since the proofs are similar, let us check only the last isomorphism.

For $K\in\BDC(\ifield_{\bN})$, we have
\begin{align*}
\rihom(\field_M, \epb{\check f}\rihom(\field_N,K))
&\simeq \rihom(\field_M, \rihom(\opb{\check f}\field_N,\epb{\check f}K)) \\
&\simeq \rihom(\field_M \tens \opb{\check f}\field_N,\epb{\check f}K) \\
&\simeq \rihom(\field_M,\epb{\check f}K).
\end{align*}
Hence, applying this for $K = \reeimv{j_{N\sep!!}}G, \roimv{j_{N\sep*}}G$, we obtain
\begin{align*}
\opb{j_M}\epb{\check f}\roimv{j_{N\sep*}}G 
&\simeq \opb{j_M}\rihom(\field_M,\epb{\check f}\roimv{j_{N\sep*}}G ) \\
&\simeq \opb{j_M}\rihom(\field_M,\epb{\check f}\rihom(\field_N,\roimv{j_{N\sep*}}G) ) \\
&\simeq \opb{j_M}\rihom(\field_M,\epb{\check f}\rihom(\field_N,\reeimv{j_{N\sep!!}}G) ) \\
&\simeq \opb{j_M}\rihom(\field_M,\epb{\check f}\reeimv{j_{N\sep!!}}G ) \\
&\simeq \opb{j_M}\epb{\check f}\reeimv{j_{N\sep!!}}G .
\end{align*}
\end{proof}

\begin{proposition}
\label{pro:bproj}
Let $f\colon (M,\bM) \to (N,\bN)$ be a morphism of bordered spaces.
For $F\in\BDC(\ifield_{(M,\bM)})$ and $G,G_1,G_2\in\BDC(\ifield_{(N,\bN)})$, one has isomorphisms
\begin{align*}
\reeim f(\opb f G \tens F) & \simeq G \tens \reeim f F, \\
\opb f (G_1\tens G_2) &\simeq \opb f G_1 \tens \opb f G_2, \\
\rihom(G,\roim f F) & \simeq \roim f \rihom(\opb f G,F), \\
\rihom(\reeim f F, G) & \simeq \roim f \rihom(F, \epb f G), \\
\epb f \rihom(G_1,G_2) & \simeq \rihom(\opb f G_1, \epb f G_2),
\end{align*}
and a morphism
\[
\opb f \rihom(G_1,G_2) \to \rihom(\opb f G_1,\opb f G_2).
\]
\end{proposition}

\begin{proof}
By Lemma~\ref{lem:fborddecomp}, replacing $(M,\bM)$ with $(\Gamma_f,\overline\Gamma_f)$, we may assume that there is a commutative diagram
\[
\xymatrix@C=8ex{
(M,\bM) \ar@{^(->}[r]^-{j_M} \ar[d]^f & \bM \ar[d]^{\check f} \\
(N,\bN) \ar@{^(->}[r]^-{j_N} & \bN.
}
\]
Then, by Lemmas~\ref{lem:f=jfj} and \ref{lem:jM} one has
\begin{align*}
\reeim f(\opb f G \tens F) 
&\simeq \opb{j_N}\reeim{\check f}\reeimv{j_{M\sep!!}}(\opb{j_M}\opb{\check f}\reeimv{j_{N\sep!!}} G \tens F)  \\
&\simeq \opb{j_N}\reeim{\check f}(\reeimv{j_{M\sep!!}}\opb{j_M}\opb{\check f}\reeimv{j_{N\sep!!}} G \tens \reeimv{j_{M\sep!!}} F) \\
&\simeq \opb{j_N}\reeim{\check f}(\opb{\check f}\reeimv{j_{N\sep!!}} G \tens \field_M \tens \reeimv{j_{M\sep!!}} F) \\
&\simeq \opb{j_N}\reeim{\check f}(\opb {\check f}\reeimv{j_{N\sep!!}} G \tens \reeimv{j_{M\sep!!}} F)  \\
&\simeq \opb{j_N}(\reeimv{j_{N\sep!!}} G \tens \reeim{\check f}\reeimv{j_{M\sep!!}} F)  \\
&\simeq G \tens \opb{j_N}\reeim{\check f}\reeimv{j_{M\sep!!}} F  \\
&\simeq G \tens \reeim f F.
\end{align*}
This proves the first isomorphism in the statement. The other isomorphisms can be proved along the same lines.
\end{proof}

\begin{lemma}
\label{lem:bcart}
Consider a Cartesian diagram in the category of bordered spaces
\begin{equation*}
\xymatrix@C=8ex{
(M',\bM') \ar[r]^{f'} \ar[d]^{g'} & (N',\bN') \ar[d]^{g} \\
(M,\bM) \ar[r]^{f}\ar@{}[ur]|-\square & (N,\bN).
}
\end{equation*}
Then there are isomorphisms of functors $\BDC(\ifield_{(M',\bM')}) \to \BDC(\ifield_{(N,\bN)})$
\[
\opb g\reeim f \simeq \reeim {f'} g^{\prime-1}, \qquad
\epb g\roim f \simeq \roim f' g^{\prime!}.
\]
\end{lemma}

\begin{proof}
By a similar argument as in the proof of the Proposition~\ref{pro:bproj}, the statement can be reduced to the corresponding statement for a Cartesian diagram
\[
\xymatrix@C=7ex@R=4ex{
\bM' \ar[r]^{\check f'} \ar[d]^{\check g'} & \bN' \ar[d]^{\check g} \\
\bM \ar[r]_{\check f}\ar@{}[ur]|-\square & \bN.
}
\]
\end{proof}

\begin{definition}\label{def:proper}
We say that a morphism of bordered spaces $f\colon (M,\bM) \to (N,\bN)$ is \emph{proper} if the following two conditions hold:
\bnum
\item $f\colon M \to N$ is proper,
\item the projection $\overline\Gamma_f \to \bN$ is proper.
\ee
\end{definition}

\begin{lemma}
\label{lem:proper}
A morphism $f\colon (M,\bM) \to (N,\bN)$ is proper if and only  if the following two conditions hold:
\bna
\item $\overline\Gamma_f \times_{\bN} N \subset \Gamma_f$.
\item the projection $\overline\Gamma_f \to \bN$ is proper.
\ee
\end{lemma}

\begin{proof}
Assume (a) and (b). Then $M \simeq \overline\Gamma_f \times_{\bN} N \to N$ is proper. Hence $f$ is proper.

Conversely, assume that $f\colon (M,\bM) \to (N,\bN)$ is proper.
Since the composite $f\colon M\to \overline\Gamma_f \times_{\bN} N\to N$ is proper,
it follows that $M\to \overline\Gamma_f \times_{\bN} N$ is proper. Hence
$\Gamma_f$ is a closed subset of $\overline\Gamma_f \times_{\bN} N$. It follows that
\[
\overline\Gamma_f \cap (\overline\Gamma_f \times_{\bN} N) = \Gamma_f.
\]
\end{proof}

\begin{proposition}
Assume that $f\colon (M,\bM) \to (N,\bN)$ is proper. Then $\reeim f \simeq \roim f$ as functors $\BDC(\ifield_{(M,\bM)}) \to \BDC(\ifield_{(N,\bN)})$.
\end{proposition}

\begin{proof}
Consider the projections
$\bM \from[\,\;p_1\,\;] \overline\Gamma_f \To[p_2] \bN$.
For $F\in\BDC(\ifield_{(M,\bM)})$,
we have the isomorphisms (cf.\ Lemma~\ref{lem:fborddecomp})
\begin{align*}
\reeim f F &\simeq \opb{j_N}\reeimv{p_{2\sep!!}}\opb{p_1}\reeimv{j_{M\sep!!}}F, \\
\roim f F &\simeq \opb{j_N}\roimv{p_{2\sep*}}\epb{p_1}\roimv{j_{M\sep*}} F \\
&\simeq \opb{j_N}\reeimv{p_{2\sep!!}}\epb{p_1}\roimv{j_{M\sep*}} F,
\end{align*}
where the last isomorphism follows from Definition~\ref{def:proper}~(ii).
Hence, it is enough to prove that
\[
\field_N \tens \reeimv{p_{2\sep!!}}L \simeq 0,
\]
where $L$ enters the distinguished triangle
\[
L \To \opb{p_1}\reeimv{j_{M\sep!!}} F \To \epb{p_1}\roimv{j_{M\sep*}} F \To[+1].
\]
Since $\opb{p_1}M\to M$ is an isomorphism, one has
\[
\field_{\opb{p_1}M} \tens \epb{p_1}\roimv{j_{M\sep*}} F \simeq
\field_{\opb{p_1}M} \tens \opb{p_1} \roimv{j_{M\sep*}}F\simeq
\opb{p_1} \reeimv{j_{M\sep!!}}F.
\]
Hence $\field_{\opb{p_1}M} \tens L \simeq 0$. Then one has
\begin{align*}
\field_N \tens \reeimv{p_{2\sep!!}} L 
&\simeq \reeimv{p_{2\sep!!}}(\field_{\opb{p_2}N} \tens L) \\
&\simeq \reeimv{p_{2\sep!!}}(\field_{\opb{p_2}N} \tens \field_{\opb{p_1}M} \tens L) 
\simeq 0,
\end{align*}
where the second isomorphism follows from the inclusion $\opb{p_2}N \subset \opb{p_1}M$ due to Lemma~\ref{lem:proper}~(a).
\end{proof}

\begin{definition}
Let $f\colon M \to N$ be a continuous map of good spaces. We say that $f$ is \emph{topologically submersive} if, for any point $x\in M$, there exist an open neighborhood $U$ of $x$ and a commutative diagram
\[
\xymatrix@C=6ex{
\db{U} \ar[r]^{f|_U} \ar@{^(->}[d]^{i} & N \\
S\times N \ar[ur]_{q_2},
}
\]
where $S$ is a subanalytic space, $q_2$ is the projection, and $i$ is an open embedding.
\end{definition}

The following proposition follows from Proposition~\ref{pro:const} and Corollary~\ref{cor:exthom}.

\begin{proposition}
\label{pro:topsub}
Let $f\colon (M,\bM) \to (N,\bN)$ be a morphism of bordered spaces. Assume that $f\colon M\to N$ is topologically submersive. Then, for any $L,G\in\BDC(\ifield_{(N,\bN)})$ there are isomorphisms in $\BDC(\ifield_{(M,\bM)})$
\begin{align*}
\opb f\rihom(L,G)
&\isoto \rihom(\opb f L,\opb f G), \\
\epb f\field_N\tens\opb f G
&\isoto \epb f G.
\end{align*}
\end{proposition}

\begin{lemma}\label{lem:epbetens}
For $k=1,2$, let $f_k\colon(M_k,\bM_k)\to(N_k,\bN_k)$ be a morphism of bordered spaces and $L_k\in\BDC(\ifield_{(N_k,\bN_k)})$. Set $f=f_1\times f_2$. Then there is a canonical morphism
\begin{equation}
\label{eq:epbetens}
\epb{f_1}L_1 \etens \epb{f_2}L_2
 \to \epb f ( L_1\etens L_2).
\end{equation}
\end{lemma}

\begin{proof}
There are morphisms
\[
\reeim f(\epb{f_1} L_1 \etens \epb{f_2} L_2) \simeq \reeimv{f_{1\sep!!}}\epb{f_1} L_1 \etens \reeimv{f_{2\sep!!}}\epb{f_2}L_2 \to L_1\etens L_2,
\]
and the desired morphism follows by adjunction.
\end{proof}

Note that the morphism \eqref{eq:epbetens} is not an isomorphism in general.

\begin{remark}
\label{rem:coperM}
For a bordered space $(M,\bM)$, consider the natural functor
\[
\iota_{(M,\bM)}\colon \BDC(\field_M) \hookrightarrow \BDC(\ifield_{(M,\bM)}).
\]
Then, for $f\colon (M,\bM) \to (N,\bN)$ a morphism of bordered spaces, one has
\begin{align*}
\iota_{(M,\bM)}\circ\opb f &\simeq \opb f \circ\iota_{(N,\bN)}, &
\iota_{(M,\bM)}\circ\epb f &\simeq \epb f \circ\iota_{(N,\bN)}, \\
\roim f \circ \iota_{(M,\bM)} &\simeq \iota_{(N,\bN)} \circ \roim f.
\end{align*}
Moreover, if the projection $\overline\Gamma_f\to\bN$ is proper, then
\[
\reeim f \circ \iota_{(M,\bM)} \simeq \iota_{(N,\bN)} \circ \reim f.
\]
\end{remark}

\subsection{$t$-structure}\label{subse:t-str}

Let $(M,\bM)$ be a bordered space and let $j\colon (M,\bM) \to \bM$ be the natural morphism. 

\begin{notation}
\bnum
\item
Let $\ind_M(\field_{\bM})$ be the full subcategory of $\ind(\field_{\bM})$ consisting of ind-sheaves $F$ on $\bM$ such that $\field_M\tens F \simeq F$.
\item
Let $\ind(\field_{(M,\bM)})$ be the quotient category $\ind(\field_{\bM})/\ind(\field_{\bM\setminus M})$.
\ee
\end{notation}

Note that $\ind_M(\field_{\bM})$ is an abelian category.

\begin{lemma}
\bnum
\item
The composition $\ind_M(\field_{\bM}) \to \ind(\field_{\bM}) \to \ind(\field_{(M,\bM)})$ is an equivalence of categories.
\item
There is an equivalence $\BDC(\ind_M(\field_{\bM}))\simeq \BDC(\ifield_{(M,\bM)})$.
\ee
\end{lemma}

Let us denote by $(\derd^{\leq 0}(\ifield_{(M,\bM)}), \derd^{\geq 0}(\ifield_{(M,\bM)}))$ the $t$-structure of $\BDC(\ifield_{(M,\bM)})$ induced by the canonical $t$-structure of $\BDC(\ind_M(\field_{\bM}))$.
By the definition, we have
\begin{align*}
\derd^{\leq 0}(\ifield_{(M,\bM)}) 
&= \{ F\in\BDC(\ifield_{(M,\bM)}) \semicolon H^n(\reeimv{j_{M\sep!!}}F)=0 \text{ for }n>0 \}, \\
\derd^{\geq 0}(\ifield_{(M,\bM)}) 
&= \{ F\in\BDC(\ifield_{(M,\bM)}) \semicolon H^n(\reeimv{j_{M\sep!!}}F)=0 \text{ for }n<0 \}.
\end{align*}

The following two propositions are easily obtained.

\begin{proposition}
\label{pro:t-str0}
\bnum
\item
The functor $\tens$ is exact, i.e.\ it induces functors
\begin{align*}
\tens &\colon \derd^{\leq 0}(\ifield_{(M,\bM)})\times \derd^{\leq 0}(\ifield_{(M,\bM)}) \to \derd^{\leq 0}(\ifield_{(M,\bM)}), \\
\tens &\colon \derd^{\geq 0}(\ifield_{(M,\bM)})\times\derd^{\geq 0}(\ifield_{(M,\bM)}) \to \derd^{\geq 0}(\ifield_{(M,\bM)}).
\end{align*}
\item
The functor $\rihom$ is left exact, i.e.\ it induces a functor
\[
\rihom \colon \derd^{\leq 0}(\ifield_{(M,\bM)})^\op \times \derd^{\geq 0}(\ifield_{(M,\bM)}) \to \derd^{\geq 0}(\ifield_{(M,\bM)}).
\]
\ee
\end{proposition}

\begin{proposition}
\label{pro:t-str}
Let $f\colon(M,\bM) \to (N,\bN)$ be a morphism of bordered spaces.
\bnum
\item
$\reeim f$ and $\roim f$ are left exact, i.e.\ they induce functors
\[
\reeim f,\roim f \colon \derd^{\geq 0}(\ifield_{(M,\bM)}) \to \derd^{\geq 0}(\ifield_{(N,\bN)}).
\]
\item
$\opb f$ is exact, i.e.\ it induces functors
\begin{align*}
\opb f &\colon \derd^{\leq 0}(\ifield_{(N,\bN)}) \to \derd^{\leq 0}(\ifield_{(M,\bM)}), \\
\opb f &\colon \derd^{\geq 0}(\ifield_{(N,\bN)}) \to \derd^{\geq 0}(\ifield_{(M,\bM)}).
\end{align*}
\item
Let $d\in\Z_{\geq 0}$ and assume that $\opb f(y)\subset M$ has soft-dimension $\leq d$ for any $y\in N$. Then 
\bna
\item
$\reeim f(\ast)[d]$ is right exact, i.e.\ $\reeim f$ induces a functor
\[
\reeim f \colon \derd^{\leq 0}(\ifield_{(M,\bM)}) \to \derd^{\leq d}(\ifield_{(N,\bN)}).
\]
\item
$\epb f(\ast)[-d]$ is left exact, i.e.\ $\epb f$ induces a functor
\[
\epb f \colon \derd^{\geq 0}(\ifield_{(N,\bN)}) \to \derd^{\geq -d}(\ifield_{(M,\bM)}).
\]
\ee\ee
\end{proposition}

We denote by
\begin{equation}
\label{eq:Hn}
H^n \colon \BDC(\ifield_{(M,\bM)}) \to \derd^0(\ifield_{(M,\bM)})
\end{equation}
the cohomology functor, where we set
\[
\derd^0(\ifield_{(M,\bM)})=\derd^{\leq 0}(\ifield_{(M,\bM)}) \cap \derd^{\geq 0}(\ifield_{(M,\bM)}) \simeq \ind(\field_{(M,\bM)}).
\]

\section{Enhanced ind-sheaves}\label{se:enhcdind}

In this section we start by adapting Tamarkin's construction to the ind-sheaf framework, introducing the category of enhanced ind-sheaves $\BEC M$. This is   a quotient category of $\BDC(\ifield_{M\times\R_\infty})$, where we consider the bordered space $\R_\infty = (\R,\R\dunion\{+\infty,-\infty\})$ instead of the real line $\R$. We show that $\BEC M$ has a structure of tensor category by convolution.
We then go on to discuss internal and external operations for enhanced ind-sheaves. In $\BEC M$ we also introduce the notions of stable object 
and of $\R$-constructible object.

\subsection{Convolution}\label{sse:bordconv}

Consider the 2-point compactification of the real line 
$\overline\R \defeq \R\dunion\{+\infty,-\infty\}$. Denote by $\PR=\R\dunion\{\infty\}$ the real projective line. Then $\overline\R$ has a structure of subanalytic space such that the natural map $\overline\R\to\PR$ is a subanalytic map.

\begin{notation}
\label{not:Rinfty}
Instead of the real line, we will consider the bordered space
\[
\R_\infty \defeq (\R,\overline\R).
\]
\end{notation}

Note that $\R_\infty$ is isomorphic to $(\R,\PR)$ as a bordered space.

Consider the morphisms of bordered spaces
\eq
&&\ba{rl}
a &\colon \R_\infty \to \R_\infty, \\[1ex]
\mu,\sigma,q_1,q_2 &\colon \R_\infty^2 \to \R_\infty,
\ea
\label{eq:muq1q2}
\eneq
where $a(t) = -t$, $\mu(t_1,t_2) = t_1+t_2$, $\sigma(t_1,t_2) = t_2 - t_1$ and $q_1,q_2$ are the natural projections.
For a good topological space $M$, we will use the same notations for the associated morphisms
\begin{align*}
a &\colon M\times\R_\infty \to M\times\R_\infty, \\
\mu,\sigma,q_1,q_2 &\colon M\times\R_\infty^2 \to M\times\R_\infty.
\end{align*}
Consider also the natural morphisms
\begin{align*}
\xymatrix@C=.5em{
M\times\R_\infty \ar[rr]^j \ar[dr]_\pi && M\times\overline\R \ar[dl]^{\overline\pi} \\
&M.
}
\end{align*}

When we want to emphasize $M$, we write $\pi_M$, $\overline\pi_M$, $j_M$, $\mu_M$, etc., instead of $\pi$, $\overline\pi$, $j$, $\mu$, etc.

\begin{definition}
The functors
\begin{align*}
\ctens &\colon \BDC(\ifield_{M\times\R_\infty}) \times \BDC(\ifield_{M\times\R_\infty}) \to \BDC(\ifield_{M\times\R_\infty}), \\
\cihom &\colon \BDC(\ifield_{M\times\R_\infty})^\op \times \BDC(\ifield_{M\times\R_\infty}) \to \BDC(\ifield_{M\times\R_\infty})
\end{align*}
are defined by
\begin{align*}
K_1\ctens K_2 &= \reeim \mu (\opb{q_1} K_1 \tens \opb{q_2} K_2), \\
\cihom(K_1,K_2)&= \roimv{q_{1\sep*}} \rihom(\opb{q_2} K_1, \epb\mu K_2).
\end{align*}
\end{definition}

\begin{remark}
As in Remark~\ref{rem:coperM}, let
\[
\iota_{M\times\R_\infty} \colon \BDC(\field_{M\times\R}) \to \BDC(\ifield_{M\times\R_\infty})
\]
be the natural functor. Then, for $F_1,F_2\in\BDC(\field_{M\times\R})$ we have
\begin{align*}
\iota_{M\times\R_\infty}(F_1)\ctens \iota_{M\times\R_\infty}(F_2) &\simeq \iota_{M\times\R_\infty}(\reim\mu(\opb{q_1} F_1\tens\opb{q_2} F_2)), \\
\cihom(\iota_{M\times\R_\infty}(F_1), \iota_{M\times\R_\infty}(F_2)) &\simeq \iota_{M\times\R_\infty}(\roimv{q_{1\sep*}} \rhom(\opb{q_2} F_2, \epb\mu F_1)).
\end{align*}
\end{remark}

The following lemma is obvious.

\begin{lemma}
\label{lem:musigma}
Let $K_1,K_2\in\BDC(\ifield_{M\times\R_\infty})$. Then one has
\begin{align*}
K_1\ctens K_2 &\simeq K_2\ctens K_1 \\
&\simeq \reeimv{q_{2\sep!!}} (\opb{q_1} K_1 \tens \opb\sigma K_2)  \\
&\simeq \reeimv{q_{1\sep!!}} (\opb{q_2}\opb a K_1 \tens \opb \mu K_2), \\
\cihom(K_1,K_2) 
&\simeq \roim \mu \rihom(\opb{q_2}\opb a K_1, \epb{q_1}K_2) \\
&\simeq \roimv{q_{1\sep*}} \rihom(\opb \sigma K_1, \epb{q_2} K_2).
\end{align*}
\end{lemma}

\begin{proposition}
\label{pro:ctenscihom}
For $K_1,K_2,K_3\in\BDC(\ifield_{M\times\R_\infty})$ one has
\begin{align*}
(K_1\ctens K_2) \ctens K_3 &\simeq
K_1 \ctens (K_2 \ctens K_3), \\
\Hom[\BDC(\ifield_{M\times\R_\infty})](K_1\ctens K_2,K_3) 
&\simeq \Hom[\BDC(\ifield_{M\times\R_\infty})](K_1,\cihom(K_2,K_3)), \\
\cihom(K_1\ctens K_2,K_3) &\simeq
\cihom(K_1,\cihom(K_2,K_3)).
\end{align*}
In particular, for $K\in\BDC(\ifield_{M\times\R_\infty})$, the functor $K\ctens\ast$ is left adjoint to $\cihom(K,\ast)$.
\end{proposition}

\begin{proof}
(i)
Consider the morphisms of bordered spaces
\[
q'_1,q'_2,q'_3,\mu'\colon M\times\R^3_\infty\to M\times\R_\infty
\]
where $q'_1,q'_2,q'_3$ are induced by the projections $\R^3\to\R$ and $\mu'$ is induced by $\R^3\owns(t_1,t_2,t_3)\mapsto t_1+t_2+t_3\in\R$. Then one can easily prove that both $(K_1\ctens K_2) \ctens K_3$ and $K_1 \ctens (K_2 \ctens K_3)$ are isomorphic to
\[
\reeim{\mu'}(\opbv{q^{\prime\sep-1}_1}K_1\tens\opbv{q^{\prime\sep-1}_2}K_2\tens\opbv{q^{\prime\sep-1}_3}K_3).
\]

\smallskip\noindent (ii)
Writing $\Hom$ instead of $\Hom[\BDC(\ifield_{M\times\R_\infty})]$, one has
\begin{align*}
\Hom(K_1\ctens K_2,K_3) 
&= \Hom(\reeim\mu(\opb{q_1}K_1\tens\opb{q_2}K_2),K_3) \\
&\simeq \Hom(\opb{q_1}K_1\tens\opb{q_2}K_2,\epb\mu K_3) \\
&\simeq \Hom(\opb{q_1}K_1,\rihom(\opb{q_2}K_2,\epb\mu K_3)) \\
&\simeq \Hom(K_1,\roimv{q_{1\sep*}}\rihom(\opb{q_2}K_2,\epb\mu K_3)) \\
&= \Hom(K_1,\cihom(K_2,K_3)).
\end{align*}

\smallskip\noindent (iii)
Writing again $\Hom$ instead of $\Hom[\BDC(\ifield_{M\times\R_\infty})]$, one has for any $K\in\BDC(\ifield_{M\times\R_\infty})$
\begin{align*}
\Hom(K, \cihom(K_1\ctens K_2,K_3))
&\simeq \Hom(K \ctens (K_1\ctens K_2) ,K_3) \\
&\simeq \Hom((K \ctens K_1)\ctens K_2 ,K_3) \\
&\simeq \Hom(K \ctens K_1, \cihom (K_2 ,K_3)) \\
&\simeq \Hom\bl K, \cihom (K_1, \cihom (K_2 ,K_3))\br.
\end{align*}
Hence, by Yoneda, one obtains
\[
\cihom(K_1\ctens K_2,K_3) \simeq \cihom (K_1, \cihom (K_2 ,K_3)).
\]
\end{proof}

\subsection{Idempotent objects}

We set
\begin{align*}
\field_{\{t\geq 0\}} &= \field_{\{(x,t)\in M\times\overline\R\semicolon t\in\R,\ t \geq 0\}}, \\
\field_{\{t = 0\}} &= \field_{\{(x,t)\in M\times\overline\R\semicolon t = 0 \}}, 
\end{align*}
and we use similar notation for $\field_{\{t> 0\}}$, $\field_{\{t\leq 0\}}$ and $\field_{\{t< 0\}}$.
These are sheaves on $M\times\overline\R$ whose stalk vanishes at points of $M\times(\overline\R\setminus\R)$. 
We also regard them as objects of $\BDC(\ifield_{M\times\R_\infty})$.

\begin{lemma}
For $K\in\BDC(\ifield_{M\times\R_\infty})$ there are isomorphisms
\[
\field_{\{t= 0\}} \ctens K \simeq K \simeq \cihom(\field_{\{t= 0\}}, K).
\]
More generally, for $a\in\R$, we have
\[
\field_{\{t= a\}} \ctens K \simeq \roimv{\mu_{a\sep*}} K \simeq \cihom(\field_{\{t= -a\}}, K),
\]
where $\mu_a\colon M\times\R_\infty\to M\times\R_\infty$ is the morphism induced by the translation $t\mapsto t+a$.
\end{lemma}

\begin{corollary}
The category $\BDC(\ifield_{M\times\R_\infty})$ has a structure of commutative tensor category with $\ctens$ as tensor product bifunctor and $\field_{\{t= 0\}}$ as unit object.
\end{corollary}

There are distinguished triangles in $\BDC(\ifield_{M\times\R_\infty})$
\begin{equation}
\label{eq:dtpm0}
\begin{cases}
\field_{\{t\geq 0\}} \To \field_{\{t = 0\}} \To \field_{\{t> 0\}}[1] \To[+1], \\
\field_{\{t\geq 0\}} \To \field_{\{t< 0\}}[1] \To \field_{M\times\R}[1] \To[+1], \\
\field_{\{t\geq 0\}} \dsum \field_{\{t\leq 0\}} \To \field_{\{t= 0\}} \To \field_{M\times\R}[1] \To[+1].
\end{cases}
\end{equation}

The following lemma is easily verified.

\begin{lemma}\label{lem:nilpo}
There are isomorphisms in $\BDC(\ifield_{M\times\R_\infty})$
\begin{align*}
\field_{\{t\geq 0\}} \ctens \field_{\{t\geq 0\}} &\isoto \field_{\{t\geq 0\}}, &
\field_{\{t\geq 0\}} \ctens \field_{\{t > 0\}}[1] &\simeq 0,\\
\field_{\{t> 0\}}[1] \ctens \field_{\{t> 0\}}[1] &\isofrom \field_{\{t > 0\}}[1], &
\field_{\{t\geq 0\}} \ctens \field_{M\times\R}[1] &\simeq 0,\\
\field_{M\times\R}[1] \ctens \field_{M\times\R}[1] &\isofrom \field_{M\times\R}[1], &
\field_{\{t > 0\}}[1] \ctens \field_{M\times\R}[1] &\isofrom \field_{M\times\R}[1],\\
\field_{\{t\geq 0\}} \ctens \field_{\{t \leq 0\}} &\simeq 0,&
\field_{\{t\geq 0\}} \ctens \field_{\{t < 0\}}[1] &\isofrom \field_{\{t\geq 0\}},\\
\field_{\{t > 0\}}[1] \ctens \field_{\{t < 0\}}[1] &\simeq \field_{M\times\R}[1].
\end{align*}
Hence, the objects $\field_{\{t\geq 0\}}$, $\field_{\{t > 0\}}[1]$, $\field_{\{t\geq 0\}} \dsum \field_{\{t\leq 0\}}$ and $\field_{M\times\R}[1]$ are idempotents in $\BDC(\ifield_{M\times\R_\infty})$.
\end{lemma}

Recall that an idempotent in a tensor category is a pair $(P,\xi)$ of an object $P$ and an isomorphism $\xi\colon P\tens P\to P$ such that $\xi\tens P = P \tens \xi$ as morphisms $P\tens P\tens P\to P\tens P$ (cf.~\cite[Lemma 4.1.2]{KS06}). Note that in each distinguished triangle $P'\to P\to P''\to[+1]$ in \eqref{eq:dtpm0}, $P$, $P'$, $P''$ are idempotents and $P'\ctens P''\simeq 0$, $P\ctens P'\simeq P'$, $P\ctens P''\simeq P''$.

\begin{corollary}
Let $K\in\BDC(\ifield_{M\times\R_\infty})$. Then
\begin{align*}
\field_{\{t\geq 0\}} \ctens K  \isoto K
&\iff \field_{\{t> 0\}}[1] \ctens K \simeq 0 \\
&\iff \field_{\{t\leq 0\}} \ctens K \simeq 0 \text{ and } \field_{M\times\R}[1] \ctens K \simeq 0 .
\end{align*}
Moreover,
\[
\field_{\{t\geq 0\}} \ctens K \simeq 0 \iff \field_{\{t> 0\}}[1] \ctens K  \isofrom K.
\]
Similar results hold when replacing the functor $\ast\ctens K$ with the functor $\cihom(\ast,K)$.
\end{corollary}

\subsection{Properties of  convolution}

\begin{lemma}
\label{lem:cihomrihompi}
For $K_1,K_2\in\BDC(\ifield_{M\times\R_\infty})$ and $L\in\BDC(\ifield_M)$ one has
\begin{align*}
\opb\pi L \tens (K_1\ctens K_2) & \simeq (\opb\pi L \tens K_1)\ctens K_2, \\
\rihom(\opb\pi L, \cihom(K_1,K_2)) & \simeq \cihom(\opb\pi L \tens K_1,K_2) \\
& \simeq \cihom(K_1,\rihom(\opb\pi L, K_2)).
\end{align*}
\end{lemma}

\begin{proof}
Since the proofs are similar, let us only discuss the second isomorphism.
Since $\pi\circ q_1 = \pi\circ q_2$, one has
\begin{align*}
\rihom(\opb\pi L, &\cihom(K_1,K_2)) \\
&= \rihom(\opb\pi L, \roimv{q_{1\sep*}}\rihom(\opb{q_2} K_1,\epb\mu K_2)) \\
&\simeq \roimv{q_{1\sep*}}\rihom(\opb{q_1}\opb\pi L, \rihom(\opb{q_2} K_1,\epb\mu K_2)) \\
&\simeq \roimv{q_{1\sep*}}\rihom(\opb{q_2}\opb\pi L, \rihom(\opb{q_2} K_1,\epb\mu K_2)) \\
&\simeq \roimv{q_{1\sep*}}\rihom(\opb{q_2}(\opb\pi L \tens K_1),\epb\mu K_2) \\
& = \cihom(\opb\pi L \tens K_1,K_2).
\end{align*}
\end{proof}

\begin{lemma}
\label{lem:cihomKt0}
For $K\in\BDC(\ifield_{M\times\R_\infty})$ and $L\in\BDC(\ifield_M)$ one has
\begin{align*}
\opb\pi L \tens K & \simeq (\opb\pi L \tens \field_{\{t=0\}})\ctens K, \\
\rihom(\opb\pi L, K) & \simeq \cihom(\opb\pi L \tens \field_{\{t=0\}},K), \\
\opb a \rihom(K,\epb\pi L) &\simeq
\cihom(K,\opb\pi L \tens \field_{\{t=0\}}).
\end{align*}
\end{lemma}

\begin{proof}
The first two isomorphisms follow from Lemma~\ref{lem:cihomrihompi} for $K_1=\field_{\{t=0\}}$ and $K_2=K$. Let us prove the third isomorphism.

Let $\delta^a\colon M\times\R_\infty\to M\times\R_\infty^2$ be the morphism induced by the anti-diagonal map $\R\to\R^2$, $t\mapsto(-t,t)$, and $i_0\colon M\to M\times\R_\infty$ the morphism induced by the inclusion $x\mapsto(x,0)$. Note that $\pi\circ i_0 =\id_M$, $\field_{\{t=0\}} \simeq \roimv{i_{0\sep*}}\field_M$, and there is a Cartesian diagram
\[
\xymatrix{
M\times\R_\infty \ar[r]^{\delta^a} \ar[d]^\pi & M\times\R_\infty^2 \ar[d]^\mu \\
M \ar[r]^-{i_0} \ar@{}[ur]|-\square & M\times\R_\infty.
}
\]
Then we have
\begin{align*}
\cihom(K,\opb\pi L \tens \field_{\{t=0\}}) 
&\simeq
\cihom(K, \roimv{i_{0\sep*}} L) \\
&=
\roimv{q_{1\sep*}}\rihom(\opb{q_2} K, \epb\mu\roimv{i_{0\sep*}} L) .
\end{align*}
On the other hand, $\epb\mu\roimv{i_{0\sep*}} L \simeq \roim \delta^a \epb\pi L$, and hence
\begin{align*}
\roimv{q_{1\sep*}}\rihom(\opb{q_2} K, &\epb\mu\roimv{i_{0\sep*}} L) \\
&\simeq \roimv{q_{1\sep*}}\rihom(\opb{q_2} K, \roim\delta^a\epb\pi L) \\
&\simeq \roimv{q_{1\sep*}}\roim\delta^a\rihom(\opbv{\delta^{a\sep-1}}\opb{q_2} K, \epb\pi L).
\end{align*}
Then the result follows from $q_1\circ \delta^a = a$ and $q_2\circ \delta^a = \id$.
\end{proof}

\begin{lemma}
\label{lem:piihomctens}
For $K_1,K_2,K_3\in\BDC(\ifield_{M\times\R_\infty})$ one has
\[
\roim\pi\rihom(K_1\ctens K_2,K_3) \simeq
\roim\pi\rihom(K_1,\cihom(K_2,K_3)).
\]
\end{lemma}

\begin{proof}
The proof is similar to part (ii) in the proof of Proposition~\ref{pro:ctenscihom}, using Lemma~\ref{lem:cihomrihompi}.
\end{proof}

\begin{lemma}
\label{lem:tenspipi}
For $K_1,K_2\in\BDC(\ifield_{M\times\R_\infty})$ there are isomorphisms
\begin{align*}
\reeim\pi(K_1\ctens K_2) &\simeq \reeim\pi K_1 \tens \reeim\pi K_2, \\
\roim\pi\cihom(K_1,K_2) &\simeq \rihom(\reeim\pi K_1,\roim\pi K_2).
\end{align*}
\end{lemma}

\begin{proof}
Note that $\pi\circ\mu = \pi\circ q_1$ and that there is a Cartesian diagram
\[
\xymatrix{
M\times\R^2_\infty \ar[r]^{q_1} \ar[d]^{q_2} & M\times\R_\infty \ar[d]^\pi \\
M\times\R_\infty \ar[r]^{\pi} \ar@{}[ur]|-\square & M.
}
\]
Then one has
\begin{align*}
\reeim\pi(K_1\ctens K_2) 
& = \reeim\pi\reeim\mu(\opb{q_1}K_1 \tens \opb{q_2}K_2) \\
&\simeq \reeim\pi\reeimv{q_{1\sep!!}}(\opb{q_1}K_1 \tens \opb{q_2}K_2) \\
&\simeq \reeim\pi(K_1 \tens \reeimv{q_{1\sep!!}}\opb{q_2}K_2) \\
&\simeq \reeim\pi(K_1 \tens \opb\pi\reeim\pi K_2) \\
&\simeq \reeim\pi K_1 \tens \reeim\pi K_2.
\end{align*}
The proof of the second isomorphism is similar.
\end{proof}

Since $\reeim\pi\field_{\{t\geq 0\}} \simeq 0$, we have the following result.

\begin{corollary}\label{cor:pieeim}
For any $K\in\BDC(\ifield_{M\times\R_\infty})$, one has
\begin{align*}
&\reeim\pi(\field_{\{t\geq 0\}} \ctens K) \simeq 0, \\
&\roim\pi\cihom(\field_{\{t\geq 0\}} , K) \simeq 0.
\end{align*}
\end{corollary}

\begin{lemma}\label{lem:piRinfty} 
For $K\in\BDC(\ifield_{M\times\R_\infty})$ and $L\in\BDC(\ifield_M)$ one has
\begin{align*}
(\opb\pi L) \ctens K & \simeq \opb\pi(L\tens\reeim\pi K), \\
\cihom(\opb\pi L, K) & \simeq \epb\pi\rihom(L,\roim\pi K), \\
\cihom(K, \epb\pi L) & \simeq \epb\pi\rihom(\reeim\pi K, L).
\end{align*}
\
In particular, we have
\eqn
(\field_{\{t\geq 0\}} \dsum \field_{\{t\leq 0\}}) \ctens \opb\pi L&\simeq& 0,\\
\cihom(\field_{\{t\geq 0\}} \dsum \field_{\{t\leq 0\}},  \opb\pi L)&\simeq&0.
\eneqn
\end{lemma}

\begin{proof}
Since the proofs are similar, let us only consider the second isomorphism.
Note that $\pi\circ q_2 = \pi\circ \mu$, and that there is a Cartesian diagram
\[
\xymatrix{
M\times\R^2_\infty \ar[r]^{\mu} \ar[d]^{q_1} & M\times\R_\infty \ar[d]^\pi \\
M\times\R_\infty \ar[r]^{\pi} \ar@{}[ur]|-\square & M.
}
\]
Then one has
\begin{align*}
\cihom(\opb\pi L, K) 
& = \roimv{q_{1\sep*}} \rihom(\opb{q_2}\opb\pi L,\epb\mu K) \\
& \simeq \roimv{q_{1\sep*}} \rihom(\opb\mu\opb\pi L,\epb\mu K) \\
& \simeq \roimv{q_{1\sep*}} \epb\mu\rihom(\opb\pi L,K) \\
& \simeq \epb\pi\roim\pi\rihom(\opb\pi L,K) \\
& \simeq \epb\pi\rihom(L,\roim\pi K).
\end{align*}
\end{proof}

By the above lemma, noticing that $\opb\pi\field_M \simeq \field_{M\times\R}$, we deduce

\begin{corollary} 
\label{cor:pipi}
For $K\in\BDC(\ifield_{M\times\R_\infty})$ one has
\begin{align*}
\field_{M\times\R} \ctens K &\simeq \opb\pi\reeim\pi K, \\
\cihom(\field_{M\times\R}, K) &\simeq \epb\pi\roim\pi K.
\end{align*}
\end{corollary}

Let us give an alternative description of the functors $\ctens$ and $\cihom$.

\begin{notation}
\label{not:sfS}
Denote by $\mathsf S$ the closure of $\{(t_1,t_2,t_3)\in\R^3\semicolon t_1+t_2+t_3=0 \}$ in $\overline\R^3$. Consider the maps $\tilde q_1,\tilde q_2,\tilde\mu\colon\mathsf S\to \overline\R$ given by $\tilde q_1(t_1,t_2,t_3) = t_1$, $\tilde q_2(t_1,t_2,t_3) = t_2$, $\tilde\mu(t_1,t_2,t_3) = -t_3 = t_1+t_2$, and denote by the same letters the corresponding maps $M\times \mathsf S\to M\times \overline\R$. This is visualized in the following picture, which shows how the three variables behave at infinity:
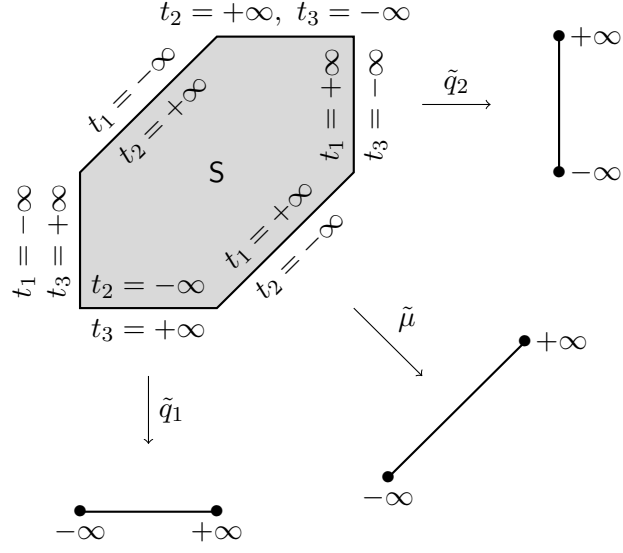
\begin{center}
{\small
\begin{tikzpicture}[x=1.8cm,y=1.8cm]
\draw[thick,fill= gray!30!white,shift={(0,.5)}] (0,1) -- node[above]{$t_2=-\infty$} node[below]{$t_3=+\infty$} %
(1,1) -- node[above,sloped]{$t_1=+\infty$} node[below,sloped]{$t_2=-\infty$} %
(2,2) -- node[above,sloped]{$t_1=+\infty$} node[below,sloped]{$t_3=-\infty$} %
(2,3) -- node[above,sloped]{$t_2=+\infty,\ t_3=-\infty$}%
(1,3) -- node[above,sloped]{$t_1=-\infty$} node[below,sloped]{$t_2=+\infty$}%
(0,2) -- cycle ;
\draw[shift={(0,.5)}] (0,1) -- node[above,sloped,text width=4em]{$t_1=-\infty$\\$t_3=+\infty$} %
(0,2);
\draw[thick] (0,0) node[below]{$-\infty$} node{$\bullet$} -- (1,0) node[below]{$+\infty$} node{$\bullet$};
\draw[thick,shift={(.25,.25)}] (2,0) node[below]{$-\infty$} node{$\bullet$} --  (3,1) node[right]{$+\infty$} node{$\bullet$};
\draw[thick,shift={(.5,.5)}] (3,2) node[right]{$-\infty$} node{$\bullet$} -- (3,3) node[right]{$+\infty$} node{$\bullet$};
\draw[->,shift={(.25,.5)}] (2.25,2.5) -- node[above]{$\tilde q_2$} (2.75,2.5);
\draw[->,shift={(.25,.25)}] (1.75,1.25) -- node[above right]{$\tilde\mu$} (2.25,0.75);
\draw[->,shift={(0,.25)}] (0.5,0.75) -- node[right]{$\tilde q_1$} (0.5,0.25);
\draw (1,2.5) node{$\mathsf{S}$} ;
\end{tikzpicture}
}
\end{center}
\end{notation}

There are commutative diagrams
\[
\vcenter{\vbox{
\xymatrix@C=9ex{
M\times\R^2_\infty \ar[d]_u \ar[r]^k & M\times\mathsf S \ar[d]^{\tilde u} \\
M\times\R_\infty \ar[r]^{j_M} & M\times\overline\R
}}}
\quad\text{for }u=q_1,q_2,\mu,
\]
where $k$ is the morphism associated with the embedding $\R^2\to\mathsf S$
given by  $(t_1,t_2)\mapsto(t_1,t_2,-t_1-t_2)$.

One has
\[
\opb{\tilde\mu}(\{t\neq -\infty\}) \cap \opb{\tilde q_1}(\{t=-\infty\}) \subset \opb{\tilde q_2}(\{t=+\infty\}).
\]
One also has
\begin{align}
\label{eq:q1q2R}
\opb{\tilde q_1}(M\times\R) \cap \opb{\tilde q_2}(M\times\R) &=
\opb{\tilde q_1}(M\times\R) \cap \opb{\tilde \mu}(M\times\R) \\
\notag
&= k(M\times\R^2).
\end{align}
We identify $M\times\R^2$ with an open subset of $M\times\mathsf S$ by $k$.
Then $M\times\R^2_\infty$ is isomorphic to $M\times(\R^2,\mathsf S)$ as a bordered space. 
For $F\in\BDC(\ifield_{M\times\mathsf S})$, one has
\[
\reeim k\opb k F \simeq \field_{M\times\R^2} \tens F,
\quad
\roim k \epb k F \simeq \rihom(\field_{M\times\R^2}, F).
\]
The following lemma is immediate.

\begin{lemma}
\label{lem:ctenstilde}
Let $K_1,K_2\in\BDC(\ifield_{M\times\R_\infty})$. With the above notations, one has isomorphisms
\begin{align*}
K_1\ctens K_2 
&\simeq \opb{j_M}\reeim{\tilde\mu}(\opb{\tilde q_1}\reeimv{j_{M\sep!!}}K_1 \tens \opb{\tilde q_2}\reeimv{j_{M\sep!!}}K_2) \\
&\simeq \opb{j_M}\reeimv{\tilde q_{1\sep!!}}(\opb{\tilde q_2}\reeimv{j_{M\sep!!}}\opb a K_1 \tens \opb{\tilde\mu}\reeimv{j_{M\sep!!}}K_2), \\
\cihom(K_1,K_2) 
&\simeq \opb{j_M}\roimv{\tilde q_{1\sep*}}\rihom(\opb{\tilde q_2}\reeimv{j_{M\sep!!}}K_1, \epb{\tilde\mu}\roimv{j_{M\sep*}}K_2) \\
&\simeq \opb{j_M}\roim{\tilde\mu}\rihom(\opb{\tilde q_2}\reeimv{j_{M\sep!!}}\opb a K_1, \epb{\tilde q_1}\roimv{j_{M\sep*}}K_2).
\end{align*}
\end{lemma}

Let us now state a result which will be fundamental in the next section.

Set for short
\[
\field_{\{t\neq \pm\infty\}} = \field_{\{(x,t)\in M\times\overline\R\semicolon t\neq \pm\infty\}} \in \BDC(\ifield_{M\times\overline\R}).
\]
Recall that $\overline\pi\colon M\times\overline\R \to M$ denotes the projection.

\begin{proposition}\label{pro:tenshomstar}
For $K\in\BDC(\ifield_{M\times\R_\infty})$ there is a distinguished triangle
\begin{equation}
\label{eq:tenshomstarDT}
\opb\pi L \To \field_{\{t\geq 0\}} \ctens K \To \cihom(\field_{\{t\geq 0\}}, K) \To[+1],
\end{equation}
where the object $L\in\BDC(\ifield_M)$ is given by
\begin{align*}
L & = \roim{\overline\pi}(\field_{\{t\neq-\infty\}} \tens \roimv{j_{M\sep*}} K) \\
&\simeq \reeim\pi\cihom(\field_{\{t\geq 0\}}, K) \\
&\simeq \roim\pi(\field_{\{t\geq 0\}} \ctens K).
\end{align*}
\end{proposition}

\begin{proof}
Consider the Cartesian diagram,
\[
\xymatrix@C=8ex{
M\times\R_\infty\times\overline\R \ar[r]^-{\overline q_1} \ar[d]^{\overline q_2} & M\times\R_\infty \ar[d]^\pi \\
M\times\overline\R \ar[r]^{\overline\pi} \ar@{}[ur]|-\square & M.
}
\]
Remark that $\epb{\overline q_2}F \simeq \opb{\overline q_2}F[1]$ for any $F\in\BDC(\ifield_{M\times\overline\R})$.

Let $(t_1,t_2) \in \R^2 \subset\overline\R^2$ be the coordinates. 
In the sequel we will denote by $\{t_2\leq t_1\}$, $\{t_2< t_1\}$, etc., the subsets of $M\times\R^2$ described by these inequalities.
Set
\[
\widetilde K = \roimv{j_{M\sep*}}K \in \BDC(\ifield_{M\times\overline\R}).
\]
One has the isomorphisms
\begin{align*}
\field_{\{t\geq 0\}} \ctens K
&\underset{(1)}\simeq \reeimv{\overline q_{1\sep!!}}(\field_{\{t_2\leq t_1\}} \tens \opb{\overline q_2}\widetilde K) \\
&\simeq \roimv{\overline q_{1\sep*}}(\field_{M\times\R^2} \tens \rihom(\field_{\{t_2 < t_1\}},\field_{M\times\R^2}) \tens \epb{\overline q_2}\widetilde K[-1]) \\
&\underset{(2)}\simeq \roimv{\overline q_{1\sep*}} (\field_{M\times\R^2} \tens \rihom(\field_{\{t_2 < t_1\}}[1],\epb{\overline q_2}\widetilde K)) ,
\end{align*}
where $(1)$ follows from Lemma~\ref{lem:musigma} and $(2)$ from Proposition~\ref{pro:const}.
Similarly, one has the isomorphism
\[
\cihom(\field_{\{t\geq 0\}}, K)
\simeq \roimv{\overline q_{1\sep*}}\rihom(\field_{\{t_1 \leq t_2\}}, \epb{\overline q_2}\widetilde K).
\]
Now, we claim that there are the isomorphisms in $\BDC(\ifield_{M\times\R_\infty\times\overline\R})$
\begin{multline}
\label{eq:tempA}
\field_{M\times\R^2} \tens \rihom(\field_{\{t_2 < t_1\}}[1],\epb{\overline q_2}\widetilde K) \\
\isoto
\field_{M\times\R\times(\overline\R\setminus\{-\infty\})} \tens
\rihom(\field_{\{t_2 < t_1\}}[1],\epb{\overline q_2}\widetilde K)
\end{multline}
and
\begin{multline}
\label{eq:tempB}
\rihom(\field_{\{t_1 \leq t_2\}}, \epb{\overline q_2}\widetilde K) \\
\isofrom
\field_{M\times\R\times(\overline\R\setminus\{-\infty\})} \tens \rihom(\field_{\{t_1 \leq t_2\}}, \epb{\overline q_2}\widetilde K).
\end{multline}
We shall give a proof later. Admitting the above isomorphisms for the moment, let us complete the proof.

We have
\begin{align*}
\field_{\{t\geq 0\}} \ctens K
&\simeq \roimv{\overline q_{1\sep*}} (\field_{M\times\R\times(\overline\R\setminus\{-\infty\})} \tens \rihom(\field_{\{t_2 < t_1\}}[1],\epb{\overline q_2}\widetilde K)), \\
\cihom(\field_{\{t\geq 0\}}, K)
&\simeq \roimv{\overline q_{1\sep*}} (\field_{M\times\R\times(\overline\R\setminus\{-\infty\})} \tens\rihom(\field_{\{t_1 \leq t_2\}}, \epb{\overline q_2}\widetilde K)).
\end{align*}

{}From the distinguished triangle
\[
\field_{\{t_1 \leq t_2\}} \To \field_{\{t_2 < t_1\}}[1] \To \field_{M\times\R^2}[1] \To[+1],
\]
we deduce a distinguished triangle
\[
\tilde L \To \field_{\{t\geq 0\}} \ctens K \To \cihom(\field_{\{t\geq 0\}}, K) \To[+1],
\]
where
\[
\tilde L = \roimv{\overline q_{1\sep*}}(\field_{M\times\R\times(\overline\R\setminus\{-\infty\})} \tens \rihom(\field_{M\times\R^2}[1], \epb{\overline q_2}\widetilde K)).
\]
One has the isomorphisms
\begin{align*}
\rihom(\field_{M\times\R^2}, \epb{\overline q_2}\widetilde K)
&\simeq \rihom(\opb{\overline q_2}\field_{M\times\R}, \epb{\overline q_2}\widetilde K) \\
&\simeq \epb{\overline q_2}\rihom(\field_{M\times\R}, \widetilde K) \\
&\simeq \epb{\overline q_2}\widetilde K.
\end{align*}
Hence,
\begin{align*}
\field_{M\times\R\times(\overline\R\setminus\{-\infty\})} \tens {} \rihom&(\field_{M\times\R^2}[1], \epb{\overline q_2}\widetilde K) \\
&\simeq \opb{\overline q_2}\field_{\{t\neq-\infty\}} \tens \epb{\overline q_2}\widetilde K[-1] \\
&\simeq \opb{\overline q_2}\field_{\{t\neq-\infty\}} \tens \opb{\overline q_2}\widetilde K \\
&\simeq \opb{\overline q_2}(\field_{\{t\neq-\infty\}} \tens \widetilde K) \\
&\simeq \epb{\overline q_2}(\field_{\{t\neq-\infty\}} \tens \widetilde K)[-1].
\end{align*}
It follows that
\begin{align*}
\tilde L 
&\simeq \roimv{\overline q_{1\sep*}}\epb{\overline q_2}(\field_{\{t\neq-\infty\}} \tens \widetilde K)[-1] \\
&\simeq \epb\pi\roim{\overline\pi}(\field_{\{t\neq-\infty\}} \tens \widetilde  K)[-1] \\
&\simeq \opb\pi\roim{\overline\pi}(\field_{\{t\neq-\infty\}} \tens \widetilde  K) .
\end{align*}
We have thus proved \eqref{eq:tenshomstarDT} with
\[
L = \roim{\overline\pi}(\field_{\{t\neq-\infty\}} \tens \roimv{j_{M\sep*}} K).
\]
Applying $\reeim\pi$ to \eqref{eq:tenshomstarDT}, we get a distinguished triangle
\[
\reeim\pi\opb\pi L \To \reeim\pi(\field_{\{t\geq 0\}} \ctens K) \To \reeim\pi\cihom(\field_{\{t\geq 0\}}, K) \To[+1].
\]
Corollary~\ref{cor:pieeim} gives $\reeim\pi(\field_{\{t\geq 0\}} \ctens K) \simeq 0$.
Noticing that $L\simeq\reeim\pi\opb\pi L[1]$, we get
\[
L \simeq \reeim\pi\cihom(\field_{\{t\geq 0\}}, K).
\]
Similarly, applying $\roim\pi$ to \eqref{eq:tenshomstarDT}, we get
\[
L \simeq \roim\pi(\field_{\{t\geq 0\}} \ctens K).
\]

It remains to prove \eqref{eq:tempA} and \eqref{eq:tempB}.
It is enough to show that for any $F\in\BDC(\ifield_{M\times\overline\R^2})$ one has
\begin{align}
\label{eq:tempC}
& \field_{M\times\R\times\{+\infty\}} \tens
\rihom(\field_{\{t_2 < t_1\}},F) \simeq 0, \\
\label{eq:tempD}
& \field_{M\times\R\times\{-\infty\}} \tens \rihom(\field_{\{t_1 \leq t_2\}}, F) \simeq 0.
\end{align}

As in Notation~\ref{not:sfS}, let $\mathsf S$ be the closure of $\{(t_1,t_2,t_3)\in\R^3\semicolon t_1+t_2+t_3 = 0 \}$ in $\overline\R^3$. Consider the map $\tilde p\colon\mathsf S\to\overline\R^2$ given by $\tilde p(t_1,t_2,t_3) = (t_1,t_2)$. Then $\opb {\tilde p}(\R^2) \isoto\R^2$.
We shall denote by the same letter the induced map $\tilde p\colon M\times \mathsf S\to M\times \overline\R^2$.

Since $\reeim {\tilde p}(\field_{\opb {\tilde p}(\{t_2 < t_1\})}) \simeq \field_{\{t_2 < t_1\}}$ and $\reeim {\tilde p}(\field_{\opb {\tilde p}(\{t_1 \leq t_2\})}) 
\simeq \field_{\{t_1 \leq t_2\}}$, we have
\begin{align*}
\rihom(\field_{\{t_2 < t_1\}},F)
&\simeq \roim {\tilde p} \rihom(\field_{\opb {\tilde p} (\{t_2 < t_1\})},\epb {\tilde p} F), \\
\rihom(\field_{\{t_1 \leq t_2\}}, F) 
&\simeq \roim {\tilde p}\rihom(\field_{\opb {\tilde p}(\{t_1 \leq t_2\})}, \epb {\tilde p} F) .
\end{align*}
Then \eqref{eq:tempC} follows from
\begin{multline*}
\field_{M\times\R\times\{+\infty\}} \tens
\rihom(\field_{\{t_2 < t_1\}},F) \\
\simeq
\roim {\tilde p}(\field_{\opb {\tilde p}(M\times\R\times\{+\infty\})} \tens
\rihom(\field_{\opb {\tilde p}(\{t_2 < t_1\})},\epb {\tilde p} F))
\end{multline*}
and
\[
\overline{\opb {\tilde p}(\{t_2 < t_1\})} \cap \opb {\tilde p}(M\times\R\times\{+\infty\}) = \emptyset.
\]
Similarly, \eqref{eq:tempD} follows from
\[
\overline{\opb {\tilde p}(\{t_1 \leq t_2\})} \cap \opb {\tilde p}(M\times\R\times\{-\infty\}) = \emptyset.
\]
\end{proof}

\begin{corollary}
\label{cor:kkK}
For $K\in\BDC(\ifield_{M\times\R_\infty})$, there are isomorphisms
\begin{align*}
\cihom(\field_{\{t\geq 0\}}, \field_{\{t\geq 0\}}\ctens K)
&\isoto \cihom(\field_{\{t\geq 0\}}, K),\\
\field_{\{t\geq 0\}}\ctens \cihom(\field_{\{t\geq 0\}}, K)
&\isoto \field_{\{t\geq 0\}}\ctens K, \\
\cihom(\field_{\{t\leq 0\}}, \field_{\{t\geq 0\}}\ctens K)
&\simeq 0 ,\\
\field_{\{t\leq 0\}}\ctens \cihom(\field_{\{t\geq 0\}}, K)
&\simeq 0.
\end{align*}
\end{corollary}

\begin{proof}
By Lemma~\ref{lem:piRinfty}, for any $L\in\BDC(\ifield_M)$ one has
\[
\cihom(\field_{\{t\geq 0\}}, \epb\pi L) \simeq 0,
\quad
\field_{\{t\geq 0\}} \ctens \opb\pi L \simeq 0.
\]
Recalling also Lemma~\ref{lem:nilpo}, the isomorphisms in the statement follow by applying the functors
$\cihom(\field_{\{\pm t\geq 0\}}, \ast)$ and $\field_{\{\pm t\geq 0\}} \ctens \ast$ to the distinguished triangle \eqref{eq:tenshomstarDT}.
\end{proof}

\begin{notation}
\label{not:psi}
For $K\in\BDC(\ifield_{M\times\R_\infty})$, consider the functors
\[
\psi_{M,\pm\infty}(K) = \opb{i_{M,\pm\infty}}\roimv{j_{M\sep*}}K,
\]
where $i_{M,\pm\infty}\colon M \to M\times\overline\R$ denotes the embedding $x\mapsto (x,\pm\infty)$.
\end{notation}

\begin{lemma}
\label{lem:psi}
For $K\in\BDC(\ifield_{M\times\R_\infty})$, one has the isomorphisms
\begin{align*}
\psi_{M,-\infty}(\field_{\{t\geq 0\}} \ctens K) &\simeq 0, \\
\psi_{M,+\infty}\,\cihom(\field_{\{t\geq 0\}}, K) &\simeq 0, \\
\psi_{M,+\infty}(\field_{\{t\geq 0\}} \ctens K) &\simeq L, \\ 
\psi_{M,-\infty}\,\cihom(\field_{\{t\geq 0\}}, K)
&\simeq L[1],
\end{align*}
where $L$ is the object defined in {\rm Proposition~\ref{pro:tenshomstar}}.
\end{lemma}

\begin{proof}
(i) Since the proofs of the first and second isomorphisms in the statement are similar, let us only check that
\begin{equation}
\label{eq:psi0}
\psi_{M,-\infty}(\field_{\{t\geq 0\}} \ctens K) \simeq 0.
\end{equation}

Set $K'=\field_{\{t\geq 0\}} \ctens K$.
Since $\field_{\{t\geq 0\}} \ctens K\simeq K'$,
Proposition~\ref{pro:tenshomstar} implies
\[
\roim{\overline\pi}(\field_{\{t\neq-\infty\}} \tens \roimv{j_{M\sep*}} K') \simeq \roim\pi K'.
\]
Since $\roim\pi K' \simeq \roim{\overline\pi}\roimv{j_{M\sep*}} K'$, we get
\[
\roim{\overline\pi}(\field_{\{t=-\infty\}} \tens \roimv{j_{M\sep*}} K') \simeq 0.
\]
One concludes since the above complex is isomorphic to $\psi_{M,-\infty}(\field_{\{t\geq 0\}} \ctens K) $.

\smallskip\noindent(ii)
Since the proofs of the third and fourth isomorphisms in the statement are similar, let us only check that
\begin{equation}
\label{eq:psiL}
\psi_{M,-\infty}\,\cihom(\field_{\{t\geq 0\}}, K)
\simeq L[1].
\end{equation}

Applying $\psi_{M,-\infty}$ to the distinguished triangle \eqref{eq:tenshomstarDT}, we obtain
\[
\psi_{M,-\infty}\,\cihom(\field_{\{t\geq 0\}}, K) \simeq \psi_{M,-\infty}(\opb\pi L[1]).
\]
Here we used \eqref{eq:psi0}. Then \eqref{eq:psiL} follows from $\psi_{M,-\infty}(\opb\pi L) \simeq L$.
\end{proof}

Let us state an easy lemma which will be of use later.

Consider the projections
\[
M \from[{\ \pi_M\ }] M\times\R_\infty \To[s_M] \R_\infty.
\]

\begin{lemma}
\label{lem:tildef}
Let $\tilde f\colon M\times\R_\infty \to N\times\R_\infty$ be the morphism of bordered spaces induced by a continuous map $f\colon M \to N$ of good topological spaces.
\bnum
\item
For $K\in\BDC(\ifield_{M\times\R_\infty})$ and $G\in\BDC(\ifield_{\R_\infty})$, there are isomorphisms
\begin{align*}
\reeim{\tilde f} (\opb{s_M}G \ctens K)
&\simeq \opb{s_N}G \ctens \reeim{\tilde f} K, \\
\roim{\tilde f} \cihom(\opb{s_M}G, K)
&\simeq \cihom(\opb{s_N}G, \roim{\tilde f} K).
\end{align*}
\item
For $L\in\BDC(\ifield_{N\times\R_\infty})$ and $G\in\BDC(\ifield_{\R_\infty})$, there are isomorphisms
\begin{align*}
\opb{\tilde f} (\opb{s_N}G \ctens L)
&\simeq \opb{s_M}G \ctens \opb{\tilde f} L, \\
\epb{\tilde f} \cihom(\opb{s_N}G, L)
&\simeq \cihom(\opb{s_M}G, \epb{\tilde f} L).
\end{align*}
\item
One has
\begin{align*}
\reeim{\tilde f} \circ \opb{\pi_M} &\simeq \opb{\pi_N} \circ \reeim f, &
\opb{\tilde f} \circ \opb{\pi_N} &\simeq \opb{\pi_M} \circ \opb f, \\
\roim{\tilde f} \circ \epb{\pi_M} &\simeq \epb{\pi_N} \circ \roim f, &
\epb{\tilde f} \circ \epb{\pi_N} &\simeq \epb{\pi_M} \circ \epb f.
\end{align*}
\ee
\end{lemma}

\subsection{Enhanced ind-sheaves}

\begin{definition}
\label{def:DerT}
Consider the full subcategories of $\BDC(\ifield_{M\times\R_\infty})$
\begin{align*}
\ind\shc_{\{t^*\leq 0\}} 
&= \{K \semicolon \field_{\{t\geq 0\}} \ctens K  \simeq 0\} \\
&= \{K \semicolon \cihom(\field_{\{t\geq 0\}}, K)  \simeq 0\}, \\
\ind\shc_{\{t^*\geq 0\}} 
&= \{K \semicolon \field_{\{t\leq 0\}} \ctens K  \simeq 0\} \\
&= \{K \semicolon \cihom(\field_{\{t\leq 0\}}, K)  \simeq 0\}, \\
\ind\shc_{\{t^* = 0\}} 
&= \ind\shc_{\{t^* \leq 0\}} \cap \ind\shc_{\{t^* \geq 0\}} \\
&= \{K \semicolon (\field_{\{t\geq0\}} \dsum \field_{\{t\leq0\}}) \ctens K \simeq 0  \} \\
&= \{K \semicolon \cihom(\field_{\{t\geq0\}} \dsum \field_{\{t\leq0\}}, K) \simeq 0  \},
\end{align*}
where the equalities hold by Corollary~\ref{cor:kkK}.
Consider also the corresponding quotient categories
\begin{align*}
\BECpm M 
&= \ind\shc_{\{\pm t^*\geq 0\}}/\ind\shc_{\{t^*= 0\}} , \\
\BEC M 
&= \BDC(\ifield_{M\times\R_\infty})/\ind\shc_{\{t^* = 0\}} .
\end{align*}
We call $\BEC M$ the triangulated category of {\em enhanced ind-sheaves}.
\end{definition}

\begin{proposition}
There are equivalences of triangulated categories
\begin{align*}
\BECpm M &\simeq \BDC(\ifield_{M\times\R_\infty})/\ind\shc_{\{\pm t^*\leq 0\}}, \\
\BEC M &\simeq \BECp M \dsum \BECm M.
\end{align*}
\end{proposition}

This follows from Proposition~\ref{pro:tam} below.

The next lemma easily follows from Corollary~\ref{cor:pipi} and
the last distinguished triangle in \eqref{eq:dtpm0}.

\begin{lemma}
One has
\begin{align*}
\ind\shc_{\{t^* = 0\}} 
&= \{K \semicolon \opb\pi\roim\pi K \isoto K\} 
= \{K \semicolon K \isoto \epb\pi\reeim\pi K\} \\
&= \{K \semicolon K \simeq \opb\pi L\text{ for some }L\in\BDC(\ifield_M) \}\\
&= \{K \semicolon K \simeq \epb\pi L\text{ for some }L\in\BDC(\ifield_M) \} \\
&= \{K \semicolon K \isoto \field_{M\times\R}[1]\ctens K \} \\
&= \{K \semicolon K \isofrom \cihom(\field_{M\times\R}[1], K) \}.
\end{align*}
\end{lemma}

Let us describe the categories $\BECpm M$ and $\BEC M$ using Proposition~\ref{pro:D/N}.

\begin{proposition}\label{pro:tam}
\be
\item [{\rm(i-a)}]
The left orthogonal to $\ind\shc_{\{\pm t^*\leq 0\}}$ is given by
\begin{align*}
{}^\bot\ind\shc_{\{\pm t^*\leq 0\}} 
&= \{K \semicolon \field_{\{\pm t\geq 0\}} \ctens K  \isoto K\} \\
&= \{K \semicolon \field_{\{\pm t> 0\}} \ctens K \simeq 0\},
\end{align*}
and there is an equivalence
\[
\BECpm M \to {}^\bot\ind\shc_{\{\pm t^*\leq 0\}}, \quad
K\mapsto \field_{\{\pm t\geq 0\}} \ctens K,
\]
with quasi-inverse given by the quotient functor.
Note that
\[
{}^\bot\ind\shc_{\{\pm t^*\leq 0\}} \subset \ind\shc_{\{\mp t^*\leq 0\}}.
\]
\item [{\rm(i-b)}]
The right orthogonal to $\ind\shc_{\{\pm t^*\leq 0\}}$ is given by
\begin{align*}
\ind\shc_{\{\pm t^*\leq 0\}}^\bot 
&= \{K' \semicolon K' \isoto \cihom(\field_{\{\pm t\geq 0\}}, K')\} \\
&= \{K' \semicolon \cihom(\field_{\{\pm t> 0\}}, K') \simeq 0\},
\end{align*}
and there is an equivalence
\[
\BECpm M \to \ind\shc_{\{\pm t^*\leq 0\}}^\bot, \quad
K\mapsto\cihom(\field_{\{\pm t\geq 0\}}, K),
\]
with quasi-inverse given by the quotient functor.
Note that
\[
\ind\shc_{\{\pm t^*\leq 0\}}^\bot \subset \ind\shc_{\{\mp t^*\leq 0\}}.
\]
\item [{\rm(ii-a)}]
The left orthogonal to $\ind\shc_{\{t^* = 0\}}$ is given by
\begin{align*}
{}^\bot\ind\shc_{\{t^* = 0\}} 
&= \{K \semicolon (\field_{\{t\geq 0\}} \dsum \field_{\{t\leq 0\}}) \ctens K  \isoto K\} \\
&= \{K \semicolon \field_{M\times\R}\ctens K \simeq 0 \} \\
&= \{K \semicolon \reeim\pi K \simeq 0\} ,
\end{align*}
and there is an equivalence
\[
\BEC M \to {}^\bot\ind\shc_{\{t^* = 0\}}, \quad
K\mapsto (\field_{\{t\geq 0\}} \dsum \field_{\{t\leq 0\}}) \ctens K,
\]
with quasi-inverse given by the quotient functor.
\item [{\rm(ii-b)}]
The right orthogonal to $\ind\shc_{\{t^* = 0\}}$ is given by
\begin{align*}
\ind\shc^\bot_{\{t^* = 0\}} 
&= \{K \semicolon \cihom(\field_{\{t\geq 0\}} \dsum \field_{\{t\leq 0\}}, K)  \isoto K\} \\
&= \{K \semicolon \cihom(\field_{M\times\R}, K) \simeq 0 \} \\
&= \{K \semicolon \roim\pi K \simeq 0\} ,
\end{align*}
and there is an equivalence
\[
\BEC M \to \ind\shc_{\{t^* = 0\}}^\bot, \quad
K\mapsto \cihom(\field_{\{t\geq 0\}} \dsum \field_{\{t\leq 0\}}, K),
\]
with quasi-inverse given by the quotient functor.
\item[{\rm(iii)}]
One has
\begin{align*}
{}^\bot\ind\shc_{\{t^* \geq 0\}} \dsum {}^\bot\ind\shc_{\{t^* \leq 0\}} &\simeq {}^\bot\ind\shc_{\{t^* = 0\}}, \\
\ind\shc^\bot_{\{t^* \geq 0\}} \dsum \ind\shc^\bot_{\{t^* \leq 0\}} &\simeq \ind\shc^\bot_{\{t^* = 0\}}.
\end{align*}
\ee
\end{proposition}

\begin{proof}
The proof is easy. Let us only note that the equality
\[
\{K \semicolon \field_{M\times\R}\ctens K \simeq 0 \} \\
= \{K \semicolon \reeim\pi K \simeq 0\}
\]
follows from Corollary~\ref{cor:pipi}.
\end{proof}

The functors
\eq&&\hs{3ex}\ba{rcl}
(\field_{\{t\geq 0\}} \dsum \field_{\{t\leq 0\}}) \ctens \ast&\cl&
\BDC(\ifield_{M\times\R_\infty})\To  \BDC(\ifield_{M\times\R_\infty}),  \\ 
\cihom(\field_{\{t\geq 0\}} \dsum \field_{\{t\leq 0\}},  \ast)
&\cl&\BDC(\ifield_{M\times\R_\infty})\To \BDC(\ifield_{M\times\R_\infty})
\ea\label{eq:LERE}
\eneq
factor through $\BEC M$ by Lemma~\ref{lem:piRinfty}.

\begin{notation}
\label{not:Tlr}
Denote by
\begin{align*}
\LE &\colon \BEC M \To {}^\bot\ind\shc_{\{t^*= 0\}} \subset \BDC(\ifield_{M\times\R_\infty}), \\ 
\RE &\colon \BEC M \To \ind\shc_{\{t^*= 0\}}^\bot \subset \BDC(\ifield_{M\times\R_\infty})
\end{align*}
the functors induced by \eqref{eq:LERE}.
\end{notation}

Note that the functors $\LE$ and $\RE$ are the left and right adjoint of the quotient functor $\BDC(\ifield_{M\times\R_\infty}) \to \BEC M$.

We have a morphism of functors $\LE \to \RE$.

\begin{lemma}
Let $F_1,F_2\in\BDC(\ifield_{M\times\R_\infty})$.
Let $K_1$, $K_2$ be the objects of $\BEC M$ corresponding to $F_1$, $F_2$
by the quotient functor.

\bnum
\item
There are isomorphisms in $\BDC(\ifield_{M\times\R_\infty})$
\begin{align*}
\cihom(\LE K_1, \LE K_2)
&\simeq  \cihom(\LE K_1, F_2) \\
&\simeq  \cihom(F_1, \RE K_2).
\end{align*}
\item
There are isomorphisms
\begin{align*}
\Hom[\BEC M](K_1, K_2)
&\simeq  \Hom[\BDC(\ifield_{M\times\R_\infty})](\LE K_1, F_2) \\
&\simeq  \Hom[\BDC(\ifield_{M\times\R_\infty})](F_1, \RE K_2).
\end{align*}
\ee
\end{lemma}

\subsection{Operations}\label{sse:hom}

By Lemma~\ref{lem:piRinfty} the compositions of functors
\begin{align}
\label{eq:ctenstilde}
& \BDC(\ifield_{M\times\R_\infty}) \times \BDC(\ifield_{M\times\R_\infty}) \To[\ctens] \BDC(\ifield_{M\times\R_\infty}) \To \BEC M, \\
\label{eq:cihomtilde}
& \BDC(\ifield_{M\times\R_\infty})^\op \times \BDC(\ifield_{M\times\R_\infty}) \To[\cihom] \BDC(\ifield_{M\times\R_\infty}) \To \BEC M,
\end{align}
factor through $\BEC M \times \BEC M$ and ${\BEC M}^\op \times \BEC M$, respectively.

\begin{definition}
\label{def:Tctens}
We denote by
\begin{align*}
\ctens &\colon \BEC M \times \BEC M \To \BEC M, \\
\cihom &\colon {\BEC M}^\op \times \BEC M \To \BEC M,
\end{align*}
the functors induced by \eqref{eq:ctenstilde} and \eqref{eq:cihomtilde}, respectively.
\end{definition}

Note that, for any $K\in\BEC M$, the composition
\[
\field_{\{t\geq 0\}} \ctens K \To K\To\cihom(\field_{\{t\geq 0\}} , K)
\]
is an isomorphism in $\BEC M$ by Proposition~\ref{pro:tenshomstar}.

\begin{definition}
By Lemma~\ref{lem:cihomKt0} one gets functors
\begin{align*}
{\opb\pi\ast}\tens{\ast} &\colon \BDC(\ifield_M) \times \BEC M \To \BEC M, \\
\rihom(\opb\pi{\ast},{\ast}) &\colon \BDC(\ifield_M)^\op \times \BEC M
\To \BEC M, \\
\rihom({\ast},\epb\pi{\ast}) &\colon \BEC M{}^\op \times \BDC(\ifield_M)
\To \BEC M.
\end{align*}
\end{definition}

\begin{remark}
The composition
\[
\BDC(\ifield_{M\times\R_\infty})\times \BDC(\ifield_{M\times\R_\infty}) \To[\otimes] \BDC(\ifield_{M\times\R_\infty}) \to \BEC M
\]
\emph{does not} factor through $\BEC M\times\BEC M$,
and the composition
\[
\BDC(\ifield_{M\times\R_\infty})^\op \times \BDC(\ifield_{M\times\R_\infty}) \To[\rihom] \BDC(\ifield_{M\times\R_\infty}) \to \BEC M
\]
\emph{does not} factor through ${\BEC M}^\op\times\BEC M$.
\end{remark}

\begin{lemma}
For $K_1,K_2,K_3\in\BEC M$ there is an isomorphism
\[
\Hom[\BEC M](K_1\ctens K_2,K_3) \simeq \Hom[\BEC M](K_1,\cihom(K_2,K_3)),
\]
i.e., for $K\in\BEC M$,
$K\ctens\ast$ is a left adjoint of $\cihom(K,\ast)$.
\end{lemma}

\begin{lemma}
\label{lem:Thommorph}
For $K_0,K_1,K_2\in\BEC M$ there are natural morphisms in $\BEC M$
\begin{align*}
& K_0 \ctens \cihom(K_0,K_1) \to K_1, \\
& \cihom(K_0,K_1) \ctens \cihom(K_1,K_2) \to \cihom(K_0,K_2), \\
& K_0 \ctens \cihom(K_1,K_2) \to \cihom(K_1,K_0\ctens K_2), \\
& \cihom(K_1,K_2) \to \cihom(K_0\ctens K_1, K_0\ctens K_2), \\
& \cihom(K_1,K_2) \to \cihom(\cihom(K_0,K_1),\cihom(K_0,K_2)), \\
& K_0 \to \cihom(\cihom(K_0,K_1),K_1).
\end{align*}
\end{lemma}

\begin{proof}
The first morphism is the image of the identity by the isomorphism
\begin{multline*}
\Hom[\BEC M](\cihom(K_0,K_1),\cihom(K_0,K_1)) \\
\isoto \Hom[\BEC M](K_0\ctens\cihom(K_0,K_1),K_1).
\end{multline*}

\noindent The second morphism follows from
\eqn
&&K_0 \ctens \cihom(K_0,K_1) \ctens \cihom(K_1,K_2) \\
&&\hs{15ex}\to K_1 \ctens \cihom(K_1,K_2)
\To K_2.
\eneqn

\noindent The third morphism is the image by the isomorphism
\begin{multline*}
\Hom[\BEC M](K_0 \ctens K_1 \ctens \cihom(K_1,K_2), K_0\ctens K_2) \\
\isoto
\Hom[\BEC M](K_0 \ctens \cihom(K_1,K_2),\cihom(K_1,K_0\ctens K_2))
\end{multline*}
of the morphism
\begin{equation}
\label{eq:KKiK}
K_0 \ctens K_1 \ctens \cihom(K_1,K_2) \to K_0\ctens K_2
\end{equation}
induced by the first morphism in the statement.

\noindent The fourth morphism is obtained from \eqref{eq:KKiK}.

\noindent The fifth morphism is obtained from the second one.

\noindent The last morphism follows from the first one.
\end{proof}

Let $f\colon M \to N$ be a continuous map of good topological spaces.
Denote by $\tilde f\colon M\times\R_\infty \to N\times\R_\infty$ the associated morphism. Then, by Lemma~\ref{lem:tildef}~(iii), the compositions of functors
\begin{align}
\label{eq:oimftilde}
& \BDC(\ifield_{M\times\R_\infty}) \To[\reeim{\tilde f},\; \roim{\tilde f}] \BDC(\ifield_{N\times\R_\infty}) \To \BEC N, \\
\label{eq:opbftilde}
& \BDC(\ifield_{N\times\R_\infty}) \To[\opb{\tilde f},\; \epb{\tilde f}] \BDC(\ifield_{M\times\R_\infty}) \To \BEC M
\end{align}
factor through $\BEC M$ and $\BEC N$, respectively.

\begin{definition}
\label{def:fT}
We denote by
\begin{align*}
\Eeeim f,\; \Eoim f &\colon \BEC M \to \BEC N, \\
\Eopb f,\; \Eepb f &\colon \BEC N \to \BEC M,
\end{align*}
the functors induced by \eqref{eq:oimftilde} and \eqref{eq:opbftilde}, respectively.
\end{definition}

\begin{definition}
For $K\in\BEC M$ and $L\in\BEC N$, set
\[
K \cetens L = \Eopb p_1 K \ctens \Eopb p_2 L \in \BEC{M\times N},
\]
where $p_1$ and $p_2$ denote the projections from $M\times N$ to $M$ and $N$, respectively.
\end{definition}

Using Notation~\ref{not:Tlr}, for $K\in\BEC M$ and $L\in\BEC N$ one has
isomorphisms in $\BEC M$ or $\BEC N$: 
\begin{align*}
\Eeeim f K &\simeq \reeim{\tilde f}\LE K \simeq \reeim{\tilde f}\RE K, \\
\Eoim f K &\simeq \roim{\tilde f}\LE K \simeq \roim{\tilde f}\RE K, \\
\Eopb f L &\simeq \opb{\tilde f}\LE L \simeq \opb{\tilde f}\RE L, \\
\Eepb f L &\simeq \epb{\tilde f}\LE L \simeq \epb{\tilde f}\RE L.
\end{align*}

Let us now show that the above operations satisfy 
similar properties to  the external operations for ind-sheaves.

The following two propositions immediately follow from their counterpart in Lemmas~\ref{lem:badj} and \ref{lem:bcomp}.

\begin{proposition}
\label{pro:Tadj}
Let $f\colon M \to N$ be a continuous map of good topological spaces.
\bnum
\item
The functor $\Eeeim f$ is left adjoint to $\Eepb f$.
\item
The functor $\Eopb f$ is left adjoint to $\Eoim f$.
\ee
\end{proposition}

\begin{proposition}\label{pro:Tcomp}
Given two continuous maps of good topological spaces $L \To[g] M \To[f] N$, one has
\[
\Eeeim{(f\circ g)} \simeq \Eeeim f \circ \Eeeim g, \qquad
\Eoim{(f\circ g)} \simeq \Eoim f \circ \Eoim g
\]
and
\[
\Eopb{(f\circ g)} \simeq \Eopb g \circ \Eopb f, \qquad
\Eepb{(f\circ g)} \simeq \Eepb g \circ \Eepb f.
\]
\end{proposition}

\begin{proposition}
\label{pro:Tproj}
Let $f\colon M \to N$ be a continuous map of good topological spaces.
For $K\in\BEC M$ and $L,L_1,L_2\in\BEC N$, one has isomorphisms
\begin{align*}
\Eeeim f(\Eopb f L \ctens K) & \simeq L \ctens \Eeeim f K, \\
\Eopb f (L_1\ctens L_2) &\simeq \Eopb f L_1 \ctens \Eopb f L_2, \\
\cihom(L,\Eoim f K) & \simeq \Eoim f\; \cihom(\Eopb f L,K), \\
\cihom(\Eeeim f K, L) & \simeq \Eoim f\; \cihom(K, \Eepb f L), \\
\Eepb f \cihom(L_1,L_2) & \simeq \cihom(\Eopb f L_1, \Eepb f L_2),
\end{align*}
and a morphism
\[
\Eopb f \cihom(L_1,L_2) \to \cihom(\Eopb f L_1,\Eopb f L_2).
\]
\end{proposition}

\begin{proof}
(i)
Since the proofs of the five isomorphisms in the statement are similar, let us only deal with the fourth one.
Consider the morphisms
\begin{align*}
q_{M1},q_{M2},\mu_M &\colon M\times\R^2_\infty \to M\times\R_\infty, \\
q_{N1},q_{N2},\mu_N &\colon N\times\R^2_\infty \to N\times\R_\infty
\end{align*}
induced by \eqref{eq:muq1q2}.
Consider the Cartesian diagrams
\[
\vcenter{\vbox{
\xymatrix{
M\times\R^2_\infty \ar[r]^{f'} \ar[d]^u & N\times\R^2_\infty \ar[d]^v \\
M\times\R_\infty \ar[r]^{{\tilde f}} \ar@{}[ur]|-\square & N\times\R_\infty
}}}
\ \text{for $(u,v) = (q_{M1},q_{N1})$, $(q_{M2},q_{N2})$, $(\mu_M,\mu_N)$.}
\]
Then one has
\begin{align*}
\cihom(\Eeeim f K, L) 
& \simeq \roimv{q_{N1\sep*}}\rihom(\opb{q_{N2}}\reeim{\tilde f}\LE K, \epb{\mu_N}\RE L) \\
& \simeq \roimv{q_{N1\sep*}}\rihom(\reeim {f'}\opb{q_{M2}}\LE K, \epb{\mu_N}\RE L) \\
& \simeq \roimv{q_{N1\sep*}}\roim f'\rihom(\opb{q_{M2}}\LE K, \epbv{f^{\prime\sep!}}\epb{\mu_N}\RE L) \\
& \simeq \roim{\tilde f}\roimv{q_{M1\sep*}}\rihom(\opb{q_{M2}}\LE K, \epb{\mu_M}\epb{\tilde f}\RE L) \\
& \simeq \Eoim f \cihom(K, \Eepb f L).
\end{align*}

\smallskip\noindent(ii)
The last morphism in the statement is obtained by adjunction from
\begin{align*}
\Eopb f L_1 \ctens \Eopb f \cihom(L_1,L_2)
&\simeq \Eopb f (L_1 \ctens \cihom(L_1,L_2)) \\
&\to \Eopb f L_2.
\end{align*}
Here, the last morphism follows from Lemma~\ref{lem:Thommorph}.
\end{proof}

The next proposition follows from Lemma~\ref{lem:bcart}.

\begin{proposition}
\label{pro:Tcart}
Consider a Cartesian diagram of good topological spaces 
\begin{equation*}
\xymatrix@C=8ex{
M' \ar[r]^{f'} \ar[d]^{g'} & N' \ar[d]^{g} \\
M \ar[r]^{f}\ar@{}[ur]|-\square & N.
}
\end{equation*}
Then there are isomorphisms of functors $\BEC M \to \BEC{N'}$
\[
\Eopb g\Eeeim f \simeq \Eeeim f' \Eopbv{g^{\prime\sep-1}}, \qquad
\Eepb g\Eoim f \simeq \Eoim f' \Eepbv{g^{\prime\sep!}}.
\]
\end{proposition}

\begin{lemma}\label{lem:ihomE}  
Let $F_1,F_2\in\BDC(\ifield_{M\times\R_\infty})$.
Let $K_1$, $K_2$ be the objects of $\BEC M$ corresponding to $F_1$, $F_2$ 
by the quotient functor. Then one has
\begin{align*}
\roim\pi\rihom(\LE K_1,\RE K_2) &\simeq
\roim\pi\rihom(\LE K_1,F_2)\\
&\simeq \roim\pi\rihom(F_1,\RE K_2).
\end{align*}
\end{lemma}

\begin{proof}
The first isomorphism follows from
\[
\roim\pi\rihom(\LE K_1,\epb\pi L)
\simeq \rihom(\reeim\pi\LE K_1,L) \simeq 0,
\]
and the second isomorphism follows from
\[
\roim\pi\rihom(\opb\pi L,\RE K_2) \simeq \rihom(L,\roim\pi \RE K_2) \simeq 0.
\]
\end{proof}

\begin{definition}\label{def:HomT}
We define the hom-functor
\[
\fhom\colon {\BEC M}^\op \times \BEC M \to \BDC(\field_M)
\]
as follows
\begin{align*}
\fhom(K_1,K_2) 
&= \alpha_M\,\roim\pi\rihom(\LE K_1,\LE K_2) \\
&\simeq \alpha_M\,\roim\pi\rihom(\LE K_1,\RE K_2) \\
&\simeq \alpha_M\,\roim\pi\rihom(\RE K_1,\RE K_2) \\
&\underset{(*)}\simeq \alpha_M\roim{\overline\pi}\rihom(\reeimv{j_{M\sep!!}}\LE K_1,\roimv{j_{M\sep*}} \RE K_2) \\
&\simeq \roim{\overline\pi}\rhom(\reeimv{j_{M\sep!!}}\LE K_1,\roimv{j_{M\sep*}} \RE K_2).
\end{align*} 
Here, $(*)$ follows from Lemma~\ref{lem:jM}~(iv) and in the last isomorphism we used the fact that $\alpha$ commutes with $\roim{\overline\pi}$.
\end{definition}

\begin{lemma}
For $K_1,K_2\in\BEC M$, one has
\begin{align*}
\Hom[\BEC M](K_1,K_2)
&\simeq H^0\rsect(M; \fhom(K_1,K_2)) \\
&\simeq \Hom[\BDC(\field_M)](\field_M, \fhom(K_1,K_2)) .
\end{align*}
\end{lemma}

\begin{lemma}
\label{lem:homT}
For $K_1,K_2,K_3\in\BEC M$, one has
\[
\fhom(K_1 \ctens K_2,K_3) \simeq
\fhom(K_1, \cihom (K_2,K_3)).
\]
In particular,
\[
\fhom(K_1,K_2) \simeq
\fhom(\field_{\{t=0\}}, \cihom (K_1,K_2)).
\]
\end{lemma}

Let $i_0 \colon M \to M\times\R_\infty$ be the embedding $x \mapsto (x,0)$.

\begin{lemma}
\label{lem:ai0}
For $K\in\BEC M$ and $L\in\BDC(\ifield_M)$, one has
\[
\fhom(\field_{\{t = 0\}} \tens \opb\pi L,K) \simeq
\alpha_M\rihom(L, \epb{i_0} \RE K).
\]
\end{lemma}

Note that $\alpha$ does not commute with $\epb{i_0}$.

\begin{proof}
There is the chain of isomorphisms 
\begin{align*}
\fhom(\field_{\{t = 0\}} \tens \opb\pi L,K) 
&\simeq \alpha_M\,\roim\pi\rihom(\field_{\{t= 0\}} \tens \opb\pi L, \RE K) \\ 
&\simeq \alpha_M\,\roim\pi\rihom(\opb\pi L, \rihom (\field_{\{t= 0\}}, \RE K)) \\ 
&\simeq \alpha_M\rihom(L,\roim\pi\roimv{i_{0\sep*}}\epb{i_0} \RE K)  \\
&\simeq \alpha_M\rihom(L,\epb{i_0} \RE K).
\end{align*}
Here the first isomorphism follows from Lemma~\ref{lem:ihomE}.
\end{proof}

The following lemma follows from the fact that $\alpha$ commutes with $\roim f$.

\begin{lemma}
\label{homepb}
For $f\colon M\to N$ a morphism of good topological spaces, $K\in\BEC M$ and $L\in\BEC N$, one has
\begin{align*}
\roim f \fhom(K,\Eepb f L) &\simeq \fhom(\Eeeim f K, L), \\
\roim f \fhom(\Eopb f L, K) &\simeq \fhom(L, \Eoim f K).
\end{align*}
\end{lemma}

\begin{remark}\bnum
\item
For $K_1,K_2\in\BEC M$ and $F\in\BDC(\field_M)$, the isomorphism
\[
\rhom(F,\fhom(K_1,K_2))
\simeq
\fhom(\opb\pi F \tens K_1, K_2)
\]
\emph{does not} hold in general.
\item
Let $f\colon M\to N$ be a morphism of good topological spaces and $L_1,L_2\in\BEC N$.
Since $\alpha$ and $\epb f$ do not commute in general, the isomorphism $\epb f\fhom(L_1,L_2) \simeq \fhom(\Eopb f L_1,\Eepb f L_2)$ does not hold in general. 
\ee
\end{remark}

\subsection{$t$-structure of $\BEC M$}

In this subsection,
we will give a $t$-structure on $\BEC M$.
Recall the t-structure on $\BDC(\ifield_{M\times\R_\infty})$
defined in \S\;\ref{subse:t-str}.

\begin{definition}\label{def:tstrTDC}
We set
\begin{align*}
\Bec[\leq0]{M} &= \{ K\in\BEC M\semicolon \LE K\in \derd^{\leq 0}(\ifield_{M\times\R_\infty}) \}, \\
\Bec[\geq0] M&= \{ K\in\BEC M\semicolon \LE K\in \derd^{\geq 0}(\ifield_{M\times\R_\infty}) \}.
\end{align*}
\end{definition}

\begin{proposition}
The pair $({\Bec[{\leq 0}] M}, \Bec [{\geq 0}]M)$ is a $t$-structure on $\BEC M$.
\end{proposition}

\begin{proof}
It is enough to show that for $K\in\BEC M$ there are isomorphisms
\begin{align*}
(\field_{\{t\geq 0\}} \dsum \field_{\{t\leq 0\}}) \ctens \tau^{\leq 0}\LE K &\simeq \tau^{\leq 0}\LE K, \\
(\field_{\{t\geq 0\}} \dsum \field_{\{t\leq 0\}}) \ctens \tau^{.0}\LE K &\simeq \tau^{>0}\LE K.
\end{align*}
In other words, we have to prove
\begin{equation}
\label{eq:taubot}
\tau^{\leq 0}\LE K,\ \tau^{> 0}\LE K \in {}^\bot\shc_{\{t^*=0\}}.
\end{equation}
Hence it is enough to show
\eq
&&\reeim{\pi}\tau^{\leq 0}\LE K\simeq \reeim{\pi}\tau^{>0}\LE K\simeq0.
\label{eq:vandir}
\eneq
We have a distinguished triangle
$$
\reeim{\pi}\tau^{\leq 0}\LE K\to\reeim{\pi}\LE K\to \reeim{\pi}\tau^{>0}\LE K\tone.$$
Since the middle term vanishes we have
$$\reeim{\pi}\tau^{>0}\LE K\simeq \reeim{\pi}\tau^{\leq 0}\LE K[1].$$
By Proposition~\ref{pro:t-str} (iii) (a), we have
$$
\text{$\reeim{\pi}\tau^{>0}\LE K\in \derd^{>0}(\ifield_{M\times\R_\infty})$
and $\reeim{\pi}\tau^{\leq 0}\LE K[1]\in \derd^{\le0}(\ifield_{M\times\R_\infty})$.}$$
Hence we obtain \eqref{eq:vandir}.
\end{proof}

Let $\tau^{\leq n}$, $\tau^{\geq n}$ and $H^n$ be the truncation functors and the cohomology functor for this $t$-structure. Then we have the quasi-commutative diagrams
\[
\vcenter{\vbox{
\xymatrix@C=6ex{
\BEC M \ar[r]^u \ar[d]^{\LE} & \BEC M \ar[d]^{\LE} \ar[r]^\id & \BEC M \\
\BDC(\ifield_{M\times\R_\infty}) \ar[r]^u & \BDC(\ifield_{M\times\R_\infty}) \ar[ur]_Q
}}}
\text{for $u = \tau^{\leq 0}$, $\tau^{\geq 0}$,$H^n$,}
\]
where $Q$ is the quotient functor.

\begin{lemma}\label{lem:geqaexact}
For $a\in\R$, the functors
\[
\field_{\{t=a\}}\ctens\ast,
\quad\field_{\{t\geq a\}}\ctens\ast,\quad \field_{\{t\leq a\}}\ctens\ast
\]
are exact endofunctors of $\BEC M$.
\end{lemma}

\begin{proof}
The functor
$\field_{\{t= a\}}\ctens\ast \simeq \roimv{\mu_{a\sep*}}(\ast)$ is an exact functor, where $\mu_a\colon M\times\R_\infty\to M\times\R_\infty$ is the morphism induced by the translation $t\mapsto t+a$.

For $K\in\BEC M$, there are isomorphisms
\begin{align*}
(\field_{\{t\geq a\}} \dsum \field_{\{t\leq a\}}) \ctens K 
&\simeq \field_{\{t =a\}} \ctens (\field_{\{t\geq 0\}} \dsum \field_{\{t\leq 0\}}) \ctens K  \\
&\simeq \field_{\{t =a\}} \ctens K.
\end{align*}
It follows that $(\field_{\{t\geq a\}} \dsum \field_{\{t\leq a\}}) \ctens \ast$ is an exact functor. Hence so are $\field_{\{t\geq a\}}\ctens\ast$ and $\field_{\{t\leq a\}}\ctens\ast$.
\end{proof}

\subsection{Stable objects}

\begin{notation}
\label{not:gg}
Consider the objects of $\BDC(\ifield_{M\times\R_\infty})$
\[
\field_{\{t\gg0\}} \defeq \indlim[a\rightarrow+\infty] \field_{\{t\geq a\}},\qquad
\field_{\{t<\ast\}}  \defeq \indlim[a\rightarrow+\infty] \field_{\{t< a\}}.
\]
\end{notation}
We have a distinguished triangle in $\BDC(\ifield_{M\times\R_\infty})$
\eq\field_{M\times\R}\To \field_{\{t\gg0\}} 
\To \field_{\{t<\ast\}} [1]\To[+1].\eneq

\begin{proposition}
\label{pro:gggeqK}
For $K\in\BDC(\ifield_{M\times\R_\infty})$ and $n\in\Z$ one has
\begin{align*}
\reeimv{j_{M\sep!!}}H^n(\field_{\{t\gg0\}} \ctens K) &\simeq \indlim[a\rightarrow+\infty] \reeimv{j_{M\sep!!}} H^n(\field_{\{t\geq a\}} \ctens K), \\
H^n(\field_{\{t\gg0\}} \ctens K) &\simeq \field_{\{t\gg0\}} \ctens H^n(\field_{\{t\geq 0\}} \ctens K).
\end{align*}
\end{proposition}

\begin{proof}
(i) The first isomorphism follows from 
Proposition 5.2.6~(i) of \cite{KS01}.

\smallskip\noindent(ii)
Let us prove the second isomorphism.
Lemma~\ref{lem:geqaexact} implies 
\begin{align*}
H^n(\field_{\{t\geq a\}} \ctens K) 
&\simeq \field_{\{t\geq a\}} \ctens H^n(K).
\end{align*}
Taking the ind-limit with respect to $a\rightarrow+\infty$, we obtain the desired result.
\end{proof}

We have the isomorphisms in $\BDC(\ifield_{M\times\R_\infty})$
\begin{align}
&\field_{\{t\gg0\}} \ctens \field_{\{t\gg0\}} \simeq \field_{\{t\gg0\}}, \\
&\field_{\{t\geq -a\}} \ctens \field_{\{t\gg0\}} \isoto \field_{\{t\gg0\}} \isoto \field_{\{t\geq a\}} \ctens \field_{\{t\gg0\}}
\end{align}
for any $a\in\R_{\geq 0}$.

\begin{notation}
\label{not:Tam}
Denote by $\field^\enh_M$ the object of $\BEC M$ associated with $\field_{\{t\gg0\}}\in\BDC(\ifield_{M\times\R_\infty})$.
More generally, for $F\in\BDC(\field_M)$, set
\[
F^\enh \defeq \field^\enh_M \tens \opb\pi F \in \BEC M.
\]
\end{notation}

Note that one has
\[
\LE \field^\enh_M \simeq \field_{\{t\gg0\}}
\quad\text{and}\quad
\RE \field^\enh_M \simeq \field_{\{t<\ast\}}[1].
\]

\begin{lemma}
The functor 
$\field_M^\enh\ctens\ast$
is an exact endofunctor of $\BEC M$.
\end{lemma}

\begin{proof}
By Proposition~\ref{pro:gggeqK}, for $K\in\BEC M$ one has
\[
H^n(\field_M^\enh \ctens K) \simeq \field_M^\enh \ctens H^n(K).
\]
Hence $\field_M^\enh\ctens\ast$ is an exact functor.
\end{proof}

\begin{proposition}
\label{pro:equivTam}
Let $K\in\BEC M$. 
Then the following conditions are equivalent.
\bnum
\item
$K \isofrom \field_{\{t\geq 0\}} \ctens K \isoto \field_{\{t\geq a\}} \ctens K$ for any $a\geq 0$,
\item
$\cihom(\field_{\{t\geq a\}}, K) \isoto \cihom(\field_{\{t\geq 0\}}, K) \isofrom K$ for any $a\geq 0$,
\item
$K \isofrom \field_{\{t\geq 0\}} \ctens K \isoto \field_M^\enh \ctens K$,
\item
$\cihom(\field_M^\enh, K) \isoto \cihom(\field_{\{t\geq 0\}}, K) \isofrom K$,
\item
$K \simeq \field_M^\enh \ctens L$ for some $L\in\BEC M$,
\item
$K \simeq \cihom(\field_M^\enh, L)$ for some $L\in\BEC M$.
\ee
\end{proposition}

\begin{proof}
The less obvious implications (i)$\implies$(iii) and (ii)$\implies$(iv) follows from Corollary~\ref{cor:tensindlim} and Proposition~\ref{pro:Prolim}.

Note also that $\cihom(\field_{\{t\geq a\}}, K) \simeq \field_{\{t\geq -a\}} \ctens K$ for any $a\in\R$. Hence, for example, (iii)$\implies$(ii) is given by
\begin{align*}
\cihom(\field_{\{t\geq a\}}, K)
&\simeq \field_{\{t\geq -a\}} \ctens K \\
&\simeq \field_{\{t\geq -a\}} \ctens \field_M^\enh \ctens K \\
&\simeq \field_M^\enh \ctens K \simeq K.
\end{align*}
\end{proof}

\begin{definition}
A stable object is an object of $\BECp M$ that satisfies the equivalent conditions of Proposition~\ref{pro:equivTam}.
\end{definition}

\begin{remark}
The notion of stable object is related to the notion of torsion object from \cite{Tam08} (compare \cite[\S5]{GS12} and Proposition~\ref{pro:pro:HomStab} below).
\end{remark}

Note that, for $K\in\BEC M$, one has isomorphisms in $\BEC M$
\begin{align*}
\field_M^\enh \ctens \cihom(\field_M^\enh, K) &\simeq \cihom(\field_M^\enh,K), \\
\cihom(\field_M^\enh, \field_M^\enh \ctens K) &\simeq \field_M^\enh \ctens K.
\end{align*}

\begin{corollary}
For $K_1,K_2\in\BEC M$ there is an isomorphism in $\BEC M$
\[
\cihom(\field_M^\enh \ctens K_1, \field_M^\enh \ctens K_2) 
\simeq
\cihom(K_1, \field_M^\enh \ctens K_2).
\]
\end{corollary}

\begin{proposition}\label{pro:pro:HomStab}
Let $F\in\BDC(\field_{M\times\R_\infty})$ and $K\in\BEC M$.
Assume that $\overline\pi(\supp(\reimv{j_{M\sep!}}F))$ is compact. Then there are isomorphisms
\begin{align*}
\Hom[\BEC M]&(\field^\enh_M \ctens F, \field^\enh_M \ctens K) \\
&\simeq \varinjlim_{a\rightarrow+\infty} \Hom[\BEC M](F, \field_{\{t\geq a\}}\ctens K) \\
&\simeq \varinjlim_{a\rightarrow+\infty} \Hom[\BEC M](\field_{\{t\geq - a\}}\ctens F, K).
\end{align*}
\end{proposition}

\begin{proof}
(i)
We have
\begin{align*}
\Hom[\BEC M]&(\field^\enh_M \ctens F, \field^\enh_M \ctens K) \\
&\simeq \Hom[\BEC M](\field_{\{t\geq0\}} \ctens F, \cihom(\field^\enh_M, \field^\enh_M \ctens K)) \\
&\simeq \Hom[\BEC M](\field_{\{t\geq0\}} \ctens F, \field^\enh_M \ctens K) \\
&\simeq \Hom[\BDC(\ifield_{M\times\overline\R})](\reeimv{j_{M\sep!!}}(\field_{\{t\geq0\}} \ctens F), \roimv{j_{M\sep*}}(\field_{\{t\gg0\}} \ctens \LE K)) \\
&\underset{(*)}\simeq \varinjlim_{a\rightarrow+\infty} \Hom[\BDC(\ifield_{M\times\overline\R})](\reeimv{j_{M\sep!!}}(\field_{\{t\geq0\}} \ctens F), \roimv{j_{M\sep*}}(\field_{\{t\geq a\}} \ctens \LE K)) \\
&\simeq \varinjlim_{a\rightarrow+\infty} \Hom[\BEC M](F, \field_{\{t\geq a\}}\ctens K).
\end{align*}
Here $(*)$ follows from Corollary~\ref{cor:tensindlim}.

\smallskip\noindent(ii)
The other isomorphism follows from
\[
\Hom[\BEC M](\field_{\{t\geq -a\}}\ctens F, K)
\simeq
\Hom[\BEC M](F, \field_{\{t\geq a\}}\ctens K).
\]
\end{proof}

\begin{lemma}\label{lem:kTamtens}
For $F\in\BDC(\field_{M\times\R_\infty})$ and $K\in\BEC M$, there is an isomorphism in $\BEC M$
\[
\field_M^\enh \ctens \cihom(F,K)
\isoto
\cihom(F,\field_M^\enh \ctens K).
\]
\end{lemma}

\begin{proof}
Let us first show that, for $L\in \BDC(\ifield_{M\times\R_\infty})$, the morphism in $\BDC(\ifield_{M\times\R_\infty})$
\begin{equation}
\label{eq:ggeqcihom}
\field_{\{t\gg0\}} \ctens \cihom(F,L)
\To
\field_{\{t\geq 0\}} \ctens \cihom(F,\field_{\{t\gg0\}} \ctens L)
\end{equation}
is an isomorphism.
For any $a\in\R$, there are isomorphisms in $\BEC M$
\begin{align*}
\field_{\{t\geq a\}} \ctens \cihom(F,L)
&\simeq  \cihom(\field_{\{t\geq -a\}}, \cihom(F,L)) \\
&\simeq  \cihom(F, \cihom(\field_{\{t\geq -a\}},L)) \\
&\simeq  \cihom(F, \field_{\{t\geq a\}} \ctens L).
\end{align*}
Hence we have an isomorphism in $\BDC(\ifield_{M\times\R_\infty})$
\[
\field_{\{t\geq a\}} \ctens \cihom(F,L)
\isoto  \field_{\{t\geq 0\}} \ctens \cihom(F, \field_{\{t\geq a\}} \ctens L).
\]
In order to see that \eqref{eq:ggeqcihom} is an isomorphism, we shall use Proposition~\ref{pro:J}.
We have
\begin{align*}
J_{M\times\overline\R}\;\reeimv{j_{M\sep!!}}&(\field_{\{t\gg0\}} \ctens \cihom(F,L)) \\
&\simeq \varinjlim_{a\rightarrow+\infty}J_{M\times\overline\R}\;\reeimv{j_{M\sep!!}}\bl\field_{\{t\geq a\}} \ctens \cihom(F,L)\br \\
&\simeq \varinjlim_{a\rightarrow+\infty}J_{M\times\overline\R}\;\reeimv{j_{M\sep!!}}\bl\field_{\{t\geq 0\}} \ctens \cihom(F,\field_{\{t\geq a\}}\ctens L)\br \\
&\simeq J_{M\times\overline\R}\;\reeimv{j_{M\sep!!}}\bl\field_{\{t\geq 0\}} \ctens \cihom(F,\field_{\{t\gg0\}} \ctens L)\br.
\end{align*}
By Proposition~\ref{pro:J}, it follows that \eqref{eq:ggeqcihom} is an isomorphism.

It remains to notice that for $K\in\BEC M$ we have isomorphisms in $\BEC M$
\begin{align*}
\field_M^\enh \ctens \cihom(F,K) 
&\simeq \field_{\{t\geq 0\}} \ctens \cihom(F,\field_M^\enh \ctens K) \\
&\simeq \cihom\bl\field_{\{t\geq 0\}}, \cihom(F,\field_M^\enh \ctens K)\br \\
&\simeq \cihom\bl F, \cihom(\field_{\{t\geq 0\}},\field_M^\enh \ctens K)\br \\
&\simeq \cihom(F, \field_M^\enh \ctens K) .
\end{align*}
\end{proof}

\begin{corollary}
\label{cor:pifieldT}
For $K\in\BEC M$ and $F\in\BDC(\field_M)$, we have
\[
\field_M^\enh \ctens \rihom(\opb\pi F, K) \simeq 
\rihom(\opb\pi F, \field_M^\enh \ctens K).
\]
\end{corollary}

\begin{proof}
This easily follows from Lemma~\ref{lem:kTamtens} and the isomorphism
\[
\rihom(\opb\pi F,K) \simeq \cihom(\opb\pi F \tens \field_{\{t=0\}},K),
\]
due to Lemma~\ref{lem:cihomKt0}.
\end{proof}

\begin{proposition}
\label{pro:kTarihom}
Let $F\in\BDC(\field_{M\times\R_\infty})$ and $G\in\BDC(\ifield_M)$. Then there is an isomorphism in $\BEC M$
\begin{equation}
\label{eq:epbpiopbpi}
\field_M^\enh\ctens\opb a \rihom(F,\epb\pi G) \simeq
\cihom(F,\field_M^\enh \tens\opb\pi G).
\end{equation}
\end{proposition}

\begin{proof}
Recall that, by Lemma~\ref{lem:cihomKt0}, one has
\[
\opb a \rihom(F,\epb\pi G) \simeq \cihom(F,\field_{\{t=0\}} \tens\opb\pi G).
\]
Hence, Lemma~\ref{lem:kTamtens} implies
\begin{align*}
\field_M^\enh\ctens\opb a \rihom(F,\epb\pi G) 
&\simeq \field_M^\enh\ctens\cihom(F,\field_{\{t=0\}} \tens\opb\pi G) \\
&\simeq \cihom(F,\field_M^\enh\ctens(\field_{\{t=0\}} \tens\opb\pi G)) \\
&\simeq \cihom(F,\field_M^\enh \tens\opb\pi G).
\end{align*}
\end{proof}

\begin{remark}
\label{rem:vanish}
By Lemma~\ref{lem:Indlim}, one has
\[
\rihom(\field_M^\enh ,\epb\pi\omega_M) 
\simeq \opb{j_M} \rihom(\field_{\{t\gg 0\}}, \omega_{M\times\overline\R})
\simeq 0.
\]
Moreover, one has
\[
\cihom(\field_M^\enh ,\field_M^\enh\tens\opb\pi \omega_M)
\simeq \field_M^\enh\tens\opb\pi \omega_M.
\]
Hence \eqref{eq:epbpiopbpi} does not hold for $F=\field_M^\enh$ and $G=\omega_M$.
\end{remark}

\begin{proposition}\label{pro:stableops}
Let $f\colon M\to N$ be a continuous map of good topological spaces.
\bnum
\item
For $K\in\BEC M$ one has
\[
\Eeeim f(\field_M^\enh \ctens K) \simeq \field_N^\enh \ctens \Eeeim f K.
\]
\item
For $L\in\BEC N$ one has
\begin{align*}
\Eopb f(\field_N^\enh \ctens L) &\simeq \field_M^\enh \ctens \Eopb f L, \\
\Eepb f(\field_N^\enh \ctens L) &\simeq \field_M^\enh \ctens \Eepb f L.
\end{align*}
\ee
\end{proposition}

\begin{proof}
The isomorphisms
\begin{align*}
\Eeeim f(\field_M^\enh \ctens K) &\simeq \field_N^\enh \ctens \Eeeim f K, \\
\Eopb f(\field_N^\enh \ctens L) &\simeq \field_M^\enh \ctens \Eopb f L
\end{align*}
follow from Proposition~\ref{pro:Tproj} and $\Eopb f \field_N^\enh \simeq \field_M^\enh$. Let us prove
\begin{equation}
\label{eq:TepbStable}
\Eepb f(\field_N^\enh \ctens L) \simeq \field_M^\enh \ctens \Eepb f L.
\end{equation}
If $L\in\BECm N$, then both sides of \eqref{eq:TepbStable} vanish. We may then assume $L\in\BECp N$, i.e.\ $L\isoto\cihom(\field_{\{t\geq 0\}},L)$.

Set $\widetilde L = \roimv{j_{N\sep*}}\RE L$, so that
\[
\widetilde L \simeq \roimv{j_{N\sep*}}\cihom(\field_{\{t\geq 0\}},\opb{j_N}\widetilde L).
\]
Let $\overline f\colon M\times\overline\R\to N\times\overline\R$ be the map induced by $f$.
By Lemma~\ref{lem:tildef}, we have
\[
\epb{\overline f}\widetilde L \simeq \roimv{j_{M\sep*}}\cihom(\field_{\{t\geq 0\}},\opb{j_M}\epb{\overline f}\widetilde L).
\]
Then, Lemma~\ref{lem:psi} implies
\[
\field_{\{t=+\infty\}} \tens \widetilde L \simeq 0, \quad
\field_{\{t=+\infty\}} \tens \epb{\overline f}\widetilde L \simeq 0.
\]
Set
\begin{align*}
C_M &= \indlim[a\rightarrow+\infty]\field_{M\times\{-\infty\leq t<a\}}[1], \\
C_N &= \indlim[a\rightarrow+\infty]\field_{N\times\{-\infty\leq t<a\}}[1],
\end{align*}
so that 
\[
C_M \simeq \roimv{j_{M\sep*}}\RE(\field^\enh_M),
\quad
C_N \simeq \roimv{j_{N\sep*}}\RE(\field^\enh_N).
\]
Using Notation~\ref{not:sfS}, consider the maps
\begin{align*}
\tilde q_{1M}, \tilde q_{2M}, \tilde\mu_M &\colon M\times\mathsf S \To M\times \overline\R, \\
\tilde q_{1N}, \tilde q_{2N}, \tilde\mu_N &\colon N\times\mathsf S \To N\times \overline\R, \\
f' &\colon M\times\mathsf S\To N\times\mathsf S,
\end{align*}
where $f'$ is the map induced by $f$.

Then, by Lemma~\ref{lem:ctenstilde}, $\field_M^\enh \ctens \Eepb f L$ is represented by the object of $\BDC(\ifield_{M\times\overline\R})$
\[
\roimv{\tilde\mu_{M\sep*}}(\field_{M\times\R^2} \tens \opb{\tilde q_{1M}} C_M \tens \opb{\tilde q_{2M}} \epb{\overline f}\widetilde L).
\]
Since
\begin{align*}
& \field_{\opb{\tilde q_{1M}}(M\times\{+\infty\})} \tens \opb{\tilde q_{1M}} C_M  \simeq 0, \\
& \field_{\opb{\tilde q_{2M}}(M\times\{+\infty\})} \tens \opb{\tilde q_{2M}} \epb{\overline f}\widetilde L \simeq 0, \\
&\opb{\tilde \mu_M}(M\times\R) \cap M \times (\mathsf S \setminus \R^2)
\subset \opb{\tilde q_{1M}}(\{t=+\infty\}) \cup \opb{\tilde q_{2M}}(\{t=+\infty\}),
\end{align*}
we obtain
\begin{multline*}
\field_{\opb{\tilde \mu_M}(M\times\R)} \tens \field_{M\times\R^2} \tens \opb{\tilde q_{1M}} C_M \tens \opb{\tilde q_{2M}} \epb{\overline f}\widetilde L \\
\simeq
\field_{\opb{\tilde \mu_M}(M\times\R)} \tens \opb{\tilde q_{1M}} C_M \tens \opb{\tilde q_{2M}} \epb{\overline f}\widetilde L.
\end{multline*}
Moreover, one has
\[
\field_{\opb{\tilde \mu_M}(M\times\R)} \tens \opb{\tilde q_{2M}} \epb{\overline f}\widetilde L \simeq \field_{\opb{\tilde \mu_M}(M\times\R)} \tens \epb{\tilde q_{2M}} \epb{\overline f}\widetilde L[-1],
\]
since $\tilde q_{2M}$ is topologically submersive and $\epb{\tilde q_{2M}}\field_{M\times\overline\R} \simeq \field_{\opb{\tilde\mu}(M\times\R)\union\opb{\tilde q_{1M}}(M\times\R)}[1]$.

Hence we conclude that $\field_M^\enh \ctens \Eepb f L$ is represented by
\[
\roimv{\tilde\mu_{M\sep*}}(\opb{\tilde q_{1M}} C_M \tens \epb{\tilde q_{2M}} \epb{\overline f}\widetilde L[-1]).
\]
On the other hand, by the same reasoning, $\field_N^\enh \ctens L$ is represented by the object of $\BDC(\ifield_{N\times\overline\R})$
\[
\roimv{\tilde\mu_{N\sep*}}(\opb{\tilde q_{1N}} C_N \tens \epb{\tilde q_{2N}} \widetilde L[-1]).
\]
Hence $\Eepb f(\field_N^\enh \ctens L)$ is represented by the object of $\BDC(\ifield_{M\times\overline\R})$
\[
\epb{\overline f}\roimv{\tilde\mu_{N\sep*}}(\opb{\tilde q_{1N}} C_N \tens \epb{\tilde q_{2N}} \widetilde L[-1]) \simeq \roimv{\tilde\mu_{M\sep*}}\epbv{f^{\prime\sep!}}(\opb{\tilde q_{1N}} C_N \tens \epb{\tilde q_{2N}} \widetilde L[-1]).
\]
Finally, Proposition~\ref{pro:opbepb} implies that
\begin{align*}
\epbv{f^{\prime\sep!}}(\opb{\tilde q_{1N}} C_N \tens \epb{\tilde q_{2N}} \widetilde L[-1])
&\simeq \opbv{f^{\prime\sep-1}}\opb{\tilde q_{1N}} C_N \tens \epbv{f^{\prime\sep!}}\epb{\tilde q_{2N}} \widetilde L[-1] \\
&\simeq \opb{\tilde q_{1M}} C_M \tens \epb{\tilde q_{2M}}\epb{\overline f} \widetilde L[-1].
\end{align*}
\end{proof}

\begin{proposition}\label{pro:embed}
The functor $e(F) = \field_M^\enh\tens\opb\pi F$ gives a fully faithful embedding
\[
e\colon \BDC(\ifield_M) \To \BEC M.
\]
\end{proposition}

\begin{proof}
For $F,G\in \BDC(\ifield_M)$ one has
\begin{align*}
\Hom[\BEC M]&(\field_M^\enh\tens\opb\pi F, \field_M^\enh\tens\opb\pi G) \\
&\simeq\Hom[\BEC M](\field_M^\enh\ctens(\field_{\{t=0\}}\tens\opb\pi F), \field_M^\enh\tens\opb\pi G) \\
&\simeq\Hom[\BEC M](\field_{\{t=0\}}\tens\opb\pi F, \cihom(\field_M^\enh, \field_M^\enh\tens\opb\pi G)) \\
&\simeq\Hom[\BEC M](\field_{\{t=0\}}\tens\opb\pi F, \field_M^\enh\tens\opb\pi G).
\end{align*}
Since
\begin{align*}
\LE(\field_{\{t=0\}}\tens\opb\pi F) &\simeq (\field_{\{t\geq0\}} \dsum \field_{\{t\leq0\}})\tens\opb\pi F, \\
\LE(\field_M^\enh\tens\opb\pi G) &\simeq \field_{\{t\gg0\}}\tens\opb\pi G, 
\end{align*}
one further has
\begin{align*}
\Hom[\BEC M]&(\field_{\{t=0\}}\tens\opb\pi F, \field_M^\enh\tens\opb\pi G) \\
&\simeq \Hom[\BDC(\ifield_{M\times\R_\infty})]((\field_{\{t\geq0\}} \dsum \field_{\{t\leq0\}})\tens\opb\pi F, \field_{\{t\gg0\}}\tens\opb\pi G) \\
&\simeq \Hom[\BDC(\ifield_{M\times\R_\infty})](\opb\pi F, \field_{\{t\gg0\}}\tens\opb\pi G) \\
&\simeq \Hom[\BDC(\ifield_M)](F, \roim\pi (\field_{\{t\gg0\}}\tens\opb\pi G)) \\
&\underset{(*)}\simeq \Hom[\BDC(\ifield_M)](F, G).
\end{align*}
Here, in $(*)$, we used the fact that
\begin{align*}
\roim\pi (\field_{\{t\gg0\}}\tens\opb\pi G)
&\simeq \roim{\overline\pi}\roimv{j_{M\sep*}} (\field_{\{t\gg0\}}\tens\opb\pi G) \\
&\simeq \roim{\overline\pi} (\indlim[a\rightarrow+\infty]\field_{\{a\leq t \leq+\infty\}}\tens\opb{\overline\pi} G),
\end{align*}
and $\roim{\overline\pi}\indlim[a\rightarrow+\infty]\field_{\{a\leq t \leq+\infty\}} \simeq \field_M$.
\end{proof}

\subsection{Duality}

\begin{definition}
We define the duality functor
\[
\Edual_M\colon\BEC M\to {\BEC M}^\op, \qquad K \mapsto \cihom(K,\omega^\enh_M),
\]
where we recall that $\omega^\enh_M \defeq \field^\enh_M \tens \opb\pi\omega_M$.
\end{definition}

\begin{proposition}
Let $f\colon M\to N$ be a continuous map of good topological spaces and $K\in\BEC M$. Then one has an isomorphism in $\BEC N$
\[
\Edual_N \Eeeim f K \simeq \Eoim f \Edual_M K.
\]
\end{proposition}

\begin{proof}
This follows from Proposition~\ref{pro:Tproj} and $\Eepb f \omega^\enh_N \simeq \omega^\enh_M$, which is a consequence of Proposition~\ref{pro:stableops}~(ii).
\end{proof}

\begin{proposition}
\label{pro:Tduala}
For $F\in\BDC(\field_{M\times\R_\infty})$, one has
\[
\Edual_M(\field^\enh_M \ctens F) \simeq \field^\enh_M \ctens \opb a \dual_{M\times\R} F.
\]
Here $a$ is the involution of $M\times\R$ defined by $a(x,t) = (x,-t)$.
\end{proposition}

\begin{proof}
We have
\begin{align*}
\Edual_M(\field^\enh_M \ctens F) 
&= \cihom(\field^\enh_M \ctens F,\omega_M^\enh) \\
&\simeq \cihom(F, \cihom(\field^\enh_M,\omega_M^\enh)) \\
&\simeq \cihom(F, \omega_M^\enh) \\
&= \cihom(F, \field^\enh_M \tens \opb\pi\omega_M) \\
&\simeq \field^\enh_M \ctens \opb a \rhom (F,\epb\pi\omega_M).
\end{align*}
Here, the last isomorphism follows from Proposition~\ref{pro:kTarihom}.
\end{proof}

\begin{corollary}
\label{cor:Tduala}
For $F\in\BDC(\field_{M})$, one has
\[
\Edual_M(\field^\enh_M \tens \opb\pi F) \simeq \field^\enh_M \tens \opb\pi \dual_M F.
\]
\end{corollary}

\begin{proof}
We have
\begin{align*}
\Edual_M(\field^\enh_M \tens \opb\pi F) 
&\simeq \Edual_M(\field^\enh_M \ctens (\field_{\{t=0\}} \tens \opb\pi F)) \\
&\simeq \field^\enh_M \ctens \opb a \dual_{M\times\R}(\field_{\{t=0\}} \tens \opb\pi F) \\
&\simeq \field^\enh_M \ctens (\field_{\{t=0\}} \tens \opb\pi \dual_M F) \\
&\simeq \field^\enh_M \tens \opb\pi \dual_M F.
\end{align*}
\end{proof}

\subsection{$\R$-constructible objects}
In this subsection, we assume that  $M$ is a subanalytic space.
Recall the natural morphism
\[
j_M\colon M\times\R_\infty \To M\times\overline\R,
\]
and  the category $\BDC(\field_{M\times\R_\infty})$ from Notation~\ref{not:fieldNbN}.

\begin{definition}
We denote by $\BDC_\Rc(\field_{M\times\R_\infty})$ the full subcategory of $\BDC(\field_{M\times\R_\infty})$ whose objects $F$ are such that $\reimv{j_{M\sep!}}F$ is an $\R$-constructible object of $\BDC(\field_{M\times\overline\R})$. 

We regard $\BDC_\Rc(\field_{M\times\R_\infty})$ as a full subcategory of $\BDC(\ifield_{M\times\R_\infty})$.
\end{definition}

Note that $\BDC_\Rc(\field_{M\times\R_\infty})$ is stable by the functors $\ctens$, $\cihom$ and $\tens$, $\rihom$.

\begin{definition}
\label{def:TRc}
We say that an object $K\in\BEC M$ is \emph{$\R$-constructible} if for any relatively compact subanalytic open subset $U\subset M$ there exists an isomorphism
\[
\opb\pi\field_U \tens K \simeq \field_M^\enh \ctens F \quad\text{for some } F\in \BDC_\Rc(\field_{M\times\R_\infty}).
\]
Denote by $\BECRc M$ the full subcategory of $\BEC M$ whose objects are $\R$-constructible.
\end{definition}

Note in particular that $\R$-constructible objects of $\BEC M$ are stable objects.

\begin{proposition}
\label{pro:Rcthick}
Let $K'\To[\varphi] K\To K''\To[+1]$ be a distinguished triangle in $\BEC M$.
If $K'$ and $K$ are $\R$-constructible, so is $K''$.
\end{proposition}

\begin{proof}
We may assume that $K' = \field_M^\enh \ctens F'$ and $K = \field_M^\enh \ctens F$ for $F,F'\in\BDC_\Rc(\field_{M\times\R_\infty})$. By replacing $F'$ with $\field_{\{t\geq 0\}}\ctens F'$, we may also assume that $F' \simeq \field_{\{t\geq 0\}}\ctens F'$. We may  assume further that $\overline\pi(\supp(\reimv{j_{M\sep!}}F'))$ is compact.
Then, by Proposition~\ref{pro:pro:HomStab},
\[
\Hom[\BEC M](K',K)
\simeq \ilim[a\rightarrow+\infty] \Hom[\BDC(\ifield_{M\times\R_\infty})](F', \field_{\{t\geq a\}} \ctens F) .
\]
Hence there exist $a\in\R$ and a morphism in $\BDC(\field_{M\times\R_\infty})$
\[
\varphi'\colon F' \to \field_{\{t\geq a\}} \ctens F
\]
such that $\varphi\colon K'\to K$ is equal to
\[
K' = \field_M^\enh \ctens F'
\To[\varphi']  \field_M^\enh \ctens (\field_{\{t\geq a\}} \ctens F)
\simeq \field_M^\enh \ctens F = K.
\]
Completing $\varphi'$ in a distinguished triangle $F'\To[\varphi'] \field_{\{t\geq a\}} \ctens F\To F''\To[+1]$, we have
$F''\in\BDC_\Rc(\field_{M\times\R_\infty})$ and $K'' \simeq \field_M^\enh \ctens F''$.
\end{proof}

\begin{corollary}
The category $\BECRc M$ is a triangulated category.
\end{corollary}

\begin{lemma}
\label{lem:H^nRc}
Let $K\in\BEC M$. Then $K$ is $\R$-constructible if and only if $H^n K$ is $\R$-constructible for any $n\in\Z$.
\end{lemma}

\begin{proof}
For $F\in\BDC(\field_{M\times\R_\infty})$, we have
\[
H^n(\field_M^\enh \ctens F) \simeq \field_M^\enh \ctens H^n(\field_{\{t\geq 0\}} \ctens F)
\]
by Proposition~\ref{pro:gggeqK}.
\end{proof}

\begin{proposition}
\label{pro:summand}
Let $K_1,K_2\in\BEC M$. If $K_1\dsum K_2$ is $\R$-constructible, then $K_1$ and $K_2$ are $\R$-constructible.
\end{proposition}

\begin{proof}
Let $f\colon K_1\dsum K_2\to K_1\dsum K_2$ be the morphism 
given by $\left(\begin{smallmatrix}0&0\\0&\id_{K_2}\end{smallmatrix}\right)$.
Then we have a distinguished triangle
$$K_1\dsum K_2\To[f] K_1\dsum K_2\to K_1\dsum K_1[1]\tone.$$
Hence, Proposition~\ref{pro:Rcthick} implies that
$K_1\dsum K_1[1]$ is $\R$-constructible.

It is therefore enough to show that 
\eq
&&\hs{5ex}\text{$K\in\BEC M$ is $\R$-constructible
if $K\dsum K[1]$ is $\R$-constructible.}
\label{cond:rcons}
\eneq
We may assume $H^n(K)=0$ unless $a\le n\le b$.
Let us show \eqref{cond:rcons} by induction on $b-a$.
By Lemma~\ref{lem:H^nRc}, $H^a(K)\simeq H^{a-1}(K\dsum K[1])$
is $\R$-constructible.
Hence $H^a(K)[-a]\dsum H^a(K)[-a+1]$ is $\R$-constructible.
There is a distinguished triangle
$$H^a(K)[-a]\dsum H^a(K)[-a+1]\To K\dsum K[1]
\To \tau^{>a}K\dsum (\tau^{>a}K)[1]\tone,$$
where $\tau^{>a}$ is the truncation functor with respect to the t-structure of
$\BEC M$.
Hence, $\tau^{>a}K\dsum (\tau^{>a}K)[1]$ is $\R$-constructible
by Proposition~\ref{pro:Rcthick}.
By the induction hypothesis, $\tau^{>a}K$ is $\R$-constructible.
Then, by the distinguished triangle
$$H^a(K)[-a]\To K \To \tau^{>a}K\tone,$$
we conclude that $K$ is $\R$-constructible.
\end{proof}

\begin{lemma}\label{lem:Rcfilt}
Let $K\in\BEC M$. Then the following conditions are equivalent.
\bnum
\item
$K\in\BECRc M$,
\item
there exist a locally finite family $\{Z_i\}_{i\in I}$ of locally closed subanalytic subsets of $M$ and $F_i\in\BDC_\Rc(\field_{M\times\R_\infty})$ such that
$M = \Union\nolimits_{i\in I} Z_i$ and
\[
\opb\pi\field_{Z_i} \tens K \simeq \field_M^\enh \ctens F_i \quad\text{for all }i\in I,
\]
\item
there exist a filtration $\emptyset = M_{-1}\subset M_0\subset\cdots\subset M_r = M$ and objects $F_k\in\BDC_\Rc(\ifield_{M\times\R_\infty})$ for $0\leq k\leq r$ such that $M_k$ is a closed subanalytic subset of $M$ and
\[
\opb\pi\field_{M_k\setminus M_{k-1}} \tens K \simeq \field_M^\enh \ctens F_k.
\]
\ee
\end{lemma}

\begin{proof}
(i)$\implies$(ii) is obvious.

\smallskip\noindent(ii)$\implies$(iii)
There exists a filtration $\{M_k\}$ such that each connected component of $M_k\setminus M_{k-1}$ is contained in some $Z_i$.

\smallskip\noindent(iii)$\implies$(i)
follows from Proposition~\ref{pro:Rcthick}.
\end{proof}

\begin{corollary}
$\R$-constructibility of $K\in\BEC M$ is a local property on $M$.
\end{corollary}

The following lemma is not used in this paper, but it might help the reader to understand the category $\BECRc M$.

\begin{lemma}
\label{lem:Tstrat}
The complex $K\in\BEC M$ is $\R$-constructible if and only if there exist 
\bnum
\item
a  locally finite family $\{Z_i\}_{i\in I}$ of locally closed subanalytic subsets of $M$,
\item
finite sets $A_i$, for $i\in I$,
\item
continuous subanalytic functions $\varphi_{i,a}\colon Z_i \to \R$ and $\psi_{i,a}\colon Z_i \to \R\union{\{+\infty\}}$ for $i\in I$ and $a\in A_i$, such that $\varphi_{i,a}(x)<\psi_{i,a}(x)$ for all $x\in Z_i$
{\rm(}here a function is called subanalytic if its graph is subanalytic in $M\times\overline\R${\rm)},
\item
integers $m_{i,a}\in\Z$ for $i\in I$ and $a\in A_i$,
\ee
such that $M=\DUnion_{i\in I}Z_i$ and there are
isomorphisms for any $i\in I$
\[
\opb\pi\field_{Z_i} \tens K \simeq \DSum_{a\in A_i} \field_M^\enh \ctens \field_{W_{i,a}} [m_{i,a}] ,
\]
where we set
\[
W_{i,a} = \{(x,t)\in Z_i\times\R\semicolon \varphi_{i,a}(x) \leq t < \psi_{i,a}(x) \}.
\]
\end{lemma}

\begin{proof}
We may assume $K = \field_M^\enh \ctens F$ for $F\in\BDC_\Rc(\field_{M\times\R_\infty})$ such that $F\simeq\field_{\{t\geq 0\}}\ctens F$.

Since $F$ is $\R$-constructible,
there exist a partition $M=\DUnion_{i\in I} Z_i$, integers $r_i\in\Z_{>0}$  ($i\in I$), and continuous subanalytic functions $\xi_{i,a}\colon Z_i \to \ol \R$  ($i\in I$, $0\leq a \leq r_i$), such that $-\infty = \xi_{i,0}(x) <\cdots <\xi_{i,r_i}(x)=+\infty$ for any $x\in Z_i$, and such that $F|_{Z_i\times\R}$ is locally constant on $\{(x,t)\semicolon x\in Z_i,\ t=\xi_{i,a}(x)\}$ (for $0 < a < r_i$) and on
$Z_i\times\R\setminus\Union_{a=1}^{r_i-1}\{t=\xi_{i,a}(x)\}$.

We may further assume that $Z_i$ is contractible. Then $\opb\pi\field_{Z_i} \tens F$ is a finite direct sum of shifts of sheaves of the form
\bnum
\item
$\field_{\{\xi_{i,a}(x)<t<\xi_{i,b}(x)\}}$ for $0\leq a<b\leq r_i$,
\item
$\field_{\{\xi_{i,a}(x)\leq t<\xi_{i,b}(x)\}}$ for $0 < a<b\leq r_i$, 
\item
$\field_{\{\xi_{i,a}(x)<t\leq \xi_{i,b}(x)\}}$ for $0\leq a<b< r_i$, 
\item
$\field_{\{\xi_{i,a}(x)\leq t\leq \xi_{i,b}(x)\}}$ for $0< a\leq b < r_i$.
\ee
Since we assumed $F\simeq\field_{\{t\geq 0\}}\ctens F$, any direct summand of $\opb\pi\field_{Z_i} \tens F$ satisfies the same condition. Hence only the case (ii) survives.
\end{proof}

\begin{notation}
For $K\in\BEC M$, we set
\[
\supp^\enh(K) = \overline\pi(\supp(\reeimv{j_{M\sep!!}}\LE K)) \subset M.
\]
\end{notation}

\begin{proposition}
\label{pro:RcTfunctorial}
Let $f\colon M\to N$ be a continuous subanalytic morphism of subanalytic spaces.
\bnum
\item
The functors $\Eopb f$ and $\Eepb f$ send $\BECRc N$ to $\BECRc M$.
\item
Let $K\in\BECRc M$ be such that $\supp^\enh(K)$ is proper over $N$. Then $\Eeeim f K \simeq \Eoim f K \in \BECRc N$.
\ee
\end{proposition}

\begin{proof}
(i) 
Note that $\Eopb f$ and $\Eepb f$ send $\BDC_\Rc(\field_{N\times\R_\infty})$ to $\BDC_\Rc(\field_{M\times\R_\infty})$.
Then the statement follows from Proposition~\ref{pro:stableops}.

\smallskip\noindent(ii)
We may assume that $K=\field_M^\enh \ctens F$ for $F\in\BDC_\Rc(\field_{M\times\R_\infty})$ such that $\overline{\pi\supp(F)}$ is compact. 
Then $\Eeeim f F\in\BDC_\Rc(\field_{N\times\R_\infty})$, and the statement follows from Proposition~\ref{pro:stableops}.
\end{proof}

\begin{theorem}
If $K\in\BECRc M$, then $\Edual_M K\in\BECRc M$ and the natural morphism
\[
K \To \Edual_M \Edual_M K
\]
is an isomorphism.
\end{theorem}

\begin{proof}
The natural morphism is constructed using Lemma~\ref{lem:Thommorph}.

We may assume $K = \field_M^\enh \ctens F$ for $F\in\BDC_\Rc(\field_{M\times\R_\infty})$.
Then
\begin{align*}
\Edual_M K
&\simeq \Edual_M(\field_M^\enh \ctens F) \\
&\simeq \field_M^\enh \ctens \opb a \dual_{M\times\R} F
\end{align*}
by Proposition~\ref{pro:Tduala}. 
Since $\dual_{M\times\R} F$ belongs to $\BDC_\Rc(\field_{M\times\R_\infty})$,
it follows that $\Edual_M K$ is $\R$-constructible.
Moreover, we have
\begin{align*}
\Edual_M\Edual_M K
&\simeq \Edual_M(\field_M^\enh\ctens \opb a\dual_{M\times\R} F) \\
&\simeq \field_M^\enh\ctens \dual_{M\times\R} \dual_{M\times\R} F \\
&\simeq \field_M^\enh\ctens F \simeq K.
\end{align*}
Hence $K\to\Edual_M\Edual_M K$ is an isomorphism.
\end{proof}

\begin{proposition}\label{prop:homdual}
Let $K,K'\in\BECRc M$. Then both $K\ctens K'$ and $\cihom(K,K')$ are $\R$-constructible, and one has isomorphisms
\bnum
\item
$\Edual_M(K\ctens K') \simeq \cihom( K,\Edual_M K')$, 
\item $\Edual_M\cihom(K,K') \simeq K \ctens \Edual_M K'$,
\item $\cihom(K,K')\simeq\cihom(\Edual_MK',\Edual_MK)$,
\vs{.5ex}
\item $\fhom(K,K')\simeq\fhom(\Edual_MK',\Edual_MK)$.
\ee
\end{proposition}

\begin{proof}
Let us first show that $K\ctens K'$ is $\R$-constructible if both $K$ and $K'$ are so. It is not restrictive to assume $K \simeq \field_M^\enh \ctens F$ and $K' \simeq \field_M^\enh \ctens F'$ for $F,F'\in\BDC_\Rc(\field_{M\times\R_\infty})$.
Then $K\ctens K' \simeq \field_M^\enh \ctens (F\ctens F')$, and hence $K\ctens K'$ is $\R$-constructible.

The first isomorphism in the statement is immediate. 

Hence $\cihom(K,K') \simeq \Edual_M(K\ctens \Edual_M K')$ is $\R$-constructible.

The second isomorphism follows from this isomorphism by applying the functor $\Edual_M$.

The third isomorphism follows from (i).

The fourth isomorphism follows from
\begin{align*}
\fhom(K,K')&\simeq \fhom\bl\field_M^\enh,\cihom(K,K')\br\\
&\simeq \fhom\bl\field_M^\enh,\cihom(\Edual_MK',\Edual_MK)\br\\
&\simeq \fhom(\Edual_MK',\Edual_MK).
\end{align*}
\QED

\begin{proposition}\label{pro:eopbTdual}
Let $f\colon M\to N$ be a continuous subanalytic morphism. For $L\in\BECRc N$ there are isomorphisms
\[
\Eepb f (\Edual_N L) \simeq \Edual_M (\Eopb f L), \quad
\Eopb f (\Edual_N L) \simeq \Edual_M(\Eepb f L).
\]
\end{proposition}

\begin{proof}
(i) There are isomorphisms
\begin{align*}
\Eepb f (\Edual_N L) 
&= \Eepb f \cihom(L,\omega^\enh_N) \\
&\simeq \cihom(\Eopb f L, \Eepb f \omega^\enh_N) \\
&\underset{(*)}\simeq \cihom(\Eopb f L, \omega^\enh_M) \\
&= \Edual_M(\Eopb f L).
\end{align*}
Here $(*)$ follows from Proposition~\ref{pro:stableops}~(ii).

\smallskip\noindent(ii) By (i), there are isomorphisms
\[
\Edual_M (\Eopb f \Edual_N L) \simeq \Eepb f\Edual_N \Edual_N L \simeq \Eepb f L.
\]
Further applying $\Edual_M$, we get $\Eopb f (\Edual_N L) \simeq \Edual_M(\Eepb f L)$.
\end{proof}

\begin{proposition}
\label{pro:DKeL} 
Let $M$ be a subanalytic space, $N$ a good topological space, and $K\in\BECRc M$, $L\in\BEC N$.
Then one has an isomorphism in $\BEC{M\times N}$
\[
\cihom(\Eopb p_1K,\field_{M\times N}^\enh \ctens \Eepb p_2L) \simeq
\Edual_M K \cetens L,
\]
where $p_1$ and $p_2$ denote the projections from $M\times N$ to $M$ and $N$, respectively.
\end{proposition}

In order to prove the above proposition, we need some preliminary results.

\begin{proposition}
\label{pro:mu!mu*}
Let $M$ be a subanalytic space, $N$ a good topological space, and consider
the morphism
\[
\mu\colon M\times N\times\R^2_\infty \to M\times N\times\R_\infty
\]
induced by $(t_1,t_2) \mapsto t_1+t_2$. 
Then, for any $F\in\BDC_\Rc(\field_{M\times\R_\infty})$ and $G\in\BDC(\ifield_{N\times\R_\infty})$, there exists a distinguished triangle in $\BDC(\ifield_{M\times N\times\R_\infty})$
\[
\reeim\mu(F\etens G) \to
\roim\mu(F\etens G) \to
\opb{\pi_{M\times N}}(L_+ \dsum L_-) \to[+1],
\]
where
\begin{align*}
L_\pm &= \psi_{M,\pm\infty}(F) \etens \psi_{N,\mp\infty}(G) 
\end{align*}
{\rm(}see {\rm Notation~\ref{not:psi})}.
Here, we identify $M\times\R_\infty\times N\times\R_\infty$ with $M\times N\times\R^2_\infty$.
\end{proposition}

\begin{proof}
Set $X=M\times N$.
With Notation~\ref{not:sfS}, consider the diagram
\[
X\times\R_\infty\To[j_X]X\times\overline\R \from[\;\;\tilde\mu\;\;] X\times\mathsf S \To[\tilde p] X\times\overline\R^2,
\]
where $\tilde p$ is induced by $(\tilde q_1,\tilde q_2)$. Set
\[
\widetilde F = \roimv{j_{M\sep*}}F \in \BDC_\Rc(\field_{M\times\overline\R}),
\quad
\widetilde G = \roimv{j_{N\sep*}}G \in \BDC(\ifield_{N\times\overline\R}).
\]
Then we have
\begin{align*}
\reeim\mu(F\etens G) &\simeq 
\opb{j_X}\roim{\tilde\mu}(\field_{X\times\R^2}\tens\opb{\tilde p}(\widetilde F \etens \widetilde G)), \\
\roim\mu(F\etens G) &\simeq 
\opb{j_X}\roim{\tilde\mu}(\rihom(\field_{X\times\R^2}, \opb{\tilde p}(\widetilde F \etens \widetilde G))).
\end{align*}
In Sublemma~\ref{subl:mumu} below, we will prove the isomorphism
\begin{multline}
\label{eq:mumutempA}
\opb{\tilde\mu}\field_{X\times\R} \tens \rihom(\field_{X\times\R^2}, \opb{\tilde p}(\widetilde F \etens \widetilde G)) \\
\simeq \opb{\tilde\mu}\field_{X\times\R} \tens \opb{\tilde p}(\widetilde F \etens \widetilde G).
\end{multline}
Admitting \eqref{eq:mumutempA}, 
we have 
$$\roim\mu(F\etens G) \simeq 
\opb{j_X}\roim{\tilde\mu}\opb{\tilde p}(\widetilde F \etens \widetilde G).$$
Hence, we obtain a distinguished triangle 
\[
\reeim\mu(F\etens G) \to
\roim\mu(F\etens G) \to
\opb{j_X}\roim{\tilde\mu}(\field_{X\times(\mathsf S\setminus\R^2)}\tens\opb{\tilde p}(\widetilde F \etens \widetilde G)) \to[+1].
\]
We have
\begin{multline}
\label{eq:mumutempB}
\opb{\tilde\mu}(X\times\R) \cap (X\times(\mathsf S\setminus\R^2)) \\ 
= \opb{\tilde\mu}(X\times\R) \cap \opb{\tilde p}(X\times\{(+\infty,-\infty), (-\infty,+\infty)\}).
\end{multline}
Moreover, we have
\[
\field_{X\times\{(+\infty,-\infty), (-\infty,+\infty)\}} \tens (\widetilde F \etens \widetilde G) \simeq \roimv{i_{+\sep*}}L_+ \dsum \roimv{i_{-\sep*}}L_-,
\]
where $i_\pm\colon X\to X\times\overline\R^2$ is the inclusion $x \mapsto (x,\pm\infty,\mp\infty)$.
Hence we obtain
\begin{multline}
\label{eq:mumutempC}
\field_{X\times\R} \tens \roim{\tilde\mu}(\field_{X\times(\mathsf S\setminus\R^2)}\tens\opb{\tilde p}(\widetilde F \etens \widetilde G)) \\
\simeq \roim{\tilde\mu}(\field_{\opb{\tilde\mu}(X\times\R)}\tens\opb{\tilde p}(\roimv{i_{+\sep*}}L_+ \dsum \roimv{i_{-\sep*}}L_-)).
\end{multline}
By the commutative diagram
\[
\xymatrix{
\opb{\tilde\mu}(X\times\R) \cap \opb{\tilde p}(X\times\{(\pm\infty,\mp\infty)\}) 
\ar[r]^-{\tilde p} \ar[d]^-{\tilde\mu}_-\bwr & X\times\{(\pm\infty,\mp\infty)\} \\
X\times\R \ar[r]^{\pi_X} & X  \ar[u]_-\bwr^-{i_\pm},
}
\]
the right hand side of \eqref{eq:mumutempC} is isomorphic to $\opb{\pi_{M\times N}}(L_+ \dsum L_-)$. Hence we obtain the desired result.
\end{proof}

\begin{sublemma}
\label{subl:mumu}
With the same notations as in the proof of {\rm\/ Proposition~\ref{pro:mu!mu*}}, we have
\[
\opb{\tilde\mu}\field_{X\times\R} \tens \rihom(\field_{X\times(\mathsf S\setminus\R^2)}, \opb{\tilde p}(\widetilde F \etens \widetilde G)) \simeq 0,
\]
where $\widetilde F = \roimv{j_{M\sep*}}F$ and $\widetilde G = \roimv{j_{N\sep*}}G$.
\end{sublemma}

\begin{proof}  
By Proposition~\ref{pro:J}, we may assume $G\in\BDC(\field_{N\times\R_\infty})$ without loss of generality. Set
\[
\Psi_{M,N}(F,G) = \rhom(\field_{X\times(\mathsf S\setminus\R^2)}, \opb{\tilde p}(\widetilde F \etens \widetilde G)) \in \BDC(\field_{X\times\mathsf S}).
\]
By \eqref{eq:mumutempB}, it is enough to show
\begin{equation}
\label{eq:sublmumu}
\field_{\opb{\tilde\mu}(X\times\R) \cap \opb{\tilde p}(X\times(\overline\R\setminus\R)^2)} \tens \Psi_{M,N}(F,G) \simeq 0.
\end{equation}

\smallskip\noindent(i)
We shall first show \eqref{eq:sublmumu} when $M=\point$, so that
$F\in\BDC_\Rc(\cor_{\R_\infty})$. Note that $\BDC_\Rc(\cor_{\R_\infty})$ is the smallest triangulated category which is stable by taking direct summands and contains $\field_\R$ and $\field_{[a,b]}$ for $-\infty<a\leq b<+\infty$. Hence we may assume $F=\field_\R$ or $F=\field_{[a,b]}$.

\smallskip\noindent(i-1)
If $F=\field_{[a,b]}$, then 
\[
\supp(\widetilde F \etens \widetilde G) \cap (X\times(\overline\R\setminus\R)^2) =\emptyset,
\]
so that \eqref{eq:sublmumu} is obvious.

\smallskip\noindent(i-2)
If $F=\field_\R$, then 
\[
\Psi_{M,N}(F,G) = \rhom(\field_{X\times(\mathsf S\setminus\R^2)}, \opb{\tilde p}\opb{\overline p_2} \widetilde G),
\]
where $\ol p_2\cl X\times\ol\R^2\to N\times \ol\R$ is the projection. 
Since
\[
\opb{\tilde\mu}(X\times\R) \cap (X\times(\mathsf S\setminus\R^2)) =
\opb{\tilde\mu}(X\times\R) \cap \opb{\tilde p}\opb{\overline p_2}(N\times(\overline\R\setminus\R)),
\]
we have
\begin{align*}
&\field_{\opb{\tilde\mu}(X\times\R) \cap \opb{\tilde p}(X\times(\overline\R\setminus\R)^2)} \tens \Psi_{M,N}(F,G) \\
&\quad\simeq \field_{\opb{\tilde\mu}(X\times\R) \cap \opb{\tilde p}(X\times(\overline\R\setminus\R)^2)} \tens \rhom(\field_{\opb{\tilde p}\opb{\overline p_2}(N\times(\overline\R\setminus\R))}, \opb{\tilde p}\opb{\overline p_2} \widetilde G) \\
&\quad\underset{(*)}\simeq \field_{\opb{\tilde\mu}(X\times\R) \cap \opb{\tilde p}(X\times(\overline\R\setminus\R)^2)} \tens \opb{\tilde p}\opb{\overline p_2} \rhom(\field_{N\times(\overline\R\setminus\R)}, \roimv{j_{N\sep*}} G) \\
&\quad\simeq 0,
\end{align*}
where (*) is due to Proposition~\ref{pro:topsub}, since $\overline p_2\, \tilde p$ is topologically submersive.

\smallskip\noindent(ii)
Let us now prove \eqref{eq:sublmumu} in the general case.
We shall show that
\[
\Psi_{M,N}(F,G)_{(x_0,y_0,z_0)} \simeq 0
\]
for any $(x_0,y_0,z_0)\in M\times N\times\mathsf S$ such that
\[
(x_0,y_0,z_0) \in \opb{\tilde\mu}(X\times\R) \cap \opb p(X\times(\overline\R\setminus\R)^2).
\]
For any $k\in\Z$, one has
\[
H^k\Psi_{M,N}(F,G)_{(x_0,y_0,z_0)} \simeq
\ilim[U,V] H^k(\overline U \times V; \Psi_{M,N}(F,G)),
\]
where $U\subset M$ ranges over the family of relatively compact subanalytic open neighborhoods of $x_0\in M$, and $V$ ranges over the family of open neighborhoods of $(y_0,z_0)\in N\times\mathsf S$.

Let $r\colon M\times\R_\infty\to\R_\infty$ be the projection, and set
\[
\Phi_{U}(F) \defeq \roim r(F\tens\field_{\overline U\times\R}) \in \BDC_\Rc(\field_{\R_\infty}).
\]
Then 
\eqn
H^k(\overline U \times V; \Psi_{M,N}(F,G))
&\simeq& H^k\bl\ol U\times (V\cap (N\times \R^2)); F\etens G\br\\
&\simeq& H^k\bl V\cap (N\times \R^2); \Phi_{U}(F)\etens G\br\\
&\simeq&H^k(V; \Psi_{\point,N}(\Phi_{U}(F),G)).
\eneqn
Hence, taking the limit on $U$ and $V$, we obtain
\[
H^k\Psi_{M,N}(F,G)_{(x_0,y_0,z_0)} \simeq \ilim[U] H^k\Psi_{\point,N}(\Phi_{U}(F),G)_{(y_0,z_0)},
\]
which vanishes by (i).
\end{proof}

As a consequence of Proposition~\ref{pro:mu!mu*} we get

\begin{corollary}
\label{cor:mu!mu*}
Let $M$ be a subanalytic space and $N$ a good topological space.
For $F\in\BDC_\Rc(\ifield_{M\times\R_\infty})$ and 
$L\in\BDC(\ifield_{N\times\R_\infty})$, 
the morphism 
\[
\reeim\mu(F\etens L) \To
\roim\mu( F\etens L)
\]
is an isomorphism in $\BEC{M\times N}$.
\end{corollary}

\begin{remark}
The above result is not true in general if we drop the assumption that $F\in\BDC_\Rc(\field_{M\times\R_\infty})$. For example, if $M=N=\point$ and
$F=L= K=\DSum_{n\in\Z}\field_{\{n\}} \in \Mod(\field_{\R_\infty})$, one has
\begin{align*}
\reeim\mu( F\etens L) &\simeq \field^{\oplus\Z} \tens K, \\
\roim\mu(F\etens L) &\simeq \field^{\Z} \tens K.
\end{align*}
\end{remark}

\begin{proposition}\label{pro:cihomaD}
Let $M$ be a subanalytic space, $N$ a good topological space.
Let $p_1\colon M\times N\to M$ and $p_2\colon M\times N\to N$ be the natural projections.
Then, for $F\in\BDC_\Rc(\field_{M\times\R_\infty})$ and $L\in\BEC N$ there is an isomorphism in $\BEC{M\times N}$
\[
\cihom(\Eopb p_1F,\Eepb p_2L) \simeq
\opb a \dual_{M\times\R}F \cetens L.
\]
\end{proposition}

\begin{proof}
Set $G=\RE  L\in \BDC(\ifield_{N\times\R_\infty})$. Consider the morphisms
\begin{align*}
r_1&\colon M\times N\times\R^2_\infty \to M\times\R_\infty, \\
r_2&\colon M\times N\times\R^2_\infty \to N\times\R_\infty, \\
\mu&\colon M\times N\times\R^2_\infty \to M\times N\times\R_\infty
\end{align*}
induced by $(t_1,t_2)\mapsto t_1$, $(t_1,t_2)\mapsto t_2$ and $(t_1,t_2)\mapsto t_1+t_2$, respectively. Then
\begin{align*}
\cihom(\Eopb p_1F,\Eepb p_2 L)
&\simeq \roim{\mu}\rihom(\opb{r_1}\opb a F, \epb{r_2}G), \\
\opb a \dual_{M\times\R}F \cetens L
&\simeq \reeim{\mu}(\opb{r_1}\opb a\dual_{M\times\R} F \tens \opb{r_2} G).
\end{align*}
By Proposition~\ref{pro:const},
\[
\rihom(\opb{r_1}\opb a F, \epb{r_2} G) \simeq
\opb{r_1}\opb a\dual_{M\times\R} F \tens \opb{r_2} G,
\]
and Corollary~\ref{cor:mu!mu*} implies that
\[
\reeim{\mu}(\opb{r_1}\opb a\dual_{M\times\R} F \tens \opb{r_2}G) \to
\roim{\mu}(\opb{r_1}\opb a\dual_{M\times\R} F \tens \opb{r_2} G)
\]
is an isomorphism in $\BEC{M\times N}$.
\end{proof}

\begin{proof}[Proof of Proposition~\ref{pro:DKeL}]
Let $p_1\colon M\times N\to M$ and $p_2\colon M\times N\to N$ be the natural projections. We have
\[
\Edual_M K \cetens L = \Eopb p_1\Edual_M K \ctens \Eopb p_2L.
\]
Hence we have a sequence of morphisms
\begin{align*}
\Eopb p_1K \ctens (\Edual_M K \cetens L)
&\simeq \Eopb p_1K \ctens \Eopb p_1\Edual_M K \ctens \Eopb p_2L \\
&\to \Eopb p_1\omega^\enh_M \ctens \Eopb p_2L \\
&= \Eopb p_1(\field^\enh_M \tens \opb{\pi_M}\omega_M) \ctens \Eopb p_2L \\
&\underset{(*)}\simeq \field^\enh_{M\times N} \ctens (\opb{\pi_{M\times N}}\opb{p_1}\omega_M \tens \Eopb p_2L) \\
&\simeq \field^\enh_{M\times N} \ctens \Eepb p_2L,
\end{align*}
where $(*)$ follows from Lemma~\ref{lem:cihomrihompi}.
Hence we obtain a morphism
\[
\Edual_M K \cetens L \to \cihom(\Eopb p_1K,\field_{M\times N}^\enh \ctens \Eepb p_2L).
\]
We shall show that it is an isomorphism for $K\in\BECRc M$.
We may assume $K\simeq \field_M^\enh \ctens F$ for $F\in\BDC_\Rc(\field_{M\times\R_\infty})$. Then
\begin{align*}
\cihom(\Eopb p_1K,&\field_{M\times N}^\enh \ctens \Eepb p_2L) \\
&\simeq \cihom(\Eopb p_1\field_M^\enh \ctens \Eopb p_1 F,\field_{M\times N}^\enh \ctens \Eepb p_2L) \\
&\simeq \cihom(\Eopb p_1 F, \cihom(\field_{M\times N}^\enh,\field_{M\times N}^\enh \ctens \Eepb p_2L)) \\
&\simeq \cihom(\Eopb p_1 F, \field_{M\times N}^\enh \ctens \Eepb p_2L) \\
&\simeq \cihom(\Eopb p_1 F, \Eepb p_2(\field_N^\enh \ctens L)).
\end{align*}
Here, the last isomorphism follows from Proposition~\ref{pro:stableops}~(ii).
By Proposition~\ref{pro:cihomaD}, one has
\begin{align*}
\cihom(\Eopb p_1 F, \Eepb p_2(\field_N^\enh \ctens L)) 
&\simeq \Eopb p_1\opb a \dual_{M\times\R} F \ctens \Eopb p_2(\field_N^\enh \ctens L) \\
&\simeq \Eopb p_1\opb a \dual_{M\times\R} F \ctens \field_{M\times N}^\enh \ctens \Eopb p_2L \\
&\simeq \Eopb p_1(\field_M^\enh \ctens \opb a \dual_{M\times\R} F) \ctens \Eopb p_2L.
\end{align*}
By Proposition~\ref{pro:Tduala}, one finally has
\begin{align*}
\Eopb p_1(\field_M^\enh \ctens \opb a \dual_{M\times\R} F) \ctens \Eopb p_2L
&\simeq \Eopb p_1\Edual_M(\field_M^\enh \ctens F) \ctens \Eopb p_2L \\
&\simeq \Eopb p_1\Edual_M K \ctens \Eopb p_2L.
\end{align*}
\end{proof}

\begin{proposition}\label{prop:exteriodual}
\label{pro:Detens}
Let $M$ and $N$ be subanalytic spaces. For $K\in\BECRc M$ and $L\in\BECRc N$ we have
\[
\Edual_{M\times N}(K\cetens L) \simeq \Edual_M K \cetens \Edual_N L.
\]
\end{proposition}

\begin{proof}
Let $p_1$ and $p_2$ be the projections from $M\times N$ to $M$ and $N$, respectively. Then we have
\begin{align*}
\Edual_{M\times N}(K\cetens L)
&= \cihom(\Eopb p_1 K \ctens \Eopb p_2 L, \omega^\enh_{M\times N}) \\ 
&\simeq \cihom(\Eopb p_1 K , \cihom (\Eopb p_2 L, \omega^\enh_{M\times N})) \\ 
&\simeq \cihom(\Eopb p_1 K , \Edual_{M\times N} (\Eopb p_2 L)) .
\end{align*}
Since $\Edual_{M\times N} (\Eopb p_2 L) \simeq \Eepb p_2\Edual_N L$ 
by Proposition~\ref{pro:eopbTdual}, one has
\begin{align*}
\Edual_{M\times N}(K\cetens L)
&\simeq \cihom(\Eopb p_1 K , \Eepb p_2\Edual_N L) \\ 
&\simeq \Edual_M K \cetens \Edual_N L
\end{align*}
by Proposition~\ref{pro:DKeL}.
\end{proof}

\begin{proposition}\label{prop:inversedual}
For $k=1,2$ let $f_k\colon M_k\to N_k$ 
be a morphism of subanalytic spaces and $L_k\in\BECRc{N_k}$.
Set $f=f_1\times f_2\colon M_1\times M_2\to N_1\times N_2$. Then we have
\begin{align*}
\Eopb f(L_1\cetens L_2) &\simeq \Eopb f_1 L_1 \cetens \Eopb f_2 L_2, \\
\Eepb f(L_1\cetens L_2) &\simeq \Eepb f_1 L_1 \cetens \Eepb f_2 L_2.
\end{align*}
\end{proposition}

\begin{proof}
The first isomorphism is immediate from Proposition~\ref{pro:Tproj}.

Let us show the second isomorphism. By the first isomorphism, we have
\begin{align*}
\Eopb f(\Edual_{N_1} L_1\cetens \Edual_{N_2} L_2) 
&\simeq \Eopb f_1 \Edual_{N_1} L_1 \cetens \Eopb f_2 \Edual_{N_2} L_2 \\
&\simeq \Edual_{M_1} \Eepb f_1 L_1 \cetens \Edual_{M_2} \Eepb f_2 L_2,
\end{align*}
where the last isomorphism follows from Proposition~\ref{pro:eopbTdual}.
Applying $\Edual_{M_1\times M_2}$, and using Proposition~\ref{pro:Detens}, we obtain
\begin{align*}
\Edual_{M_1\times M_2} (\Eopb f(\Edual_{N_1} L_1\cetens \Edual_{N_2} L_2) )
&\simeq  \Eepb f \Edual_{N_1\times N_2}(\Edual_{N_1} L_1\cetens \Edual_{N_2} L_2) \\
&\simeq \Eepb f(L_1\cetens L_2)
\end{align*}
and
\[
\Edual_{M_1\times M_2} (\Edual_{M_1} \Eepb f_1 L_1 \cetens \Edual_{M_2} \Eepb f_2 L_2) \simeq \Eepb f_1 L_1 \cetens \Eepb f_2 L_2.
\]
\end{proof}

\begin{proposition}
\label{homText}
For $K\in\BECRc M$ and $K'\in\BEC M$, one has
\begin{align*}
\cihom(K,\field_M^\enh\ctens K') &\simeq \Eepb\delta(\Edual K \cetens K'), \\
\fhom(K,\field_M^\enh\ctens K') 
&\simeq \fhom(\field_M^\enh,\Eepb\delta(\Edual_M K \cetens K')) \\
&\simeq \opb\delta\fhom(\field_\Delta^\enh,\Edual_M K \cetens K'),
\end{align*}
where $\delta\colon\Delta\to M\times M$ denotes the diagonal embedding.
\end{proposition}

\begin{proof}
(i) 
Let $p_1,p_2\colon M\times M\to M$ be the projections.
By Proposition~\ref{pro:stableops}~(ii), one has
\begin{align*}
\field_M^\enh\ctens K'
&\simeq \field_M^\enh\ctens \Eepb\delta \Eepb p_2 K' \\
&\simeq \Eepb\delta(\field_{M\times M}^\enh\ctens  \Eepb p_2 K').
\end{align*}
Then one has
\begin{align*}
\cihom(K,\field_M^\enh\ctens K') 
&\simeq \cihom(\Eopb\delta \Eopb p_1 K,\Eepb\delta(\field_{M\times M}^\enh\ctens  \Eepb p_2 K')) \\
&\underset{(*)}\simeq \Eepb\delta\cihom(\Eopb p_1 K,\field_{M\times M}^\enh\ctens\Eepb p_2 K') \\
&\underset{(**)}\simeq \Eepb\delta(\Edual_M K \cetens K'),
\end{align*}
where $(*)$ follows from Proposition~\ref{pro:Tproj} and $(**)$ from Proposition~\ref{pro:DKeL}.

\smallskip\noindent(ii)
The second isomorphism follows from (i) and Lemma~\ref{lem:homT}.

\smallskip\noindent(iii)
The third isomorphism follows by applying $\opb\delta$ to
\[
\roim\delta\fhom(\field_M^\enh,\Eepb\delta(\Edual_M K \cetens K'))
\simeq \fhom(\Eeeim\delta\field_M^\enh,\Edual_M K \cetens K').
\]
\end{proof}

\subsection{Ring action}

Let $S$ be a good topological space, and $\sha$ a sheaf of $\field$-algebras on $S$.
Recall from \cite{KS01} that the category of $\sha$-modules in the category of ind-sheaves is defined by\footnote{The category $\ind(\sha)$ is denoted by $\ind(\beta\sha)$ in \cite{KS01}.}
\begin{align*}
\ind(\sha)& = \left\{(F,\varphi)\;\semicolon\; 
\parbox{47ex}{$F\in\ind(\field_S)$,\\ 
$\varphi\colon\sha\to\shEnd(F)$ 
is a $\field$-algebras homomorphism}\right\}.
\end{align*}
Here, $\shEnd(F)$ is the sheaf of $\field$-algebras given by $U\mapsto\Endo[\ind(\field_U)](F|_U)$.

\begin{definition}
Let $f\colon (M,\bM) \to S$ be a morphism of bordered spaces, and $\sha$ a sheaf of $\field$-algebras on $S$. 
Recall that $f$ is decomposed as $(M,\bM) \isofrom (\Gamma_f,\overline\Gamma_f) \to S$.
We set
\[
\BDC(\ind\sha_{(M,\bM)}) = \BDC\bl\ind(\opb{p_2}\sha)\br
/\BDC\bl\ind((\opb{p_2}\sha)_{\overline\Gamma_f\setminus\Gamma_f})\br,
\]
where $p_2\colon\overline\Gamma_f\to S$ is the projection.
\end{definition}

\begin{remark}
If $f$ is induced by a map $\check f\colon \bM \to S$, then one has an equivalence
\[
\BDC(\ind\sha_{(M,\bM)}) \simeq \BDC\bl\ind\opb{\check f}\sha\br/
\BDC\bl\ind((\opb{\check f}\sha)_{\bM\setminus M})\br.
\]
\end{remark}

Let us set
\[
\sha_{(M,\bM)} = \opb{p_2}\sha,
\]
where $p_2\colon\overline\Gamma_f\to S$.
It is a sheaf of $\field$-algebras on $\overline\Gamma_f$.
One can define the functors\footnote{For $M=\bM=S$, the functors $\ltens[\sha]$ and $\rhom[\sha]$ are denoted by $\beta(\ast)\ltens[\beta\sha]\ast$ and $\rihom[\beta\sha](\beta(\ast),\ast)$, respectively, in \cite{KS01}.}
\begin{align*}
\rihom &\colon \BDC(\ifield_{(M,\bM)})^\op \times \BDC(\ind\sha_{(M,\bM)}) \To \BDC(\ind\sha_{(M,\bM)}), \\
\ltens[\sha] &\colon \BDC(\sha_{(M,\bM)}^\op) \times \BDC(\ind\sha_{(M,\bM)}) \To \BDC(\ind\field_{(M,\bM)}), \\
\rhom[\sha] &\colon \BDC(\sha_{(M,\bM)})^\op \times \BDC(\ind\sha_{(M,\bM)}) \To \BDC(\ind\field_{(M,\bM)}).
\end{align*}

\begin{lemma}
Let $F\in\BDC(\field_{(M,\bM)}) \simeq \BDC(\field_M)$, $\shm\in\BDC(\sha_{(M,\bM)})$, $\shn\in\BDC(\sha_{(M,\bM)}^\op)$ and $\shk\in\BDC(\ind\sha_{(M,\bM)})$. Then there are isomorphisms
\begin{align*}
\rihom(F,\shn\ltens[\sha]\shk) &\simeq \shn\ltens[\sha]\rihom(F,\shk), \\
\rihom(F,\rhom[\sha](\shm,\shk)) &\simeq \rhom[\sha](\shm,\rihom(F,\shk)) \\
&\simeq \rhom[\sha](F\tens\shm,\shk).
\end{align*}
\end{lemma}

\medskip

Recall that $\pi\colon M \times \R_\infty \to M$ denotes the projection.

\begin{definition}
For $\sha$ a sheaf of $\field$-algebras on $M$, we set
\[
\BEC[\ind\sha]{} = \BDC(\ind\sha_{M\times\R_\infty})/\{K\semicolon \opb\pi\roim\pi K \isoto K\}.
\]
\end{definition}

We have a forgetful functor
\[
\BEC[\ind\sha]{} \to \BEC M. 
\]

\begin{remark}
The results on $\BEC M$ can be extended to this context with $\BEC[\ind\sha]{}$.
\end{remark}

\section{Review of tempered functions}\label{se:tempered}

We recall here some constructions of \cite{KS96,KS01}.
In particular, we recall the ind-sheaf $\Ot_X$ of tempered holomorphic functions on a complex analytic manifold $X$, which plays a fundamental role in this paper. We end this section by adapting the notion of bordered space to the framework of analytic spaces.

\subsection{Real setting}

Let $M$ be a real analytic manifold and let $U\subset M$ be an open subset.

One says that a function $\varphi\colon U \to \C$ has polynomial growth
at $x_\circ\in M\setminus U$ if there exist a
sufficiently small compact neighborhood $K$ of $x_\circ$ and constants
$C>0$, $r\in\Z_{>0}$ such that
\begin{equation}
\label{eq:poly}
|\varphi(x)| \leq 
C \, \dist(K\setminus U,x)^{-r} \quad \text{for any }
x\in K \cap U.
\end{equation}
(Here ``$\dist$'' denotes the Euclidean distance with respect to a local
coordinate system.)

One says that a smooth function $\varphi \in \Ci_M(U)$ is tempered at $x_0\in M\setminus U$ if all of its derivatives have polynomial growth at $x_0$.

Denote by $\Db_M$ the sheaf of Schwartz's distributions on $M$.

\begin{definition}[{\cite[Definition 7.2.5]{KS01}}]
\bnum
\item
For a subanalytic open subset $U\subset M$, we define $\Cit_M(U)$ as the set of $C^\infty$-functions defined on $U$ which are tempered at every point of $M\setminus U$. Then $\Cit_M$ is a subanalytic sheaf.
\item
For a subanalytic open subset $U\subset M$,
we define the sheaf of $\Cfield$-algebras $\shc_{U|M}^{\infty,\textrm{temp}} \defeq \hom(\Cfield_U,\Cit_M)$.
\item
The subanalytic sheaf of tempered distributions on $M$ is defined by
\begin{align*}
\Dbt_M(V) &\defeq \Db_M(M) / \sect_{M\setminus V}(M; \Db_M) \\
&\simeq \Im(\Db_M(M)\to\Db_M(V))
\end{align*}
for any subanalytic open subset $V\subset M$.
We still denote by $\Dbt_M$ the corresponding subanalytic ind-sheaf.
\ee
\end{definition}

There is a morphism $\Dbt_M\to\Db_M$ of ind-sheaves.

For any open subset $V\subset M$ we have
\[
\shc_{U|M}^{\infty,\textrm{temp}}(V) = \{\varphi\in\Ci_M(V\cap U) \semicolon \varphi \text{ is tempered at any point of }V\setminus U\}.
\]

One has the following lemma.

\begin{lemma}
For any $\R$-constructible sheaf $F$,
\[
H^k\rihom(F,\Dbt_M) = 0\quad\text{for any }k\neq0.
\]
\end{lemma}

\begin{proof}
For any $\R$-constructible sheaf $G$ and any $k\neq 0$, one has
\[
H^k\RHom(G,\rihom(F,\Dbt_M)) \simeq H^k\RHom(G\tens F,\Dbt_M) \simeq 0,
\]
where the last isomorphism follows from \cite[Proposition 7.2.6 (i)]{KS01}.
\end{proof}

\begin{proposition}\label{pro:CDtoD}
Let $U\subset M$ be a subanalytic open subset.
The product $\Ci_M \tens \Db_M \to \Db_M$ induces a $\Cfield_M$-algebra homomorphism
\[
\shc_{U|M}^{\infty,\textrm{temp}} \to \shEnd(\ihom(\Cfield_U,\Dbt_M)).
\]
In other words, $\ihom(\Cfield_U,\Dbt_M) \in \ind(\shc_{U|M}^{\infty,\textrm{temp}})$.
\end{proposition}

\begin{proof}
Let $V\subset M$ be a relatively compact subanalytic open subset.
By~\cite[Lemma 3.3]{Kas84}, 
the product induces a natural morphism
\[
\Cit_M(U\cap V) \tens \Dbt_M(U\cap V) \to \Dbt_M(U\cap V).
\]
\end{proof}

For a closed subset $Z\subset M$, denote by $\shi^\infty_{M,Z} \subset \shc^\infty_{M}$ the subsheaf of functions which vanish on $Z$ up to infinite order.
Recall the Whitney functor of \cite{KS96}
\[
\ast\wtens \shc^\infty_M \colon \BDC_\Rc(\Cfield_M) \to \BDC(\Cfield_M).
\]
It is characterized by setting $\Cfield_U\wtens \shc^\infty_{M} \defeq \shi^\infty_{M,M\setminus U}$ for any subanalytic open subset  $U\subset M$.

One says that a function $\varphi\in\shc^\infty_{M}(U)$ is rapidly decreasing
at $x_\circ\in M\setminus U$ if there exists a
sufficiently small compact neighborhood $K$ of $x_\circ$ such that for any $r\in\Z_{>0}$ and $\alpha\in\Z_{\geq 0}^n$ there is a constant
$C>0$ with
\[
|\partial_x^\alpha\varphi(x)| \leq C\,\dist(K\setminus U,x)^r \quad \text{for any }
x\in K \cap U.
\]
(Here ``$\dist$'' and $\partial^\alpha$ are taken with respect to a local
coordinate system.)

One says that $\varphi\in\shc^\infty_{M}(U)$ is rapidly decreasing at the boundary of $U$ if it is rapidly decreasing at each point of the boundary of $U$.

\begin{lemma}\label{lem:W}
A section of $\shc^\infty_{M}(U)$ extends to a global section of $\Cfield_U\wtens \shc^\infty_{M}$ if and only if it is rapidly decreasing at the boundary of $U$.
\end{lemma}

\subsection{Complex setting}

Let $X$ be a complex analytic manifold.
Denote by $X_\R$ the real analytic manifold underlying $X$. It is identified with the diagonal of $X\times\overline X$, where $\overline X$ is the conjugate complex manifold of $X$.
Recall that $(\overline X)_\R = X_\R$ and that sections of $\O_{\overline X}$ are the complex conjugates of sections of $\O_X$.

Recall that, by Dolbeault resolution, one has
\[
\O_X \simeq \rhom[\D_{\overline X}](\O_{\overline X}, \Db_{X_\R}).
\]

\begin{definition}[{\cite[\S7.3]{KS01}}]
One sets
\begin{align*}
\Ot_X &= \rhom[\D_{\overline X}](\O_{\overline X}, \Dbt_{X_\R}) \in \BDC(\ind\D_X), \\
\Ovt_X &= \Omega_X \tens[\O_X] \Ot_X \in \BDC(\ind\D_X^\op).
\end{align*}
\end{definition}

The canonical morphism $\Dbt_{X_\R} \to \Db_{X_\R}$ induces a canonical morphism $\Ot_X\to\O_X$ in $\BDC(\ind\D_X)$. Note that $\Ot_X\in \BDC_\suban(\iCfield_X)$. It is not concentrated in degree zero, in general.

\begin{notation}
The classical de Rham and solution functors are
\begin{align*}
\dr_X &\colon \BDC(\D_X) \to \BDC(\Cfield_X), &
\shm &\mapsto \Omega_X \ltens[\D_X] \shm, \\
\sol_X &\colon \BDC(\D_X)^\op \to \BDC(\Cfield_X), &
\shm &\mapsto \rhom[\D_X] (\shm,\O_X),
\end{align*}
and the tempered de Rham and solution functors are
\begin{align*}
\drt_X &\colon \BDC(\D_X) \to \BDC(\iCfield_X), &
\shm &\mapsto \Ovt_X \ltens[\D_X] \shm, \\
\solt_X &\colon \BDC(\D_X)^\op \to \BDC(\iCfield_X), &
\shm &\mapsto \rhom[\D_X] (\shm,\Ot_X).
\end{align*}
\end{notation}

One has
\[
\sol_X \simeq \alpha_X\solt_X,
\quad
\dr_X \simeq \alpha_X\drt_X.
\]
Recall that, by \cite[Lemma~7.4.11]{KS01}, for $\shl\in\BDC_\reghol(\D_X)$ one has
\[
\solt_X(\shl) \simeq \sol_X(\shl),
\quad
\drt_X(\shl) \simeq \dr_X(\shl).
\]
For $\shm\in\BDC_\coh(\D_X)$, one has
\[
\solt_X(\shm) \simeq \drt_X(\ddual_X\shm)[-d_X].
\]
Note that
\[
\drt_X(\O_X) \simeq \dr_X(\O_X) \simeq \Cfield_X[d_X].
\]

Let us recall some functorial properties of the tempered de Rham and solution functors.

\begin{theorem}[{\cite[Theorems 7.4.1, 7.4.6 and 7.4.12]{KS01}}]
\label{thm:ifunct}
Let $f\colon X\to Y$ be a complex analytic map.
\bnum
\item
There is an isomorphism in $\BDC(\ind\opb f\D_Y)$
\[
\epb f \Ot_Y[d_Y] \simeq \D_{Y\leftarrow X} \ltens[\D_X] \Ot_X [d_X].
\]
\item
For any $\shn\in\BDC(\D_Y)$ there is an isomorphism in $\BDC(\iCfield_X)$
\[
\drt_X(\dopb f\shn) [d_X] \simeq 
\epb f \drt_Y(\shn) [d_Y].
\]
\item
Let $\shm\in\BDC_\good(\D_X)$, and assume that $\supp\shm$ is proper over $Y$. Then there is an isomorphism in $\BDC(\iCfield_Y)$
\[
\drt_Y(\doim f\shm) \simeq 
\reeim f\drt_X(\shm) .
\]
\item
Let $\shl\in\BDC_\reghol(\D_X)$. 
Then there is an isomorphism in $\BDC(\ind\D_X)$
\[
\Ot_X \ltens[\O_X] \shl \simeq 
\rihom(\sol_X(\shl),\Ot_X).
\]
In particular, for a closed hypersurface $Y\subset X$, one has
\[
\Ot_X \ltens[\O_X] \O_X(*Y) \simeq \rihom(\Cfield_{X\setminus Y},\Ot_X).
\]
\ee
\end{theorem}

\subsection{Back to the real setting}

\begin{proposition}[{\cite[Theorem 5.10]{KS96}}]
\label{pro:DbtOt}
Let $X$ be a complexification of a real analytic manifold $M$, 
and denote by $i\colon M \to X$ the embedding. Then
\[
\epb i \Ot_X[d_X] \simeq \Dbt_M \tens \ori_M.
\]
\end{proposition}

\begin{lemma}
\label{lem:fDbt}
Let $f\colon M \to N$ be a morphism of real analytic manifolds. Then
\[
\epb f\Dbt_N \simeq \D_{N\leftarrow M} \ltens[\D_M] \Dbt_M,
\]
where $\D_{N\leftarrow M} = \D_{Y\leftarrow X}|_M\tens \ori_M \tens \opb f \ori_N$ for a complexification $X\to Y$ of $f$.
\end{lemma}

\begin{proof}
Consider the diagram
\[
\xymatrix@C=7ex{
M \ar[r]^{f} \ar[d]^{i_M} & N \ar[d]^{i_N} \\
X \ar[r]^{\tilde f} & Y.
}
\]
Then one has the isomorphisms
\begin{align*}
\D_{N\leftarrow M} \ltens[\D_M] \Dbt_M
&\underset{(*)}\simeq \opb{i_M}\D_{Y\leftarrow X} \ltens[\opb{i_M}\D_X] \epb{i_M}\Ot_X\tens\opb f\ori_N[d_M] \\
&\simeq \epb{i_M}(\D_{Y\leftarrow X} \ltens[\D_X] \Ot_X)\tens\opb f\ori_N[d_M] \\
&\underset{(**)}\simeq \epb{i_M}\epb{\tilde f}\Ot_Y \tens\opb f\ori_N[d_N] \\
&\simeq \epb f\epb{i_N}\Ot_Y \tens\opb f\ori_N[d_N] \\
&\underset{(*)}\simeq \epb f\Dbt_N,
\end{align*}
where $(*)$'s follow from Proposition~\ref{pro:DbtOt}, and $(**)$ follows from Theorem~\ref{thm:ifunct}~(i).
\end{proof}

\subsection{Real analytic bordered spaces}

\begin{definition}
The category of \emph{real analytic bordered spaces} is the category whose objects are pairs $(M,\bM)$ where $\bM$ is a real analytic manifold and $M\subset \bM$ is an open subanalytic subset.
Morphisms $f\colon (M,\bM) \to (N,\bN)$ are real analytic maps $f\colon M\to N$ such that
\bnum
\item
$\Gamma_f$ is a subanalytic subset of $\bM\times\bN$, and
\item
$\overline\Gamma_f\to\bM$ is proper.
\ee
\end{definition}

Hence a morphism of real analytic bordered spaces is a morphism of bordered spaces.

\begin{lemma}
Let $f\colon(M,\bM)\to(N,\bN)$ be a morphism of real analytic bordered spaces.
Then $f$ is an isomorphism if the following conditions are satisfied
\bnum
\item $f\colon M\to N$ is an isomorphism of real analytic manifolds,
\item $\overline\Gamma_f\to\bN$ is proper.
\ee
\end{lemma}

Recall that $j_M\colon(M,\bM)\to\bM$ and $j_N\colon(N,\bN)\to\bN$ denote the natural morphisms.

\begin{proposition}\label{pro:Dbtbordered}
Let $f\colon(M,\bM)\to(N,\bN)$ be an isomorphism of real analytic bordered spaces. Then there is an isomorphism in $\BDC(\iCfield_{(M,\bM)})$
\[
\opb{j_M}\Dbt_{\bM} \simeq \opb f\opb{j_N}\Dbt_{\bN}.
\]
\end{proposition}

\begin{proof}
We shall regard $\roimv{j_{M\sep*}}\opb{j_M}\Dbt_{\bM}$ and $\roimv{j_{M\sep*}}\opb f\opb{j_N}\Dbt_{\bN}$ as subanalytic sheaves on $\bM$. Hence it is enough to show that
\[
\Dbt_{\bM}(\opb f(V)) \simeq \Dbt_{\bN}(V)
\]
for any relatively compact subanalytic open subset $V$ of $\bN$ contained in $N$.

By \cite[Theorem 6.1]{KS96}, the topological dual of the above isomorphism is given by
\[
\sect(\bM;\Cfield_{\opb f(V)}\wtens \shc^\infty_{\bM}) \simeq 
\sect(\bN;\Cfield_{V}\wtens \shc^\infty_{\bN}).
\]
Hence, by Lemma~\ref{lem:W}, the proposition follows from Lemma~\ref{lem:repdec} below.
\end{proof}

\begin{lemma}
\label{lem:repdec}
With the same notations as in the above proposition, let $V$ be a relatively compact subanalytic open  subset of $\bN$ contained in $N$, and let $u\in\shc^\infty_{\bN}(V)$. Then $u$ is rapidly decreasing at the boundary of $V$ if and only if $f^*(u)\in \shc^\infty_{\bM}(\opb f(V))$ is rapidly decreasing at the boundary of $\opb f(V)$.
\end{lemma}

\begin{proof}
Denote by $q_1\colon\overline\Gamma_f\to\bM$ and $q_2\colon\overline\Gamma_f\to\bN$ the projections. 
Note that, since $f$ is an isomorphism of real analytic bordered spaces, one has 
\[
\Gamma_f = \overline\Gamma_f\times_{\bM}M = \overline\Gamma_f\times_{\bN}N.
\]

Assume that $u$ is rapidly decreasing at the boundary of $V$. For $x_\circ\in\partial(\opb f(V))$ let us choose a sufficiently small open neighborhood $W$ of $x_0$ and local coordinates $(x_1,\dots,x_n)$.
Since $\opb{q_1}(x_\circ)$ is compact, shrinking $W$ if necessary, there exist finitely many relatively compact  subanalytic open subsets $\{V_i\}$ and $\{V_i'\}$ of $\bN$ such that
\bna
\item 
$\overline{V_i'} \subset V_i$,
\item
$\opb{q_1}(W) \subset \Union_i(W\times V_i')$,
\item
there exist local coordinates $(y_1^i,\dots,y_n^i)$ on $V_i$.
\ee
Then $f(\opb f(V)\cap W) \subset \Union_i V_i'$.

It follows that the derivatives $\partial_x^\alpha f^*(u)$ are linear combinations of derivatives $\partial_{y^i}^\beta u$ with coefficients given by products of terms of the form $\partial_x^\gamma y_k^i$.
Since $\partial_{y^i}^\beta u$ are rapidly decreasing and $\partial_x^\gamma y_k^i$ have polynomial growth, it follows that $f^*(u)|_{\opb f(V\cap V_i')}$ is rapidly decreasing at $x_\circ$ for any $i$. Hence $f^*(u) \in \shc^\infty_{\bM}(\opb f V)$ is rapidly decreasing at $x_\circ$.
\end{proof}

\section{Exponential \texorpdfstring{$\D$}{D}-modules}\label{se:expo}

Let $X$ be a complex analytic manifold. According to the results of
Mochizuki~\cite{Moc09,Moc11} and  Kedlaya~\cite{Ked10,Ked11}
(see \S\ref{sse:normal} below), a fundamental model for irregular holonomic $\D_X$-modules is the exponential $\D_X$-module associated with a meromorphic connection $d+d\varphi$
for a meromorphic function  $\varphi\in\O_X(*Y)$
with poles on a hypersurface $Y$.
In this section we describe the tempered de Rham complex of such exponential $\D_X$-modules.

\subsection{Exponential $\D$-modules}
Let $X$ be a complex analytic manifold.

\begin{definition}
\label{def:expY}
Let $Y\subset X$ be a complex analytic hypersurface. Set $U=X\setminus Y$.
For $\varphi\in\O_X(*Y)$, set
\begin{align*}
\D_X e^\varphi &= \D_X/\{P\semicolon Pe^\varphi=0 \text{ on } U\}, \\
\she^\varphi_{U|X}&=\D_X e^\varphi(*Y).
\end{align*}
\end{definition}

Hence $\D_X e^\varphi\subset\she^\varphi_{U|X}$.
Note that $\she^\varphi_{U|X}$ is a holonomic $\D_X$-module which satisfies
\[
\she^\varphi_{U|X} \simeq \she^\varphi_{U|X}(*Y), \qquad \ss(\she^\varphi_{U|X}) = Y.
\]
Note that the map $\O_X(*Y)\to[\cdot e^\varphi]\she^\varphi_{U|X}$ induces an isomorphism as $\O_X$-modules.

\begin{lemma}
\label{lem:YEphi}
For $\varphi\in\O_X(*Y)$ one has
\[
(\ddual_X\she^\varphi_{U|X})(*Y) \simeq \she^{-\varphi}_{U|X} .
\]
\end{lemma}

\begin{proof}
The morphism $\D_X e^{-\varphi}(*Y) \dtens \D_X e^\varphi(*Y) \to \O_X (*Y)$ induces a morphism $\ddual_X\she^\varphi_{U|X} \to \she^{-\varphi}_{U|X}$. Since it is an isomorphism outside of $Y$, the statement follows.
\end{proof}

\begin{remark}
The isomorphism $\ddual_X\she^\varphi_{U|X} \simeq \she^{-\varphi}_{U|X}$ does not hold in general. For example, let $X=\C^2\owns(u,v)$, $Y=\{v=0\}$ and $\varphi(u,v) = u^2/v^2$. Then $\she^\varphi_{U|X} \simeq \D_X v^{-2} e^\varphi$ and there is an epimorphism 
$$\she^\varphi_{U|X} \twoheadrightarrow \shb_{\{(0,0)\}}\simeq\D_X/(\D_Xu+\D_Xv).$$
Hence $\ddual_X\she^\varphi_{U|X}$ contains $\shb_{\{(0,0)\}}$ as a submodule. 
\end{remark}

\subsection{Tempered de Rham}
Our aim in this subsection is to describe the tempered de Rham complex of an exponential $\D$-module.

Let $X$, $Y$, $U$ and $\varphi$ be as in
Definition~\ref{def:expY}.
For $c\in\R$, set for short
\[
\{\Re \varphi < c\} = \{x\in U\semicolon \Re \varphi(x) < c\} \subset X.
\]

\begin{notation}
\label{not:<?}
We set
\begin{align*}
\Cfield_{\{\Re \varphi <\ast\}} &\defeq \indlim[c\rightarrow+\infty]\Cfield_{\{\Re \varphi < c\}} \in \ind(\Cfield_X), \\
E^\varphi_{U|X} &\defeq \rihom(\Cfield_U,\Cfield_{\{\Re \varphi <\ast\}}) \in \BDC(\iCfield_X).
\end{align*}
\end{notation}

For example, denoting by $z \in \C \subset \PP$ the affine coordinate of the complex projective line, one has 
\begin{equation}
\label{eq:HjEt}
H^j E^z_{\C|\PP} \simeq
\begin{cases}
\Cfield_{\{\Re z <\ast\}} &\text{for }j=0, \\
\Cfield_{\{\infty\}} &\text{for }j=1, \\
0 &\text{otherwise}.
\end{cases}
\end{equation}

\begin{proposition}
\label{pro:Solphi}
Let $Y\subset X$ be a closed complex analytic hypersurface, and set $U=X\setminus Y$.
For $\varphi\in\O_X(*Y)$, there is an isomorphism in $\BDC(\iCfield_X)$
\[
\drt_X(\she^{-\varphi}_{U|X}) \simeq E^\varphi_{U|X}[d_X].
\]
\end{proposition}

The fundamental case where $X=\C$ and $\varphi(z)=1/z$ was considered in \cite[Proposition~7.3]{KS03}.

In order to prove the above proposition, we need some preliminary results.

\begin{lemma}\label{lem:IhomU}
With the above notations, one has
\[
\drt_X(\she^{-\varphi}_{U|X})  \isoto \rihom(\Cfield_U,\drt_X(\she^{-\varphi}_{U|X})).
\]
\end{lemma}

\begin{proof}
One has
\begin{align*}
\Ovt_X \ltens[\D_X] \she^{-\varphi}_{U|X} &\simeq
\Ovt_X \ltens[\D_X] (\she^{-\varphi}_{U|X} \dtens \O_X(*Y)) \\
&\simeq (\O_X(*Y) \ltens[\O_X] \Ovt_X) \ltens[\D_X] \she^{-\varphi}_{U|X} \\
&\simeq \rihom(\Cfield_U, \Ovt_X \ltens[\D_X] \she^{-\varphi}_{U|X}).
\end{align*}
The last isomorphism follows from Theorem~\ref{thm:ifunct}~(iv).
\end{proof}

Let $M$ be a real analytic manifold, and $i\colon M\to X$ a complexification of $M$.
For $\shm\in\BDC(\D_X)$, let us set
\begin{align*}
\drt_M(\shm) &= \Dbvt_M \ltens[\D_X] \shm \\
&\simeq \epb i \drt_X(\shm)[d_X] \in\BDC(\iCfield_M),
\end{align*}
where $\Dbvt_M = \Dbt_M \tens \ori_M \tens[\opb i \O_X] \opb i \Omega_X \simeq \epb i \Ovt_X[d_X]$ is the subanalytic ind-sheaf of tempered distribution densities.

Note that, considering the complexification $X_\R\subset X\times\overline X$, one has
\begin{equation}
\label{eq:drXXR}
\drt_X(\shm) \simeq \drt_{X_\R}(\shm\detens\O_{\overline X})[-d_X].
\end{equation}

Let $\PR$ be the real projective line and denote by $x$ the coordinate on $\R=\PR\setminus\{\infty\}$.
Note that the object of $\BDC(\iCfield_\PR)$
\begin{align*}
\rihom(\Cfield_\R , \Cfield_{\{x <\ast\}}) 
&\simeq \ihom(\Cfield_\R , \Cfield_{\{x <\ast\}}) \\
&\simeq \indlim[c\rightarrow+\infty]\Cfield_{\{x<c\}\union\{\infty\}}
\end{align*}
is concentrated in degree zero.

\begin{lemma}\label{lem:Rt}
Let $\PR$ be the real projective line.
Denote by $x$ the coordinate on $\R=\PR\setminus\{\infty\}$ and by $z$ the coordinate on $\C=\PP\setminus\{\infty\}$.
Then there is an isomorphism in $\BDC(\iCfield_\PR)$
\[
\drt_\PR(\she^{-z}_{\C|\PP}) \simeq 
\ihom(\Cfield_\R , \Cfield_{\{x <\ast\}})[1].
\]
\end{lemma}

\begin{proof}
One has 
\begin{align*}
\Dbvt_\PR \ltens[\D_\PP] \she^{-z}_{\C|\PP} 
&\simeq \ihom(\Cfield_\R,\Dbvt_\PR) \ltens[\D_\PP] \she^{-z}_{\C|\PP} \\
&\simeq (\she^{-z}_{\C|\PP}) ^\mop \ltens[\D_\PP] \ihom(\Cfield_\R,\Dbt_\PR) \\
&\simeq \bigl( \ihom(\Cfield_\R,\Dbt_\PR) \to[\partial_x-1] \ihom(\Cfield_\R, \Dbt_\PR) \bigr) \mathbin{{=}{:}} \shs,
\end{align*}
where the complex $\shs$ is in degree $-1$ and $0$.
Here, the first isomorphism follows from the real analogue of Lemma~\ref{lem:IhomU}, $\mop$ is the functor in \eqref{def:r} 
and the last isomorphism follows from 
$\she^{-z}_{\C|\PP} \simeq \D_\PP/\D_\PP(\partial_z+1)$
and $(\she^{-z}_{\C|\PP}) ^\mop\simeq\D_\PP/(\partial_z-1)\D_\PP$.

Hence, we have to prove the isomorphisms of subanalytic sheaves
\[
H^{-1}\shs \simeq \ihom(\Cfield_\R , \Cfield_{\{x <\ast\}}), \qquad H^0\shs \simeq 0.
\]

Let $U\subset \PR$ be an open subanalytic subset, so that $U\cap\R$ is a finite union of open intervals.

The first isomorphism follows from the fact that $e^x\in\Dbt_\PR(U\cap\R)$ if and only 
if $U\cap\R \subset \{ x<c \}$ for some $c$.

To show that $H^0\shs \simeq 0$ it is enough to consider the commutative diagram
\[
\xymatrix{
\Dbt_\PR(\R) \ar@{->>}[r]^{\partial_x-1} \ar[d] & \Dbt_\PR(\R) \ar@{->>}[d] \\
\Dbt_\PR(U\cap\R) \ar[r]^{\partial_x-1} & \Dbt_\PR(U\cap\R)
}
\]
and notice that the vertical arrow, as well as the top horizontal arrow, is surjective.
\end{proof}

\begin{lemma}[{cf.~\cite[Proposition~7.3]{KS03}}]
\label{lem:t}
Let $\PP$ be the complex projective line and
denote by $z$ the coordinate on $\C=\PP\setminus\{\infty\}$.
There is an isomorphism in $\BDC(\iCfield_\PP)$
\[
\drt_\PP(\she^{-z}_{\C|\PP}) \simeq 
E^z_{\C|\PP}[1].
\]
\end{lemma}

\begin{proof}
Consider the real analytic bordered spaces $(\C_\R,\PP_\R)$ and $(\R^2,\PR^2)$.
Then the morphism $f\colon (\R^2,\PR^2) \to (\C_\R,\PP_\R)$ given by $(x,y) \mapsto x+\sqrt{-1}y$ is an isomorphism of real analytic bordered spaces.
Consider the morphisms
\[
\xymatrix{
\PR^2 & (\R^2,\PR^2) \ar[l]_-{k} \ar[r]^-f & (\C_\R,\PP_\R) \ar[r]^-j & \PP_\R .
}
\]
By Proposition~\ref{pro:Dbtbordered},
\[
\opb f \opb j \Dbt_{\PP_\R} \simeq \opb k\Dbt_{\PR^2}.
\]
By \eqref{eq:drXXR} and Lemma~\ref{lem:IhomU},
\[
\drt_\PP(\she^{-z}_{\C|\PP})  \simeq
\roim j\opb j((\she^{-z}_{\C|\PP} \detens \O_{\overline\PP})^\mop \ltens[\D_{\PP\times\overline\PP}] \Dbt_{\PP_\R})[-1].
\]
Note that $\opb j((\she^{-z}_{\C|\PP} \detens \O_{\overline\PP})^\mop \ltens[\D_{\PP\times\overline\PP}] \Dbt_{\PP_\R})$ is represented by the complex
\[
\opb j\Dbt_{\PP_\R} \To[(\partial_z-1,\;\partial_{\,\overline z})]
(\opb j\Dbt_{\PP_\R})^2 \To[(-\partial_{\,\overline z},\;\partial_z-1)]
\opb j\Dbt_{\PP_\R}.
\]
Applying $\opb f$, we get the complex
\[
\opb k\Dbt_{\PR^2} \To[(\partial_x-1,\;\partial_y-\sqrt{-1})]
(\opb k\Dbt_{\PR^2})^2 \To[(-\partial_y+\sqrt{-1},\;\partial_x-1)]
\opb k\Dbt_{\PR^2}.
\]
This last complex represents $\opb k((\she^{-u}_{\C|\PP} \detens \she^{-\sqrt{-1}v}_{\C|\PP})^\mop \ltens[\D_{\PP^2}] \Dbt_{\PR^2})$,
where $(u,v)\in\C^2$ is a complexification of $(x,y)\in\R^2$.

We have thus proved
\[
\drt_\PP(\she^{-z}_{\C|\PP})  \simeq \roim j\roim f\opb k((\she^{-u}_{\C|\PP} \detens \she^{-\sqrt{-1}v}_{\C|\PP})^\mop \ltens[\D_{\PP^2}] \Dbt_{\PR^2})[-1].
\]

By Proposition~\ref{pro:CDtoD}, the function $e^{-\sqrt{-1}y}\in\Cit_{\PR^2}(\R^2)$ induces an automorphism
of $\opb k\Dbt_{\PR^2}$. This automorphism interchanges the actions of $\partial_y$ and of $\partial_y-\sqrt{-1}$. 
Hence, for a $\D_\PP$-module $\shm$, it induces an isomorphism
\begin{equation}
\label{eq:eiy}
\opb k ((\shm \detens \she^{-\sqrt{-1}v}_{\C|\PP})^\mop \ltens[\D_{\PP^2}] \Dbt_{\PR^2})
\simeq \opb k( (\shm \detens \O_\PP)^\mop \ltens[\D_{\PP^2}] \Dbt_{\PR^2}).
\end{equation}

We then have, denoting by $p_1$ the first projection $\PR^2\to\PR$, 
\begin{align*}
\drt_\PP(\she^{-z}_{\C|\PP})
&\simeq \roim j\roim f\opb k((\she^{-u}_{\C|\PP} \detens \O_\PP)^\mop \ltens[\D_{\PP^2}] \Dbt_{\PR^2}) [-1]\\
&\simeq \roim j\roim f\opb k((\she^{-u}_{\C|\PP})^\mop \ltens[\D_\PP] \D_{\PP\from[p_1]\PP^2} \ltens[\D_{\PP^2}] \Dbt_{\PR^2})[-1] \\
&\underset{(1)}\simeq \roim j\roim f\opb k((\she^{-u}_{\C|\PP})^\mop \ltens[\D_\PP] 
\epb{p_1} \Dbt_{\PR})[-1] \\
&\simeq \roim j\roim f\opb k\opb{p_1}\drt_\PR(\she^{-u}_{\C|\PP}) \\
&\underset{(2)}\simeq \roim j\roim f\opb k\opb{p_1}\Cfield_{\{x <\ast\}}[1] \\
&\simeq \roim j\opb j \Cfield_{\{\Re z <\ast\}}[1] ,
\end{align*}
where $(1)$ follows from Lemma~\ref{lem:fDbt} and $(2)$ follows from Lemma~\ref{lem:Rt}.
\end{proof}

\begin{lemma}
\label{lem:s/t}
Denote by $(u,v)$ the coordinates of   $\C^2$.
There is an isomorphism in $\BDC(\iCfield_{\C^2})$
\[
\drt_{\C^2}(\she^{-u/v}_{\{v\neq0\}|\C^2}) \simeq 
E^{u/v}_{\{v\neq0\}|\C^2}[2].
\]
\end{lemma}

\begin{proof}
Recall that $z$
denotes the coordinate on $\C=\PP\setminus\{\infty\}$.
Denote by $\widetilde\C^2$ the blow-up of the origin in $\C^2$.
Recall that $\widetilde \C^2 \subset \C^2 \times \PP$ is the surface of equation $uz_0 = vz_1$, where $(z_0:z_1)\in\PP$ 
are homogeneous coordinates with $z=z_1/z_0$. 
Consider the maps
\begin{equation*}
\xymatrix{
\C^2  & \widetilde \C^2 \ar[l]_p \ar[r]^q & \PP
}
\end{equation*}
induced by the projections from $\C^2 \times \PP$.
Since $\opb q(\infty) \subset \opb p(\{v=0\})$, one has
\begin{align}
\label{eq:pUqC}&\ba{l}
 p^{-1}(\{v\neq 0\}) \subset q^{-1}(\C), \\[1ex]
\she^{-u/v}_{\{v\neq0\}|\C^2} \simeq \O_{\C^2}(*\{v=0\})
\dtens\doim p\dopb q \she^{-z}_{\C|\PP}.
\ea\end{align}
It follows
\begin{align*}
\drt_{\C^2}(\she^{-u/v}_{\{v\neq0\}|\C^2} )
&\simeq \drt_{\C^2} \bigl(\O_{\C^2}(*\{v=0\})\dtens\doim p\dopb q \she^{-z}_{\C|\PP}\bigr) \\
&\simeq \rihom (\Cfield_{\{v\neq 0\}}, \drt_{\C^2}(\doim p\dopb q \she^{-z}_{\C|\PP})),
\end{align*}
where the last isomorphism follows from Theorem~\ref{thm:ifunct}~(iv).
Note that
\begin{align*}
\drt_{\C^2}(\doim p\dopb q \she^{-z}_{\C|\PP}) 
&\simeq \roim p \epb q(\drt_\PP(\she^{-z}_{\C|\PP})) [-1]\\
&\simeq \roim p \epb q\rihom(\Cfield_\C,\Cfield_{\{\Re z <\ast\}})\\
&\simeq \roim p\rihom(\opb q\Cfield_\C,\epb q\Cfield_{\{\Re z <\ast\}}) \\
&\simeq \roim p\rihom(\opb q\Cfield_\C,\opb q\Cfield_{\{\Re z <\ast\}}) [2].
\end{align*}
Here, 
the first isomorphism follows from Theorem~\ref{thm:ifunct}~(ii) and (iii),
the second isomorphism follows from Lemma~\ref{lem:t}, and
the last isomorphism follows from the fact that $q$ is smooth with fiber $\C$.
Hence
\begin{align*}
\rihom (\Cfield_{\{v\neq 0\}}, & \drt_{\C^2}(\doim p\dopb q \she^{-z}_{\C|\PP})) \\
&\simeq \rihom (\Cfield_{\{v\neq 0\}}, \roim p\rihom(\opb q\Cfield_\C,\opb q\Cfield_{\{\Re z <\ast\}}))[2] \\
& \simeq \roim p \rihom (\opb p \Cfield_{\{v\neq 0\}} \tens \opb q\Cfield_\C,\opb q\Cfield_{\{\Re z <\ast\}})[2] \\
& \underset{(1)}\simeq \roim p \rihom (\opb p \Cfield_{\{v\neq 0\}},\opb q\Cfield_{\{\Re z <\ast\}})[2] \\
& \simeq \roim p \rihom (\opb p \Cfield_{\{v\neq 0\}},\opb p \Cfield_{\{v\neq 0\}}\tens\opb q\Cfield_{\{\Re z <\ast\}})[2] \\
& \underset{(2)}\simeq \roim p \rihom (\opb p \Cfield_{\{v\neq 0\}},\opb p \Cfield_{\{\Re(u/v) <\ast\}})[2] \\
& \underset{(3)}\simeq  \roim p\rihom (\opb p\Cfield_{\{v\neq 0\}},\epb p \Cfield_{\{\Re(u/v) <\ast\}})[2] \\
& \simeq  \rihom (\Cfield_{\{v\neq 0\}}, \Cfield_{\{\Re(u/v) <\ast\}})[2] .
\end{align*}
Here,
$(1)$ follows from \eqref{eq:pUqC},
$(2)$ follows from the equality
\[
q^{-1}(\{\Re z<c\})\cap p^{-1}(\{v\neq 0\}) = p^{-1}(\{\Re(u/v) < c\})
\ \text{ for $c\in\R$,}
\]
and $(3)$ follows from the fact that $p$ is an isomorphism over $\{v\neq 0\}$.
\end{proof}

\begin{proof}[Proof of~Proposition~\ref{pro:Solphi}]
As in the previous lemma, denote by $(u,v)$ the coordinates in $\C^2$.
Write $\varphi=a/b$ for $a,b\in\O_X$ such that $Y=\opb b(0)$, and consider the map
\[
f = (a,b) \colon X\to \C^2.
\]
Since $f^{-1}(\{v=0\}) = b^{-1}(0) = Y$, one has
\begin{align}
\label{eq:UU'}
&f^{-1}(\{v\neq 0\}) = U, \\
\label{eq:EYY'}
&\she^{\varphi}_{U|X} \simeq \dopb f\she^{u/v}_{\{v\neq 0\}|\C^2}.
\end{align}
Note that
\begin{align*}
\drt_X(\dopb f\she^{-u/v}_{\{v\neq0\}|\C^2}) 
&\simeq \epb f (\drt_{\C^2}( \she^{-u/v}_{\{v\neq0\}|\C^2}))[2-d_X] \\
&\simeq \epb f\rihom(\Cfield_{\{v\neq 0\}},\Cfield_{\{\Re(u/v) <\ast\}})[4-d_X],
\end{align*}
where the first isomorphism follows from Theorem~\ref{thm:ifunct}~(ii), and the second isomorphism follows from Lemma~\ref{lem:s/t}.
Hence
\begin{align*}
\drt_X (\she^{-\varphi}_{U|X}) 
& \simeq \epb f\rihom(\Cfield_{\{v\neq 0\}},\Cfield_{\{\Re(u/v) <\ast\}})[4-d_X] \\
& \simeq \rihom(\opb f\Cfield_{\{v\neq 0\}}, \epb f\Cfield_{\{\Re(u/v) <\ast\}})[4-d_X] \\
& \underset{(1)}\simeq \rihom(\Cfield_U, \epb f\Cfield_{\{\Re(u/v) <\ast\}})[4-d_X] \\
& \underset{(2)}\simeq \rihom(\Cfield_U, \opb f\Cfield_{\{\Re(u/v) <\ast\}})[d_X] \\
& \simeq \rihom(\Cfield_U, \Cfield_{\{\Re \varphi <\ast\}})[d_X],
\end{align*}
where $(1)$ follows from \eqref{eq:UU'}, and
$(2)$ follows from Proposition~\ref{pro:opbepb}.
\end{proof}

\section{Normal form of holonomic \texorpdfstring{$\D$}{D}-modules}\label{se:normal}

On a complex curve, the classical results of Levelt-Turittin and of Hukuhara-Turittin describe the formal structure of a flat meromorphic connection and its asymptotic expansion on sectors.
Analogous statements in higher dimension have  recently been obtained 
by Mochizuki~\cite{Moc09,Moc11}
and Kedlaya~\cite{Ked10,Ked11}, after preliminary results and conjectures by Sabbah~\cite{Sab00}.

In this section we recall these statements in the language of $\D$-modules, and establish some lemmas that will be used later.
In particular, Lemma~\ref{lem:redux} below will be a key ingredient in our proof of the irregular Riemann-Hilbert correspondence.

\subsection{Real blow-up}\label{se:realblow}

Let $X$ be a complex manifold and $D\subset X$ a smooth closed hypersurface.
The total real blow-up 
\[
\varpi_\tot\colon \widetilde X_D^\tot \to X
\]
of $X$ along $D$ is the real analytic map of real analytic manifolds
locally defined as follows.

We take coordinates $(z,w)\in\C\times\C^{n-1}$ on $X$ such that $D=\{z=0\}$. Then one has
\[
\widetilde X_D^\tot = \{(t,\zeta,w)\in \R\times\C\times\C^{n-1} \semicolon |\zeta|=1 \}
\]
and
\[
\varpi_\tot\colon \widetilde X_D^\tot \to X, \quad (t,\zeta,w) \mapsto (t\zeta,w).
\]
Note that $\varpi_\tot$ is an unramified 2-sheeted covering over $X\setminus D$, so that we may write
\[
\opb{\varpi_\tot}(X\setminus D) = (X\setminus D) \times\{+,-\}.
\]

Consider the subsets locally defined by
\begin{align*}
\widetilde X_D^{> 0} &= \{(t,\zeta,w)\in \widetilde X_D^\tot \semicolon t > 0 \}
= (X\setminus D) \times\{+\}, \\
\widetilde X_D &= \{(t,\zeta,w)\in \widetilde X_D^\tot \semicolon t\geq 0 \} = \overline{\widetilde X_D^{> 0}}, \\
\widetilde X_D^{0} &= \{(t,\zeta,w)\in \widetilde X_D^\tot \semicolon t = 0 \} = \widetilde X_D \setminus \widetilde X_D^{> 0}.
\end{align*}

We call the subanalytic space $\widetilde X_D$ 
the \emph{real blow-up} of $X$ along $D$, and we denote by
\[
\varpi\colon\widetilde X_D \to X
\]
the map induced by $\varpi_\tot$.
Note that $\varpi$ induces an isomorphism
\[
\varpi\colon\widetilde X_D^{> 0} \isoto X\setminus D,
\]
and one has
\[
\widetilde X_D^{0} = \opb\varpi(D) = S_D X,
\]
where $S_D X = (T_D X\setminus D)/\R_{>0}$ 
denotes the normal sphere bundle to $D$ in $X$.

\medskip
Let now $D\subset X$ be a normal crossing divisor, and write (locally)
\begin{equation}
\label{eq:DD1Dr}
D = D_1 \union \cdots \union D_r,
\end{equation}
where $D_k\subset X$ are smooth hypersurfaces of $X$.
The total real blow-up 
\[
\varpi_\tot\colon \widetilde X_D^\tot \to X
\]
of $X$ along $D$ is defined by
\[
\widetilde X_D^\tot = \widetilde X_{D_1}^\tot \times_X \cdots \times_X \widetilde X_{D_r}^\tot.
\]
Note that $\varpi_\tot$ is an unramified $2^r$-sheeted covering over $X\setminus D$, so that we may write
\[
\opb{\varpi_\tot}(X\setminus D) = (X\setminus D) \times\{+,-\}^r.
\] 
Set
\begin{align*}
\widetilde X_D^{> 0} &= \widetilde X_{D_1}^{> 0} \times_X \cdots \times_X \widetilde X_{D_r}^{> 0} = (X\setminus D) \times\{(+,\dots,+)\}, \\
\widetilde X_D &= \overline{\widetilde X_D^{> 0}}, \\
\widetilde X_D^{0} &= \widetilde X_D \setminus \widetilde X_D^{> 0}.
\end{align*}

We call the subanalytic space $\widetilde X_D$ 
the \emph{real blow-up} of $X$ along $D$, and we denote by
\[
\varpi\colon\widetilde X_D \to X
\]
the proper map induced by $\varpi_\tot$.
Note that $\varpi$ induces an isomorphism
\[
\varpi\colon\widetilde X_D^{> 0} \isoto X\setminus D.
\]

\begin{remark}
The spaces $\widetilde X_D$, $\widetilde X_D^{> 0}$ and $\widetilde X_D^{0}$ are determined canonically.
On the contrary, the space $\widetilde X_D^\tot$ is not canonical.
For example, 
writing $D = D_2 \union \cdots \union D_r$ near a point $x\in D\setminus D_1$,
$\varpi_\tot$ becomes a $2^{r-1}$-sheeted covering over $X\setminus D$.
\end{remark}

\subsection{Sheaves of functions on the real blow-up}\label{sse:blowfun}

Let $X$ be a complex manifold and $D\subset X$ 
a normal crossing divisor. 
Set for short $\widetilde X = \widetilde X_D$.

\begin{notation}
\bnum
\item
Set $\shc_{\widetilde X}^{\infty,\textrm{temp}} = \opb i\shc_{\widetilde X^{> 0}|\widetilde X^{\tot}}^{\infty,\textrm{temp}}$, where $i\colon \widetilde X \to \widetilde X^\tot$ is the closed embedding. In other words, $\shc_{\widetilde X}^{\infty,\textrm{temp}}$ is the sheaf of $\C$-algebras on $\widetilde X$ defined by
\[
\widetilde X \underset{\text{open}}\supset V \mapsto \{u\in\shc_{\widetilde X^\tot}^\infty(V\cap\widetilde X^{> 0}) \semicolon u \text{ is tempered at any point of }V\cap \widetilde X^0\}.
\]
\item
Let $\sha_{\widetilde X}$ be the sheaf of rings on $\widetilde X$ defined by
\[
\widetilde X \underset{\text{open}}\supset V \mapsto \{u\in\shc_{\widetilde X}^{\infty,\textrm{temp}}(V) \semicolon u \text{ is holomorphic on }V\cap \widetilde X^{>0}\}.
\]
\item
Set $\D_{\widetilde X}^\sha = \sha_{\widetilde X} \tens[\opb\varpi\O_X] \opb\varpi\D_X$.
\item
Denote by $\D_{\widetilde X}^{\shc^{\infty,\textrm{temp}}}$ the ring of differential operators with $C^{\infty,\textrm{temp}}_{\widetilde X}$ coefficients.
\ee
\end{notation}

\begin{lemma}
\label{lem:A*D}
One has
\[
\sha_{\widetilde X} \simeq \opb\varpi\O_X(*D) \tens[\opb\varpi\O_X] \sha_{\widetilde X}\,.
\]
\end{lemma}

\begin{remark}
By Lemma~\ref{lem:A*D}, there is an action of $\opb\varpi\D_X$ on $\sha_{\widetilde X}$. Hence $\D_{\widetilde X}^\sha$ has a natural algebra structure. Note also that there are natural $\Cfield$-algebra morphisms
\begin{align*}
& \opb\varpi\D_X \to \D_{\widetilde X}^\sha, \\
& \D_{\widetilde X}^\sha \tens[\C] \opb\varpi\D_{\overline X} \to
\D_{\widetilde X}^{\shc^{\infty,\textrm{temp}}}.
\end{align*}
\end{remark}

\begin{notation}
Consider the ind-sheaf on $\widetilde X$
\[
\Dbt_{\widetilde X} \defeq \opb i\ihom(\Cfield_{\widetilde X^{> 0}}, \Dbt_{\widetilde X^\tot}),
\]
where $i\colon \widetilde X \to \widetilde X^\tot$ is the closed embedding.
\end{notation}

Note that one has
$$\opb i\ihom(\Cfield_{\widetilde X^{> 0}}, \Dbt_{\widetilde X^\tot})
\simeq\rihom(\Cfield_{\widetilde X^{> 0}}, \epb i\Dbt_{\widetilde X^\tot}),$$
where $\Cfield_{\widetilde X^{> 0}}$ on the left hand side denotes
a sheaf on $\widetilde X^\tot$ and  on the right  hand side 
a sheaf on $\widetilde X$.

\begin{lemma}
The ind-sheaf $\Dbt_{\widetilde X}$ has a structure of $\D_{\widetilde X}^{\shc^{\infty,\textrm{temp}}}$-module. In particular, it has a structure of $(\D_{\widetilde X}^\sha \tens[\Cfield] \opb\varpi\D_{\overline X})$-module.
\end{lemma}

This immediately follows from Proposition~\ref{pro:CDtoD}.

\begin{notation}
We set
\[
\Ot_{\widetilde X} = 
\rhom[\opb\varpi\D_{\overline X}](\opb\varpi\O_{\overline X}, \Dbt_{\widetilde X})
\in\BDC(\ind\D_{\widetilde X}^\sha),
\]
the Dolbeault complex with coefficients in $\Dbt_{\widetilde X}$.
\end{notation}

\begin{theorem}
\label{thm:forOt}
There is an isomorphism in $\BDC(\ind\opb\varpi\D_X)$
\[
\operatorname{for}(\Ot_{\widetilde X}) \simeq
\epb\varpi\rihom(\Cfield_{X\setminus D},\Ot_X),
\]
where $\operatorname{for}\colon \BDC(\ind\D_{\widetilde X}^\sha) \to \BDC(\ind\opb\varpi\D_X)$ is the forgetful functor.
\end{theorem}

\begin{proof}
It is enough to prove the isomorphism
\[
\Dbt_{\widetilde X} \simeq \epb\varpi\rihom(\Cfield_{X\setminus D},\Dbt_X).
\]
Consider a complexification of morphisms of real analytic manifolds
\[
\vcenter{\vbox{
\xymatrix{
\widetilde X^\tot \ar[r] & X_\R \\
\widetilde X^0 \ar[u] \ar[ur],
}}}
\mbox{\large$\hookrightarrow$}
\vcenter{\vbox{
\xymatrix{
\widetilde X^\tot_\C \ar[r] & X_\C \\
\widetilde X^0_\C\,. \ar[u] \ar[ur]
}}}
\]
Then $\Dbt_{\widetilde X}$ is a module over
\[
\D_{\widetilde X^\tot_\C}(*\widetilde X^0_\C) \defeq
\D_{\widetilde X^\tot_\C} \tens[\O_{\widetilde X^\tot_\C}] \O_{\widetilde X^\tot_\C}(*\widetilde X^0_\C).
\]
Hence
\begin{align*}
\epb\varpi\rihom(\Cfield_{X\setminus D},\Dbt_X)
&\simeq \rihom(\Cfield_{\widetilde X^{>0}}, \epb\varpi\Dbt_X) \\
&\simeq \rihom(\Cfield_{\widetilde X^{>0}}, \D_{X_\C \leftarrow \widetilde X^\tot_\C} \ltens[\D_{\widetilde X^\tot_\C}] \Dbt_{\widetilde X}) \\
&\simeq \D_{X_\C \leftarrow \widetilde X^\tot_\C} \ltens[\D_{\widetilde X^\tot_\C}] \rihom(\Cfield_{\widetilde X^{>0}}, \Dbt_{\widetilde X}) \\
&\simeq \D_{X_\C \leftarrow \widetilde X^\tot_\C} \ltens[\D_{\widetilde X^\tot_\C}] \Dbt_{\widetilde X} \\
&\simeq \D_{X_\C \leftarrow \widetilde X^\tot_\C} \ltens[\D_{\widetilde X^\tot_\C}] \D_{\widetilde X^\tot_\C}(*\widetilde X^0_\C) \ltens[\D_{\widetilde X^\tot_\C}(*\widetilde X^0_\C)]  
\Dbt_{\widetilde X},
\end{align*}
where the second isomorphism follows from Lemma~\ref{lem:fDbt}.
To conclude, note that
\[
\D_{X_\C \leftarrow \widetilde X^\tot_\C} \ltens[\D_{\widetilde X^\tot_\C}] \D_{\widetilde X^\tot_\C}(*\widetilde X^0_\C) \simeq 
\D_{\widetilde X^\tot_\C}(*\widetilde X^0_\C).
\]
\end{proof}

\begin{remark}
The importance of Theorem~\ref{thm:forOt} is in showing that $\epb\varpi\rihom(\Cfield_{X\setminus D},\Ot_X)$ has a structure of $\D_{\widetilde X}^\sha\,$-module.
\end{remark}

\begin{corollary}
\label{cor:varpiOt}
There is an isomorphism in $\BDC(\ind\D_X)$
\[
\roim\varpi\Ot_{\widetilde X} \simeq
\rihom(\Cfield_{X\setminus D}, \Ot_X).
\]
\end{corollary}

\begin{proof}
By the above theorem, we have
\begin{align*}
\Ot_{\widetilde X} 
&\simeq \epb\varpi\rihom(\Cfield_{X\setminus D},\Ot_X) \\
&\simeq \rihom(\opb\varpi\Cfield_{X\setminus D},\epb\varpi\Ot_X).
\end{align*}
Hence
\begin{align*}
\roim\varpi\Ot_{\widetilde X} 
&\simeq \roim\varpi\rihom(\opb\varpi\Cfield_{X\setminus D},\epb\varpi\Ot_X) \\
&\simeq \rihom(\reeim\varpi\opb\varpi\Cfield_{X\setminus D},\Ot_X) \\
&\simeq \rihom(\Cfield_{X\setminus D},\Ot_X).
\end{align*}
\end{proof}

\begin{proposition}\label{pro:AOt}
One has
\[
\sha_{\widetilde X} \simeq \alpha_{\widetilde X}\Ot_{\widetilde X}.
\]
\end{proposition}

\begin{proof}
By the definition of $\sha_{\widetilde X}$, using \cite[Theorem 10.5]{KS96} one has
\[
\sha_{\widetilde X} \simeq H^0 \alpha_{\widetilde X}\Ot_{\widetilde X}.
\]
Let $U$ be a relatively compact subanalytic open subset of $\widetilde X^\tot$ and set $V=\varpi(U\cap \widetilde X^{>0})$. Then we have
\begin{align*}
\rsect(U;\alpha_{\widetilde X}\Ot_{\widetilde X})
&\simeq \RHom(\Cfield_{U\cap \widetilde X^{>0}}, \Ot_{\widetilde X}) \\
&\simeq \RHom(\Cfield_V, \Ot_X),
\end{align*}
where the last isomorphism follows from Corollary~\ref{cor:varpiOt}.
Hence the vanishing of the higher cohomology groups of the complex $\alpha_{\widetilde X}\Ot_{\widetilde X}$ follows from the fact that, if $V$ is a relatively compact subanalytic convex open subset of $\C^n$, then,
\[
H^k\RHom(\Cfield_V,\Ot_{\C^n}) = 0 \quad \text{for $k\neq 0$}.
\]
This last fact follows e.g.\ from \cite[Theorem 5.10]{DS96}.
\end{proof}

\subsection{Normal forms}\label{sse:normal}

Let $X$ be a complex manifold and $D\subset X$ a normal crossing divisor. 
Let $(z_1,\dots,z_n)$ be a system of local coordinates
of $X$  such that $D=\{z_1\cdots z_r=0\}$.

\begin{notation}
For $\shm\in\BDC(\D_X)$, set
\[
\shm^\sha = \D_{\widetilde X}^\sha \ltens[\opb\varpi\D_X] \opb\varpi\shm.
\]
\end{notation}

\begin{lemma}
If $\shm$ is a holonomic $\D_X$-module such that $\ss(\shm)\subset D$ and $\shm\isoto\shm(*D)$,
then one has
\begin{equation}
\shm^\sha \simeq \D_{\widetilde X}^\sha \tens[\opb\varpi\D_X] \opb\varpi\shm.
\end{equation}
\end{lemma}

\begin{proof}
This follows from
\begin{align*}
\D_{\widetilde X}^\sha \ltens[\opb\varpi\D_X] \opb\varpi\shm 
&\simeq
(\sha_{\widetilde X} \ltens[\opb\varpi\O_X] \opb\varpi\D_X) \ltens[\opb\varpi\D_X] \opb\varpi\shm \\
&\simeq \sha_{\widetilde X} \ltens[\opb\varpi\O_X] \opb\varpi\shm
\end{align*}
by noticing that $\shm$ is flat over $\O_X$.
\end{proof} 

It is well known that if $\shm$ is a regular holonomic $\D_X$-module such that $\shm \simeq \shm(*D)$ and $\ss(\shm) \subset D$, then $\shm^\sha$ is isomorphic to a finite direct sum of copies of $(\O_X)^\sha$, locally on $\widetilde X^0$.
(Note that $z_k^\lambda(\log z_k)^m$ is a section of $\sha_{\widetilde X}$, locally on $\widetilde X^0$, for $\lambda\in\C$, $k=1,\dots,r$ and $m\in\Z_{\geq 0}$.)

\begin{definition}
\label{def:normal}
We say that a holonomic $\D_X$-module $\shm$ has \emph{a normal form} along $D$ if
\bnum
\item
$\shm \simeq \shm(*D)$,
\item
$\ss(\shm) \subset D$,
\item
for any $x\in \widetilde X^0$, there exist an open neighborhood $U\subset X$ of $\varpi(x)$ and finitely many $\varphi_i\in\sect(U;\O_X(*D))$
such that
\[
(\shm^\sha)|_V \simeq
\left.\left( \soplus\nolimits_i (\she_{U\setminus D|U}^{\varphi_i})^\sha \right)\right|_V
\]
for some neighborhood $V\subset\opb\varpi(U)$ of $x$.
\ee
\end{definition}

A ramification of $X$ along $D$ on a neighborhood $U$ of $x\in D$ is a finite map
\[
p \colon X' \to U
\]
of the form $p(z) = (z_1^{m_1},\dots,z_r^{m_r},z_{r+1},\dots,z_n)$ for some $(m_1,\dots,m_r)\in(\Z_{>0})^r$.
Here $(z_1,\dots,z_n)$ is a 
local coordinate system such that $D=\{z_1\cdots z_r=0\}$.

\begin{definition}
\label{def:quasi-normal}
We say that a holonomic $\D_X$-module $\shm$ has \emph{a quasi-normal form} along $D$ if it satisfies (i) and (ii) in Definition~\ref{def:normal}, and if for any $x\in D$ there exists a ramification $p\colon X'\to U$ on a neighborhood $U$ of $x$ such that $\dopb p (\shm|_U)$ has a normal form along $\opb p (D\cap U)$.
\end{definition}

\begin{remark}
With the above notations, $\dopb p(\shm|_U)$ and $\doim p\dopb p(\shm|_U)$ are concentrated in degree zero, and
$\shm|_U$ is a direct summand of $\doim p\dopb p (\shm|_U)$.
\end{remark}

\begin{theorem}[see \cite{Maj84,Sab00,Moc09,Moc11,Ked10,Ked11}]
\label{thm:normal}
Let $X$ be a complex manifold, $\shm$ a holonomic $\D_X$-module and $x\in X$.
Then there exist an open neighborhood $U$ of $x$, a closed analytic hypersurface $Y\subset U$, a complex manifold $X'$ and a projective morphism $f\colon X'\to U$ such that
\bnum
\item $\ss(\shm)\cap U\subset Y$,
\item $D\defeq\opb f(Y)$ is a 
normal crossing divisor  of $X'$,
\item $f$ induces an isomorphism $X'\setminus D \to U \setminus Y$,
\item $(\dopb f \shm)(*D)$ has a quasi-normal form along $D$.
\ee
\end{theorem}

Remark that, under assumption (iii), $(\dopb f \shm)(*D)$ is concentrated in degree zero.

\medskip

The above fundamental result provides the following tool to prove statements concerning holonomic objects.

\begin{lemma}\label{lem:redux}
Let $P_X(\shm)$ be a statement concerning a complex manifold $X$ and a holonomic object $\shm\in\BDC_\hol(\D_X)$. Consider the following conditions.
\bna
\item
Let $X=\Union\nolimits_{i\in I}U_i$ be an open covering. Then $P_X(\shm)$ is true if and only if $P_{U_i}(\shm|_{U_i})$ is true for any $i\in I$.
\item
If $P_X(\shm)$ is true, then $P_X(\shm[n])$ is true for any $n\in\Z$.
\item
Let $\shm'\to\shm\to\shm''\to[+1]$ be a distinguished triangle in $\BDC_\hol(\D_X)$. If $P_X(\shm')$ and $P_X(\shm'')$ are true, then $P_X(\shm)$ is true.
\item
Let $\shm$ and $\shm'$ be holonomic $\D_X$-modules. 
If $P_X(\shm\dsum\shm')$ is true, then $P_X(\shm)$ is true.
\item Let $f\colon X\to Y$ be a projective morphism and $\shm$ a good holonomic $\D_X$-module. If $P_X(\shm)$ is true, then $P_Y(\doim f\shm)$ is true.
\item If $\shm$ is a holonomic $\D_X$-module with a normal form along a normal crossing divisor of $X$, then $P_X(\shm)$ is true.
\ee
If conditions {\rm (a)--(f)} are satisfied, 
then $P_X(\shm)$ is true for any complex manifold $X$ and any $\shm\in\BDC_\hol(\D_X)$.
\end{lemma}

\begin{proof}
Let $X$ be a complex manifold and $\shm\in\BDC_\hol(\D_X)$. Let us show that $P_X(\shm)$ is true.

\smallskip\noindent (i)
Let $a\leq b$ be integers such that $\shm\in\derd^{[a,b]}_\hol(\D_X)$.
Then one says that  $\shm$ has amplitude $\leq b-a$. 
By applying (b) and (c) to the distinguished triangle
\[
\tau^{\leq a}\shm \To \shm \To \tau^{>a}\shm \To[+1]
\]
and arguing by induction on the amplitude of $\shm$, we may assume that $\shm$ is concentrated in degree zero. In other words, we may assume that $\shm$ is a holonomic $\D_X$-module.
Since the question is local on $X$ by (a), we may further assume that $\shm$ is good.

\smallskip\noindent (ii)
Assume that $\shm$ is a good holonomic $\D_X$-module with a quasi-normal form along a normal crossing divisor $D\subset X$.

Locally, there exists a ramification $p\colon X'\to X$ as in Definition~\ref{def:quasi-normal}, such that $\dopb p\shm$ has a normal form. Then, $P_{X'}(\dopb p\shm)$ is true by (f). Hence
$P_X(\doim p\dopb p \shm)$ is true by (e). Since $\shm$ is a direct summand of $\doim p\dopb p \shm$, it follows from (d) that $P_X(\shm)$ is true.

\smallskip\noindent (iii)
Let $\shm$ be a good holonomic $\D_X$-module. We will argue by induction on $\dim X$ and by induction on the dimension of $Y \defeq \supp\shm$. 

\smallskip\noindent (iii-1)
Assume first $Y=X$. Then, locally on $X$, there exist a closed hypersurface $Z\subset X$ and a projective morphism $f\colon X'\to X$ such that $D \defeq \opb f Z$ is a normal crossing divisor of  $X'$, $f$ induces an isomorphism $X'\setminus D\isoto X\setminus Z$, and $(\dopb f \shm)(*D)$ has a quasi-normal form. Hence $P_{X'}(\dopb f \shm(*D))$ is true by (ii).
Since $\dopb f \shm(*D)$ is good and $\shm(*Z) \simeq \doim f \dopb f\shm(*D)$, $P_X(\shm(*Z))$ is true by (e).
Let us consider a distinguished triangle
\[
\shm\To\shm(*Z)\To\shn\To[+1].
\]
Then $\dim\supp\shn < \dim Y$, and hence $P_X(\shn)$ is true by the induction hypothesis. Therefore $P_X(\shm)$ is true by (b) and (c).

\smallskip\noindent (iii-2)
Assume now that $Y\neq X$. Let $Y_\sing$ be its singular locus, and let $f\colon Y'\to X$ be a projective morphism such that $Y'$ is a complex manifold, $f(Y')= Y$, $Z'\defeq\opb f Y_\sing$ is a closed hypersurface of $Y'$, and $f$ induces an isomorphism $Y'\setminus Z'\isoto Y\setminus Y_\sing$.
Then $\shn \defeq \dopb f\shm(*Z')[d_{Y'}-d_X]$ is a good holonomic $\D_{Y'}$-module.
Since $\dim Y'<\dim X$, $P_{Y'}(\shn)$ is true by the induction hypothesis on $\dim X$. 
Hence $P_X(\doim f\shn)$ is also true by (e).
Consider a distinguished triangle
\[
\shm \To \doim f\shn \To \shl \To[+1].
\]
Since $\supp\shl \subset Y_\sing$, the induction hypothesis on $\dim Y$ implies that $P_X(\shl)$ is true. Hence $P_X(\shm)$ is also true by (b) and (c).
\end{proof}

\section{Enhanced tempered functions}\label{se:enhanced}

We define in this section the enhanced ind-sheaves of tempered distributions and of tempered holomorphic functions.

\subsection{Enhanced tempered distributions}

Denote by $\PR$ and $\PP$ the real and complex projective line, respectively.
Let $t\in\R\subset\PR$ and  $\tau\in\C\subset\PP$ be the affine coordinates, with $t=\tau|_\R$. 
Let $M$ be a real analytic manifold, and consider the natural morphism of bordered spaces
\[
j\colon M\times\R_\infty \to M\times\PR .
\]

\begin{definition}
\label{def:DbT}
Set
\[
\DbT_M = 
\epb j \rhom[\D_\PP](\she_{\C|\PP}^\tau,\Dbt_{M\times\PR})[1] \in \BDC(\iCfield_{M\times\R_\infty}),
\]
and denote by $\DbE_M$ the associated object of $\BEC[\iCfield]X$. 
\end{definition}

Here the shift has been chosen 
so that Propositions \ref{pro:OTetens} and \ref{prop:regireg} below 
hold.
Note that, by an argument similar to that in the proof of Lemma~\ref{lem:Rt}, one has
\[
H^k(\DbT_M) = 0 \quad\text{for }k\neq-1.
\]

\begin{remark}
There are monomorphisms
$$\Cfield_{\{t<\ast\}} \tens \opb\pi\Dbt_M
 \monoto H^{-1}(\DbT_M)\monoto \opb\pi\Db_M.$$
The first one is induced by $v(x)\mapsto e^tv(x)$,
and the second is induced by $u(x,t)\mapsto e^{-t}u(x,t)$.
They are not isomorphisms (if $\dim M\ge 1$).
In fact, for $M=\R$ and $U=\{(x,t)\in M\times\R\semicolon x>0,\ t<-1/x\}$,
one has
$e^te^{1/x} \in \Hom(\Cfield_U, H^{-1}(\DbT_M))$ but
$e^{1/x}\in  \Hom(\Cfield_U, \opb\pi\Db_M)\simeq\Hom(\C_{\{x>0\}},\Db_M)$
does not belong to
$$\Hom(\Cfield_U, \Cfield_{\{t<\ast\}} \tens \opb\pi\Dbt_M)
\simeq \Hom(\Cfield_U, \opb\pi\Dbt_M)\simeq\Hom(\C_{\{x>0\}},\Dbt_M).$$
\end{remark}

\begin{proposition}
\label{pro:DbTgeq0}
There are isomorphisms in $\BDC(\iCfield_{M\times\R_\infty})$
\begin{align*}
\DbT_M 
&\isoto \cihom(\Cfield_{\{t\geq0\}}, \DbT_M) \\
&\isofrom \cihom(\Cfield_{\{t\geq a\}}, \DbT_M) \quad \text{for any $a\geq 0$}.
\end{align*}
\end{proposition}

\begin{proof}
(i)
Let us prove the isomorphism
\[
\DbT_M \isoto \cihom(\Cfield_{\{t\geq0\}}, \DbT_M).
\]
Denote by $p\colon M\times\PR\to M$ the projection.
Let $U\subset M\times\PR$ be an open subanalytic subset such that $U\cap(M\times\R)\cap \opb p(x)$ is connected for all $x\in M$. 
Note that $\roim j\DbT_M$ belongs to $\BDC_\suban(\iCfield_{M\times\PR})$.
By Lemma~\ref{lem:vanrelsuban}, it is then enough to show
\[
\RHom(\Cfield_U,\cihom(\Cfield_{\{t>0\}}, \DbT_M))
\simeq 0.
\]
One has
\[
\RHom(\Cfield_U,\cihom(\Cfield_{\{t>0\}}, \DbT_M))
\simeq \RHom(\Cfield_U \ctens \Cfield_{\{t>0\}}, \DbT_M).
\]
Set $V= p(U)\subset M$ and $U\cap(M\times\R)=\{(x,t)\in V\times\R \semicolon \varphi(x)<t<\psi(x) \}$, where $\varphi,\psi\colon V \to \overline\R$ are subanalytic functions with $\varphi(x)<\psi(x)$ for all $x\in V$. Then
\[
\Cfield_U \ctens \Cfield_{\{t>0\}} \simeq \Cfield_W[-1],
\]
where $W=\{(x,t)\in V\times\R\semicolon \varphi(x)<t \}$.
Note that $\varphi$ takes value in $\overline\R\setminus\{+\infty\}$.
Hence we have to prove that the bottom arrow in the commutative diagram below is an isomorphism.
\[
\xymatrix{
\Hom(\Cfield_{M\times\R}, \Dbt_{M\times\PR}) \ar@{->>}[r]^{\partial_t-1} \ar@{->>}[d] &
\Hom(\Cfield_{M\times\R}, \Dbt_{M\times\PR}) \ar@{->>}[d] \\
\Hom(\Cfield_W, \Dbt_{M\times\PR}) \ar[r]^{\partial_t-1} & 
\Hom(\Cfield_W, \Dbt_{M\times\PR}) .
}
\]
Since the top arrow is surjective and the vertical arrows are surjective, also the bottom arrow is surjective.
By Lemma~\ref{lem:uet} below, the bottom arrow is injective.

\smallskip\noindent(ii)
In order to prove the isomorphism
\[
\cihom(\Cfield_{\{t\geq a\}}, \DbT_M) \isoto \cihom(\Cfield_{\{t\geq0\}}, \DbT_M),
\]
it is enough to show that
\[
\cihom(\Cfield_{\{t< a\}}, \DbT_M) \isoto \cihom(\Cfield_{\{t< 0\}}, \DbT_M).
\]
Hence, as in (i), it is enough to show that
\[
\RHom(\Cfield_U,\cihom(\Cfield_{\{t< a\}}, \DbT_M)) \isoto \RHom(\Cfield_U,\cihom(\Cfield_{\{t< 0\}}, \DbT_M))
\]
for any subanalytic open subset $U\subset M\times\PR$ such that
\[
U\cap(M\times\R) = \{(x,t)\in V\times\R \semicolon \varphi(x)<t<\psi(x) \},
\]
where $V=p(U)$.
One has $\Cfield_U \ctens \Cfield_{\{t<a\}} \simeq \Cfield_{W_a}[-1]$,
where
\[
W_a = \{(x,t)\in V\times\R\semicolon t-a<\psi(x) \}.
\]
Hence we have to show that the following morphism is a quasi-isomorphism
\begin{multline*}
\left( \Hom(\Cfield_{W_a}, \Dbt_{M\times\PR}) \to[\partial_t-1] \Hom(\Cfield_{W_a}, \Dbt_{M\times\PR}) \right) \\
\To
\left( \Hom(\Cfield_{W_0}, \Dbt_{M\times\PR}) \to[\partial_t-1] \Hom(\Cfield_{W_0}, \Dbt_{M\times\PR}) \right).
\end{multline*}
Since the arrows $\partial_t-1$ are surjective, we have to show that the natural morphism
\begin{multline*}
\ker\left( \Hom(\Cfield_{W_a}, \Dbt_{M\times\PR}) \to[\partial_t-1] \Hom(\Cfield_{W_a}, \Dbt_{M\times\PR}) \right) \\
\to
\ker\left( \Hom(\Cfield_{W_0}, \Dbt_{M\times\PR}) \to[\partial_t-1] \Hom(\Cfield_{W_0}, \Dbt_{M\times\PR}) \right)
\end{multline*}
is an isomorphism.
Indeed, its inverse is given by $u(x,t) \mapsto e^a\,u(x,t-a)$.
\end{proof}

\begin{lemma}
\label{lem:uet}
Let $u\in\sect(M;\Db_M)$ and assume that $u(x)e^t \in \sect(M\times\{t>0\};\Db_{M\times\PR})$ is tempered at $t=\infty$. Then $u=0$.
\end{lemma}

\begin{proof}
For any $v\in\shc^\infty_c(M)$, set $c= \int v(x)u(x) dx$. Then the function $ce^t= \int v(x)u(x)e^t dx$ is tempered at $t=\infty$, and hence $c=0$.
\end{proof}

\subsection{Enhanced tempered holomorphic functions}

Let $X$ be a complex manifold. 
Consider the natural morphism of bordered spaces
\[
i\colon X\times\R_\infty \to X\times\PP .
\]
Let $\tau\in\C\subset\PP$ be the affine coordinate such that $\tau|_\R = t$, the affine coordinate of $\R$. 

\begin{definition}
Set
\begin{align*}
\OEn_X 
&= \epb i ((\she_{\C|\PP}^{-\tau})^\mop\ltens[\D_\PP]\Ot_{X\times\PP})[1] \\
&\simeq \epb i \rhom[\D_\PP](\she_{\C|\PP}^\tau,\Ot_{X\times\PP})[2] \in \BEC[\ind\D]X, \\
\OvE_X &= \Omega_X \ltens[\O_X] \OEn_X \\
&\simeq \epb i(\Ovt_{X\times\PP} \ltens[\D_\PP] \she_{\C|\PP}^{-\tau})[1]  \in \BEC[\ind\D^\op]X.
\end{align*}
 Recall that $\mop\colon\BDC(\D_\PP) \to \BDC(\D_\PP^\op)$ is the functor given by $\shm^\mop = \Omega_\PP \ltens[\O_\PP] \shm$.
\end{definition}

\begin{theorem}
\label{thm:OTgeq0}
There is an isomorphism in $\BDC(\iCfield_{X\times\R_\infty})$
\[
\RE\OEn_X \simeq \epb i ((\she_{\C|\PP}^{-\tau})^\mop\ltens[\D_\PP]\Ot_{X\times\PP})[1],
\]
and there are isomorphisms in $\BEC[\ind\D]X$
\begin{align*}
\OEn_X &\isoto \cihom(\Cfield_{\{t\geq 0\}}, \OEn_X) \\
&\isofrom \cihom(\Cfield_{\{t\geq a\}}, \OEn_X) \quad
\text{for any $a\geq 0$}.
\end{align*}
\end{theorem}

\begin{proof}
This follows from Proposition~\ref{pro:DbTgeq0}, noticing that
\[
\OEn_X \simeq \rhom[\opb\pi\D_{\overline X}](\opb\pi\O_{\overline X},\DbE_{X_\R}),
\]
where $X_\R$ denotes the real analytic manifold underlying $X$.
\end{proof}

As a consequence of Theorem~\ref{thm:OTgeq0} and Proposition~\ref{pro:equivTam}, we get the following result.

\begin{corollary}
\label{cor:CTamOT}
There are isomorphisms in $\BEC[\ind\D]X$
\begin{align*}
\OEn_X 
&\simeq \cihom(\Cfield_X^\enh, \OEn_X) \\
&\simeq \Cfield_X^\enh \ctens \OEn_X.
\end{align*}
\end{corollary}

\begin{proposition}
\label{pro:OTetens}
There is a canonical morphism
\[
\OEn_X \cetens \OEn_Y \to \OEn_{X\times Y}.
\]
\end{proposition}

In order to prove this proposition, we need a complex analytic analogue of the construction in Notation~\ref{not:sfS}.

\begin{notation}
Denote by $\mathbb S'$ the closure of $\{(x_1,x_2,x_3)\in\C^3\semicolon x_1+x_2+x_3=0 \}$ in $\PP\times\PP\times\PP$. 
Then $\mathbb S'$ has a quadratic singularity at $(\infty,\infty,\infty)$.
Denote by $\mathbb S$ the blow-up of $\mathbb S'$ with center $(\infty,\infty,\infty)$.
Then $\mathbb S$ is a smooth projective surface.
Consider the maps
\[
\PP \from[{\ \tilde\mu\ }] \mathbb S\To[\tilde p] \PP\times\PP
\]
induced by $(x_1,x_2,x_3) \mapsto -x_3$, and $(x_1,x_2,x_3) \mapsto (x_1,x_2)$, respectively. 
We denote by the same letters the induced maps
\[
X\times Y\times\PP \from[\tilde\mu] X\times Y\times\mathbb S \to[\tilde p] X\times Y\times\PP\times\PP.
\]
\end{notation}

\begin{remark}
The algebraic surface $\mathbb S$ is also obtained as the blow-up of the complex projective plane $\PP^2(\C)$ with center at three points on a line.
\end{remark}

\begin{proof}[Proof of Proposition~\ref{pro:OTetens}]
Consider the diagrams of bordered spaces
\[
\xymatrix{
X\times\R_\infty \ar[d]^{i_X} & X\times Y\times\R_\infty\times\R_\infty \ar[l]_-{p_1} \ar[r]^-{p_2} \ar[d]^j & Y\times\R_\infty \ar[d]^{i_Y} \\
X\times\PP & X\times Y\times\PP\times\PP \ar[l]_-{\overline p_1} \ar[r]^-{\overline p_2} & Y\times\PP
}
\]
and
\[
\xymatrix{
X\times Y\times\mathbb S \ar[r]^{\tilde\mu} \ar[d]^{\tilde p} & X\times Y\times\PP \\ 
X\times Y\times\PP\times\PP & X\times Y\times\R_\infty\times\R_\infty \ar[ul]_{j_{\mathbb S}} \ar[r]^\mu \ar[l]^j & X\times Y\times\R_\infty \ar[ul]_{i_{X\times Y}}.
}
\]
Recall that
\begin{align*}
\OEn_X 
&= \epb{i_X} ((\she_{\C|\PP}^{-\tau_1})^\mop\ltens[\D_\PP]\Ot_{X\times\PP})[1], \\
\OEn_Y 
&= \epb{i_Y} ((\she_{\C|\PP}^{-\tau_2})^\mop\ltens[\D_\PP]\Ot_{Y\times\PP})[1],
\end{align*}
where $\tau_1$ and $\tau_2$ are coordinates on $\C\subset\PP$.

There are morphisms
\begin{align*}
\opb{p_1}\OEn_X \tens \opb{p_2}\OEn_Y 
&\to \epb j(\opb{\overline p_2}((\she_{\C|\PP}^{-\tau_1})^\mop\ltens[\D_\PP]\Ot_{X\times\PP}) \tens \opb{\overline p_2}((\she_{\C|\PP}^{-\tau_2})^\mop\ltens[\D_\PP]\Ot_{Y\times\PP}))[2] \\
&\to \epb j((\she_{\C|\PP}^{-\tau_1}\etens \she_{\C|\PP}^{-\tau_2})^\mop \ltens[\D_\PP\etens\D_\PP] (\opb{\overline p_1}\Ot_{X\times\PP} \tens \opb{\overline p_2}\Ot_{Y\times\PP}))[2] \\
&\to \epb j((\she_{\C^2|\PP\times\PP}^{-\tau_1-\tau_2})^\mop \ltens[\D_{\PP\times\PP}] \Ot_{X\times Y\times\PP\times\PP})[2],
\end{align*}
where the first morphism follows from Lemma~\ref{lem:epbetens}.

Since $j = \tilde p\circ j_{\mathbb S}$, we have
\begin{align*}
\epb j(&(\she_{\C^2|\PP\times\PP}^{-\tau_1-\tau_2})^\mop \ltens[\D_{\PP\times\PP}] \Ot_{X\times Y\times\PP\times\PP})[2] \\
&\simeq \epb {j_{\mathbb S}}\epb{\tilde p}((\she_{\C^2|\PP\times\PP}^{-\tau_1-\tau_2})^\mop \ltens[\D_{\PP\times\PP}] \Ot_{X\times Y\times\PP\times\PP})[2] \\
&\simeq \epb{j_{\mathbb S}}(\opb{\tilde p}(\she_{\C^2|\PP\times\PP}^{-\tau_1-\tau_2})^\mop \ltens[\opb{\tilde p}\D_{\PP\times\PP}] \epb{\tilde p}\Ot_{X\times Y\times\PP\times\PP})[2] \\
&\underset{(*)}\simeq \epb {j_{\mathbb S}}(\opb{\tilde p}(\she_{\C^2|\PP\times\PP}^{-\tau_1-\tau_2})^\mop \ltens[\opb{\tilde p}\D_{\PP\times\PP}] \D_{\PP\times\PP\from[{\;\tilde p\;}]\mathbb S} \ltens[\D_{\mathbb S}] \Ot_{X\times Y\times\mathbb S})[2] \\
&\simeq \epb{j_{\mathbb S}}(\dopb{\tilde p}(\she_{\C^2|\PP\times\PP}^{-\tau_1-\tau_2})^\mop \ltens[\D_{\mathbb S}] \Ot_{X\times Y\times\mathbb S})[2] ,
\end{align*}
where $(*)$ follows from Theorem~\ref{thm:ifunct} (i).
We have a morphism 
\begin{align*}
\dopb{\tilde p}(\she_{\C^2|\PP\times\PP}^{-\tau_1-\tau_2})
&\to \dopb{\tilde p} (\she_{\C^2|\PP\times\PP}^{-\tau_1-\tau_2})  
(*\opb{\tilde p}(\PP\times\PP\setminus\C^2)) \\
&\simeq \dopb{\tilde \mu}(\she_{\C|\PP}^{-\tau})(*\opb{\tilde  
p}(\PP\times\PP\setminus\C^2)).
\end{align*}
Hence we obtain
\begin{align*}
\opb{p_1}\OEn_X \tens \opb{p_2}\OEn_Y 
&\to \epb{j_{\mathbb S}}(\dopb{\tilde \mu}(\she_{\C|\PP}^{-\tau})(*\opb{\tilde p}(\PP\times\PP\setminus\C^2))^\mop \ltens[\D_{\mathbb S}] \Ot_{X\times Y\times\mathbb S})[2] \\
&\simeq \epb{j_{\mathbb S}}\rihom(\Cfield_{X\times Y\times\C^2}, \dopb{\tilde \mu}(\she_{\C|\PP}^{-\tau})^\mop \ltens[\D_{\mathbb S}] \Ot_{X\times Y\times\mathbb S})[2].
\end{align*}
Since $\opb{j_{\mathbb S}}(\C_{X\times Y\times\C^2})\simeq \C_{ X\times Y\times \R^2}$, 
one has
\begin{align*}
\epb{j_{\mathbb S}}\rihom&(\Cfield_{X\times Y\times \C^2}, \dopb{\tilde \mu}(\she_{\C|\PP}^{-\tau})^\mop \ltens[\D_{\mathbb S}] \Ot_{X\times Y\times\mathbb S})[2] \\
&\simeq \epb{j_{\mathbb S}}(\dopb{\tilde \mu}(\she_{\C|\PP}^{-\tau})^\mop 
\ltens[\D_{\mathbb S}] \Ot_{X\times Y\times\mathbb S})[2] \\
&\simeq \epb{j_{\mathbb S}}(\opb{\tilde\mu}(\she_{\C|\PP}^{-\tau})^\mop \ltens[\opb{\tilde\mu}\D_\PP] \D_{\PP\from[{\;\tilde\mu\;}]\mathbb S} \ltens[\D_{\mathbb S}] \Ot_{X\times Y\times\mathbb S})[2] \\
&\underset{(*)}\simeq \epb{j_{\mathbb S}}(\opb{\tilde\mu}(\she_{\C|\PP}^{-\tau})^\mop \ltens[\opb{\tilde\mu}\D_\PP] \epb{\tilde\mu} \Ot_{X\times Y\times\PP})[1] \\
&\simeq \epb{j_{\mathbb S}}\epb{\tilde\mu}((\she_{\C|\PP}^{-\tau})^\mop \ltens[\D_\PP] \Ot_{X\times Y\times\PP})[1] \\
&\simeq \epb\mu\epb{i_{X\times Y}}((\she_{\C|\PP}^{-\tau})^\mop \ltens[\D_\PP] \Ot_{X\times Y\times\PP})[1],
\end{align*}
where $(*)$ follows from Theorem~\ref{thm:ifunct} (i).
We thus get a morphism
\[
\opb{p_1}\OEn_X \tens \opb{p_2}\OEn_Y \to \epb\mu\OEn_{X\times Y}.
\]
The desired morphism follows by adjunction.
\end{proof}

\section{Riemann-Hilbert correspondence}\label{se:RH}

We have now all the ingredients to state and prove a Riemann-Hilbert correspondence for holonomic $\D$-modules which are not necessarily regular. It is an analogue of the classical Riemann-Hilbert correspondence for regular holonomic $\D$-modules, in the framework of enhanced ind-sheaves.

\subsection{Enhanced de Rham and solution functors}

Recall that
\[
i\colon X\times\R_\infty \to X\times\PP
\]
is the natural morphism of bordered spaces, $\tau\in\C\subset\PP$ is the affine coordinate and $t=\tau|_\R$.

\begin{definition}
For $\shm\in\BDC(\D_X)$, set
\begin{align*}
\drE_X(\shm) &= \OvE_X \ltens[\D_X] \shm \\
&\simeq \epb i \drt_{X\times\PP}(\shm\detens\she_{\C|\PP}^{-\tau})[1], \\
\solE_X(\shm) &= \rhom[\D_X](\shm,\OEn_X)  \\
&\simeq \epb i \solt_{X\times\PP}(\shm\detens\she_{\C|\PP}^\tau)[2].
\end{align*}
They induce the functors
\begin{align*}
\drE_X &\colon \BDC(\D_X) \to \BEC[\iCfield]X, \\
\solE_X &\colon \BDC(\D_X)^\op \to \BEC[\iCfield]X.
\end{align*}
\end{definition}

Note that one has
\[
\solE_X(\shm) 
\simeq \drE_X(\ddual_X\shm)[-d_X].
\]

{}From Theorem~\ref{thm:ifunct}, one deduces

\begin{theorem}
\label{thm:Tfunct}
Let $f\colon X\to Y$ be a complex analytic map.
\bnum
\item
There is an isomorphism in $\BEC[\ind\opb f\D]Y$
\[
\Eepb f \OEn_Y[d_Y] \simeq \D_{Y\from X} \ltens[\D_X] \OEn_X [d_X].
\]
\item
For any $\shn\in\BDC(\D_Y)$ there is an isomorphism in $\BEC[\iCfield]X$
\[
\drE_X(\dopb f \shn)[d_X] \simeq \Eepb f \drE_Y(\shn)[d_Y].
\]
\item
Let $\shm\in\BDC_\good(\D_X)$, and assume that $\supp\shm$ is proper over $Y$. Then there is an isomorphism in $\BEC[\iCfield]Y$
\[
\drE_Y(\doim f\shm) \simeq  \Eeeim f\drE_X(\shm).
\]
\item
Let $\shl\in\BDC_\reghol(\D_X)$ and $\shm\in\BDC(\D_X)$. Then
\[
\drE_X(\shl\dtens\shm) \simeq \rihom(\opb\pi\sol_X(\shl), \drE_X(\shm)),
\]
where $\sol_X(\shl) = \rhom[\D_X](\shl,\O_X)$.
In particular, for a closed hypersurface $Y\subset X$, one has
\[
\drE_X\bl\shm(*Y)\br \simeq \rihom\bl\opb\pi\C_{X\setminus Y}, \drE_X(\shm)\br.
\]
\ee
\end{theorem}

\begin{proposition}\label{prop:regireg}
For $\shl\in\BDC_\reghol(\D_X)$ one has an isomorphism in $\BEC[\iCfield]M$
\[
\drE_X(\shl) \simeq  e(\dr_X(\shl)) \defeq \Cfield_X^\enh \tens \opb\pi\dr_X(\shl).
\]
In particular, one has
\[
\drE_X(\O_X) \simeq \Cfield_X^\enh[d_X].
\]
\end{proposition}

\begin{proof}
(i)\ When $X=\point$, by Lemma~\ref{lem:Rt} we have
\[
\drE_\point(\Cfield) \simeq \Cfield_{\{t<\ast\}}[1] \simeq \Cfield_\point^\enh.
\]
Hence, Theorem~\ref{thm:Tfunct}~(ii) and Proposition~\ref{pro:stableops}~(ii) imply
\[
\drE_X(\O_X) \simeq \Eepb a_X \Cfield_\point^\enh[-d_X] \simeq \Cfield_X^\enh[d_X],
\]
where $a_X\colon X\to\point$ is the canonical map. 

\smallskip\noindent(ii)\ 
By (i), setting $\shm=\O_X$ in Theorem~\ref{thm:Tfunct}~(iv) one has
\[
\drE_X(\shl) \simeq \rihom(\opb\pi\sol_X(\shl), \Cfield_X^\enh[d_X]).
\]
Moreover,
\begin{align*}
\rihom(\opb\pi\sol_X(\shl), \Cfield_X^\enh[d_X]) 
&\simeq \Cfield_X^\enh \ctens \rihom(\opb\pi\sol_X(\shl), \Cfield_{\{t=0\}}[d_X]) \\
&\simeq \Cfield_X^\enh \ctens (\opb\pi\dual_X(\sol_X(\shl)[d_X]) \tens \Cfield_{\{t=0\}}) \\
&\simeq \Cfield_X^\enh \tens \opb\pi\dual_X(\sol_X(\shl)[d_X]) \\
&\simeq \Cfield_X^\enh \tens \opb\pi\dr_X(\shl),
\end{align*}
where the first isomorphism follows from Corollary~\ref{cor:pifieldT}. 
\end{proof}

\subsection{Real blow-up}

Let $D\subset X$ be a normal crossing divisor, and denote by $\widetilde X$ the real blow-up of $X$ along $D$. Similarly, denote by $\widetilde{X\times\PP}$ the real blow-up of $X\times\PP$ along $D\times\PP$.
There is a natural identification $\widetilde{X\times\PP} = \widetilde X\times\PP$. Hence, following the notations in section~\ref{sse:blowfun},
we have the sheaves of rings on $\widetilde X\times\PP$ 
\begin{align*}
\sha_{\widetilde X\times\PP} &\subset \opb\varpi\rhom(\Cfield_{(X\setminus D)\times\PP},\O_{X\times\PP}), \\
\D_{\widetilde X\times\PP}^\sha &= \sha_{\widetilde X\times\PP} \tens[\opb\varpi\O_{X\times\PP}] \opb\varpi\D_{X\times\PP},
\end{align*}
and the complex
\[
\Ot_{\widetilde X\times\PP} \in \BDC(\ind\D_{\widetilde X\times\PP}^\sha).
\]

Consider the natural morphisms
\[
\widetilde X \from[{\ \pi_{\widetilde X}\ }] \widetilde X \times \R_\infty \To[\tilde\imath] \widetilde X \times \PP.
\]

\begin{definition}
Set
\begin{align*}
\OEn_{\widetilde X} 
&= \epb{\tilde\imath} ((\she_{\C|\PP}^{-\tau})^\mop \ltens[\D_\PP] \Ot_{\widetilde X\times\PP})[1] \\
&\simeq \epb{\tilde\imath} \rhom[\D_\PP](\she_{\C|\PP}^\tau,\Ot_{\widetilde X\times\PP})[2] \in \BEC[\ind\D^\sha]{\widetilde X}, \\
\OvE_{\widetilde X} &= \opb{\pi_{\widetilde X}}\opb\varpi\Omega_X \ltens[\opb{\pi_{\widetilde X}}\opb\varpi\O_X] \OEn_{\widetilde X} \\
&\simeq \epb{\tilde\imath} (\Ovt_{\widetilde X\times\PP}\ltens[\D_\PP]\she_{\C|\PP}^{-\tau}) [1] \in \BEC[\ind(\D_{\widetilde X}^{\sha})^\op]{}
\end{align*}
and
\[
\drE_{\widetilde X}(\shl) = \OvE_{\widetilde X} \ltens[\D_{\widetilde X}^\sha] \shl \in \BEC[\iCfield]{\widetilde X}\quad\text{for $\shl\in\BDC(\D_{\widetilde X}^\sha)$.}
\]
\end{definition}

Theorem~\ref{thm:forOt} and Corollary~\ref{cor:varpiOt} imply

\begin{theorem}
\label{thm:forOTtilde}
There are 
an isomorphism in $\BEC[\ind\D^\op]X$
\[
\Eoim\varpi\OvE_{\widetilde X} \simeq
\rihom(\opb\pi\Cfield_{X\setminus D}, \OvE_X),
\]
and an isomorphism in $\BEC[\ind\opb\varpi\D_X{}^\op]{}$ 
\[
\operatorname{for}(\OvE_{\widetilde X}) \simeq
\Eepb\varpi\rihom(\opb\pi\Cfield_{X\setminus D}, \OvE_X),
\]
where $\operatorname{for}\colon \BEC[\ind(\D^\sha_{\widetilde X})^\op]{}\to \BEC[\ind\opb\varpi\D_X{}^\op]{}$ is the forgetful functor.
\end{theorem}

\begin{corollary}\label{cor:drTXtilde}
For $\shm\in\BDC_\hol(\D_X)$ such that $\shm\isoto\shm(*D)$, we have
\begin{align*}
\drE_X(\shm) &\simeq \Eoim\varpi\drE_{\widetilde X}(\shm^\sha), \\
\drE_{\widetilde X}(\shm^\sha) &\simeq \Eepb\varpi\drE_X(\shm).
\end{align*}
\end{corollary}

\begin{proof}
By the first isomorphism in Theorem~\ref{thm:forOTtilde}, one has
\begin{align*}
\Eoim\varpi\drE_{\widetilde X}(\shm^\sha)
&= \Eoim\varpi(\OvE_{\widetilde X} \ltens[\D_{\widetilde X}^\sha]\shm^\sha) \\
&\simeq \Eoim\varpi(\OvE_{\widetilde X} \ltens[\opb\varpi\D_X]\opb\varpi\shm) \\
&\simeq (\Eoim\varpi\OvE_{\widetilde X}) \ltens[\D_X]\shm \\
&\simeq \rihom(\opb\pi\Cfield_{X\setminus D},\OvE_X) \ltens[\D_X]\shm \\
&\simeq \OvE_X \ltens[\D_X](\O_X(*D)\dtens\shm )\\
&\simeq \drE_X(\shm ).
\end{align*}
The proof of the second isomorphism in the statement is similar, using the second isomorphism in Theorem~\ref{thm:forOTtilde}.
\end{proof}

\subsection{Constructibility}

Let $Y\subset X$ be a complex analytic hypersurface and $\varphi\in\O_X(*Y)$. 
Set $U=X\setminus Y$.
Let $\tau\in\C\subset\PP$ be the affine coordinate such that $\tau|_\R = t$. 
We set
\[
\{t=\Re\varphi\} = \{(x,t)\in U\times\R\semicolon t=\Re\varphi(x)\}
\subset X\times\PR
\]
and define the object $E^\enh_{U|X}(\varphi)$ of  $\BEC[\iCfield]X$ by 
\[
E^\enh_{U|X}(\varphi) = \Cfield^\enh_X \ctens \rihom(\Cfield_{U\times\R}, \Cfield_{\{t=\Re\varphi\}}).
\]

Recall the notation $\she_{U|X}^\varphi$ from Definition~\ref{def:expY}.

\begin{lemma}\label{lem:drTEphi}
Let $Y\subset X$ be a closed hypersurface. 
Let $\varphi\in\O_X(*Y)$ be a meromorphic function with poles at $Y$.
Then we have an isomorphism in $\BEC[\iCfield]X$
\[
\drE_X(\she_{U|X}^\varphi) 
\simeq
E^\enh_{U|X}(\varphi)[d_X] .
\]

In particular, $\drE_X(\she_{U|X}^\varphi)$ is $\R$-constructible.
\end{lemma}

\begin{proof}
We have
\[
\she_{U|X}^\varphi \detens \she_{\C|\PP}^{-\tau} \simeq \she_{U\times\C|X\times\PP}^{\varphi-\tau}.
\]
By Proposition~\ref{pro:Solphi},
\[
\drt_{X\times\PP}(\she_{U\times\C|X\times\PP}^{\varphi-\tau}) = 
\rihom(\Cfield_{U\times\C}, \indlim[a\rightarrow+\infty] \Cfield_{\{\Re(\tau-\varphi) < a\}} )[d_X+1],
\]
and by the definition,
\[
\drE_X(\she_{U|X}^\varphi) \simeq \epb i \drt_{X\times\PP}(\she_{U\times\C|X\times\PP}^{\varphi-\tau}) [1].
\]
Hence
\begin{align*}
\drE_X(\she_{U|X}^\varphi) 
&\simeq \rihom(\Cfield_{U\times\R}, \epb i \indlim[a\rightarrow+\infty] \Cfield_{\{\Re(\tau-\varphi) < a\}} [d_X+2]) \\
&\simeq \rihom(\Cfield_{U\times\R}, \indlim[a\rightarrow+\infty] \Cfield_{\{t-\Re\varphi < a\}} [d_X+1]),
\end{align*}
where the last isomorphism follows from
\[
\Cfield_{U\times\R} \tens \epb i \Cfield_{\{\Re(\tau-\varphi) < a\}} \simeq \Cfield_{\{t-\Re\varphi < a\}}[-1].
\]
In $\BEC[\iCfield]X$ we have
\[
\indlim[a\rightarrow+\infty] \Cfield_{\{t-\Re\varphi <a\}}[1]
\simeq \indlim[a\rightarrow+\infty] \Cfield_{\{t-\Re\varphi \geq a\}}
\simeq \Cfield_X^\enh \ctens \Cfield_{\{t = \Re\varphi\}}.
\]
Thus we obtain
\begin{align*}
\drE_X(\she_{U|X}^\varphi) 
&\simeq \rihom(\Cfield_{U\times\R}, \Cfield_X^\enh \ctens \Cfield_{\{t = \Re\varphi\}}) [d_X] \\
&\simeq \Cfield_X^\enh\ctens\rihom(\Cfield_{U\times\R}, \Cfield_{\{t = \Re\varphi\}}) [d_X].
\end{align*}
Here, the last isomorphism follows from Corollary~\ref{cor:pifieldT}.
\end{proof}

\begin{theorem}
For $\shm\in\BDC_\hol(\D_X)$, the object $\drE_X(\shm)$ 
of  $\BEC[\iCfield]X$ is $\R$-constructible.
\end{theorem}

\begin{proof}
(i)
Assume first that $\shm$ is a holonomic $\D_X$-module with a normal form along a normal crossing divisor $D$. Then
\[
\shm^\sha \defeq \D_{\widetilde X}^\sha\ltens[\opb\varpi\D_X]\opb\varpi\shm
\] 
is locally a direct sum of $\D_{\widetilde X}^\sha$-modules of the form $(\she_{X\setminus D|X}^\varphi)^\sha$ for $\varphi\in\O_X(*D)$ as in Lemma~\ref{lem:drTEphi}. By Corollary~\ref{cor:drTXtilde}, one has
\[
\drE_{\widetilde X}((\she_{X\setminus D|X}^\varphi)^\sha)
\simeq \Eepb\varpi \drE_X(\she_{X\setminus D|X}^\varphi).
\]
Since $\drE_X(\she_{X\setminus D|X}^\varphi)$ is $\R$-constructible, Proposition~\ref{pro:RcTfunctorial} implies that $\drE_{\widetilde X}((\she_{X\setminus D|X}^\varphi)^\sha)$ is $\R$-constructible. Hence also $\drE_{\widetilde X}(\shm^\sha)$ is $\R$-constructible. 
By Corollary~\ref{cor:drTXtilde}, $\drE_X(\shm ) \simeq \Eoim\varpi\drE_{\widetilde X}(\shm^\sha)$ is $\R$-constructible.

\smallskip\noindent (ii)
We shall apply Lemma~\ref{lem:redux} to the statement
\[
P_X(\shm) = \text{``$\drE_X(\shm)$ is $\R$-constructible''}.
\]
Hypotheses (a) and (b) are obvious, (c) follows from Proposition~\ref{pro:Rcthick}, (d) from Proposition~\ref{pro:summand}, (e) from Theorem~\ref{thm:Tfunct}~(iii) and Proposition~\ref{pro:RcTfunctorial}, and (f) from step (i).
\end{proof}

\begin{theorem}
\label{thm:drTetens}
For $\shm\in\BDC_\hol(\D_X)$ and $\shn\in\BDC_\hol(\D_Y)$, there is a canonical isomorphism
\[
\drE_X(\shm) \cetens \drE_Y(\shn) \isoto \drE_{X\times Y}(\shm \detens \shn).
\]
\end{theorem}

\begin{proof}
The morphism is defined by using Proposition~\ref{pro:OTetens}.

By d\'evissage, using Lemma~\ref{lem:redux}, we may assume that both $\shm$ and $\shn$ are holonomic $\D$-modules having a normal form along a normal crossing divisor. Denote by $D_X\subset X$ and $D_Y\subset Y$ the normal crossing divisors of the singularities of $\shm$ and $\shn$, respectively.
Note that $\shm\detens\shn$ has singularities at $D_{X\times Y} \defeq (D_X\times Y) \cup (X\times D_Y)$.

Consider the real blow-ups $\varpi_X\colon\widetilde X\to X$ and $\varpi_Y\colon\widetilde Y\to Y$. Note that $\widetilde{X\times Y} \simeq \widetilde X \times \widetilde Y$.

There is a natural morphism
\[
\OvE_{\widetilde X} \cetens \OvE_{\widetilde Y} \to \OvE_{\widetilde {X\times Y}}.
\]
Hence there are morphisms
\begin{align*}
(\OvE_{\widetilde X}\ltens[\D_X]\shm) \cetens (\OvE_{\widetilde Y}\ltens[\D_Y]\shn) 
&\to \OvE_{\widetilde {X\times Y}}\ltens[\D_X\etens \D_Y](\shm\etens\shn) \\
&\simeq \OvE_{\widetilde {X\times Y}}\ltens[\D_{X\times Y}](\shm\detens\shn) .
\end{align*}
The composite of the above morphisms is isomorphic to
\[
(\OvE_{\widetilde X}\ltens[\D_{\widetilde X}^\sha]\shm^\sha) \cetens (\OvE_{\widetilde Y}\ltens[\D_{\widetilde Y}^\sha]\shn^\sha)
\to (\OvE_{\widetilde {X\times Y}})\ltens[\D_{\widetilde {X\times Y}}^\sha](\shm\detens\shn)^\sha,
\]
i.e.\ to
\[
\drE_{\widetilde X}(\shm^\sha) \cetens \drE_{\widetilde Y}(\shn^\sha)
\to \drE_{\widetilde {X\times Y}}((\shm\detens\shn)^\sha).
\]
By Corollary~\ref{cor:drTXtilde}, it is enough to show that this morphism is an isomorphism.
Then, by Theorem~\ref{thm:normal}, we may assume
$\shm\simeq\she_{X\setminus D_X|X}^\varphi$ and
$\shn\simeq\she_{Y\setminus D_Y|Y}^\psi$
for $\varphi\in\O_X(*D_X)$ and $\psi\in\O_Y(*D_Y)$.
Hence, by Corollary~\ref{cor:drTXtilde}, one has
\[
\drE_{\widetilde X}(\shm^\sha) \simeq \Eepb \varpi_X \drE_X(\shm),
\]
and similarly for $\shm$ replaced by $\shn$ and $\shm\detens\shn$.

On the other hand, Proposition~\ref{prop:inversedual} implies
\eqn
&&\Eepb \varpi_X \drE_X(\shm)\cetens \Eepb \varpi_Y \drE_Y(\shn)
\simeq \Eepb \varpi_{X\times Y} \bl\drE_X(\shm)
\cetens \drE_Y(\shn)\br.\eneqn
We have thus reduced the theorem to the case
$\shm = \she_{X\setminus D_X|X}^\varphi$ and $\shn = \she_{Y\setminus D_Y|Y}^\psi$, and we conclude by using the lemma below.
\end{proof}

\begin{lemma}
Let $X$, $Y$ be complex manifolds, $D_X\subset X$, $D_Y\subset Y$ closed hypersurfaces, and $\varphi\in\O_X(*D_X)$, $\psi\in\O_Y(*D_Y)$.
Then we have an isomorphism in $\BEC[\iCfield]{X\times Y}$
\[
E^\enh_{X\setminus D_X|X}(\varphi) \cetens E^\enh_{Y\setminus D_Y|Y}(\psi) \simeq E^\enh_{(X\setminus D_X)\times(Y\setminus D_Y)|X\times Y}(\varphi+\psi).
\]
\end{lemma}

\begin{proof}
We have
\[
\Edual_X\bigl(E^\enh_{X\setminus D_X|X}(\varphi)[d_X]\bigr)
\simeq \Cfield^\enh_X\ctens \Cfield_{\{t=-\Re\varphi\}}[d_X].
\]
One checks easily that
\eqn
&&
\bl\Cfield^\enh_X\ctens \Cfield_{\{t=-\Re\varphi\}}[d_X]\br \cetens \bl\Cfield^\enh_Y\ctens \Cfield_{\{t=-\Re\psi\}}[d_Y])\br\\
&&\hs{30ex}\simeq\Cfield^\enh_{X\times Y}\ctens \Cfield_{\{t=-\Re(\varphi+\psi)\}}[d_X+d_Y].
\eneqn
Applying $\Edual_{X\times Y}$ and noticing that $\Edual$ commutes with $\cetens$
by Proposition~\ref{prop:exteriodual},  we obtain the desired result.
\end{proof}

\subsection{Duality}

Let $\sht$ be a tensor category with unit object $\mathbbm 1$.
Recall that an adjunction in $\sht$ is a datum $(X_1,X_2,\eta,\varepsilon)$ where $X_1,X_2\in\sht$ and
\[
\mathbbm 1 \To[\eta] X_1\tens X_2, \quad
X_2\tens X_1 \To[\varepsilon]  \mathbbm 1
\]
are morphisms such that the compositions
\begin{align*}
X_2 \simeq X_2\tens \mathbbm 1 \To[\eta] X_2\tens X_1\tens X_2 
\To[\varepsilon]  \mathbbm 1\tens X_2 \simeq X_2, \\
X_1 \simeq \mathbbm 1 \tens X_1 \To[\eta] X_1\tens X_2 \tens X_1 
\To[\varepsilon] X_1 \tens \mathbbm 1  \simeq X_1
\end{align*}
are the identities. In this case, $\Hom[\sht](Z, X_2) \simeq \Hom[\sht](Z\tens X_1,\mathbbm 1)$ functorially in $Z\in\sht$, and one calls $X_2$ a right dual of $X_1$.

\medskip
Let $X$ be a complex manifold.
We shall adapt the construction above  to the categories $\BDC_\hol(\D_X)$ and $\BECRc[\iCfield]X$.
\medskip

Define the maps
\[
p^n_{i_1\cdots i_m}\colon X^n\to X^m \quad\text{by }(x_1,\dots,x_n) \mapsto (x_{i_1},\dots,x_{i_m}). 
\]
In particular, 
$p^1_{11}$ is the diagonal embedding $\delta\colon X\to X\times X$.

Recall that $\shb_{\Delta_X}$ is the holonomic $\D_{X\times X}$-module
associated with the diagonal set $\Delta_X$ (see \eqref{eq:B}).
\begin{lemma}
\label{lem:HomDMM'}
For $\shm,\shm'\in\BDC_\good(\D_X)$ one has
\begin{align}
\label{eq:homdetens1}
\Hom[\BDC(\D_X)](\shm,\shm') &\simeq
\Hom[\BDC(\D_{X^3})](\shb_{\Delta_X}[-d_X]\detens\shm, \shm'\detens\shb_{\Delta_X}[d_X]), \\
\label{eq:homdetens2}
\Hom[\BDC(\D_X)](\shm,\shm') &\simeq
\Hom[\BDC(\D_{X^3})](\shm\detens\shb_{\Delta_X}[-d_X], \shb_{\Delta_X}[d_X]\detens\shm').
\end{align}
\end{lemma}

\begin{proof}
Let us prove only \eqref{eq:homdetens1}.
We have
\begin{align*}
\shb_{\Delta_X}[-d_X]\detens\shm &\simeq \doimv{p^2_{112\sep*}} \dopbv{p_2^{2\sep*}}\shm [-d_X], \\
\shm'\detens\shb_{\Delta_X}[d_X] &\simeq
\doimv{p^2_{122\sep*}}\dopbv{p_1^{2\sep*}}\shm'[d_X].
\end{align*}
By Proposition~\ref{pro:Dadj},
\begin{multline*}
\Hom[\BDC(\D_{X^3})](\doimv{p^2_{112\sep*}}\dopbv{p_2^{2\sep*}}\shm[-d_X], \doimv{p^2_{122\sep*}}\,\dopbv{p_1^{2\sep*}}\shm'[d_X]) \\ \simeq
\Hom[\BDC(\D_{X^2})](\dopbv{p_2^{2\sep*}}\shm, \dopbv{p^{2\sep*}_{112}}\doimv{p^2_{122\sep*}}\,\dopbv{p_1^{2\sep*}}\shm'[d_X]).
\end{multline*}
Since
\begin{equation}
\label{eq:XX3}
\ba{c}
\xymatrix@C=8ex{
X \ar[r]^-{\delta} \ar[d]^-{\delta} & X\times X \ar[d]^-{p^2_{122}} \\
X\times X \ar[r]^-{p^2_{112}}\ar@{}[ur]|-\square & X\times X\times X
}\ea
\end{equation}
is a transversal Cartesian diagram, Proposition~\ref{pro:transCart} gives
\begin{align*}
\dopbv{p^{2\sep*}_{112}}\doimv{p^2_{122\sep*}}\dopbv{p_1^{2\sep*}}\shm'
&\simeq \doim \delta \dopb \delta\dopbv{p^{2\sep*}_1}\shm' \\
&\simeq \doim \delta\shm'.
\end{align*}
Hence
\begin{align*}
\Hom[\BDC(\D_{X^3})](\shb_{\Delta_X}[-d_X]\detens\shm,& \shm'\detens\shb_{\Delta_X}[d_X]) \\
&\simeq \Hom[\BDC(\D_{X^2})](\dopbv{p^{2\sep*}_2}\shm, \doim \delta\shm'[d_X]) \\
&\underset{(*)}\simeq \Hom[\BDC(\D_X)](\shm, \doimv{p^2_{2\sep*}}\doim \delta\shm') \\
&\simeq \Hom[\BDC(\D_X)](\shm, \shm'),
\end{align*}
where $(*)$ follows from Proposition~\ref{pro:Dadj}.
\end{proof}

\begin{definition}\label{def:Dadjun}
An \emph{adjunction} in $\BDC_\hol(\D_X)$  is a datum
$(\shm_1,\shm_2,\eta,\varepsilon)$, where
$\shm_1,\shm_2\in\BDC_\hol(\D_X)$ and
\begin{align*}
\shb_{\Delta_X}[-d_X] \To[\eta] &\shm_1\detens\shm_2, \\
&\shm_2\detens\shm_1 \To[\varepsilon] \shb_{\Delta_X}[d_X]
\end{align*}
are morphisms such that:
\bna
\item
the composition
\[
\shb_{\Delta_X}[-d_X]\detens\shm_1 \To[\eta] \shm_1\detens\shm_2\detens\shm_1 \To[\varepsilon] \shm_1\detens\shb_{\Delta_X}[d_X]
\]
corresponds to $\id_{\shm_1}$ by \eqref{eq:homdetens1},
\item
the composition
\[
\shm_2\detens\shb_{\Delta_X}[-d_X] \To[\eta] \shm_2\detens\shm_1\detens\shm_2 \To[\varepsilon] \shb_{\Delta_X}[d_X]\detens\shm_2
\]
corresponds to $\id_{\shm_2}$ by \eqref{eq:homdetens2}.
\ee
\end{definition}

\begin{proposition}
\bnum
\item For $\shm\in\BDC_\hol(\D_X)$ there is a natural adjunction $(\shm,\ddual_X\shm,\eta,\varepsilon)$, that we denote by $(\shm,\ddual_X\shm)$ for short.
\item
If $(\shm_1,\shm_2,\eta,\varepsilon)$ is an adjunction in $\BDC_\hol(\D_X)$, then $\shm_2\simeq\ddual_X\shm_1$.
\ee
\end{proposition}

\begin{proof}
Since (i) is obvious, we will prove only (ii).

\smallskip\noindent(ii-a)
First, let us show that there is a functorial isomorphism in $\shl\in\BDC_\hol(\D_X)$
\begin{equation}
\label{eq:LM2}
\Hom[\BDC(\D_X)](\shl,\shm_2) \simeq \Hom[\BDC(\D_{X^2})](\shl\detens\shm_1,\shb_{\Delta_X}[d_X]).
\end{equation}
Consider the map sending $\varphi\in\Hom[\BDC(\D_X)](\shl,\shm_2)$ to the morphism $\psi$ given by the composition
\[
\shl\detens\shm_1 \To[\varphi] \shm_2\detens\shm_1 \To[\varepsilon]\shb_{\Delta_X}[d_X].
\]
Consider the map sending  $\psi\in \Hom[\BDC(\D_{X^2})](\shl\detens\shm_1,\shb_{\Delta_X}[d_X])$ to the morphism $\varphi$ which corresponds by \eqref{eq:homdetens2} to the composition
\[
\shl\detens\shb_{\Delta_X}[-d_X] \To[\varepsilon] \shl\detens\shm_1\detens\shm_2 
\To[\psi] \shb_{\Delta_X}[d_X] \detens \shm_2.
\]
Then it is easy to check that these maps are inverse to each other.

\smallskip\noindent(ii-b)
By applying \eqref{eq:LM2} first to the natural adjunction $(\shm_1,\ddual_X\shm_1)$, and then to the adjunction $(\shm_1,\shm_2,\eta,\varepsilon)$, we have
\begin{align*}
\Hom[\BDC(\D_X)](\shl,\ddual_X\shm_1) 
&\simeq \Hom[\BDC(\D_{X^2})](\shl\detens\shm_1,\shb_{\Delta_X}[d_X]) \\
&\simeq  \Hom[\BDC(\D_X)](\shl,\shm_2).
\end{align*}
Hence, by Yoneda, $\shm_2\simeq\ddual_X\shm_1$.
\end{proof}

Now, we have a similar formulation for $\BECRc[\iCfield]X$.
Recall from Notation~\ref{not:Tam} that
\begin{align*}
\Cfield_{\Delta_X}^\enh &= \Cfield_{X\times X}^\enh \tens \opb{\pi_{X\times X}}\Cfield_{\Delta_X}, \\
\omega_{\Delta_X}^\enh &= \Cfield_{X\times X}^\enh \tens \opb{\pi_{X\times X}}\omega_{\Delta_X}.
\end{align*}

\begin{lemma}
For $K,K'\in\BECRc[\iCfield]X$ one has
\begin{align}
\label{eq:homTetens1}
\Hom[{\BEC[\iCfield]X}](K,K') &\simeq
\Hom[{\BEC[\iCfield]{X^3}}](\Cfield_{\Delta_X}^\enh\cetens K, K'\cetens\omega_{\Delta_X}^\enh), \\
\label{eq:homTetens2}
\Hom[{\BEC[\iCfield]X}](K,K') &\simeq
\Hom[{\BEC[\iCfield]{X^3}}](K\cetens\Cfield_{\Delta_X}^\enh, \omega_{\Delta_X}^\enh\cetens K').
\end{align}
\end{lemma}

\begin{proof}
The proof is parallel to that of Lemma~\ref{lem:HomDMM'}.

Let us prove only \eqref{eq:homTetens1}.
We have
\begin{align*}
\Cfield_{\Delta_X}^\enh\cetens K &\simeq \Eeeimv{p^2_{112\sep!!}}\Eopbv{p_2^{2\sep-1}}K, \\
K'\cetens\omega_{\Delta_X}^\enh &\simeq
\Eoimv{p^2_{122\sep*}}\Eepbv{p^{2\sep!}_1}K'.
\end{align*}
Hence
\begin{align*}
\Hom[{\BEC[\iCfield]{X^3}}]&(\Cfield_{\Delta_X}^\enh\cetens K, K'\cetens\omega_{\Delta_X}^\enh) \\
&\simeq \Hom[{\BEC[\iCfield]{X^3}}](\Eeeimv{p^2_{112\sep!!}}\Eopbv{p_2^{2\sep-1}}K, \Eoimv{p^2_{122\sep*}}\Eepbv{p^{2\sep!}_1}K') \\
&\simeq \Hom[{\BEC[\iCfield]{X^2}}](\Eopbv{p_2^{2\sep-1}}K, \Eepbv{p^{2\sep!}_{112}}\Eoimv{p^2_{122\sep*}}\Eepbv{p^{2\sep!}_1}K') \\
&\underset{(*)}\simeq \Hom[{\BEC[\iCfield]{X^2}}](\Eopbv{p_2^{2\sep-1}}K, \Eoim \delta\Eepb \delta \Eepbv{p^{2\sep!}_1}K') \\
&\simeq \Hom[{\BEC[\iCfield]{X^2}}](\Eopbv{p^{2\sep-1}_2}K, \Eoim \delta K') \\
&\simeq \Hom[{\BEC[\iCfield]X}](K, \Eoimv{p^2_{2\sep*}} \Eoim \delta K') \\
&\simeq \Hom[{\BEC[\iCfield]X}](K, K'),
\end{align*}
where $(*)$ follows from the fact that there is a Cartesian diagram \eqref{eq:XX3}.
\end{proof}

\begin{definition}\label{def:Tadjun}
An \emph{adjunction} in $\BECRc[\iCfield]X$ is a datum $(K_1,K_2,\eta,\varepsilon)$,
where $K_1,K_2\in\BECRc[\iCfield]X$ and 
\begin{align*}
\Cfield_{\Delta_X}^\enh \To[\eta] &K_1\cetens K_2, \\
&K_2\cetens K_1 \To[\varepsilon] \omega_{\Delta_X}^\enh,
\end{align*}
are morphisms such that:
\bna
\item 
the composition
\[
\Cfield_{\Delta_X}^\enh\cetens K_1 \To[\eta] K_1\cetens K_2\cetens K_1 \To[\varepsilon] K_1\cetens\omega_{\Delta_X}^\enh
\]
corresponds to $\id_{K_1}$ by \eqref{eq:homTetens1},
\item
the composition
\[
K_2\cetens\Cfield_{\Delta_X}^\enh \To[\eta] K_2\cetens K_1\cetens K_2 \To[\varepsilon] \omega_{\Delta_X}^\enh\cetens K_2
\]
corresponds to $\id_{K_2}$ by \eqref{eq:homTetens2}.
\ee
\end{definition}

Similarly to the case of $\BDC_\hol(\D_X)$, we obtain: 

\begin{proposition}
\bnum
\item For $K\in\BECRc[\iCfield]X$ there is a natural adjunction $(K,\Edual_X K,\eta,\varepsilon)$, that we denote by $(K,\Edual_X K)$ for short.
\item
If $(K_1,K_2,\eta,\varepsilon)$ is an adjunction in $\BECRc[\iCfield]X$, then $K_2\simeq\Edual_X K_1$.
\ee
\end{proposition}

\bigskip
Note that
\[
\drE_{X\times X}(\shb_{\Delta_X}[-d_X]) \simeq \Cfield_{\Delta_X}^\enh,
\quad
\drE_{X\times X}(\shb_{\Delta_X}[d_X]) \simeq \omega_{\Delta_X}^\enh.
\]

\begin{proposition}
Let $(\shm_1,\shm_2,\eta,\varepsilon)$ be an adjunction in $\BDC_\hol(\D_X)$.
Then $(\drE_X(\shm_1),\drE_X(\shm_2),\drE_{X\times X}(\eta),\drE_{X\times X}(\varepsilon))$ is an adjunction in $\BECRc[\iCfield]X$.
\end{proposition}

\begin{proof}
This easily follows from the functorial properties of $\drE_X$.
\end{proof}

In particular, the natural adjunction $(\shm,\ddual_X\shm)$ in $\BDC_\hol(\D_X)$ induces an adjunction $(\drE_X(\shm),\drE_X(\ddual_X\shm))$ in $\BECRc[\iCfield]X$.
We thus get the following result.

\begin{theorem}
\label{thm:drTsolT}
For $\shm\in\BDC_\hol(\D_X)$ there is an isomorphism
\[
\Edual_X\drE_X(\shm) \simeq \drE_X(\ddual_X\shm).
\]
\end{theorem}

Recalling that
\[
\drE_X(\ddual_X\shm) \simeq \solE_X(\shm)[d_X],
\]
we deduce

\begin{corollary}
For $\shm\in\BDC_\hol(\D_X)$ there is an isomorphism
\[
\solE_X(\shm)[d_X] \simeq \Edual_X(\drE_X(\shm)).
\]
\end{corollary}

Hence we obtain the following corollary of Theorems~\ref{thm:Tfunct}, \ref{thm:drTetens} and Proposition~\ref{prop:exteriodual}.

\begin{corollary}
Let $f\colon X\to Y$ be a complex analytic map.
\bnum
\item
For any $\shn\in\BDC_\hol(\D_Y)$ there is an isomorphism in $\BEC[\iCfield]X$
\[
\solE_X(\dopb f\shn) \simeq \Eopb f\solE_Y(\shn).
\]
\item
Let $\shm\in\BDC_\hol(\D_X)\cap \BDC_\good(\D_X)$, and assume that $\supp\shm$ is proper over $Y$. Then there is an isomorphism in $\BEC[\iCfield]Y$
\[
\solE_Y(\doim f\shm)[d_Y] \simeq  \Eoim f\solE_X(\shm)[d_X].
\]
\item
Let $\shm\in\BDC_\hol(\D_X)$ and $\shn\in\BDC_\hol(\D_Y)$. Then there is an isomorphism in $\BEC[\iCfield]{X\times Y}$
\[
\solE_X(\shm) \cetens \solE_Y(\shn) \isoto \solE_{X\times Y}(\shm \detens \shn).
\]
\ee
\end{corollary}

We obtain the following corollary of Theorem~\ref{thm:Tfunct}~(iv).

\begin{corollary}
\label{cor:SolXminusY}
If $\shm\in\BDC_\hol(\D_X)$ and $Y\subset X$ is a closed hypersurface, then
\[
\solE_X(\shm(*Y)) \simeq \opb\pi\Cfield_{X\setminus Y} \tens \solE_X(\shm).
\]
\end{corollary}

We also obtain the following corollary of Lemma~\ref{lem:drTEphi}.

\begin{corollary}
\label{cor:solTEphi}
If $Y\subset X$ is a closed hypersurface and $\varphi\in\O_X(*Y)$, then
\[
\solE_X(\she^\varphi_{X\setminus Y|X}) \simeq \Cfield_X^\enh \ctens \Cfield_{\{t=-\Re\varphi \}}.
\]
\end{corollary}

\subsection{Riemann-Hilbert correspondence}

Let $X$ be a complex manifold.
Recall the hom-functor from Definition~\ref{def:HomT}
\[
\fhom\colon {\BEC[\iCfield]X}^\op \times \BEC[\iCfield]X \to \BDC(\Cfield_X).
\]

\begin{proposition}
\label{pro:shmhomTam}
There is a functorial morphism in $\shm\in\BDC(\D_X)$
\begin{equation}
\label{eq:rec-morph}
\shm \To \fhom(\solE_X(\shm),\OEn_X).
\end{equation}
\end{proposition}

\begin{proof}
Recall the natural morphisms
\[
X\times\R_\infty \To[j_X] X\times\overline\R \To[\overline\pi_X] X.
\]
By the definition of $\solE_X$, one has
\[
\roimv{j_{X\sep*}}\RE \solE_X(\shm) \simeq \rhom[\opb{\overline\pi_X}\D_X](\opb{\overline\pi_X}\shm,\roimv{j_{X\sep*}}\RE \OEn_X).
\]
Hence, there is a morphism in $\BDC(\opb{\overline\pi_X}\D_X)$
\[
\opb{\overline\pi_X}\shm \to \rhom(\roimv{j_{X\sep*}}\RE \solE(\shm),\roimv{j_{X\sep*}}\RE \OEn_X),
\]
which induces by adjunction the desired morphism.
\end{proof}

Consider the diagonal embedding
\[
\delta \colon X \to X\times X.
\]

\begin{lemma}\label{lem:RC2}
For $\shm\in\BDC_\hol(\D_X)$, one has
\[
\fhom(\solE_X(\shm),\OEn_X) \simeq
\fhom\bl\Cfield_{\{t=0\}}, \Eepb\delta(\drE_X(\shm)\cetens\OEn_X)\br[d_X].
\]
In particular, there is a functorial morphism in $\shm\in\BDC_\hol(\D_X)$
\begin{equation}
\label{eq:rec-truemorph}
\shm \to \fhom\bl\Cfield_{\{t=0\}}, \Eepb\delta(\drE_X(\shm)\cetens\OEn_X)\br[d_X].
\end{equation}
\end{lemma}

\begin{proof}
By Proposition~\ref{homText} and Corollary~\ref{cor:CTamOT}, for $K\in\BECRc[\iCfield]X$ one has
\[
\fhom(\Edual_X K,\OEn_X) \simeq \fhom(\Cfield_{\{t=0\}}, \Eepb\delta(K\cetens\OEn_X)).
\]
Moreover, by Theorem~\ref{thm:drTsolT}, one has
\[
\Edual_X(\drE_X(\shm)) \simeq \solE_X(\shm)[d_X].
\]
\end{proof}

We can now state our Riemann-Hilbert correspondence.

\medskip

Consider the quasi-commutative diagram of functors
\begin{equation}
\label{eq:RH}
\ba{c}\xymatrix@C=6.4ex{
\BDC_\hol(\D_X) \ar[r]^-{\solE_X} \ar@{=}[d] & {\BECRc[\iCfield]X}^\op \ar[rrr]^-{\fhom(\ast,\OEn_X)} \ar[d]^{\Edual_X(\ast[d_X])}_\wr &&& \BDC(\D_X) \ar@{=}[d] \\
\BDC_\hol(\D_X) \ar[r]_-{\drE_X} & \BECRc[\iCfield]X \ar[rrr]_(.42){\qquad\fhom(\Cfield_{\{t=0\}}, \Eepb\delta(\ast\cetens\OEn_X))[d_X]} &&& \BDC(\D_X) .
}\ea
\end{equation}

\begin{theorem}\label{th:RH}
\label{thm:RH}
\bnum
\item
For $\shm\in\BDC_\hol(\D_X)$, the morphisms \eqref{eq:rec-morph} and \eqref{eq:rec-truemorph} are isomorphisms. This means in particular that we can reconstruct $\shm$ from $\drE_X(\shm)$.
\item
The functor
\[
\drE_X \colon \BDC_\hol(\D_X) \to \BECRc[\iCfield]X
\]
is fully faithful.
\ee
\end{theorem}

We will prove (i) in Section~\ref{sse:Reconstruction} and (ii) in
Section~\ref{sse:Fullyfaithfulness}.

\medskip

Let us check that the correspondence \eqref{eq:RH} is compatible with the classical Riemann-Hilbert correspondence.

\begin{proposition}
There is a quasi-commutative diagram
\[
\xymatrix@C=5em{
\BDC_\reghol(\D_X) \ar[r]^{\dr_X} \ar[d]
& \BDC_\Cc(\Cfield_X) \ar[rr]^{\rhom(\dual_X(\ast),\Ot_X)[d_X]} \ar[d]^e
&& \BDC_\reghol(\D_X) \ar[d] \\
\BDC_\hol(\D_X) \ar[r]^{\drE_X} 
& \BECRc[\iCfield]X \ar[rr]^{\fhom(\Cfield_{\{t=0\}}, \Eepb\delta(\ast\cetens\OEn_X))[d_X]}
&& \BDC(\D_X) ,
}
\]
where $e(F) = \Cfield_X^\enh \tens \opb\pi F$ is the fully faithful functor of {\rm Proposition~\ref{pro:embed}}.
\end{proposition}

\begin{proof}
(i) The quasi-commutativity of the left hand side square follows from 
Proposition~\ref{prop:regireg}.

\smallskip\noindent(ii)
Denote by $i_0\colon X\to X\times\R_\infty$ the morphism given by $x\mapsto(x,0)$.
The quasi-commutativity of the right hand side square follows from
\begin{align*}
\fhom&(\Edual_X(\Cfield_X^\enh \tens \opb\pi F),\OEn_X) \\
&\underset{(1)}\simeq \fhom(\Cfield_X^\enh \tens \opb\pi \dual_X F,\OEn_X) \\
&\simeq \fhom(\Cfield_X^\enh \ctens (\Cfield_{\{t=0\}} \tens \opb\pi \dual_X F),\OEn_X) \\
&\simeq \fhom(\Cfield_{\{t=0\}} \tens \opb\pi \dual_X F,\OEn_X) \\
&\underset{(2)}\simeq \rhom(\dual_X F,\epb{i_0}\RE\OEn_X) \\
&\underset{(3)}\simeq \rhom(\dual_X F, \Ot_X).
\end{align*}
Here, $(1)$ follows from Corollary~\ref{cor:Tduala}, $(2)$ follows from Lemma~\ref{lem:ai0}, and $(3)$ follows from Lemma~\ref{lem:i0OT} below.
\end{proof}

\begin{lemma}\label{lem:i0OT}
One has
\[
\epb{i_0}\RE\OEn_X \simeq \Ot_X,
\]
where $i_0\colon X\to X\times\R_\infty$ denotes the morphism given by $x\mapsto(x,0)$.
\end{lemma}

\begin{proof}
By Theorem~\ref{thm:OTgeq0}, we have
\[
\RE\OEn_X = \epb i((\she_{\C|\PP}^{-\tau})^\mop\ltens[\D_\PP]\Ot_{X\times\PP})[1].
\]
Let $s\colon\{0\}\to\PP$ be the inclusion, and denote by the same letter the induced map $s\colon X\to X\times\PP$. Then one has $s = i\circ i_0$, and
\begin{align*}
\epb{i_0}\RE\OEn_X
&\simeq \epb s((\she_{\C|\PP}^{-\tau})^\mop\ltens[\D_\PP]\Ot_{X\times\PP})[1] \\
&\underset{(*)}\simeq (\she_{\C|\PP}^{-\tau})^\mop\ltens[\D_\PP] \D_{X\times\PP \from[s] X} \ltens[\D_X] \Ot_X \\
&\simeq (\dopb s \she_{\C|\PP}^{-\tau})\ltens[\C]\Ot_X\simeq\Ot_X.
\end{align*}
Here $(*)$ follows from Theorem~\ref{thm:ifunct}~(i).
\end{proof}

\subsection{Reconstruction}\label{sse:Reconstruction}

By Lemma~\ref{lem:RC2}, 
the following result implies
Theorem~\ref{thm:RH}~(i).

\begin{theorem}
\label{thm:reconstruction}
Let $\shm\in\BDC_\hol(\D_X)$. Then, the morphism 
in {\rm Proposition~\ref{pro:shmhomTam}}
$$\shm \To \fhom(\solE_X(\shm),\OEn_X)$$
is an isomorphism. 
\end{theorem}

\begin{proof}
Consider the statement
\begin{equation}
\label{eq:P-rec}
P_X(\shm) = \text{``\,one has $\shm \isoto \fhom(\solE_X(\shm),\OEn_X)$\,''}.
\end{equation}
Then the hypotheses of Lemma~\ref{lem:redux} are all easily verified, except (e) and (f). We will prove (e) in Lemma~\ref{lem:rec-e} below, and (f) in Lemma~\ref{lem:rec-f} below. Then the theorem follows from Lemma~\ref{lem:redux}.
\end{proof}

\begin{lemma}
\label{lem:rec-e}
Let $f\colon X\to Y$ be a projective morphism and $\shm$ a good holonomic $\D_X$-module.
Under  notation \eqref{eq:P-rec},
if $P_X(\shm)$ is true, then $P_Y(\doim f \shm)$ is true.
\end{lemma}

\begin{proof}
One has
\begin{align*}
\fhom(\solE_Y(\doim f\shm),\OEn_Y)
&\simeq \fhom(\Eoim f\solE_X(\shm)[d_X-d_Y],\OEn_Y) \\
&\simeq \roim f\fhom(\solE_X(\shm),\Eepb f\OEn_Y[d_Y-d_X]) \\
&\simeq \roim f\fhom(\solE_X(\shm),\D_{Y\leftarrow X} \ltens[\D_X] \OEn_X) \\
&\simeq \roim f (\D_{Y\leftarrow X} \ltens[\D_X] \fhom(\solE_X(\shm),\OEn_X)) \\
&\simeq \roim f (\D_{Y\leftarrow X} \ltens[\D_X] \shm) 
= \doim f \shm,
\end{align*}
where the last isomorphism follows from the fact that $P_X(\shm)$ is true.
\end{proof}

We now have to show that Theorem~\ref{thm:reconstruction} holds if $\shm$ is a holonomic $\D_X$-module with a normal form along a normal crossing divisor. We begin with the following result, analogous to \cite[Proposition 7.3]{DAg12}.

\begin{lemma}
\label{lem:rec-f-temp}
Let $Y\subset X$ be a complex analytic hypersurface and $\varphi\in\O_X(*Y)$. 
Set $U=X\setminus Y$.
Then there is an isomorphism in $\BDC(\ind(\D_X\tens[\Cfield]\D_{\overline X}))$
\[
\roim\pi\rihom(\Cfield_{\{t<\Re\varphi\}}[1],\DbT_{X_\R})
\simeq \she^{-\varphi}_{U|X} \ltens[\O_X] \Dbt_{X_\R}.
\]
\end{lemma}

\begin{proof}
Recall that $j\colon X\times\R_\infty \to X\times \PR$ and $\overline\pi\colon X\times\PR \to X$ denote the natural morphisms. Recall that
\[
\DbT_{X_\R} = 
\epb j \rhom[\D_\PP](\she_{\C|\PP}^\tau,\Dbt_{X_\R\times\PR})[1].
\]
Hence $\roim\pi\rihom(\Cfield_{\{t<\Re\varphi\}}[1],\DbT_{X_\R})$ is represented by the complex
\begin{equation}
\label{eq:piihomtleqphi}
\oim{\overline\pi}\ihom(\Cfield_{\{t<\Re\varphi\}},\Dbt_{X_\R\times\PR})
\To[\partial_t -1]
\oim{\overline\pi}\ihom(\Cfield_{\{t<\Re\varphi\}},\Dbt_{X_\R\times\PR})
\end{equation}
with components in degree $0$ and $1$.
We consider them as subanalytic ind-sheaves.

For any relatively compact open subanalytic subset $V\subset X$, \eqref{eq:piihomtleqphi} induces a complex
\begin{equation}
\label{eq:piihomtleqphiV}
\Dbt_{X_\R\times\PR}(\{t<\Re\varphi\}\cap(V\times\R))
\To[\partial_t -1]
\Dbt_{X_\R\times\PR}(\{t<\Re\varphi\}\cap(V\times\R)).
\end{equation}
Hence it is enough to show that \eqref{eq:piihomtleqphiV} is surjective and that its kernel
\[
\ker(\partial_t-1) = \{u(t,x)\in \Dbt_{X_\R\times\PR}(\{t<\Re\varphi\}\cap(V\times\R)) \semicolon \partial_t u = u \}
\]
is given by
\begin{equation}
\label{eq:kertau-1}
\{e^{t-\varphi(x)}v(x)\semicolon v(x)\in\Dbt_{X_\R}(V\cap U) \}.
\end{equation}
Note that the morphism 
\[
\bl\she^{-\varphi}_{U|X} \tens[\O_X] \Dbt_{X_\R}\br(V\cap U) 
\simeq \C\,e^{-\vphi}\tens\Dbt_{X_\R}(V\cap U)  \to \ker(\partial_t-1)
\]
is given by $e^{-\varphi} \tens v(x) \mapsto e^{t-\varphi(x)}v(x)$.

The surjectivity follows from the surjectivity of
\[
\Dbt_{X_\R\times\PR}(V\times\R)
\to[\partial_t -1]
\Dbt_{X_\R\times\PR}(V\times\R).
\]

Neglecting the tempered growth, it is obvious that
\begin{multline*}
\{u(t,x)\in \Db_{X_\R\times\PR}(\{t<\Re\varphi\}\cap(V\times\R)) \semicolon \partial_t u = u \} \\
=\{e^{t-\varphi(x)}v(x)\semicolon v(x)\in\Db_{X_\R}(V\cap U) \}.
\end{multline*}
Hence, \eqref{eq:kertau-1} coincides with $\ker(\partial_t-1)$
by the following sublemma.
\end{proof}

\begin{sublemma}
\label{sublemma}
For $v(x)\in\Db_{X_\R}(V\cap U)$, one has 
\[
e^{t-\varphi(x)}v(x) \in \Dbt_{X_\R\times\PR}(\{t<\Re\varphi\}\cap(V\times\R))
\] 
if and only if 
\[
v(x)\in\Dbt_{X_\R}(V\cap U).
\]
\end{sublemma}

\begin{proof}
Assume $v(x)\in\Dbt_{X_\R}(V\cap U)$.
Since $e^{t-\varphi(x)}$ belongs to $\Cit_{X_\R\times\PR}(\{t<\Re\varphi\})$, one has
\[
e^{t-\varphi(x)}v(x) \in \Dbt_{X_\R\times\PR}(\{t<\Re\varphi\}\cap(V\times\R)).
\]
Conversely, assume $e^{t-\varphi(x)}v(x)\in\Dbt_{X_\R\times\PR}(\{t<\Re\varphi\}\cap(V\times\R))$.
Take a $C^\infty$-function $\chi(t)$ on $\PR$ whose support is contained in 
$\set{t\in\R}{-2<t<-1}$ and such that $\int e^t\chi(t)dt=1$. 
Set 
\eqn
&&W_1=\set{(x,t)\in (U\cap V)\times \R}{t<\Re\varphi(x)}\\
&&Z=\set{(x,t)\in (U\cap V) \times \R}{\Re\varphi(x)-2\le t\le\Re\varphi(x)-1}\\
&&W_2=\bigl((U\cap V)\times \PR\bigr)\setminus Z
\eneqn
Then $W_1$ and $W_2$ are subanalytic open subsets and we have
$(U\cap V) \times \PR=W_1\cup W_2$.
Since
$\chi(t-\Re\varphi(x))e^{\sqrt{-1}\Im\varphi(x)}$ belongs to $\Cit_{X_\R\times\PR}(
W_1)$, we obtain
\[
\chi(t-\Re\varphi(x))e^{t-\Re\varphi(x)}v(x) \in \Dbt_{X_\R\times\PR}(W_1).
\]
Since $\chi(t-\Re\varphi(x))e^{\sqrt{-1}\Im\varphi(x)}v(x)$ vanishes on $W_1\cap W_2$,
there exists $w(x,t) \in \Dbt_{X_\R\times\PR}((U\cap V) \times \PR)$ such that
$$\text{$w(x,t)\vert_{W_1}
= \chi(t-\Re\varphi(x))e^{t-\Re\varphi(x)}v(x)$ and
$w(x,t)\vert_{W_2}=0$.}$$
Hence, $v(x) = \int w(x,t)dt \in \Dbt_{X_\R}(U\cap V)$.
\end{proof}

We deduce the following result, analogous to \cite[Proposition 8.1]{DAg12}.

\begin{proposition}
\label{pro:rec-ephi}
Using the same notations as in {\rm Lemma~\ref{lem:rec-f-temp}}, one has
\[
\roim\pi\rihom(\LE \solE_X(\she^\varphi_{U|X}),\RE \OEn_X)
\simeq \she^\varphi_{U|X} \ltens[\O_X] \Ot_X.
\]
In particular,
\[
\fhom(\solE_X(\she^\varphi_{U|X}),\OEn_X)
\simeq \she^\varphi_{U|X}.
\]
\end{proposition}

\begin{proof}
We have
\begin{align*}
&\LE \solE_X(\she^\varphi_{U|X}) \underset{(*)}\simeq \Cfield_{\{t\gg0\}}\ctens\Cfield_{\{t<-\Re\varphi\}}[1], \\
&\RE \OEn_X \simeq \rhom[\D_{\overline X}](\O_{\overline X},\DbT_{X_\R}), \\
&\Ot_X \simeq \rhom[\D_{\overline X}](\O_{\overline X},\Dbt_{X_\R}),
\end{align*}
where $(*)$ follows from Corollary~\ref{cor:solTEphi}.

The statement then follows by applying the functor $\rhom[\D_{\overline X}](\O_{\overline X},\ast)$ to the isomorphism of Lemma~\ref{lem:rec-f-temp}.
\end{proof}

Now it remains to prove the following result, required in the proof of Theorem~\ref{thm:reconstruction}.

\begin{lemma}
\label{lem:rec-f}
Let $\shm$ be a holonomic $\D_X$-module with a normal form along a normal crossing divisor. Then
\[
\shm \isoto \fhom(\solE_X(\shm),\OEn_X).
\]
\end{lemma}

\begin{proof}
Let $D\subset X$ be a normal crossing divisor and $\shm$ a holonomic $\D_X$-module with a normal form along $D$. Set $U=X\setminus D$.
We keep the same notations as in Section~\ref{se:realblow} 
such as  $\varpi\colon\widetilde X\to X$ and $\D^\sha_{\widetilde X}$. 
We also consider the natural morphisms
\[
X \times \R_\infty\from[\tilde\varpi] \widetilde X \times \R_\infty \to[\pi_{\widetilde X}] \widetilde X.
\]
For a $\D^\sha_{\widetilde X}$-module $\shl$, we set
\[
\solE_{\widetilde X}(\shl) = \rhom[\D^\sha_{\widetilde X}](\shl, \OEn_{\widetilde X}) \in \BEC[\iCfield]{\widetilde X}.
\]
Similarly to the construction of \eqref{eq:rec-morph}, 
we have a morphism
\begin{equation}
\label{eq:Topbmorph}
\shl
\to \fhom(\solE_{\widetilde X}(\shl),\OEn_{\widetilde X}).
\end{equation}

\smallskip\noindent(i)
We shall first show
\begin{equation}
\label{eq:Topba}
\opb{\tilde\varpi}\opb\pi\Cfield_U \tens \solE_{\widetilde X}(\shm^\sha) 
\simeq \Eopb\varpi\solE_X(\shm).
\end{equation}
Since
\begin{equation}
\label{eq:tempOT}
\OEn_{\widetilde X} \simeq \Eepb\varpi\rihom(\opb\pi\Cfield_U,\OEn_X),
\end{equation}
we have
\begin{align*}
\solE_{\widetilde X}(\shm^\sha)  
&= \rhom[\shd_{\widetilde X}^\sha](\shm^\sha,\OEn_{\widetilde X}) \\
&\simeq \rhom[\opb\varpi\D_X](\opb\varpi\shm, \Eepb\varpi\rihom(\opb\pi\Cfield_U,\OEn_X)) \\
&\simeq \Eepb\varpi \rihom(\opb\pi\Cfield_U,\rhom[\D_X](\shm,\OEn_X)) \\
&\simeq \rihom(\opb{\tilde\varpi}\opb\pi\Cfield_U,\Eepb\varpi \solE_X(\shm)) \\
&\simeq \rihom(\opb{\tilde\varpi}\opb\pi\Cfield_U,\Eopb\varpi \solE_X(\shm)),
\end{align*}
where the last isomorphism follows from the fact that $\varpi$ is an isomorphism over $U$. 
Hence we obtain
\begin{align*}
\opb{\tilde\varpi}\opb\pi\Cfield_U \tens \solE_{\widetilde X}(\shm^\sha) 
&\simeq \opb{\tilde\varpi}\opb\pi\Cfield_U \tens  \Eopb\varpi \solE_X(\shm) \\
&\simeq \Eopb\varpi (\opb\pi\Cfield_U \tens   \solE_X(\shm)) \\
&\simeq \Eopb\varpi\solE_X(\shm).
\end{align*}
Here, the last isomorphism follows from Corollary~\ref{cor:SolXminusY}.

\medskip
\noindent(ii)
Next, we shall show
\begin{equation}
\label{eq:Topbb}
\shm^\sha 
\isoto \fhom(\solE_{\widetilde X}(\shm^\sha),\OEn_{\widetilde X}).
\end{equation}
Since the question is local, we can assume $\shm=\she_{U|X}^\varphi$ for $\varphi\in\O_X(*D)$.
Then we have
\begin{align*}
\rihom&(\LE \solE_{\widetilde X}(\shm^\sha),\RE\OEn_{\widetilde X}) \\
& \simeq \rihom(\LE \solE_{\widetilde X}(\shm^\sha),\rihom(\opb{\tilde\varpi}\opb\pi\Cfield_U,\epb{\tilde\varpi}\RE\OEn_X)) \\
& \simeq \rihom(\LE \solE_{\widetilde X}(\shm^\sha) \tens \opb{\tilde\varpi}\opb\pi\Cfield_U,\epb{\tilde\varpi}\RE\OEn_X) \\
& \underset{(*)}\simeq \rihom(\opb{\tilde\varpi}\LE \solE_X(\shm),\epb{\tilde\varpi}\RE \OEn_X) \\
& \simeq \epb{\tilde\varpi}\rihom(\LE \solE_X(\shm),\RE \OEn_X),
\end{align*}
where $(*)$ follows \eqref{eq:Topba}. 
Hence
\begin{align*}
\roimv{\pi_{\widetilde X\sep*}}\rihom&(\LE \solE_{\widetilde X}(\shm^\sha),\RE \OEn_{\widetilde X}) \\
&\simeq \roimv{\pi_{\widetilde X\sep*}}\epb{\tilde\varpi}\rihom(\LE\solE_X(\shm),\RE \OEn_X) \\
&\simeq \epb\varpi\roim\pi\rihom(\LE\solE_X(\shm),\RE \OEn_X) \\
&\simeq \epb\varpi(\shm \ltens[\O_X] \Ot_X),
\end{align*}
where the last isomorphism follows from Proposition~\ref{pro:rec-ephi}.
We have
\begin{align*}
\epb\varpi(\shm \ltens[\O_X] \Ot_X)
&\simeq \epb\varpi(\shm \ltens[\O_X] \rihom(\Cfield_U,\Ot_X)) \\
&\simeq \opb\varpi\shm \ltens[\opb\varpi\O_X] \Ot_{\widetilde X}.
\end{align*}
Hence, by applying $\alpha_{\widetilde X}$, we obtain
\begin{equation}
\label{eq:Topbd}
\fhom(\solE_{\widetilde X}(\shm^\sha),\OEn_{\widetilde X})
\simeq \alpha_{\widetilde X}(\opb\varpi\shm \ltens[\opb\varpi\O_X] \Ot_{\widetilde X}) 
\simeq \shm^\sha
\end{equation}
by Proposition~\ref{pro:AOt}.

\smallskip\noindent(iii)
Now we shall prove the statement
\[
\shm \isoto \fhom(\solE_X(\shm),\OEn_X).
\]
By Proposition~\ref{pro:AOt}, we have
\begin{align*}
\roim\varpi\shm^\sha
&\simeq \alpha_X\roim\varpi(\Ot_{\widetilde X} \ltens[\opb\varpi\O_X] \opb\varpi\shm) \\
&\simeq (\alpha_X\rihom(\Cfield_U,\Ot_X)) \ltens[\O_X] \shm \\
&\simeq \O_X(*D) \ltens[\O_X] \shm \simeq \shm.
\end{align*}
We have
\begin{align*}
\fhom&(\solE_{\widetilde X}(\shm^\sha),\OEn_{\widetilde X}) \\
& \simeq
\fhom(\solE_{\widetilde X}(\shm^\sha),\rihom(\opb{\tilde\varpi}\opb\pi\Cfield_U,\Eepb\varpi\OEn_X)) \\
& \simeq
\fhom(\solE_{\widetilde X}(\shm^\sha) \tens \opb{\tilde\varpi}\opb\pi\Cfield_U,\Eepb\varpi\OEn_X) \\
& \simeq
\fhom(\Eopb\varpi\solE_X(\shm),\Eepb\varpi\OEn_X),
\end{align*}
where the last isomorphism follows from \eqref{eq:Topba}.
It follows that
\[
\roim\varpi\fhom(\solE_{\widetilde X}(\shm^\sha),\OEn_{\widetilde X})
\simeq \fhom(\Eeeim\varpi\Eopb\varpi\solE_X(\shm),\OEn_X)
\]
by Lemma~\ref{homepb}.
Hence, applying $\roim\varpi$ to \eqref{eq:Topbb}, we get
\begin{equation}
\label{eq:Topbe}
\shm \isoto \fhom(\Eeeim\varpi\Eopb\varpi\solE_X(\shm),\OEn_X).
\end{equation}
By Corollary~\ref{cor:SolXminusY}, we have
$\solE_X(\shm) \simeq \opb\pi\Cfield_U\tens\solE_X(\shm)$.
Moreover, $\varpi\colon\widetilde X\to X$ is an isomorphism over $U$. Hence we have
\begin{align*}
\Eeeim\varpi\Eopb\varpi\solE_X(\shm)
&\simeq \Eeeim\varpi\Eopb\varpi(\opb\pi\Cfield_U\tens\solE_X(\shm)) \\
&\simeq \opb\pi\Cfield_U\tens\Eeeim\varpi\Eopb\varpi\solE_X(\shm) \\
&\simeq \opb\pi\Cfield_U\tens\solE_X(\shm) \\
&\simeq \solE_X(\shm) .
\end{align*}
We thus obtain the desired result.
\end{proof}

Thus the proof of Theorem~\ref{thm:reconstruction} is complete.

As a consequence of Theorem~\ref{thm:reconstruction}, we get the following result (which is also a consequence of Lemmas~\ref{lem:i0OT} and \ref{lem:ai0}).

\begin{corollary}
\label{cor:OOT}
There is an isomorphism in $\BDC(\D_X)$
\[
\O_X 
\simeq \fhom( \Cfield_X^\enh , \OEn_X).
\]
\end{corollary}

\subsection{Fully faithfulness}\label{sse:Fullyfaithfulness}

Let us now show that the functor $\drE_X$ is fully faithful.

\begin{theorem}
For $\shm,\shn\in\BDC_\hol(\D_X)$, there is an isomorphism
\[
\rhom[\D_X](\shm,\shn) \isoto \fhom(\drE_X(\shm), \drE_X(\shn)).
\]
In particular, the functor
\[
\drE_X\colon \BDC_\hol(\D_X) \to \BEC[\iCfield]X
\]
is fully faithful.
\end{theorem}

\Proof
By Theorem~\ref{thm:drTsolT} and Proposition~\ref{prop:homdual} (iv),
we have
$$\fhom(\drE_{X}\shm,\drE_{X}\shn)
\simeq\fhom(\solE_X\shn,\solE_X\shm).$$
Then, we have
\eqn
\fhom(\solE_X\shn,\solE_X\shm)
&\simeq&\fhom\bl\solE_X\shn,\rhom[\D_X](\shm, \OEn_X)\br\\
&\simeq&\rhom[\D_X]\bl\shm,\fhom\bl\solE_X\shn,\OEn_X)\br\\
&\simeq&\rhom[\D_X](\shm,\shn).
\eneqn
Here the last isomorphism follows from Theorem~\ref{thm:reconstruction}.
\QED

\subsection{Stokes phenomenon}\label{sse:Stokes}

Liner ordinary differential equations with irregular singularities are subjected to the Stokes phenomenon (see for example \cite[Section 15]{Was65} or \cite[\S9.7]{Hil76}). Following \cite[\S7]{DK12}, we show here through an example how, in our setting, the Stokes phenomenon arises in a purely topological fashion.

\medskip

Let $X$ be an open disc in $\C$ centered at $0$.
(We will shrink $X$ if necessary.) 
Consider the real blow-up $\varpi\colon \widetilde X\to X$ of $X$ along $\{0\}$, and recall that $\widetilde X^0 = \opb\varpi(0)$ is the set of normal directions to $0$ in $X$.

Let $\varphi,\psi\in\O_X(*0)$, and assume that $\psi-\varphi$ has an effective pole at $0$.
For $U = X\setminus \{0\}$, set
\[
\shm_0 \defeq \she^{\varphi}_{U|X} \dsum \she^{\psi}_{U|X}.
\]
Let $\shm$ be a holonomic $\D_X$-module such that
$\shm \simeq \shm(*0)$,
$\ss(\shm)=\{0\}$, and one has
\begin{equation}
\label{eq:isoStokes}
(\shm^\sha)|_\theta \simeq
(\shm_0^\sha)|_\theta\quad\text{for any $\theta\in \widetilde X^0$.}
\end{equation}
Note that $\shm$ has a normal form along $\{0\}$.

The Stokes curves are the real analytic arcs $\ell_i\subset X$ defined by
\[
\{\Re(\psi-\varphi) = 0\} = \bigsqcup\nolimits_{i\in I}\ell_i.
\]
(Here we possibly shrink $X$ to avoid crossings of the $\ell_i$'s and to ensure that they admit $|z|$ as parameter.)
Since $\she^{\varphi}_{U|X} \simeq \she^{\varphi+\varphi_0}_{U|X}$ for $\varphi_0\in\O_X$, the Stokes curves depend on the choice of $\varphi$ and $\psi$.

The Stokes lines $L_i$, defined as the half-lines tangent to $\ell_i$ at $0$, are independent of the choice of $\varphi$ and $\psi$. 

The Stokes multipliers of $\shm$ describe how the isomorphism \eqref{eq:isoStokes} changes when $\theta$ crosses a Stokes line.

Let us show how these data are topologically encoded in $\drE_X(\shm)$.

Set
\begin{align*}
F &\defeq \C_X^\enh \ctens \C_{\{t=\Re\varphi\}} \simeq \indlim[a\rightarrow+\infty]\C_{\{t-\Re\varphi \geq a \}} , \\
G &\defeq \C_X^\enh \ctens \C_{\{t=\Re\psi\}} \simeq \indlim[a\rightarrow+\infty]\C_{\{t-\Re\psi \geq a \}} .
\end{align*}
By Corollary~\ref{cor:drTXtilde}, Lemma~\ref{lem:drTEphi} and \eqref{eq:isoStokes},
\begin{equation}
\label{eq:drH}
\drE_X(\shm) \simeq \rihom(\opb\pi\C_U ,H)[1],
\end{equation}
where $H$ is an enhanced ind-sheaf such that $H \simeq \opb\pi\Cfield_U \tens H$ and
\[
\opb\pi\C_{S} \tens H \simeq \opb\pi\C_{S} \tens (F\dsum G)
\]
for any sufficiently small open sector $S$.

Let $\fb^\pm$ be the vector space of upper/lower triangular matrices in $\operatorname{M}_2(\C)$, and let $\ft=\fb^+\cap \fb^-$ be the vector space of diagonal matrices.
Using Proposition~\ref{pro:pro:HomStab} one gets

\begin{lemma}\label{lem:FcGc}
Let $S$ be an open sector.
\bnum 
\item
If $S\subset \{\pm\Re(\varphi-\psi)>0\}$, then
\[
\Endo[{\BEC[\iCfield]X}](\opb\pi\C_{S} \tens (F\dsum G)) \simeq \fb^\pm.
\]
\item
If $S\supset L_i$ for some $i\in I$ and $S\cap L_j=\emptyset$ for $i\neq j$, then
\[
\Endo[{\BEC[\iCfield]X}](\opb\pi\C_{S} \tens (F\dsum G)) \simeq \ft.
\]
\ee
\end{lemma}

This proves that the Stokes lines are encoded in $H$.
Let us show how to recover the Stokes multipliers of $\shm$ as gluing data for $H$.

Let $S_i$ be an open sector which contains $L_i$ and is disjoint from $L_j$ for $i\neq j$. We choose $S_i$ so that $\Union\nolimits_{i\in I} S_i = U$.

Then for each $i\in I$, there is an isomorphism
\[
\alpha_i\colon \opb\pi\C_{S_i} \tens H \isoto \opb\pi\C_{S_i} \tens (F\dsum G).
\]
Note that $\alpha_i$ is unique only up to left multiplication by elements of $\mathfrak{t} \cap \operatorname{GL}_2(\C)$ by Lemma~\ref{lem:FcGc}~(ii).

Take a cyclic ordering of $I$ such that the Stokes lines get ordered counterclockwise. 

Since $\{S_i\}_{i\in I}$ is an open cover of $U$, the enhanced ind-sheaf $H$ is reconstructed from $F\dsum G$ via the gluing data given by the Stokes multipliers
\[
A_i = \alpha_{i+1}\alpha_i^{-1}|_{\opb \pi(S_i\cap S_{i+1})},
\]
where $A_i \in \fb^\pm \cap \operatorname{GL}_2(\C)$ if $\pm\Re(\varphi-\psi)>0$ on $S_i\cap S_{i+1}$  by Lemma~\ref{lem:FcGc}~(i).

Note that, replacing $A_i$ with $A_i' = \gamma_{i+1} A_i \gamma_i^{-1}$ for $\gamma_i \in \mathfrak{t} \cap \operatorname{GL}_2(\C)$, one gets an enhanced ind-sheaf isomorphic to $H$.

\end{document}